\documentclass[12pt]{amsart}
\pdfoutput=1
1
\usepackage{amsmath}
\usepackage{mathrsfs, amssymb, upgreek}
\usepackage{longtable,boldline}
\usepackage{amsthm}
\usepackage[all,pdf,cmtip]{xy}
\usepackage{mathtools}
\usepackage{hhline, multirow}
\usepackage{array}
\usepackage{url}
\usepackage{graphicx}
\usepackage{geometry}
\usepackage{float}

\usepackage{eucal}
\usepackage{enumitem}
\usepackage{comment} 

\usepackage{tikz}
\usetikzlibrary{shapes.geometric}
\usetikzlibrary{intersections}
\newcommand{\xdasharrow}[2][->]{
\tikz[baseline=-\the\dimexpr\fontdimen22\textfont2\relax]{
\node[anchor=south,font=\scriptsize, inner ysep=1.5pt,outer xsep=2.2pt](x){#2};
\draw[shorten <=3.4pt,shorten >=3.4pt,dashed,#1](x.south west)--(x.south east);
}
}

\newcommand{\moplus}{\mathop{\textstyle{\bigoplus}}\limits}

\usepackage{hyperref} 

\usepackage[normalem]{ulem}

\usepackage{colortbl}

\DeclareMathSymbol{\shortminus}{\mathbin}{AMSa}{"39}

\newcommand{\CC}{\mathbb{C}}
\newcommand{\ZZ}{\mathbb{Z}}
\newcommand{\PP}{\mathbb{P}}
\newcommand{\QQ}{\mathbb{Q}}

\newcommand{\type}[1]{$\mathrm{#1}$}
\newcommand{\ntype}[3]{\mbox{$\mathbf{#2#3}$}}
\newcommand{\typeb}{$(\mathrm{B}\raisebox{-.02ex}{\scriptsize{$\tfrac{1}{2}$}})$}
\newcommand{\typemm}[2]{\mbox{$({#1}\textrm{-}{#2})$}}

\newcommand{\io}{\upiota}

\newcommand{\whom}{\widehat\Omega}

\newcommand{\dd}{\operatorname{d}}

\newcommand{\p}{\operatorname{p}_{\mathrm{a}}}

\newcommand{\h}{\mathbf{h}}

\newcommand{\rX}{\mathrm{X}}
\newcommand{\rY}{\mathrm{Y}}

\newcommand{\cD}{\mathcal{D}}
\newcommand{\cG}{\mathcal{G}}
\newcommand{\cN}{\mathscr{N}}
\newcommand{\cK}{\mathcal{K}}

\newcommand{\balpha}{{\boldsymbol{\upalpha}}}

\newcommand{\Eff}{\overline{\operatorname{Eff}}}
\newcommand{\Nef}{\operatorname{Nef}}
\newcommand{\Mov}{\overline{\operatorname{Mov}}}
\newcommand{\NE}{\overline{\operatorname{NE}}}

\newcommand{\SO}{\operatorname{SO}}
\newcommand{\Sp}{\operatorname{Sp}}
\newcommand{\PGL}{\operatorname{PGL}}

\newcommand{\Sing}{\operatorname{Sing}}

\newcommand{\Pic}{\operatorname{Pic}}
\newcommand{\Cl}{\operatorname{Cl}}
\newcommand{\Gr}{\operatorname{Gr}}
\newcommand{\CGr}{\operatorname{CGr}}
\newcommand{\OGr}{\operatorname{OGr}}
\newcommand{\LGr}{\operatorname{LGr}}
\newcommand{\GtGr}{\operatorname{G_2Gr}}
\newcommand{\Bs}{\operatorname{Bs}}

\newcommand{\can}{{\operatorname{can}}}

\newcommand{\sm}{{\operatorname{sm}}}

\newcommand{\xref}[1]{\textup{\ref{#1}}}

\newcommand{\hB}{\widehat{B}}

\newcommand{\hS}{\widehat{S}}

\newcommand{\hX}{\widehat{X}}
\newcommand{\hY}{\widehat{Y}}

\newcommand{\tE}{\tilde{E}}
\newcommand{\tM}{\widetilde{M}}
\newcommand{\tS}{\tilde{S}}
\newcommand{\tX}{\tilde{X}}
\newcommand{\tY}{\tilde{Y}}

\newcommand{\tD}{\tilde{D}}
\newcommand{\tF}{\tilde{F}}

\newcommand{\cE}{\mathscr{E}}

\newcommand{\cF}{\mathscr{F}}
\newcommand{\cH}{\mathscr{H}}
\newcommand{\cI}{\mathscr{I}}
\newcommand{\cL}{\mathscr{L}}
\newcommand{\cO}{\mathscr{O}}

\newcommand{\cR}{\mathscr{R}}
\newcommand{\cS}{\mathcal{S}}
\newcommand{\cT}{\mathscr{T}}
\newcommand{\cU}{\mathscr{U}}
\newcommand{\cX}{\mathscr{X}}
\newcommand{\cZ}{\mathscr{Z}}

\DeclareMathOperator{\Bl}{Bl}
\DeclareMathOperator{\Fl}{Fl}
\DeclareMathOperator{\codim}{codim}

\DeclareMathOperator{\Proj}{Proj}
\DeclareMathOperator{\rank}{rank}
\DeclareMathOperator{\Hom}{Hom}
\DeclareMathOperator{\Ext}{Ext}

\DeclareMathOperator{\Ker}{Ker}

\DeclareMathOperator{\Sym}{Sym}
\newcommand{\rc}{\mathrm{c}}

\newcommand{\rM}{\mathrm{M}}
\newcommand{\rT}{\mathrm{T}}
\newcommand{\tZ}{\widetilde{Z}}

\newcommand{\rC}{\mathrm{C}}

\newcommand{\g}{\mathrm{g}}
\newcommand{\hh}{\operatorname{h}}
\newcommand{\xrightiso}[1]{ \xrightarrow[{\ \raisebox{0.5ex}[0ex][0ex]{$\sim$}\ }]{#1} }
\newcommand{\xleftiso}[1]{ \xleftarrow[{\ \raisebox{0.5ex}[0ex][0ex]{$\sim$}\ }]{#1} }

\newcommand{\Blw}[1]{\operatorname{Bl}_{[#1]}}

\makeatletter
\@addtoreset{equation}{section}
\makeatother

\newenvironment{aenumerate}{\begin{enumerate}[label={\textup{(\alph*)}}]}{\end{enumerate}}
\newenvironment{arenumerate}{\begin{enumerate}[label={\textup{(\arabic*)}}]}{\end{enumerate}} 
\newenvironment{thmenumerate}{\begin{enumerate}[wide, label={\textup{(\alph*)}}]}{\end{enumerate}}

\theoremstyle{plain}
\newtheorem{theorem}{Theorem}[section]
\newtheorem{lemma}[theorem]{Lemma}
\newtheorem{proposition}[theorem]{Proposition}
\newtheorem{corollary}[theorem]{Corollary}
\newtheorem*{claim*}{Claim}

\theoremstyle{definition}
\newtheorem{definition}[theorem]{Definition}

\newtheorem*{definition*}{Definition}

\newtheorem*{notation*}{Notation}

\newtheorem{remark}[theorem]{Remark}

\title{One-nodal Fano threefolds with Picard number one}
\author{Alexander Kuznetsov}
\address{
Steklov Mathematical Institute of Russian Academy of Sciences, Moscow, Russia
\newline\indent
Laboratory of Algebraic Geometry, HSE, 6 Usacheva str., Moscow, Russia}
\email{akuznet@mi-ras.ru} 

\author{Yuri Prokhorov}
\address{
Steklov Mathematical Institute of Russian Academy of Sciences, Moscow, Russia
\newline\indent
Department of Algebra, Faculty of Mathematics, Moscow State
University, Moscow, 119 991, Russia
\newline\indent
Laboratory of Algebraic Geometry, HSE, 6 Usacheva str., Moscow, Russia}
\email{prokhoro@mi-ras.ru} 

\thanks{We were partially supported by the HSE University Basic Research Program.}

\begin{document}

\begin{abstract}
We classify all $1$-nodal degenerations of smooth Fano threefolds with Picard number~$1$ (both nonfactorial and factorial) 
and describe their geometry.
In particular, we describe a relation between such degenerations and smooth Fano threefolds of higher Picard rank 
and with unprojections of complete intersection varieties.
\end{abstract}

\maketitle
\tableofcontents

\section{Introduction}

The goal of this paper is to provide a systematic and precise classification of {\sf $1$-nodal Fano threefolds of Picard number~$1$},
i.e., projective threefolds~$X$ with Picard group~$\Pic(X) \cong \ZZ$, ample anticanonical class~$-K_X$, and a single ordinary double point. 

Many partial classification results of this sort can be found in the literature, see, e.g.,
\cite{Jahnke2006},
\cite{Jahnke-Peternell-Radloff-II},
\cite{Cutrone-Marshburn},
\cite{P:V22},
\cite{Prokhorov2017},
\cite{P:ratF-1},
\cite{P:ratFano2:22},
\cite{Takeuchi:DP},
\cite{CKMS}.
However, none of these papers contains a full and detailed classification, with an accurate description of all geometric aspects.
We are filling this gap, at the same time trying to make our exposition self-contained and rigorous
and classification structured.
\medskip 

There are a few reasons to study $1$-nodal Fano threefolds, 
or more generally, Fano threefolds with terminal Gorenstein singularities:
\begin{itemize}[wide]
\item 
According to the Minimal Model Program, Fano varieties with terminal Gorenstein singularities
provide building blocks for arbitrary rationally connected varieties,
and $1$-nodal Fano threefolds is the simplest non-smooth class of such varieties.
\item 
Namikawa proved in~\cite{Na97} that any Fano threefold with terminal Gorenstein singularities 
admits a deformation to a nodal Fano threefold, and then a smoothing.
Therefore, $1$-nodal Fano threefolds correspond to general points 
of the boundary divisors of the moduli stacks of Fano threefolds 
(where the interior corresponds to smooth threefolds).
\item 
Derived categories of $1$-nodal Fano threefolds can be used to relate 
the interesting parts of derived categories of smooth Fano threefolds of different deformation types,
see~\cite{KS23}.
\end{itemize}

All these reasons already provide a good motivation for our study.
But the most important reason is that the resulting classification looks beautiful
and provides a new perspective on the structure of Fano classification problems in general. 

It is worth here pointing out an important subtlety.
Although formally (or even \'etale locally) all ordinary double points look the same,
there is an important distinction in the global geometry of $1$-nodal varieties.
Indeed, if~$(X,x_0)$ is a threefold with a single ordinary double point~$x_0 \in X$,
there is a natural left-exact sequence
\begin{equation}
\label{eq:pic-cl}
0 \longrightarrow \Pic(X) \longrightarrow \Cl(X) \xrightarrow{\ r_{x_0}\ } \ZZ,
\end{equation}
where the last term can be identified with the local class group~$\Cl(X,x_0)$
and the map~$r_{x_0}$ can be defined in terms of the blowup~$\hX \coloneqq \Bl_{x_0}(X)$ 
and the embedding~$\iota_E \colon E \to \hX$ of its exceptional divisor
as the morphism
\begin{equation*}
\Cl(X) = \Pic(\hX) / \ZZ \cdot E \xrightarrow{\ \iota_E^*\ } \Pic(E) / \ZZ \cdot E\vert_E \cong \ZZ.
\end{equation*}

On the one hand, if the morphism~$r_{x_0} \colon \Cl(X) \to \ZZ$ is zero, 
every Weil divisor on~$X$ is Cartier, i.e., we have an equality~$\Cl(X) = \Pic(X)$.
In this case, we say that~$X$ is {\sf factorial}.

On the other hand, if the morphism~$r_{x_0} \colon \Cl(X) \to \ZZ$ is nonzero, 
its image is isomorphic to~$\ZZ$, hence~$\Cl(X) \cong \Pic(X) \oplus \ZZ$.
In this case, we say that~$X$ is {\sf nonfactorial}.

As we will see, most Fano threefolds admit both factorial and nonfactorial $1$-nodal degenerations,
in some cases even several deformation types of nonfactorial degenerations:
\begin{table}[H]
\begin{equation*}
\begin{array}{c!{\vrule width 0.1em}c|c|c|c|c|c|c|c|c|c!{\vrule width 0.1em}c|c|c|c|c!{\vrule width 0.1em}c!{\vrule width 0.1em}c}
\io(X) & \multicolumn{10}{c!{\vrule width 0.1em}}{1} & \multicolumn{5}{c!{\vrule width 0.1em}}{2} & 3 & 4
\\
\hline
\g(X) & 2 & 3 & 4 & 5 & 6 & 7 & 8 & 9 & 10 & 12 & 5 & 9 & 13 & 17 & 21 & 28 & 33
\\
\noalign{\hrule height 0.1em}
\text{\hphantom{non}factorial degenerations} & 1 & 1 & 1 & 1 & 1 & 1 & 1 & 1 & 1 & \cellcolor{gray!0}0 & 
1 & 1 & 1 & 1 & \cellcolor{gray!0}0 & \cellcolor{gray!0}0 & \cellcolor{gray!0}0
\\
\text{nonfactorial degenerations} & 1 & \cellcolor{gray!0}0 & 1 & 1 & 1 & 1 & 2 & 2 & 2 & 4 & 
\cellcolor{gray!0}0 & \cellcolor{gray!0}0 & \cellcolor{gray!0}0 & \cellcolor{gray!0}0 & 1 & 1 & \cellcolor{gray!0}0
\end{array}
\end{equation*}
\caption{Number of degeneration types of $1$-nodal Fano threefolds}
\label{table:numbers}
\end{table} 
\noindent
Here~$\io(X)$ and~$\g(X)$ denote the index and genus of a Fano threefold~$X$ 
(their definition is recalled below in~\eqref{eq:def-io-g}).
Note in particular, that Fano threefolds~$X$ with trivial Hodge number~$\h^{1,2}(X) = 0$ 
have no factorial degenerations, see Lemma~\ref{lem:factorial-criteria} for a simple explanation. 

Geometry of factorial and nonfactorial $1$-nodal threefolds is in fact quite different.
On the one hand, the description of factorial $1$-nodal threefolds is quite similar to the description in the smooth case --- 
as we will see in Theorem~\ref{thm:intro-factorial-ci}
they are complete intersections of the same type in the same weighted projective spaces or homogeneous Mukai varieties.
Nonfactorial $1$-nodal threefolds, on a contrary, are quite different ---
some of them are still complete intersections, but in different higher-dimensional varieties, see~\cite{Muk22},
but for the most of them no complete intersection description is available.
However, we will show that in many cases they can be represented 
as \emph{unprojections} of complete intersections in simpler varieties, see Theorem~\ref{thm:intro-nf-ci}.
For these reasons we state the classification results separately.

Before we start presenting our results, we recall
the most important invariants of Fano threefolds.
For an algebraic variety~$X$ we denote by
\begin{equation*}
\uprho(X) \coloneqq \rank \Pic(X) 
\end{equation*}
the rank of its Picard group.
Further, if~$X$ is a Fano threefold
with canonical Gorenstein singularities,
we denote by~$\io(X)$, $\dd(X)$, and~$\g(X)$ its {\sf Fano index}, {\sf degree}, and {\sf genus},
defined by
\begin{equation}
\label{eq:def-io-g}
\begin{aligned}
\io(X) &\coloneqq \max \{ i \mid K_X \in i \cdot \Pic(X) \},\\
\dd(X) &\coloneqq \frac{(-K_X)^{3}}{\io(X)^{3}},\\
\g(X) &\coloneqq \tfrac12(-K_X)^{3} + 1.
\end{aligned}
\end{equation}
Note that the degree~$\dd(X)$ and genus~$\g(X)$ carry essentially the same information about~$X$.
However, it is traditional to use degree for Fano threefolds~$X$ with~$\io(X) \ge 2$ 
and genus in the case of threefolds with~$\io(X) = 1$.

We also need to remind some facts about the classification of smooth Fano threefolds
due to Fano, Iskovskikh, Mori, and Mukai, see~\cite{IP99} for an overview.

When~$\uprho(X) = 1$ there are~$17$ deformation families of smooth Fano threefolds, 
uniquely characterized by the index and genus (or degree).
These threefolds are:
\begin{itemize}[wide]
\item 
If~$\io(X) = 4$ then~$\dd(X) = 1$ (hence~$\g(X) = 33$) and~$X \cong \PP^3$, the projective $3$-space.
\item 
If~$\io(X) = 3$ then~$\dd(X) = 2$ (hence~$\g(X) = 28$) and~$X \cong Q^3$, the smooth quadric.
\item 
If~$\io(X) = 2$ then~$\g(X) = 4\dd(X) + 1$ with~$\dd(X) \in \{1,2,3,4,5\}$.
\item 
If~$\io(X) = 1$ then~$\g(X) \in \{2,3,4,5,6,7,8,9,10,12\}$.
\end{itemize}
Threefolds of index~$2$ are often called {\sf del Pezzo threefolds},
and threefolds of index~$1$ are often called {\sf prime Fano threefolds}.

When~$\uprho(X) \ge 2$, the classification is more complicated; 
it can be found in~\cite{Mori-Mukai:MM} or, in an interactive from, in~\cite{fanography}.
In this case, deformation types of Fano threefolds are no longer characterized by the index and genus or degree,
so, to distinguish between deformation types we use the notation of~\cite{Mori-Mukai:MM} and~\cite{fanography}:
we will say that~$X$ is {\sf a Fano threefold of type~\typemm{\rho}{m}}
if~$\uprho(X) = \rho$ and~$X$ has number~$m$ in~\cite{Mori-Mukai:MM}.
Similarly, when~$\uprho(X) = 1$ and~$\g(X) \le 10$, 
following the notation of~\cite{fanography} we will say that~$X$ is {\sf a Fano threefold of type~\typemm{1}{m}} 
if~$\io(X)=1$ and~$\g(X) =m+1$.

Let~$X$ be a Fano threefold with (at worst) canonical Gorenstein singularities.
We say that~$X$ is {\sf a smoothable Fano threefold of type~\typemm{\rho}{m}} 
if there is a flat projective family of threefolds~$\cX \to B$ 
over a smooth pointed curve~$(B,o)$
such that~$\cX_o \cong X$ and all other fibers~$\cX_b$, $b \ne o$,
are smooth Fano threefolds of type~\typemm{\rho}{m};
in particular, $\uprho(\cX_b) = \rho$.

We say that a normal projective threefold~$X$ with (at worst) canonical singularities is {\sf a weak Fano threefold} 
if its anticanonical class~$-K_X$ is a nef and big $\QQ$-Cartier divisor. 
By~\cite[Theorem~3.3]{Kollar-Mori:book} if~$X$ is a weak Fano threefold
a sufficiently high multiple~$-nK_X$ of the anticanonical class 
defines a morphism~$X \to \PP^N$ with connected fibers, called {\sf the anticanonical contraction of~$X$}, 
and its image~$X_\can \subset \PP^N$ is a normal projective variety
with (at worst) canonical singularities, 
called {\sf the anticanonical model of~$X$}.

\subsection{Classification in the nonfactorial case}

We start by explaining a relation between nonfactorial $1$-nodal Fano threefolds and smooth Fano threefolds with higher Picard rank.
A combination of the next result with Theorem~\ref{thm:intro-nf-contractions} 
leads to a conceptual explanation for the classification.

\begin{theorem}
\label{thm:intro-nf}
Let~$(X,x_0)$ be a nonfactorial $1$-nodal Fano threefold with~$\uprho(X) = 1$ and let 
\begin{equation*}
\hX \coloneqq \Bl_{x_0}(X)
\end{equation*}
be the blowup of~$X$ at the node~$x_0 \in X$.
Then one of the following three cases takes place:
\begin{enumerate}
\item 
\label{it:intro-nf-ample}
The class~$-K_{\hX}$ is ample; then~$\hX$ is a smooth Fano threefold of type~\typemm{3}{m} with
\begin{equation*}
m \in \{2,\, 5,\, 21,\, 31\}.
\end{equation*}

\item 
\label{it:intro-nf-bpf}
The class~$-K_{\hX}$ is not ample but nef and big; then~$\hX$ is a smooth weak Fano threefold
and its anticanonical model~$\hX_\can$ is a smoothable Fano threefold of 
type~\typemm{2}{m} with
\begin{equation}
\label{eq:m-list}
m \in \{1,\, 2,\, 3,\, 4,\, 5,\, 6,\, 7,\, 9,\, 10,\, 12,\, 13,\, 14\}.
\end{equation}
\item 
\label{it:intro-nf-g2}
The class~$-K_{\hX}$ is base point free but not big;
then~$\io(X) = 1$, $\g(X) = 2$, and~$X$ 
is a complete intersection in the weighted projective space~$\PP(1^4,2,3)$ 
of the cone over a smooth quadric surface in~$\PP(1^4) = \PP^3$ with vertex~$\PP(2,3)$ and a sextic hypersurface.
\end{enumerate}
\end{theorem} 

The most interesting case is~\ref{it:intro-nf-bpf}.
Our second result describes in this case a relation between the birational geometry of~$X$,
the anticanonical model~$\bar{X} \coloneqq \hX_\can$ of the blowup~\mbox{$\hX = \Bl_{x_0}(X)$}, 
and a smoothing~$X'$ of~$\bar{X}$.
This relation leads, eventually, to a completely explicit description of the corresponding threefolds, 
see Table~\ref{table:nf} and Theorem~\ref{thm:intro-nf-ci}.

\begin{theorem}
\label{thm:intro-nf-contractions}
Let~$(X,x_0)$ be a nonfactorial $1$-nodal Fano threefold with~$\uprho(X) = 1$ 
such that the anticanonical class~$-K_{\hX}$ of the blowup~$\hX \coloneqq \Bl_{x_0}(X)$ is not ample but nef and big.
Let~$X'$ be a smoothing of the anticanonical model~$\bar{X} \coloneqq \hX_\can$ of~$\hX$,
so that~$X'$ is a smooth Fano threefold of type~\typemm{2}{m} with~$m$ in~\eqref{eq:m-list}.
If~$m \ge 2$ and
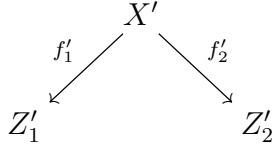
\vspace{-.2ex}
\begin{equation*}
\xymatrix{
& X' \ar[dl]_{f'_1} \ar[dr]^{f'_2}
\\
Z'_1 &&
Z'_2
}
\end{equation*}
are the extremal $K$-negative contractions of~$X'$ then there is a Sarkisov link diagram
\begin{equation}
\label{eq:intro-sl}
\vcenter{\xymatrix{
& 
X_1 \ar[dl]_{f_1} \ar[dr]^{\pi_1} \ar@{<-->}[rr]^\chi &&
X_2 \ar[dl]_{\pi_2} \ar[dr]^{f_2} &
\\
Z_1 &&
X && 
Z_2,
}}
\end{equation}
where~$\pi_i \colon X_i \to X$ are small resolutions with smooth irreducible 
exceptional curves,
$\chi$ is the Atiyah flop in these curves, 
and~$f_i \colon X_i \to Z_i$ are extremal $K$-negative contractions
that have the same type as~$f'_i$. 
More precisely, 
$Z_i$ and~$Z'_i$ are smooth and deformation equivalent, and
\begin{aenumerate}
\item 
\label{it:intro-b1}
if~$f'_i$ is the blowup of a smooth curve of degree~$d$ and genus~$g$
then~$f_i$ is the blowup of a smooth curve of degree~$d - 1$ and genus~$g - \io(Z_i) + 1$;
\item 
\label{it:intro-c1}
if~$f'_i$ is a conic bundle over~$\PP^2$ with discriminant of degree~$d$
then~$f_i$ is a conic bundle over~$\PP^2$ with discriminant of degree~$d-1$;
\item 
\label{it:intro-d1}
if~$f'_i$ is a del Pezzo fibration over~$\PP^1$ with fibers of degree~$d$
then~$f_i$ is a del Pezzo fibration over~$\PP^1$ with fibers of degree~$d+1$.
\end{aenumerate}
\end{theorem}

\begin{remark}
In the case where~$m = 1$, i.e., $X'$ has type~\typemm{2}{1}, the situation is similar:
the extremal contractions~$f'_1 \colon X' \to Z'_1$ and~$f'_2 \colon X' \to Z'_2$ 
are the blowup of a smooth curve of degree~$d = 1$ and genus~$g = 1$ on a smooth del Pezzo threefold~$Z'_1$ of degree~$1$
and a del Pezzo fibration of degree~$2$ over~$Z'_2 = \PP^1$, respectively,
while the contractions~$f_1 \colon X_1 \to Z_1$ and~$f_2 \colon X_2 \to Z_2$
are the blowup of the (unique) Gorenstein singular point on a singular del Pezzo threefold~$Z_1$ of degree~$1$
and a del Pezzo fibration of degree~$2$ over~$Z_2 = \PP^1$. 
In fact, in a sense that can be made precise a Gorenstein singular point 
can be considered as a ``virtual curve'' of degree~$d - 1 = 0$ and genus~$g - 2 + 1 = 0$.
\end{remark} 

Using Theorem~\ref{thm:intro-nf-contractions} it is easy to explain why only varieties of type~\typemm{2}{m} 
with~$m$ from the list~\eqref{eq:m-list} correspond to $1$-nodal Fano threefolds, see Remark~\ref{rem:other-m}.

Comparing the results of Theorems~\ref{thm:intro-nf} and~\ref{thm:intro-nf-contractions} 
with the Mori--Mukai classification~\cite{Mori-Mukai:MM},
it is easy to determine explicitly the types of extremal contractions~$f_i \colon X_i \to Z_i$ in~\eqref{eq:intro-sl}.
\begin{table}[H]
\renewcommand{\arraystretch}{1.2}
\begin{equation*}
\begin{array}
{l!{\vrule width 0.11em}c|c!{\vrule width 0.05em}ll!{\vrule width 0.05em}ll!{\vrule width 0.05em}l!{\vrule width 0.05em}c!{\vrule width 0.05em}clc!{\vrule width 0.06em}}
\text{type} &\io(X) &\g(X) & Z_1 & \Gamma_1 / \Delta_1 & Z_2 & \Gamma_2 / \Delta_2 
& \hfill \bar{X}\hfill & 
\h^{1,2}(X') & |\Sing( \bar{X})| 
\\
\noalign{\hrule height 0.1em}
\rowcolor{gray!10} 
& 3 & 28 & \PP^1 & (9,0) & \PP^1 & (9, 0) & {\typemm{3}{31}} & 0 & 0
\\
\rowcolor{gray!10} 
& 2 & 21 & \PP^2 & (0,0) & \PP^1 & (8,2) & \typemm{3}{21} & 0 & 0
\\
\hypertarget{12na}{\ntype{1}{12}{na}}
& 1 & 12 & \PP^3 & (0, 5) & \PP^3 & (0, 5) & \typemm{2}{12} & 3 & 4
\\
\hypertarget{12nb}{\ntype{1}{12}{nb}}
& 1& 12 & Q^3 & (0, 5) & \PP^2 & (0, 3) & \typemm{2}{13} & 2 & 3
\\
\hypertarget{12nc}{\ntype{1}{12}{nc}}
&1& 12 & \rY_5 & (0, 4) & \PP^1 & (6, 6) & \typemm{2}{14} & 1 & 2
\\
\rowcolor{gray!10} 
\hypertarget{12nd}{\ntype{1}{12}{nd}}
& 1& 12 & \PP^2 & (0, 0) & \PP^1 & (5, 8) & \typemm{3}{5} & 0 & 0
\\
\hline
\hypertarget{10na}{\ntype{1}{10}{na}}
& 1& 10 & \PP^3 & (2, 6) & \PP^2 & (2, 4) & \typemm{2}{9} & 5 & 4
\\
\hypertarget{10nb}{\ntype{1}{10}{nb}}
&1& 10 & \rY_4 & (0, 3) & \PP^1 & (5, 12) & \typemm{2}{10} & 3 & 2
\\
\hline
\hphantom{\mathbf{0}}
\hypertarget{9na}{\ntype{1}{9}{na}}
& 1 & 9 & Q^3 & (3, 7) & \PP^1 & (5, 14) & \typemm{2}{7} & 5 & 3
\\
\rowcolor{gray!10} 
\hphantom{\mathbf{0}}
\hypertarget{9nb}{\ntype{1}{9}{nb}}
&1 & 9 & \PP^1 & (8, 8) & \PP^1 & (4, 16) & \typemm{3}{2} & 3 & 0 
\\
\hline
\hphantom{\mathbf{0}}
\hypertarget{8na}{\ntype{1}{8}{na}}
& 1 & 8 & \rY_3 & (0, 2) & \PP^1 & (4, 20) & \typemm{2}{5} & 6 & 2 
\\
\hphantom{\mathbf{0}}
\hypertarget{8nb}{\ntype{1}{8}{nb}}
&1 & 8 & \PP^2 & (5, 5) & \PP^2 & (5, 5) & \typemm{2}{6} & 9 & 5 
\\
\hline
\hphantom{\mathbf{0}}
\hypertarget{7n}{\ntype{1}{7}{n}}
& 1 & 7 & \PP^3 & (7, 8) & \PP^1 & (4, 24) & \typemm{2}{4} & 10 & 4
\\
\noalign{\hrule height 0.1em}
\hphantom{\mathbf{0}}
\hypertarget{6n}{\ntype{1}{6}{n}}
&1 & 6 & \rY_2 & (0, 1) & \PP^1 & (3, 32) & \typemm{2}{3} & 11 & 2
\\
\hphantom{\mathbf{0}}
\hypertarget{5n}{\ntype{1}{5}{n}}
& 1 & 5 &\PP^2 & (14, 7) & \PP^1 & (3, 40) & \typemm{2}{2} & 20 & 7
\\
\hphantom{\mathbf{0}}
\hypertarget{4n}{\ntype{1}{4}{n}}
& 1 & 4 & \rY^{\mathrm{s}}_1 & (0, 0) & \PP^1 & (2, 54) & \typemm{2}{1} & 21 & \infty
\\
\rowcolor{gray!50} 
\hphantom{\mathbf{0}}
\hypertarget{2n}{\ntype{1}{2}{n}} 
& 1 & 2 & \PP^1 & (1, 120) & \PP^1 & (1, 120) & \quad- & - & -
\end{array}
\end{equation*}
\caption{Nonfactorial $1$-nodal Fano threefolds with~$\uprho(X) = 1$}
\label{table:nf}
\end{table}

\noindent{}Gray rows in the table correspond to cases from part~\ref{it:intro-nf-ample} (light gray) 
and part~\ref{it:intro-nf-g2} (dark gray) of Theorem~\ref{thm:intro-nf}.
Notation for the types of nonfactorial 1-nodal Fano threefolds introduced in the first column of Table~\ref{table:nf}
will be used throughout the paper.
The columns~$Z_1$ and~$Z_2$ describe the deformation types of the targets 
of the contractions~\mbox{$f_i \colon X_i \to Z_i$} in~\eqref{eq:intro-sl},
where~notation~$\rY_d$ and~$\rY_1^s$ means that the target is a smooth del Pezzo threefold of degree~$d$,
or a del Pezzo threefold of degree~$1$ with one Gorenstein singular point, respectively.
The columns~$\Gamma_i/\Delta_i$ describe 
the critical loci of the contractions~$f = f_i \colon X_i \to Z_i$:
\begin{itemize}[wide]
\item 
if~$f$ is the blowup of a curve~$\Gamma$ this is~$(\g(\Gamma),\, \deg(\Gamma))$, 
\item 
if~$f$ is the blowup of a Gorenstein singular point this is~$(0, 0)$, 
\item 
if~$f$ is a conic bundle with discriminant~$\Delta \subset \PP^2$ this is~$(\p(\Delta) - 1,\, \deg(\Delta))$,
\item 
if~$f$ is a $\PP^1$-bundle this is~$(0,0)$,
\item 
if~$f$ is a del Pezzo fibration of degree~$d$ with discriminant~$\Delta \subset \PP^1$ this is~$(d, \, \deg(\Delta))$.
\end{itemize}
The column~$\bar{X}$ describes the Mori--Mukai type of~$ \bar{X} = \hX_\can$ (see Theorem~\ref{thm:intro-nf}
and note that if~$\hX$ is a Fano variety then~$\bar{X} = \hX$).
Finally, the last two columns list the Hodge number~$\h^{1,2}$ of a smoothing~$X'$ of~$\bar{X}$
and the number of singular points on~$ \bar{X}$ when~$X$ is general
(if~$X$ is of type~\hyperlink{4n}{\ntype{1}{4}{n}} 
the singular locus of~$ \bar{X}$ is always 1-dimensional, see Remark~\ref{rem:sing-g4};
this can also happen for special varieties~$X$ of 
types~\hyperlink{5n}{\ntype{1}{5}{n}}
and~\hyperlink{6n}{\ntype{1}{6}{n}},
see Remark~\ref{rem:sing-g56}). 

In~\S\ref{sec:nf-ci} we give an explicit construction of nonfactorial $1$-nodal threefolds from Table~\ref{table:nf}.
More precisely, in Corollary~\ref{cor:ci-bir} we characterize the blowups~$\Bl_\Gamma(Z)$ 
such that~$\Bl_{\Gamma}(Z)_\can$ is a nonfactorial $1$-nodal Fano threefold,
in Proposition~\ref{prop:resolution-ce} and Remark~\ref{rem:-CB} we describe conic bundles,
and in Lemma~\ref{lem:dp-bundle} del Pezzo fibrations with analogous properties.

\subsection{Proofs of Theorems~\ref{thm:intro-nf} and~\ref{thm:intro-nf-contractions}}
\label{ss:intro-proofs}

To prove Theorems~\ref{thm:intro-nf} and~\ref{thm:intro-nf-contractions} 
we study the anticanonical model
\begin{equation*}
\bar{X} \coloneqq \hX_\can = \Bl_{x_0}(X)_\can
\end{equation*}
of the blowup~$\hX = \Bl_{x_0}(X)$ of~$X$ at the node.
To explain the argument, assume for simplicity that~$\io(X) = 1$ and~$\g(X) \ge 7$
(the advantage of this is that~$-K_X$ is very ample 
and~\mbox{$X \subset \PP^{\g(X) + 1}$} is an intersection of quadrics, see Corollary~\ref{cor:intersection-of-quadrics}).
Then we prove that there is only a finite number of anticanonical lines on~$X$ through the node~$x_0$,
and the strict transforms of these lines in~$\hX$ 
are the only curves contracted by the anticanonical morphism
\begin{equation*}
\xi \colon \hX \longrightarrow \bar{X},
\end{equation*}
hence~$\xi$ is small.
It follows from this that~$\bar{X}$ is a Fano variety with terminal Gorenstein singularities,
hence it is smoothable by~\cite{Na97}.

Recall notation introduced in diagram~\eqref{eq:intro-sl}.
Assuming also that~$\io(f_1) = \io(f_2) = 1$ 
(where~$\io(f_i)$ are the Fano indices of the contractions~$f_i \colon X_i \to Z_i$ defined in~\eqref{eq:io-f}),
we prove a very important linear equivalence (see Corollary~\ref{cor:alpha}):
\begin{equation}
\label{eq:hhhe}
-K_{\hX} \sim H_1 + H_2
\qquad\text{in~$\Pic(\hX)$,}
\end{equation} 
where~$H_1$ and~$H_2$ are the pullbacks to~$\hX$ of
the ample generators of the Picard groups~$\Pic(Z_1)$ and~$\Pic(Z_2)$, respectively.
Since~$H_1$ and~$H_2$ are nef, it follows from~\eqref{eq:hhhe}
that the classes~$H_1$ and~$H_2$ 
restrict trivially to any curve contracted by~$\xi \colon \hX \to \bar{X}$,
hence they are pullbacks of Cartier divisor classes on~$\bar{X}$,
and therefore there is a commutative diagram
\begin{equation}
\label{eq:intro-barx-diagram}
\vcenter{\xymatrix{
X_1 \ar[d]_{f_1} &
\hX \ar[l]_{\sigma_1} \ar[d]_\xi \ar[r]^{\sigma_2} &
X_2 \ar[d]^{f_2} 
\\
Z_1 &
\bar{X} \ar[l]^{\bar{f}_1} \ar[r]_{\bar{f}_2} &
Z_2,
}}
\end{equation}
where~$\sigma_i \colon \hX \to X_i$ is the blowup 
of the exceptional curve~$C_i \subset X_i$ of the contraction~\mbox{$X_i \to X$}.
It follows that the morphism~$\bar{f}_i \colon \bar{X} \to Z_i$ is the relative anticanonical model
of the composition~$f_i \circ \sigma_i \colon \hX \to Z_i$
and this leads to a relation between~$f_i$ and~$\bar{f}_i$;
for instance, if~$f_i$ is the blowup of a smooth curve~$\Gamma_i \subset Z_i$ and~$f_i(C_i) \ne \Gamma_i$
then it turns out that~$\bar{f}_i$ is the blowup of the reducible curve~$\bar{\Gamma}_i = \Gamma_i \cup f_i(C_i)$.
In particular, it follows that
\begin{equation*}
\p(\bar{\Gamma}_i) \ge \io(Z_i) - 1.
\end{equation*}

On the other hand, since the morphism~$\xi$ is small, \eqref{eq:hhhe} implies the linear equivalence
\begin{equation*}
-K_{\bar{X}}\sim \bar{f}_1^*H_1 + \bar{f}_2^*H_2
\qquad\text{in~$\Pic(\bar{X})$.}
\end{equation*}
Now we consider a smoothing~$\cX \to B$ of the threefold~$\bar{X}$ and check that
the divisor classes~$\bar{f}_1^*H_1, \bar{f}_2^*H_2 \in \Pic(\bar{X})$ 
extend to nef over~$B$ classes~$\cH_1,\cH_2 \in \Pic(\cX)$, 
and that the above relation also extends, so that we obtain a (relative over~$B$) \emph{nef decomposition}
\begin{equation*}
-K_{\cX}\sim \cH_1 + \cH_2
\qquad\text{in~$\Pic(\cX)$.}
\end{equation*}
We check that the classes~$\cH_i$ induce $K$-negative extremal contractions
\begin{equation*}
\xymatrix@1{\cZ_1 & \cX \ar[l]_{\varphi_1} \ar[r]^{\varphi_2} & \cZ_2}
\end{equation*}
where~$\cZ_1$ and~$\cZ_2$ are flat over~$B$.
Over the central point~$o \in B$ these recover the bottom line of~\eqref{eq:intro-barx-diagram}
and over any other point~$b \in B$ we obtain the extremal contractions
\begin{equation*}
\xymatrix@1{\cZ_{1,b} & \cX_b \ar[l]_{\varphi_{1,b}} \ar[r]^{\varphi_{2,b}} & \cZ_{2,b}}
\end{equation*}
of a smoothing~$\cX_b$ of~$\bar{X}$.
Moreover, we show that these are ``flat families of extremal contractions'';
for instance, if~$\bar{f}_i$ is the blowup of a reducible curve~$\bar{\Gamma}_i \subset Z_i$ 
then~$\varphi_{i,b}$ is the blowup of a smooth curve~$\cG_{i,b} \subset \cZ_{i,b}$
that has the same degree and arithmetic genus as~$\bar{\Gamma}_i$.
When combined with the relation of~$f_i$ and~$\bar{f}_i$ observed above, 
this proves Theorem~\ref{thm:intro-nf-contractions}.

Furthermore, $\cX_b$ is a smooth Fano threefold with~$\uprho(\cX_b) = 2$, hence of type~\typemm{2}{m},
and one can show that~$\io(\varphi_{1,b}) = \io(\varphi_{2,b}) = 1$ (because~\mbox{$\io(f_1) = \io(f_2) = 1$}),
that~$\varphi_{i,b}$ is not the blowup of a point,
and if~$\varphi_{i,b}$ is the blowup of a smooth curve~$\cG_{i,b} \subset \cZ_{i,b}$ then
\begin{equation*}
\g(\cG_{i,b}) \ge \io(\cZ_{i,b}) - 1.
\end{equation*}
A combination of these obstructions shows that~$m$ is in the list~\eqref{eq:m-list} and proves Theorem~\ref{thm:intro-nf}.

The cases where $\io(f_1) > 1$ or~$\io(f_2) > 1$
and those where~$-K_X$ is not base point free, or not very ample, 
or the image of~$X$ in~$\PP^{\g(X)+1}$ is not an intersection of quadrics,
are treated separately.
In all these cases the variety~$X$ is described explicitly, and all results stated above
are proved by a case-by-case analysis, see~\S\ref{ss:bi} and~\S\ref{sec:special}.

\subsection{Complete intersections in products and unprojections}

As we tried to show in~\S\ref{ss:intro-proofs}, 
the anticanonical model~$\bar{X}$ of the blowup~$\hX = \Bl_{x_0}(X)$ 
is useful for classification of nonfactorial $1$-nodal Fano threefolds.
Even more interesting is that~$\bar{X}$ admits a very simple description 
that leads eventually to an explicit description of~$X$.

Recall that 
the blowup of a Weil divisor class~$\Blw{D}(Y)$ on a normal variety~$Y$
is defined as the projective spectrum of a certain reflexive sheaf of graded algebras on~$Y$
(see equation~\eqref{eq:bl-d-z} in Appendix~\ref{App:blowup}).
Note that in all cases of Table~\ref{table:nf} we have~$Z_2 = \PP^k$ for some~$k$.

\begin{theorem}
\label{thm:intro-nf-ci}
Let~$X$ be a nonfactorial $1$-nodal Fano threefold with~$\uprho(X) = 1$, \mbox{$\io(X) = 1$}, and~\mbox{$\g(X) \ge 6$}
such that the class~$-K_{\hX}$ is not ample.

\begin{thmenumerate}
\item 
\label{thm:nf-ci-x-barx}
The anticanonical model~$\bar{X}$ of~$\hX$ is the image of the morphism~$\hX \to Z_1 \times Z_2 = Z_1 \times \PP^k$
and the image of its exceptional divisor~$E \subset \hX$ over~$x_0$ is a smooth quadric surface
\begin{equation*}
\Sigma = L_1 \times L_2 \subset Z_1 \times Z_2,
\end{equation*}
where~$L_1 \subset Z_1$ and~$L_2 \subset Z_2 = \PP^k$ are lines, so that we have a chain of embeddings
\begin{equation}
\label{eq:barx-chain}
\Sigma \subset \bar{X} \subset Z_1 \times Z_2 = Z_1 \times \PP^k.
\end{equation} 
Moreover, $\bar{X} \subset Z_1 \times Z_2$ is a complete intersection of ample divisors 
of bidegrees~$(a_i,b_i)$ such that~$\sum a_i = \io(Z_1) - 1$ and~$\sum b_i = \io(Z_2) - 1 = k$;
the pairs~$(a_i,b_i)$ are listed in the next table
\begin{table}[H]
$
\renewcommand{\arraystretch}{1.2}
\begin{array}{c|cccccccccc}
\text{{\rm type}} & 
\hyperlink{6n}{\ntype{1}{6}{n}} & 
\hyperlink{7n}{\ntype{1}{7}{n}} & 
\hyperlink{8na}{\ntype{1}{8}{na}} & 
\hyperlink{8nb}{\ntype{1}{8}{nb}} & 
\hyperlink{9na}{\ntype{1}{9}{na}} & 
\hyperlink{10na}{\ntype{1}{10}{na}} & 
\hyperlink{10nb}{\ntype{1}{10}{nb}} & 
\hyperlink{12na}{\ntype{1}{12}{na}} &
\hyperlink{12nb}{\ntype{1}{12}{nb}} & 
\hyperlink{12nc}{\ntype{1}{12}{nc}}
\\
\hline
Z_1 & \rY_2 & \PP^3 & \rY_3 & \PP^2 & Q^3 & \PP^3 & \rY_4 & \PP^3 & Q^3 & \rY_5
\\
Z_2 & \PP^1 & \PP^1 & \PP^1 & \PP^2 & \PP^1 & \PP^2 & \PP^1 & \PP^3 & \PP^2 & \PP^1
\\
(a_i,b_i) & (1,1) & (3,1) & (1,1) & (2,2) & (2,1) & (1,1) \oplus (2,1) & (1,1) & (1,1)^{\oplus 3} & (1,1)^{\oplus 2} & (1,1)
\end{array}
$
\caption{Complete intersection description of~$\bar{X}$}
\label{table:ci-nf}
\end{table}

\item 
\label{thm:nf-ci-barx-x}
Let~$\bar{X} \subset Z_1 \times Z_2$ be a complete intersection from Table~\xref{table:ci-nf} 
such that~\eqref{eq:barx-chain} holds for a surface~$\Sigma = L_1 \times L_2 \subset Z_1 \times Z_2$.
If~$\bar{X}$ has terminal singularities and~$\Blw{\Sigma}(\bar{X})$ is smooth 
then there exists a unique nonfactorial $1$-nodal Fano threefold~$(X,x_0)$ such that
\begin{equation}
\label{eq:hx-blsigma}
\Bl_{x_0}(X) \cong \Blw{-\Sigma}(\bar{X})
\end{equation}
\end{thmenumerate}
\end{theorem}

Part~\ref{thm:nf-ci-barx-x} of Theorem~\ref{thm:intro-nf-ci} implies that~$X$ 
can be obtained from a complete intersection~$\bar{X}$ by blowing up the Weil divisor class~$-\Sigma$
and then contracting the strict transform~$E$ of~$\Sigma$;
in the terminology of~\cite{PR04}
this means that~$X$ is an {\sf unprojection} of~$\bar{X}$.

\begin{remark}
\label{rem:barx-g456}
For~$\g(X) \in \{4,5,6\}$ the threefold~$\bar{X}$ may have 
non-isolated canonical singularities (for~$\g(X) = 4$ this is always the case);
in these cases one can prove similar results:
\begin{itemize}[wide]
\item 
If~$\g(X) = 6$ 
and~$\bar{X}$ has canonical singularities,
then~$\bar{X}$ is still a divisor of bidegree~$(1,1)$ in~$Z_1\times \PP^1$
(where~$Z_1$ is a smooth del Pezzo threefold of degree~$2$), 
its singular locus is a smooth rational curve in~$L_1 \times \PP^1$ of bidegree~$(3,1)$,
and~$\hX$ is obtained from~$\bar{X}$ by its blowup. 

\item 
If~$\g(X) = 5$ 
the canonical morphism~$\bar{X} \to Z_1 \times Z_2 = \PP^2 \times \PP^1$,
is a double covering branched at a divisor of bidegree~$(4,2)$
which is tangent to a surface~$\bar{\Sigma} \coloneqq \PP^1 \times \PP^1 \subset \PP^2 \times \PP^1$.
If~$\bar{X}$ has terminal singularities, 
then each irreducible component~$\Sigma$ of the preimage of~$\bar{\Sigma}$ is a smooth quadric surface 
and~\eqref{eq:hx-blsigma} holds as before.
If~$\bar{X}$ has canonical singularities (this happens if~$\bar{\Sigma}$ is contained in the branch divisor),
then the singular locus of~$\bar{X}$ is a smooth curve in~$\PP^1 \times \PP^1$ of bidegree~$(3,2)$,
and~$\hX$ is obtained from~$\bar{X}$ by its blowup. 

\item 
If~$\g(X) = 4$ 
then~$\bar{X} \hookrightarrow Z_1 \times \PP^1$
is a divisor of bidegree~$(1,1)$ containing a singular surface~$\Sigma = L_1 \times \PP^1$,
where~$Z_1$ is a del Pezzo threefold of degree~$1$ with one Gorenstein singular point~$y_0$
and~$L_1$ is a curve on~$Z_1$ of degree~$1$, arithmetic genus~$1$, and geometric genus~$0$, passing through~$y_0$.
In this case~$\bar{X}$ has canonical singularities, 
its singular locus is the smooth rational curve~$\{y_0\} \times \PP^1$,
and~$\hX$ is obtained from~$\bar{X}$ by its blowup. 
\end{itemize}

For more details about these cases see Proposition~\textup{\ref{prop:hyperelliptic}}\ref{he:nf} 
and Proposition~\textup{\ref{prop:trigonal}}\ref{it:tri:g5}, \ref{it:tri:g6}.
\end{remark} 

\subsection{The factorial case}
\label{ss:intro-fact}

As we mentioned before, the factorial case is simpler than the nonfactorial one;
in particular, factorial $1$-nodal Fano threefolds with~$\uprho(X) = 1$
have a description analogous to smooth threefolds.
Moreover, without changing the arguments 
we can also include the case of generalized cusps, see Definition~\ref{def:node-cusp} in consideration.

Consider the following varieties
\begin{equation*}
\rM_2 \coloneqq \PP(1^4,3),
\quad 
\rM_3 \coloneqq \PP(1^5,2),
\quad 
\rM_4 \coloneqq \PP^5,
\quad 
\rM_5 \coloneqq \PP^6,
\quad
\rM_6 \coloneqq \CGr(2,5) \subset \PP^{10},
\end{equation*}
where the first four are (weighted) projective spaces and the last is
the cone over the Grassmannian~$\Gr(2,5) \subset \PP^9$.
Recall also from~\cite[\S2]{Muk92} the homogeneous Mukai varieties:
\begin{itemize}[wide]
\item 
$\rM_{7\hphantom{0}} = \OGr_+(5,10)$, the orthogonal isotropic Grassmannian,
\item 
$\rM_{8\hphantom{0}} = \Gr(2,6)$, the Grassmannian,
\item 
$\rM_{9\hphantom{0}} = \LGr(3,6)$, the symplectic isotropic Grassmannian,
\item 
$\rM_{10} = \GtGr(2,7)$, the adjoint Grassmannian of the simple group of Dynkin type~$\mathbf{G}_2$.
\end{itemize}
Furthermore, denote by~$X_{d_1,\dots,d_k}$ a complete intersection of hypersurfaces of degree~$d_1,\dots,d_k$.

\begin{theorem}
\label{thm:intro-factorial-ci}
Let~$X$ be a factorial Fano threefold with a single node or cusp and~\mbox{$\uprho(X) = 1$}.
Then 
\begin{enumerate}
\item 
\label{it:fact-i2}
either~$\io(X) = 2$, $\dd(X) \in \{1,2,3,4\}$, and~$X$ is a complete intersection
\begin{equation*}
X_6 \subset \PP(1^3,2,3),
\quad 
X_4 \subset \PP(1^4,2),
\quad 
X_3 \subset \PP^4,
\quad 
X_{2,2} \subset \PP^5,
\end{equation*}
\item 
\label{it:fact-i1-gsmall}
or~$\io(X) = 1$, $\g(X) \in \{2,3,4,5,6\}$, and~$X \subset\rM_{\g(X)}$ is a complete intersection
\begin{equation*}
\begin{array}{c}
X_6 \subset \rM_2 = \PP(1^4,3),
\quad 
X_{2,4} \subset \rM_3 = \PP(1^5,2),
\quad 
X_{2,3} \subset \rM_4 = \PP^5,
\quad 
X_{2,2,2} \subset \rM_5 = \PP^6,\\[1ex]
X_{1^2,2} \subset \rM_6 = \CGr(2,5),
\end{array}
\end{equation*}
\item 
\label{it:fact-i1-gbig}
or~$\io(X) = 1$, $\g(X) \in \{7,8,9,10\}$, and
\begin{equation*}
X \subset \rM_{\g(X)}
\end{equation*}
is a dimensionally transverse linear section.
\end{enumerate}
Conversely, any complete intersection as above with a single node or cusp is factorial.
\end{theorem}

As we already mentioned above, 
the vanishing of the Hodge number~$\h^{1,2}(X)$ is an obstruction 
for the existence of a factorial $1$-nodal (or 1-cuspidal) degeneration of a smooth threefold~$X$;
for this reason prime Fano threefolds of genus~$12$, quintic del Pezzo threefolds, 
the quadric~$Q^3$ and the projective space~$\PP^3$ do not have such degenerations. 

For factorial del Pezzo threefolds we have the following analogue of Theorem~\ref{thm:intro-nf}.

\begin{theorem}
\label{thm:intro-fi2}
Let~$(X,x_0)$ be a factorial Fano threefold with a single node or cusp, \mbox{$\uprho(X) = 1$}, and~\mbox{$\io(X) = 2$},
so that~$\dd(X) \in \{1,2,3,4\}$. 
Let~$\hX \coloneqq \Bl_{x_0}(X)$.
Then
\begin{enumerate}
\item 
\label{it:intro-fi2-ample}
If~$\dd(X) \in \{2,3,4\}$ then the class~$-K_{\hX}$ is ample 
and~$\hX$ is a smooth Fano threefold 
of type~\typemm{2}{m}, where
\begin{equation*}
m \in \{8,15,23\}.
\end{equation*}

\item 
\label{it:intro-fi2-bpf}
If~$\dd(X) = 1$ then~$-K_{\hX}$ is not ample but nef and big
and the anticanonical model~$\hX_\can$ is a nonfactorial $1$-nodal Fano threefold 
with~$\uprho(\hX_\can) = 1$ and~$\g(\hX_\can) = 4$.
\end{enumerate}
\end{theorem}

Note that in case~\ref{it:intro-fi2-bpf} the variety~$\hX_\can$ 
coincides with the variety of type~\hyperlink{4n}{\ntype{1}{4}{n}} in Table~\ref{table:nf},
while~$\hX$ coincides with one of its small resolutions 
and the map~$\hX \to X$ coincides with the corresponding extremal contraction,
thus forming one half of the Sarkisov link~\eqref{eq:intro-sl}.

For factorial Fano threefolds with~$\io(X) = 1$ we have analogues 
both for Theorem~\ref{thm:intro-nf} and Theorem~\ref{thm:intro-nf-ci},
which we state together.
A weak Fano variety is {\sf almost Fano} if it has at worst terminal Gorenstein singularities 
and its anticanonical contraction is small.

\begin{theorem}
\label{thm:intro-fi1}
Let~$(X,x_0)$ be a factorial Fano threefold with a single node or cusp, \mbox{$\uprho(X) = 1$}, and~\mbox{$\io(X) = 1$},
so that~$2 \le \g(X) \le 10$.
Let~$\hX \coloneqq \Bl_{x_0}(X)$.
\begin{enumerate}
\item 
If~$3 \le \g(X) \le 10$ then the class~$-K_{\hX}$ is not ample but nef and big, $\hX$ is a smooth weak Fano variety,
and its anticanonical model~$\bar{X} \coloneqq \hX_\can$ is a smoothable Fano threefold with canonical Gorenstein singularities
of index~$\io(\bar{X}) = 1$ and genus
\begin{equation*}
\g(\bar{X}) = \g(X) - 1 \in \{2,3,4,5,6,7,8,9\}.
\end{equation*}
If~$\g(X) \ge 5$ then~$\hX$ is almost Fano and~$\bar{X}$ has terminal Gorenstein singularities.
\item 
If~$\g(X) = 2$ then the class~$-K_{\hX}$ is base point free 
and the anticanonical morphism of~$\hX$ is an elliptic fibration~$\hX \to \PP^2$.
\end{enumerate} 
Moreover, if~$\g(X) \ge 3$ 
then~$\bar{X} \subset \rM_{\g(X) - 1}$ is a complete intersection of the same type 
as in Theorem~\textup{\ref{thm:intro-factorial-ci}}\ref{it:fact-i1-gsmall}--\ref{it:fact-i1-gbig}
containing an irreducible quadric surface~$\Sigma \subset \bar{X}$.

Conversely, if~$\bar{X} \subset \rM_{g - 1}$ is a complete intersection 
as in Theorem~\textup{\ref{thm:intro-factorial-ci}}\ref{it:fact-i1-gsmall}--\ref{it:fact-i1-gbig}
containing an irreducible  quadric surface~$\Sigma$
such that~$\bar{X}$ has terminal singularities and~$\Blw{\Sigma}(\bar{X})$ is smooth 
then there exists a unique factorial Fano threefold~$(X,x_0)$ with a single node or cusp such that
\begin{equation*}
\Bl_{x_0}(X) \cong \Blw{-\Sigma}(\bar{X}).
\end{equation*}
\end{theorem}

As in Theorem~\ref{thm:intro-nf-ci} the last isomorphism means that~$X$ is an unprojection of~$\bar{X}$.

For a factorial Fano threefold~$X$ such that~$\uprho(X) = 1$, $\io(X) = 1$, $\g(X) \ge 3$,
and~$\bar{X}$ has terminal singularities,
the morphism~$\xi \colon \hX \to \bar{X}$ is a flopping contraction, 
so one can consider its flop and extend it to a Sarkisov link
\begin{equation*}
\vcenter{\xymatrix{
& 
\hX \ar[dl]_{\pi} \ar[dr]^{\xi} \ar@{<-->}[rr] &&
\hX_+ \ar[dl]_{\xi_+} \ar[dr]^{\pi_+}
\\
X &&
\bar{X} &&
X_+,
}}
\end{equation*}
where~$\pi$ is the blowup of the singularity~$x_0 \in X$, $\xi$ and~$\xi_+$ are small $K$-trivial contractions, 
and~$\pi_+$ is an extremal $K$-negative contraction.
In the proof of Theorem~\ref{thm:intro-fi1} in~\S\S\ref{ss:f-ci}--\ref{ss:proofs-fi1} we describe these links explicitly.
We summarize some information about these links in Table~\ref{table:factorial}.
\begin{table}[H]
\begin{equation*}
\renewcommand{\arraystretch}{1.2}
\begin{array}{c|c!{\vrule width 0.1em}lc|l|c|c} 
\io(X) & \g(X) & X_+ & \Gamma_+ / \Delta_+ &\ \ \,\bar{X} & \h^{1,2}(\bar{X}^\sm) & |\Sing(\bar{X})| 
\\\noalign{\hrule height 0.1em}
\rowcolor{gray!10}
2 & 17 & Q^3 & (1,4) & \typemm{2}{23}& 1 & 0 
\\\hline
\rowcolor{gray!10}
2 & 13 & \PP^3 & (4,6) & \typemm{2}{15}& 5 & 0 
\\\hline
\rowcolor{gray!10}
2 & 9 &\PP^2 & (9,6) & \typemm{2}{8}& 9 & 0 
\\\hline
2 & 5 & \PP^1 & (2,54)& \typemm{1}{4}& 20 & 1 
\\\hline
1 & 10 &\rY_5 & (1, 6) & \typemm{1}{8} & 3 & 3
\\
\hline
1 & 9 & \rY_4 & (0, 4) & \typemm{1}{7} & 5 & 4
\\
\hline 
1 & 8 & Q^3 & (4, 8) & \typemm{1}{6} & 7 & 4 
\\
\hline
1 & 7 & \PP^3 & (6, 8) & \typemm{1}{5} & 10 & 5 
\\
\hline
1 & 6 & \PP^2 & (9, 6) & \typemm{1}{4} & 14 & 6 
\\
\hline
1 & 5 & \PP^1 & (4, 36) & \typemm{1}{3} & 20 & 8 
\\
\hline 
1 & 4 & \rX^{\mathrm{s}}_4 & (0,0) & \typemm{1}{2} & 30 & 12 
\\
\hline
1 & 3 & \rX^{\mathrm{s}}_3 & (0,0) & \typemm{1}{1} & 52 & 24
\\
\hline
\rowcolor{gray!50} 
1 & 2 & \PP^2 & - & \quad- & - &
\end{array}
\end{equation*}
\caption{Factorial Fano threefolds with one node}
\label{table:factorial}
\end{table}
\noindent 
where we use the same notation and conventions as in Table~\ref{table:nf},
in particular~$\rY_5$ and~$\rY_4$ denote the deformation classes of smooth del Pezzo threefolds of degree~$5$ and~$4$,
while~$\rX_4^{\mathrm{s}}$ and~$\rX_3^{\mathrm{s}}$ denote the deformation classes of prime Fano threefolds of genus~$4$ and~$3$
with at most one Gorenstein singular point.
Note, however, that for~$\g(X) \in \{3,4\}$ the anticanonical model~$\bar{X}$ of~$\hX$ 
may have non-isolated canonical singularities;
in this case the Sarkisov link does not exist
(see the proof of Theorem~\textup{\ref{thm:intro-fi1}} in~\S\ref{ss:f-ci}). 

\subsection{Notation and conventions}
\label{ss:conventions}

We work over an algebraically closed field~$\Bbbk$ of characteristic~$0$;
all schemes and morphisms of schemes are $\Bbbk$-linear.

Throughout the paper we use freely the standard definitions and conventions of the Minimal Model Program (MMP), 
see, e.g., \cite{Kollar-Mori:book};
in particular, we write~$K_X$ for the canonical class of~$X$ and denote linear equivalence of Cartier or Weil divisors by~$\sim$. 

Recall that {\sf a contraction} is a projective morphism~$f \colon Y \to Z$ with connected fibers between normal varieties,
so that~$f_*\cO_Y \cong \cO_Z$.
A contraction~$f$ is {\sf $K$-negative} if~$Y$ is $\QQ$-Gorenstein and~$-K_Y$ is $f$-ample,
and~$f$ is {\sf extremal} if~$Y$ is $\QQ$-factorial and~$\uprho(Y/Z) = 1$.

Our convention for the projectivization of a vector bundle~$\cE$ on a scheme~$X$ is
\begin{equation*}
\PP_X(\cE) \coloneqq \Proj_X \left( \moplus_{k=0}^\infty\Sym^k(\cE^\vee) \right),
\end{equation*}
and we normalize its {\sf relative hyperplane class}~$M \in \Pic(\PP_X(\cE))$ by the rule 
\begin{equation*}
p_*\cO_{\PP_X(\cE)}(M) \cong \cE^\vee,
\end{equation*}
where~$p \colon \PP_X(\cE) \to X$ is the projection.

We often use the following conventions:
\begin{itemize}[wide]
\item 
given a morphism~$f \colon Y \to Z$ and a Cartier divisor class~$D \in \Pic(Z)$ we denote its pullback~$f^*D$ to~$Y$ simply by~$D$;
\item 
given a 
birational transformation~$\chi \colon Y \dashrightarrow Y'$ which is an isomorphism in codimension~$1$ and a Weil divisor~$D$ on~$Y$
we denote its strict transform~$\chi_*D$ to~$Y'$ simply by~$D$.
\end{itemize}
This allows us to make notation less heavy and avoid introducing names for all the maps.

We denote by~$\Bl_Z(Y)$ the blowup of a subscheme~$Z \subset Y$ 
and by~$\Blw{D}(Y)$ the blowup of a Weil divisor class~$D \in \Cl(Y)$, 
see Appendix~\ref{App:blowup}.

Given a normal Cohen--Macaulay projective variety~$\rM$ with~$\Cl(\rM) \cong \ZZ$ we denote by
\begin{equation*}
X_{d_1,\dots,d_k} \subset \rM
\end{equation*}
any complete intersection of Weil divisors corresponding to elements~$d_1,\dots,d_k \in \Cl(\rM)$.

\subsection{Acknowledgements}

We would like to thank Laurent Manivel for the help with the proof of Lemma~\ref{lem:omd-mukai}
and the anonymous referees for their comments.
We would also like to thank National Center for Theoretical Sciences (Taipei) for hospitality.

\section{Preliminaries}

Recall that in the Introduction we defined {\sf a weak Fano variety}~$X$
as a normal projective variety with at worst canonical 
singularities 
such that the anticanonical divisor~$-K_X$ is a nef and big $\QQ$-Cartier divisor.
We also introduced the {\sf anticanonical contraction}
\begin{equation*}
\xi \colon X \longrightarrow X_\can,
\end{equation*}
as the morphism defined by a sufficiently high multiple of the anticanonical class.
The variety~$X_\can$ also has canonical singularities and the class~$-K_{X_\can}$ is ample.
If, moreover, $X$ is Gorenstein then~$X_\can$ is also Gorenstein,
and the contraction $\xi$ is \emph{crepant}, i.e.,
\begin{equation}
\label{eq:crepant-contr}
\xi^*\cO_{X_\can}(K_{X_\can}) \cong \cO_X(K_X),
\end{equation}
this follows easily from~\cite[Theorem~3.3]{Kollar-Mori:book}.

Sometimes, it is convenient to consider a more restrictive class:
we will say that a weak Fano variety~$X$ is {\sf almost Fano} 
if~$X$ has at worst terminal Gorenstein singularities and the anticanonical contraction~$\xi$ is small,
i.e., if~$\xi$ does not contract divisors.
In this case, it follows from~\eqref{eq:crepant-contr} 
that~$X_\can$ also has at worst terminal Gorenstein singularities.

Recall that terminal Gorenstein singularities of threefolds are compound Du Val,
hence isolated hypersurface singularities; see, e.g.,~\cite[Theorem~1.1]{Reid:MM}.

The following types of terminal Gorenstein 
singularities often appear in MMP:

\begin{definition}
\label{def:node-cusp}
A singularity~$(X,x_0)$ is {\sf a node} or {\sf an ordinary double point} if it is a hypersurface singularity
and the Hessian matrix of the local equation of~$X$ at~$x_0$ is nondegenerate.

Similarly, a singularity~$(X,x_0)$ is {\sf a (generalized) cusp} if it is a hypersurface singularity,
the Hessian matrix of the local equation of~$X$ at~$x_0$ has corank~$1$, and~$\Bl_{x_0}(X)$ is smooth.
\end{definition}

\begin{remark}
\label{rem:blowup-nc}
If~$(X,x_0)$ is a node or a cusp then the blowup~$\Bl_{x_0}(X)$ is smooth and
its exceptional divisor~$E \subset \Bl_{x_0}(X)$ is a quadric (smooth for a node and of corank~$1$ for a cusp).
Moreover, if~$\dim(X) = 3$ and~$\cO_E(1)$ denotes the hyperplane class of the quadric then
\begin{equation*}
\cO_E(E) \cong \cO_E(K_{\Bl_{x_0}(X)}) \cong \cO_E(-1),
\end{equation*}
and the discrepancy of~$E$ equals~$1$.
The special feature of a cusp is that~$\Pic(E)$ is generated by~$E\vert_E$, 
hence the local class group is zero, the morphism~$r_{x_0}$ in~\eqref{eq:pic-cl} vanishes,
and therefore~$\Pic(X) = \Cl(X)$.
Thus, any 1-cuspidal threefold is factorial.
\end{remark} 

Similarly, {\sf the quotient singularity of type~$\frac12(1,1,1)$} 
is a terminal $\QQ$-Gorenstein isolated threefold singularity~$(X,x_0)$
that locally around~$x_0$ admits a finite morphism of degree~$2$ 
from a smooth threefold which is \'etale over the complement of~$x_0$.

\begin{remark}
\label{rem:contr}
If~$(X,x_0)$ is a quotient singularity of type~$\tfrac12(1,1,1)$, 
the blowup~$\Bl_{x_0}(X)$ is also smooth, 
and if~$E \subset \Bl_{x_0}(X)$ is its exceptional divisor then~$E\cong \PP^2$, 
\begin{equation*}
\cO_E(E) \cong \cO_{\PP^2}(-2),
\qquad 
\cO_E(K_{\Bl_{x_0}(X)}) \cong \cO_{\PP^2}(-1),
\end{equation*}
and the discrepancy of~$E$ equals~$1/2$.
\end{remark} 

\subsection{$K$-negative extremal contractions}
\label{ss:contractions}

Let~$Y$ be a smooth projective rationally connected threefold 
and let~$f \colon Y \to Z$ be a $K$-negative extremal contraction.
Then~$Z$ is also rationally connected.
If, furthermore, $\uprho(Y) = 2$ then~$\uprho(Z) = 1$, the class~$-K_Z$ is ample,
and by~\cite{Mori1982} the morphism~$f$ is of one of the following types: 

\begin{enumerate}
\item [\type{(B_2)}]
$Z$ is a smooth Fano threefold and $f$ is the \emph{blowup of a smooth point},

\item [\type{(B_1^1)}]
$Z$ is a smooth Fano threefold and $f$ is the \emph{blowup of a smooth curve}~$\Gamma \subset Z$, 

\item [\type{(B_1^0)}]
$Z$ is a Fano threefold 
with a node or cusp~$z \in Z$
and~$f$ is the \emph{blowup of~$z$},

\item [\typeb]
$Z$ is a Fano threefold with a~$\tfrac12(1,1,1)$-singularity~$z \in Z$ and~$f$ is the \emph{blowup of~$z$},

\item [\type{(C_2)}]
$f$ is a~\emph{$\PP^1$-bundle} over~$Z \cong \PP^2$,

\item [\type{(C_1)}]
$f$ is a \emph{conic bundle} over~$Z \cong \PP^2$ with a \emph{discriminant curve}~$\Delta \subset \PP^2$, 
$\deg(\Delta) \ge 3$,

\item [\type{(D_3)}]
$f$ is a~\emph{$\PP^2$-bundle} over~$Z \cong \PP^1$,

\item [\type{(D_2)}]
$f$ is a flat \emph{quadric surface fibration} over~$Z \cong \PP^1$,

\item [\type{(D_1)}]
$f$ is a flat \emph{del Pezzo fibration} over~$Z \cong \PP^1$ of degree~$\deg(Y/Z) \in \{1,2,3,4,5,6\}$.
\end{enumerate}

Furthermore, let~$\Upsilon$ be the class of a minimal rational curve contracted by~$f$
in the space of $1$-cycles on~$Y$ modulo numerical equivalence.
Then the following sequence
\begin{equation}
\label{eq:exact-seq:Pic} 
0 \longrightarrow 
\Pic(Z) \overset{f^*}\longrightarrow 
\Pic(Y) \xrightarrow{\ \cdot\Upsilon \ }
\ZZ \longrightarrow 
0
\end{equation} 
is exact, see~\cite[(3.2)]{Mori-Mukai1983} or~\cite[Remark~1.4.4]{IP99}.
In other words, there is a divisor class~$M_f \in \Pic(Y)$, unique up to a twist by~$f^*(\Pic(Z))$, such that~$M_f \cdot \Upsilon = 1$.

The {\sf Fano index of~$f$} is defined by
\begin{equation}
\label{eq:io-f}
\io(f) \coloneqq -K_Y \cdot \Upsilon.
\end{equation} 
Thus, $-K_Y \sim \io(f)M_f + f^*D$ for some~$D \in \Pic(Z)$.

Inspecting the above classification of $K$-negative extremal contractions
it is easy to see that~$\io(f) \in \{1,2,3\}$, moreover
\begin{itemize}[wide]
\item 
$\io(f) = 3$ if and only if~$f$ is of type~\type{(D_3)};
\item 
$\io(f) = 2$ if and only if~$f$ is of type~\type{(B_2)}, \type{(C_2)}, or~\type{(D_2)};
\item 
$\io(f) = 1$ if and only if~$f$ is of type~\type{(B_1^1)}, \type{(B_1^0)}, \typeb, \type{(C_1)}, or~\type{(D_1)}.
\end{itemize}
In particular, for types~\type{(C)} and~\type{(D)} the lower index 
equals~$\io(f)$.
On the other hand, in type~\type{(B)} the lower index is equal to the 
discrepancy of the exceptional divisor,
and when the discrepancy is~$a = 1$, the upper index is equal to the dimension of the blowup center. 

\subsection{Numerical invariants}

The Fano index~$\io(Y)$ of a Fano threefold~$Y$ with canonical Gorenstein singularities was defined by~\eqref{eq:def-io-g}.
The following bound for~$\io(Y)$ is well known.

\begin{theorem}[\cite{Kobayashi1973,Fujita1975}]
\label{hm:KO}
If~$Y$ is a Fano threefold with canonical Gorenstein singularities then~$\io(Y) \le 4$.
Moreover,
\begin{itemize}[wide]
\item 
if~$\io(Y) = 4$, then $Y\cong \PP^3$,
\item 
if $\io(Y) = 3$, then $Y$ is a \textup(possibly singular\textup) quadric in~$\PP^{4}$.
\end{itemize}
\end{theorem} 

If $Y$ is a weak Fano threefold with canonical Gorenstein singularities, 
then~$\Pic(Y)$ is a finitely generated torsion free abelian group, see~\cite[Proposition~2.1.2]{IP99}.
Hence there is a unique primitive ample Cartier divisor class~$H_Y$ such that
\begin{equation*}
-K_Y \sim \io(Y) H_Y.
\end{equation*}
The class~$H_Y$ is called {\sf the fundamental class} of~$Y$.
Note that~$\dd(Y) = H_Y^{3}$ by~\eqref{eq:def-io-g}.
The genus~$\g(Y)$ was also defined in~\eqref{eq:def-io-g}.
Since~$Y$ has canonical Gorenstein singularities, 
we have~$H^{>0}(Y, \cO_Y(-K_Y)) = 0$ by the Kawamata--Viehweg vanishing theorem, 
and therefore
\begin{equation}
\label{eq:h0-mk}
\dim H^0(Y, \cO_Y(-K_Y)) = \g(Y) + 2,
\end{equation}
by the Riemann--Roch Theorem, cf.~\cite[Corollary~2.1.14(ii)]{IP99}.
Equality~\eqref{eq:h0-mk} can be used as an alternative (sometimes more convenient) definition of the genus.

Finally, if~$Y$ is a smooth projective threefold~$Y$ we write
\begin{equation}
\hh(Y) \coloneqq \h^{1,2}(Y) = \dim H^2(Y, \Omega^1_Y)
\end{equation} 
for a {\sf Hodge number} of~$Y$ (the only off-diagonal Hodge number if~$Y$ is rationally connected). 

The following table lists the genera and Hodge numbers for all smooth Fano threefolds of Picard rank~$1$
(see, e.g., \cite[Table~\S12.2]{IP99}).
\begin{table}[H]
\begin{equation*}
\renewcommand{\arraystretch}{1.4} 
\begin{array}
{c!{\vrule width 0.9pt}rrrrrrrrrr!{\vrule width 0.7pt}rrrrr!
{\vrule width 0.7pt}c!{\vrule width 0.7pt}c}
\io(Y) & \multicolumn{10}{c!{\vrule width 0.7pt}}{1} & \multicolumn{5}{c!{\vrule width 0.7pt}}{2} & 3 & 4
\\
\hline
\g(Y) & 2 & 3 & 4 & 5 & 6 & 7 & 8 & 9 & 10 & 12 & 5 & 9 & 13 & 17 & 21 & 28 & 33
\\
\hline
\hh(Y) & 52 & 30 & 20 & 14 & 10 & 7 & 5 & 3 & 2 & 0 & 21 & 10 & 5 & 2 & 0 & 0 & 0
\end{array}
\end{equation*}
\caption{Genera and Hodge numbers of Fano threefolds of Picard rank~$1$}
\label{table:hodge-genus}
\end{table} 

Recall from~\cite[Theorem~11]{Na97} that for any Fano threefold~$X$ with terminal Gorenstein singularities 
there is a {\sf smoothing}, i.e., a flat projective family of threefolds~$\cX \to B$ over a smooth pointed curve~$(B,o)$ 
such that~$\cX_o \cong X$ and~$\cX_b$ is smooth for any~$b \ne o$ in~$B$.
We denote by~$\Pic_{\cX/B}$ the \'etale Picard sheaf, see~\cite[\S9.2]{Kleiman}.

\begin{proposition}[{cf.~\cite[Theorem~1.4]{Jahnke2011}}]
\label{prop:picard-sheaf}
Let~$X$ be a Fano threefold with terminal Gorenstein singularities
and let~$p \colon \cX \to B$ be a smoothing of~$X$.
Then~$\cX$ is a factorial fourfold with terminal hypersurface singularities
and the \'etale Picard sheaf~$\Pic_{\cX/B}$ is locally constant.
In particular, for all~$b \in B$ close to~$o \in B$ if~$b \ne o$ then
the fiber~$\cX_b$ is a smooth Fano threefold with the same invariants:
\begin{equation*}
\uprho(\cX_b) = \uprho(X),
\qquad 
\io(\cX_b) = \io(X), 
\qquad\text{and}\qquad
\g(\cX_b) = \g(X).
\end{equation*}
\end{proposition}

\begin{proof}
The \'etale Picard sheaf~$\Pic_{\cX/B}$ is locally constant by~\cite[Proposition~A.7]{KS23}
and the equalities of the index and genus follow from the argument of~\cite[Corollary~A.8]{KS23}.
\end{proof} 

In what follows we use notation~$X^{\mathrm{sm}}$ for a smooth fiber of a smoothing of~$X$.
Using this, the equalities of Proposition~\ref{prop:picard-sheaf} can be rewritten as
\begin{equation}
\label{eq:smoothing-invariants}
\uprho(X^{\mathrm{sm}}) = \uprho(X),
\qquad 
\io(X^{\mathrm{sm}}) = \io(X),
\qquad 
\g(X^{\mathrm{sm}}) = \g(X).
\end{equation}

\subsection{Some intersection theory}

The next well-known lemma collects some intersection numbers 
(see, e.g., \cite[Lemma~2.11]{Iskovskih1977} or~\cite[Lemma~4.1.6]{IP99}).

\begin{lemma}
\label{lem:intersection}
Let~$f \colon Y \to Z$ be a $K$-negative extremal contraction
where~$Y$ is a smooth rationally connected threefold with~$\uprho(Y) = 2$.
If~$f$ is of type~\type{(C)} or~\type{(D)} then
\begin{align*}
H_Z^2 \cdot (-K_Y) &= 2, &
H_Z \cdot (-K_Y)^2 &= 12 - \deg(\Delta), &&
\text{for type~\type{(C)},}
\\
H_Z^2 \cdot (-K_Y) &= 0, &
H_Z \cdot (-K_Y)^2 &= \deg(Y/Z), &&
\text{for type~\type{(D)},}
\end{align*}
where~$H_Z$ is the line class of~$\PP^2$ for type~\type{(C)},
or the point class of~$\PP^1$ for type~\type{(D)}.
Similarly, if~$f$ is birational, $E \subset Y$ is its exceptional divisor, and~$\Gamma \subset Z$ is its center, then
\begin{align*}
f^*D^2 \cdot E &= 0, &
f^*D \cdot E^2 &= - D \cdot \Gamma, &
E^3 &= 
\begin{cases}
1, & \text{for type~\type{(B_2)},}\\
2, & \text{for type~\type{(B_1^0)},}\\
4, & \text{for type~\typeb,}\\
2 - 2\g(\Gamma) + K_Z \cdot \Gamma, & \text{for type~\type{(B_1^1)},}
\end{cases}
\end{align*}
where~$D$ is a Cartier divisor on~$Z$. 
\end{lemma}

The next result is useful for comparing intersection numbers of varieties related by a flop 
(see, e.g., \cite[Lemma~4.1.2]{IP99}).

\begin{lemma}
\label{lem:int-flop}
If
\begin{equation*}
\xymatrix{
Y\ar@{-->}[rr]^\chi\ar[dr]_{\xi} && Y_+\ar[dl]^{\xi_+}
\\
&\bar Y
} 
\end{equation*}
is a flop of threefolds with terminal singularities, $D_1,\, D_2 \in \Pic(Y)$, and~$\bar{D} \in \Pic(\bar{Y})$ then
\begin{equation*}
D_1 \cdot D_2 \cdot \xi^*\bar{D} = (\chi_*D_1) \cdot (\chi_*D_2) \cdot \xi_+^*\bar{D},
\end{equation*}
where~$\chi_*D_{i} \in \Pic(Y_+)$ is the strict transform of~$D_{i}$ under the flop.
\end{lemma}

\begin{proof}
Since both sides are linear in~$\bar{D}$, it is enough to prove it when~$\bar{D}$ is base point free.
In this case we may assume that~$\xi$ and~$\xi_+$ are isomorphisms over a neighborhood of~$\bar{D}$,
hence~$\chi$ is an isomorphism on a neighborhood of~$\xi^{-1}(\bar{D})$, 
and then the equality is obvious.
\end{proof} 

For a simple flop one can also relate~$(\chi_*D)^3$ to~$D^3$, see~\cite[Lemma~7.4(2)]{Friedman91}, 
but we will not need this result.

\subsection{Numerical invariants and $K$-negative extremal contractions}

The following results relate~$\g(Y)$ and~$\hh(Y)$ with the invariants of a contraction~$f\colon Y \to Z$.

First, we discuss the case where~$Y$ and~$Z$ are both smooth.

\begin{lemma}
\label{lem:hh-g-xi}
If\/~$Y$ is a smooth weak Fano threefold 
and $f \colon Y \to Z$ is a birational $K$-negative extremal contraction with smooth~$Z$ then 
\begin{align*}
\g(Y) &= 
\begin{cases}
\g(Z) - 4, & 
\text{if~$f$ has type~\type{(B_2)},}\\
\g(Z) + \g(\Gamma) - \io(Z)\deg(\Gamma) - 1, & 
\text{if~$f$ has type~\type{(B_1^1)}.}
\end{cases}
\intertext{Moreover, if~$f \colon Y \to Z$ is birational or a conic bundle over~$\PP^2$ then}
\hh(Y) &= 
\begin{cases}
\hh(Z), & 
\text{if~$f$ has type~\type{(B_2)},}\\
\hh(Z) + \g(\Gamma), & 
\text{if~$f$ has type~\type{(B_1^1)},}\\
\tfrac12 \deg(\Delta)(\deg(\Delta) - 3), & \text{if~$f$ has type~\type{(C)}}.
\end{cases}
\end{align*}
\end{lemma}

\begin{proof}
The genus formulas follow from~\eqref{eq:def-io-g}, Lemma~\ref{lem:intersection}, 
and the canonical class formula~$K_{Y} = f^*K_{Z} + aE$,
where~$E \subset Y$ is the exceptional divisor of~$f$ and~$a$ is the discrepancy.
The Hodge number formula is obvious for smooth blowups
and standard for conic bundles (see, e.g., \cite[(4.13)]{MoriMukai:86}). 
\end{proof} 

The situation with the blowup of a Gorenstein singular point is more tricky.

\begin{proposition}
\label{prop:hh-g-xsm}
Let~$(X, x_0)$ be a Fano threefold with a single node or cusp~$x_0 \in X$
and let~$\hX \coloneqq \Bl_{x_0}(X)$ be its blowup.
Then
\begin{equation}
\label{eq:prop:hh-g-xsm}
\g(\hX) = \g(X) - 1
\qquad\text{and}\qquad
\hh(\hX) = 
\begin{cases}
\hh(X^\sm), 
& \text{if~$X$ is nonfactorial,}\\
\hh(X^\sm) - 1, 
& \text{if~$X$ is factorial.}
\end{cases}
\end{equation}
If, furthermore, $X$ is nonfactorial and~$X' \to X$ is a small resolution 
then~$\g(X') = \g(X)$ and~$\hh(X') = \hh(X^\sm)$.
\end{proposition}

\begin{proof}
The equality~$\g(X) = \g(\hX) + 1$ follows from~\eqref{eq:def-io-g} and Lemma~\ref{lem:intersection},
because
\begin{equation*}
(-K_{\hX})^3 = (-K_X - E)^3 = (-K_X)^3 - E^3 = (2\g(X) - 2) - 2.
\end{equation*}
Similarly, if~$X'$ is a small resolution of~$X$ then~$(-K_{X'})^3 = (-K_X)^3$, hence~$\g(X') = \g(X)$.

To relate the Hodge numbers of~$\hX$ and~$X^\sm$ we assume that~$\Bbbk = \CC$ 
and compute the topological Euler characteristics.
First, we have~$\hX \setminus E = X \setminus \{x_0\}$. 
Next, if~$F$ is the Milnor fiber of a retraction~$X^\sm \to X$, 
then~$X^\sm \setminus F$ is homotopy equivalent to~$X \setminus \{x_0\}$, hence  
\begin{equation*}
\upchi_{\mathrm{top}}(\hX) - \upchi_{\mathrm{top}}(X^\sm) = \upchi_{\mathrm{top}}(E) - \upchi_{\mathrm{top}}(F).
\end{equation*}
Now we consider two cases:
if~$x_0$ is a node then~$E \cong \PP^1 \times \PP^1$ and~$F$ is homotopy equivalent to a sphere,
and if~$x_0$ is a cusp then~$E$ is a quadratic cone and~$F$ is homotopy equivalent to a bouquet of two spheres.
In either case~$\upchi_{\mathrm{top}}(E) - \upchi_{\mathrm{top}}(F) = 4$, hence
\begin{equation*}
\upchi_{\mathrm{top}}(\hX) - \upchi_{\mathrm{top}}(X^\sm) = 4.
\end{equation*}
On the one hand, if~$X$ is nonfactorial then~$\uprho(\hX) = \uprho(X) + 2 = \uprho(X^\sm) + 2$, hence 
\begin{equation*}
2\hh(\hX) = 
2 + 2\uprho(\hX) - \upchi_{\mathrm{top}}(\hX) = 
2 + 2\uprho(X^\sm) + 4 - \upchi_{\mathrm{top}}(X^\sm) - 4 = 
2\hh(X^\sm).
\end{equation*}
On the other hand, if~$X$ is factorial then~$\uprho(\hX) = \uprho(X) + 1 = \uprho(X^\sm) + 1$, hence 
\begin{equation*}
2\hh(\hX) = 
2 + 2\uprho(\hX) - \upchi_{\mathrm{top}}(\hX) = 
2 + 2\uprho(X^\sm) + 2 - \upchi_{\mathrm{top}}(X^\sm) - 4 = 
2\hh(X^\sm) - 2.
\end{equation*}
This proves the second equality in~\eqref{eq:prop:hh-g-xsm}.

Finally, if~$X'$ is a small resolution of~$X$ (hence~$x_0$ is a node and~$X$ is nonfactorial) 
then~$\hX$ is the blowup of~$X'$ along a smooth rational curve,
hence~$\hh(X') = \hh(\hX)$ by Lemma~\ref{lem:hh-g-xi}.
\end{proof}

\begin{lemma}
\label{lem:flop-h}
If~$\chi \colon Y \dashrightarrow Y_+$ is a flop of smooth weak Fano threefolds then~$\hh(Y) = \hh(Y_+)$.
\end{lemma}

\begin{proof}
Since~$Y$ and~$Y_+$ are isomorphic in codimension one, we have~$\uprho(Y) = \uprho(Y_+)$.
On the other hand, by construction of flops~\cite[\S2]{Kollar:flops} the topological Euler numbers of~$Y$ and~$Y_+$ are equal.
Hence, $2\hh(Y) = 2 + 2\uprho(Y) - \upchi_{\mathrm{top}}(Y) = 2 + 2\uprho(Y_+) - \upchi_{\mathrm{top}}(Y_+) = 2\hh(Y_+)$.
\end{proof} 

\subsection{Blowup at the singular point}

In this subsection we describe the anticanonical contraction of the blowup at the singular point 
of a $1$-nodal or $1$-cuspidal Fano threefold~$(X,x_0)$
in the case where~$-K_X$ is very ample.
Then~$X \subset \PP^{\g(X) + 1}$ by~\eqref{eq:h0-mk}
and we denote by~\mbox{$\rT_{x_0}(X) \subset \PP^{\g(X) + 1}$} the embedded tangent space at~$x_0$ of~$X$.
Notation~$\Blw{-\Sigma}(\bar{X})$ in the next lemma 
is used for the blowup of a Weil divisor class, see Appendix~\ref{App:blowup}.

\begin{lemma}
\label{lem:barx}
Let~$(X,x_0)$ be a Fano threefold with a single node or cusp and such that~$-K_X$ is very ample, so that~$X \subset \PP^{\g(X) + 1}$.
Let~$\hX \coloneqq \Bl_{x_0}(X)$ and let~$E \subset \hX$ be the exceptional divisor.

\begin{thmenumerate}
\item
\label{it:barx-weak}
The variety $\hX$ is a weak Fano variety and its anticanonical contraction~\mbox{$\xi \colon \hX \longrightarrow \bar{X}$} 
is obtained from the Stein factorization of the composition
\begin{equation*}
\tilde\xi \colon \hX = \Bl_{x_0}(X) \hookrightarrow \Bl_{x_0}(\PP^{\g(X) + 1}) \xrightarrow{\ p\ } \PP^{\g(X)},
\end{equation*}
where~$p$ is induced by the linear projection~$\PP^{\g(X) + 1} \dashrightarrow \PP^{\g(X)}$ with center~$x_0$,
and~$\bar{X}$ is a Fano threefold with canonical Gorenstein singularities, $\io(\bar{X}) = 1$, and~$\g(\bar{X}) = \g(X) - 1$.
Moreover, the exceptional divisor~$E \subset \hX$ is $\xi$-ample 
and~$\xi$ induces an isomorphism~$E \xrightiso{\ } \Sigma \subset \bar{X}$.

\item
\label{it:barx-almost}
If~$\dim(X \cap \rT_{x_0}(X)) \le 1$ then~$\xi$ is small, 
$\hX$ is an almost Fano variety, and~$\bar{X}$ is a Fano variety with terminal Gorenstein singularities.
Moreover,
\begin{equation*}
\hX \cong \Blw{-\Sigma}(\bar{X}).
\end{equation*}
\item
\label{it:barx-iq}
If, moreover, $X \subset \PP^{\g(X) + 1}$ is an intersection of quadrics 
then~$\bar{X} \subset \PP^{\g(X)}$ is the image of the linear projection;
in particular its anticanonical class~$-K_{\bar{X}}$ very ample, 
\end{thmenumerate}
\end{lemma}

\begin{proof}
\ref{it:barx-weak}
Let~$H$ be the hyperplane class of~$X \subset \PP^{\g(X) + 1}$. 
It is easy to see that the composition~$\tilde\xi$ is given by the complete linear system~$|H - E| = |-K_{\hX}|$,
hence the anticanonical contraction~$\xi \colon \hX \to \bar{X}$ is obtained from the Stein factorization of~$\tilde\xi$.
It also follows that~$-K_{\hX}$ is nef and big (because~$X$ is not a cone), hence~$\hX$ is a weak Fano variety
and~$\bar{X}$ is a Fano threefold with canonical Gorenstein singularities.
The equality~$\io(\bar{X}) = 1$ is obvious 
because~$\xi^*K_{\bar{X}}\sim K_{\hX}\sim E - H$ is primitive in~$\Cl(\bar{X}) = \Cl(\hX)$,
and the equality~$\g(\bar{X}) = \g(X) - 1$ follows from~\eqref{eq:prop:hh-g-xsm}.
Finally, the exceptional divisor of~$\Bl_{x_0}(\PP^{\g(X) + 1})$ is $\tilde\xi$-ample 
and~$\tilde{\xi}$ maps it isomorphically onto~$\PP^{\g(X)}$,
therefore~$E$ has the same properties with respect to the induced morphism~$\xi$.

\ref{it:barx-almost}
By part~\ref{it:barx-weak} the curves contracted by~$\xi$ are the strict transforms of lines on~$X$ through~$x_0$.
Any such line is contained in~$X \cap \rT_{x_0}(X)$, so if~$\dim(X \cap \rT_{x_0}(X)) \le 1$ then~$\xi$ is small.
Consequently, $\hX$ is an almost Fano variety and~$\bar{X}$ is a Fano variety with terminal Gorenstein singularities. 
Finally, the isomorphism~$\hX \cong \Blw{-\Sigma}(\bar{X})$ follows from Corollary~\ref{cor:small-blowup} and Definition~\eqref{eq:bl-d-z}.

\ref{it:barx-iq}
If~$X$ is an intersection of quadrics, there is a right-exact sequence
\begin{equation*}
\cO_{\Bl_{x_0}(\PP^{\g(X) + 1})}(E - 2H)^{\oplus N_1} \oplus
\cO_{\Bl_{x_0}(\PP^{\g(X) + 1})}(2E - 2H)^{\oplus N_2}
\longrightarrow \cO_{\Bl_{x_0}(\PP^{\g(X) + 1})} \longrightarrow \cO_{\hX} \to 0,
\end{equation*}
and since the fibers of the morphism~$p \colon \Bl_{x_0}(\PP^{\g(X) + 1}) \to \PP^{\g(X)}$ are $1$-dimensional and 
\begin{equation*}
\mathbf{R}^1p_*(\cO_{\Bl_{x_0}(\PP^{\g(X) + 1})}(E - 2H)) =
\mathbf{R}^1p_*(\cO_{\Bl_{x_0}(\PP^{\g(X) + 1})}(2E - 2H)) = 0,
\end{equation*}
it follows that the natural morphism~\mbox{$\cO_{\PP^{\g(X)}} \to \tilde\xi_*\cO_{\hX}$} is surjective,
hence~$\bar{X} = \tilde\xi(\hX) \subset \PP^{\g(X)}$ and~$-K_{\bar{X}}$ is very ample.
\end{proof}

\section{Nonfactorial threefolds: first results}
\label{sec:first}

In this section we start the study of nonfactorial $1$-nodal Fano threefolds.
In~\S\ref{ss:notation} we set up the notation that will be used throughout Sections~\ref{sec:first}--\ref{sec:nf-ci}.
in~\S\ref{ss:pic-constraints} we discuss the structure of the class group~$\Cl(X)$,
in~\S\ref{ss:numerical-constraints} we establish some numerical constraints,
and in~\S\ref{ss:bi} we classify~$X$ such that one of its small resolutions has an extremal contraction of large index.

\subsection{Notation}
\label{ss:notation}

Let~$(X,x_0)$ be a nonfactorial $1$-nodal Fano threefold with~$\uprho(X) = 1$.
Let~$\pi\colon \hX \to X$ be the blowup of the node~$x_0 \in X$ and let
\begin{equation*}
E \cong \PP^1 \times \PP^1
\end{equation*}
be its exceptional divisor.
The nonfactoriality assumption implies that there are two birational 
contractions~$\sigma_1 \colon \hX \to X_1$ and~$\sigma_2 \colon \hX \to X_2$
to smooth projective varieties 
such that~$C_i \coloneqq \sigma_i(E)$ are smooth rational curves 
and~$\Bl_{C_1}(X_1) \cong \hX \cong \Bl_{C_2}(X_2)$.
Moreover, 
\begin{equation}
\label{eq:cn-ci}
\cN_{C_i/X_i} \cong \cO_{C_i}(-1) \oplus \cO_{C_i}(-1),
\end{equation} 
and there is a commutative diagram
\begin{equation}
\label{eq:diag-pi-sigma}
\vcenter{\xymatrix@=1em{
& \hX \ar[dl]_{\sigma_1} \ar[dd]^(.3)\pi \ar[dr]^{\sigma_2}
\\
X_1 \ar[dr]_{\pi_1} \ar@{<-->}[rr] &&
X_2 \ar[dl]^{\pi_2} 
\\
& X
}}
\end{equation}
where~$\pi_1$ and~$\pi_2$ are the small birational contractions with exceptional curves~$C_1$ and~$C_2$,
and the birational morphism~$\sigma_2 \circ \sigma_1^{-1} = \pi_2^{-1} \circ \pi_1 \colon X_1 \dashrightarrow X_2$ is the flop of~$C_1$. 

\begin{remark}
\label{rem:xi-bu}
Since~$\pi_1$ and~$\pi_2$ are small, applying Corollary~\ref{cor:small-blowup} we obtain isomorphisms
\begin{equation*}
X_1 \cong \Blw{D_1}(X),
\qquad 
X_2 \cong \Blw{D_2}(X),
\end{equation*}
where~$D_1,D_2 \in \Cl(X)$ are representatives of elements of~$\Cl(X) / \Pic(X) \cong \ZZ$ with opposite signs.
Conversely, for any Weil divisor class~$D \in \Cl(X) \setminus \Pic(X)$ we have~$\Blw{D}(X) \cong X_i$ for some~$i = 1,2$;
in particular, for any small birational morphism~$X' \to X$ from a normal variety~$X'$ 
we have~$X' \cong X_1$ or~$X' \cong X_2$,
again by Corollary~\ref{cor:small-blowup}.
\end{remark} 

Since~$X$ is nonfactorial and~$\uprho(X) = 1$, we have~$\uprho(X_1) = \uprho(X_2) = 2$.
Moreover, since~$X$ is Gorenstein and~$\pi_i$ are small, we have
\begin{equation*}
K_{X_i} = \pi_i^*K_X,
\end{equation*}
and since~$X$ is Fano, the classes~$-K_{X_i}$ are semiample.
Therefore, there are unique $K$-negative extremal contractions~$f_i \colon X_i \to Z_i$.
We summarize all the varieties and maps discussed above into the Sarkisov link diagram~\eqref{eq:intro-sl}. 

Since~$\uprho(X_1) = \uprho(X_2) = 2$, we have~$\uprho(Z_1) = \uprho(Z_2) = 1$.
We denote by
\begin{equation*}
H \in \Pic(X),
\qquad 
H_1 \in \Pic(Z_1),
\qquad 
H_2 \in \Pic(Z_2)
\end{equation*} 
the ample generators of the Picard groups and, by abuse, 
we use the same notation for the pullbacks of these classes to~$X_1$, $X_2$, $\hX$ and other varieties.
Since~$Z_i$ are rationally connected and have terminal singularities and Picard rank~$1$, they are Fano varieties.
We denote by~$\io(X)$ and~$\io(Z_i)$ the Fano indices, so that
\begin{equation*}
-K_X = \io(X)H,
\qquad 
-K_{Z_1} = \io(Z_1)H_1,
\qquad 
-K_{Z_2} = \io(Z_2)H_2.
\end{equation*}
Note that~$\io(X) \in \{1,2,3,4\}$ by Theorem~\ref{hm:KO}, and also~$\io(Z_i) \in \{1,2,3,4\}$ if~$Z_i$ is Gorenstein.

Since~$X_i$ are smooth varieties and~$\pi_i$ are small birational contractions, we have 
\begin{equation}
\label{eq:pic-iso}
\Pic(X_1) \cong 
\Cl(X_1) \xrightiso{\ \pi_{1*}\ } 
\Cl(X) \xleftiso{\ \pi_{2*}\ } 
\Cl(X_2) \cong 
\Pic(X_2).
\end{equation}
We will always identify these groups by means of these isomorphisms. 

We write~$\Upsilon_1$ and~$\Upsilon_2$ for the numerical classes of minimal rational curves contracted by~$f_1$ and~$f_2$, respectively.
Note that
\begin{equation}
\label{eq:h-ups}
H_1 \cdot \Upsilon_1 = H_2 \cdot \Upsilon_2 = 0,
\quad
H_1 \cdot \Upsilon_2 > 0,
\quad 
H_2 \cdot \Upsilon_1 > 0.
\end{equation}
Indeed, the equalities follow from the projection formula for~$f_i$
and the inequalities follow from Lemma~\ref{lem:eff-mov} below.
Finally, 
\begin{equation}
\label{eq:io-fi}
\io(f_i) = -K_{X_i}\cdot \Upsilon_i = \io(X) H \cdot \Upsilon_i
\end{equation} 
by~\eqref{eq:io-f}, 
hence~$\io(X)$ divides both~$\io(f_1)$ and~$\io(f_2)$. 

\subsection{The class group constraints}
\label{ss:pic-constraints}

Let~$X$ be a nonfactorial $1$-nodal Fano threefold with Picard rank~\mbox{$\uprho(X) = 1$}.
In this subsection we discuss the structure of the class group of~$X$.
Recall the identifications~\eqref{eq:pic-iso}. 

We denote by~$\Nef(X_i) \subset \Cl(X) \otimes \QQ$ the nef cone of~$X_i$,
by~$\Mov(X) \subset \Cl(X) \otimes \QQ$ the closure of the cone of movable classes on~$X$,
and by~$\Eff(X) \subset \Cl(X) \otimes \QQ$ the pesudoeffective cone of~$X$, respectively.
If~$f_i$ is birational, we denote by~$E_i$ its exceptional divisor.
Note that under the isomorphism~\eqref{eq:pic-iso} we have
\begin{equation*}
\Mov(X_1) = \Mov(X) = \Mov(X_2)
\quad \text{and}\quad 
\Eff(X_1) = \Eff(X) = \Eff(X_2).
\end{equation*}
Finally, we denote by~$\NE(X_i)$ the Mori cone of~$X_i$.

\begin{lemma}
\label{lem:eff-mov}
The cones~$\NE(X_i)$ and~$\Nef(X_i)$ are generated by the following classes
\begin{equation*}
\NE(X_i) = \QQ_{\ge 0}[C_i] + \QQ_{\ge 0}[\Upsilon_i],
\qquad 
\Nef(X_i) = \QQ_{\ge 0}[H] + \QQ_{\ge 0}[H_i]
\end{equation*}
and if~$\io(f_i) = 1$ then~$\Nef(X_i) \cap \Cl(X) = \ZZ_{\ge 0}[H] + \ZZ_{\ge 0}[H_i]$.
Moreover,
\begin{equation*}
\Mov(X) = \QQ_{\ge 0}[H_1] + \QQ_{\ge 0}[H_2] = \Nef(X_1) \cup \Nef(X_2).
\end{equation*}
Finally, if~$D \subset X$ is an irreducible divisor then either~$D \in \Mov(X)$ 
or~$D = E_i$ is the exceptional divisor of~$f_i$ for~$i \in \{1,2\}$ and~$f_i$ is birational.
Thus,
\begin{equation*}
\Eff(X) = \QQ_{\ge 0}[D_1] + \QQ_{\ge 0}[D_2],
\end{equation*}
where~$D_i = E_i$ if~$f_i$ is birational, and~$D_i = H_i$ otherwise.
\end{lemma}

\begin{proof}
The descriptions of~$\NE(X_i)$ and~$\Nef(X_i)$ obviously follow from the definitions of the corresponding curves and divisor classes
and the fact that if~$\io(f_i) = 1$ then~$H$ generates~$\Pic(X_i/Z_i)$, hence~$H$ and~$H_i$ generate~$\Cl(X)$, 
hence any~$D \in \Nef(X_i) \cap \Cl(X)$ is a linear combination of~$H$ and~$H_i$ with nonnegative integral coefficients.

Now let~$M$ be a movable divisor on~$X$.
Consider~$M$ as a divisor on~$X_1$.
Then~$M \cdot \Upsilon_1 \ge 0$ because~$\Upsilon_1$ sweeps a divisor or the entire~$X_1$ and~$M$ is movable.
Note that this argument proves the inequalities in~\eqref{eq:h-ups}.
Now, if~$M \cdot C_1 \ge 0$, it follows that~\mbox{$M \in \Nef(X_1)$}.
If, on the other hand, we have~$M \cdot C_1 < 0$ then, since $X_1\dashrightarrow X_2$ is a flop, 
we have~$M \cdot C_2 > 0$, 
where now~$M$ is considered as a divisor on~$X_2$.
As before, $M \cdot \Upsilon_2 \ge 0$, hence~$M \in \Nef(X_2)$.
This proves that~$\Mov(X) = \Nef(X_1) \cup \Nef(X_2)$.
Moreover, since~$\Mov(X)$ is convex and the interiors of~$\Nef(X_i)$ do not intersect,
it follows that~$\Mov(X)$ is generated by~$H_1$ and~$H_2$.
It also follows that~$H$ is a linear combination of~$H_1$ and~$H_2$ with positive rational coefficients.

Now let~$D$ be an irreducible non-movable divisor.
Then~$D \sim a_1H_1 + a_2H_2$ with~$a_i \in \QQ$, and one of the coefficients must be negative.
Assume~$a_1 < 0$. 
Considering~$D$ as a divisor on~$X_2$, we see that~$D \cdot \Upsilon_2 < 0$ by~\eqref{eq:h-ups}.
But then all the curves contracted by~$f_2$ are contained in~$D$, hence~$f_2$ is birational and~$D = E_2$.
Similarly, if~$a_2 < 0$ then~$f_1$ is birational and~$D = E_1$.
This proves 
the description of~$\Eff(X)$. 
\end{proof} 

The parameter~$\balpha(X)$ defined in the next proposition and the relations~\eqref{eq:mkx} and~\eqref{eq:mkhx}
will be used extensively throughout Sections~\ref{sec:first}--\ref{sec:nf-ci}.

\begin{proposition}
\label{prop:pic}
The classes~$H_1$ and~$H_2$ freely generate the vector space~$\Cl(X) \otimes \QQ$ and
\begin{align}
\label{eq:mkx}
\balpha(X) H &\sim \tfrac{\io(f_2)}{\io(X)}H_1 + \tfrac{\io(f_1)}{\io(X)}H_2
&& \text{in~$\Cl(X)$,}
\\
\label{eq:mkhx}
\balpha(X) H - \tfrac{\io(f_1)\io(f_2)}{\io(X)^2} E &\sim \tfrac{\io(f_2)}{\io(X)}H_1 + \tfrac{\io(f_1)}{\io(X)}H_2
&& \text{in~$\Pic(\hX)$}
\end{align} 
for some integer~$\balpha(X) \ge 1$.
Moreover,
\begin{equation}
\label{eq:deg-ci}
H_1 \cdot C_1 = \tfrac{\io(f_1)}{\io(X)},
\qquad 
H_2 \cdot C_2 = \tfrac{\io(f_2)}{\io(X)}.
\end{equation} 
If~$\io(f_1) = 1$ then~$H_1$ and~$H$ generate~$\Cl(X)$
and if~$\io(f_2) = 1$ then~$H_2$ and~$H$ generate~$\Cl(X)$.
\end{proposition}

\begin{remark}
A posteriori we will see that~$\balpha(X) = 1$, see Corollary~\ref{cor:alpha}.
It would be interesting to find a conceptual a priori reason for this equality.
\end{remark} 

\begin{proof}
Let~$M_i$ be a relatively ample generator of~$\Pic(X_i/Z_i)$, so that~$M_i \cdot \Upsilon_i = 1$ (see~\eqref{eq:exact-seq:Pic}).
Then the classes~$M_i$ and~$H_i$ generate~$\Pic(X_i)$, hence we can write
\begin{equation}
\label{eq:Hsim}
H \sim a_iM_i + b_iH_i
\end{equation}
with~$a_i > 0$.
Since~$H$ is primitive, $a_i$ and~$b_i$ are coprime.
Since~$H$ is a pullback along the map~$\pi_i \colon X_i \to X$ and~$C_i$ is contracted by~$\pi_i$, we have~$H \cdot C_i = 0$, therefore
\begin{equation*}
0 = a_i(M_i \cdot C_i) + b_i(H_i \cdot C_i).
\end{equation*}
Now, applying~\cite[Proposition~A.14]{KS23}, we see that
there is a linear combination of~$M_i$ and~$H_i$ that has intersection index~$1$ with~$C_i$;
therefore, the integers~$M_i \cdot C_i$ and~$H_i \cdot C_i$ are coprime.
Since, moreover,~$a_i$ and~$b_i$ are coprime, $a_i > 0$, and~$H_i \cdot C_i > 0$, we have
\begin{equation*}
a_i = H_i \cdot C_i,
\qquad 
b_i = - M_i \cdot C_i.
\end{equation*}
Multiplying the linear equivalence~\eqref{eq:Hsim} with~$\Upsilon_i$ and using~\eqref{eq:h-ups} and~\eqref{eq:io-fi}, we obtain
\begin{equation*}
\tfrac{\io(f_i)}{\io(X)} = 
\tfrac{1}{\io(X)}(-K_{X_i}) \cdot \Upsilon_i = 
H \cdot \Upsilon_i = 
(a_iM_i + b_iH_i) \cdot \Upsilon_i = 
a_i = 
H_i \cdot C_i.
\end{equation*}
This proves~\eqref{eq:deg-ci}. 

Now consider the group~$\Cl(X)$.
Its rank is equal to~$2$, hence in~$\Cl(X)$ there is a relation
\begin{equation*}
cH \sim c_1H_1 + c_2H_2,
\end{equation*}
where $c,c_1,c_2 \in \ZZ$ are pairwise coprime.
Since~$\Cl(X) = \Pic(\hX) / \ZZ \cdot [E]$, this relation lifts to a relation in~$\Pic(\hX)$ that looks as
\begin{equation*}
cH - c'E \sim c_1H_1 + c_2H_2
\end{equation*}
for some~$c' \in \ZZ$.
Restricting this relation to~$\Pic(E) \cong \ZZ \oplus \ZZ$ and taking into 
account that the restrictions of~$H_1$ and~$H_2$ have bidegree~$\big(\tfrac{\io(f_1)}{\io(X)},0\big)$ 
and~$\big(0,\tfrac{\io(f_2)}{\io(X)}\big)$ by~\eqref{eq:deg-ci},
while the restriction of~$H$ is zero and the restriction of~$E$ has bidegree~$(-1,-1)$, we obtain
\begin{equation*}
c_1\, \tfrac{\io(f_1)}{\io(X)} = c' = c_2\, \tfrac{\io(f_2)}{\io(X)}.
\end{equation*}
It follows that the relation in~$\Pic(\hX)$ takes the form~\eqref{eq:mkhx}
and the relation in~$\Cl(X)$ takes the form~\eqref{eq:mkx}.
Moreover, since~$H \in \Mov(X)$, and~$\Mov(X)$ is a convex cone generated by~$H_1$ and~$H_2$ (see Lemma~\ref{lem:eff-mov})
the coefficient~$\balpha(X)$ must be positive.

Since~$H$ and~$H_1$ generate~$\Cl(X_1) \otimes \QQ = \Cl(X) \otimes \QQ$ and~$\balpha(X) \ne 0$, 
it follows from~\eqref{eq:mkx} that~$H_1$ and~$H_2$ also generate~$\Cl(X) \otimes \QQ$.

If~$\io(f_1) = 1$ then~$H$ generates~$\Pic(X_1/Z_1)$, hence~$H$ and~$H_1$ generate~$\Pic(X_1) = \Cl(X)$.
The same argument works if~$\io(f_2) = 1$.
\end{proof}

\subsection{Numerical constraints} 
\label{ss:numerical-constraints}

Recall the notation for contraction types from~\S\ref{ss:contractions}.

\begin{lemma}
\label{lem:no-blowup-smooth-point}
Neither~$f_1$ nor~$f_2$ can be of type~\type{(B_2)}.
\end{lemma}

\begin{proof}
If~$f_i$ is of type~\type{(B_2)}, Lemma~\ref{lem:hh-g-xi}, Proposition~\ref{prop:hh-g-xsm}, and~\eqref{eq:smoothing-invariants} give
\begin{equation*}
\g(X^\sm) = \g(X_i) = \g(Z_i) - 4
\qquad\text{and}\qquad 
\hh(X^\sm) = \hh(X_i) = \hh(Z_i),
\end{equation*}
and inspecting Table~\ref{table:hodge-genus} it is easy to see 
that there is no pair~$(X^\sm,Z_i)$ of smooth Fano threefolds whose genera and Hodge numbers satisfy the above equalities.
\end{proof}

\begin{lemma}
\label{lem:curve-genus-degree}
If~$f_i$ is of type~\type{(B_1^1)} then
\begin{equation*}
\g(\Gamma_i) = \hh(X^\sm) - \hh(Z_i),
\quad\text{and}\quad 
\io(Z_i)\deg(\Gamma_i) = 
(\g(Z_i) - \hh(Z_i)) - (\g(X^\sm) - \hh(X^\sm)) - 1.
\end{equation*}
If~$f_i$ is of type~\type{(B_1^0)} or~\type{(B_1^1)}
then~$\io(Z_i) \ge 2$.
\end{lemma}

\begin{proof}
The formulas for~$\g(\Gamma_i)$ and~$\deg(\Gamma_i)$ follow 
from Lemma~\ref{lem:hh-g-xi}, Proposition~\ref{prop:hh-g-xsm}, and~\eqref{eq:smoothing-invariants}. 

If~$f_i$ is of type~\type{(B_1)} then~$\io(f_i) = 1$, hence~$\io(X) = 1$ by~\eqref{eq:io-fi}.
Assume also~$\io(Z_i) = 1$.
We show below that this leads to a contradiction.

If~$f_i$ is of type~\type{(B_1^0)}, hence~$Z_i$ is a factorial Fano threefold with a node or cusp,
we denote by~$Z_i^\sm$ a smoothing of~$Z_i$ and note that
\begin{equation*}
\hh(X^\sm) = \hh(X_i) = \hh(Z_i^\sm) - 1
\qquad\text{and}\qquad
\g(X^\sm) = \g(X_i) = \g(Z_i^\sm)
\end{equation*}
by Proposition~\ref{prop:hh-g-xsm} and~\eqref{eq:smoothing-invariants}.
Inspecting Table~\ref{table:hodge-genus} we see that there is no pair~$(X^\sm,Z_i^\sm)$ of smooth Fano threefolds, both of index~$1$,
whose genera and Hodge numbers satisfy the above equalities.

Similarly, if~$f_i$ is of type~\type{(B_1^1)} and~$\g(\Gamma_i) = 0$, we have~$\hh(X^\sm) = \hh(Z_i)$.
Since both have index~$1$, we conclude from Table~\ref{table:hodge-genus} that~$\g(X^\sm) = \g(Z_i)$,
and then the equality proved above gives~$\io(Z_i)\deg(\Gamma_i) = -1$, which is absurd. 

Finally, if~$f_i$ is of type~\type{(B_1^1)} and~$\g(\Gamma_i) > 0$, we have~$H = \io(Z_i)H_i - E_i$, hence
\begin{equation*}
H^2 \cdot E_i = 
(\io(Z_i)H_i - E_i)^2 \cdot E_i =
-2\io(Z_i)H_i \cdot E_i^2 + E_i^3 = 2 - 2\g(\Gamma_i) + \io(Z_i)\deg(\Gamma_i),
\end{equation*}
where we use Lemma~\ref{lem:intersection} to compute the right side.
Since the anticanonical contraction of~$X_i$ is small, the left hand side is positive, therefore
\begin{equation*}
\io(Z_i)\deg(\Gamma_i) > 2\g(\Gamma_i) - 2.
\end{equation*}
Combining this with the genus formulas of Lemma~\ref{lem:hh-g-xi} and~\eqref{eq:smoothing-invariants}, we obtain 
\begin{equation*}
\g(X^\sm) = \g(Z_i) + \g(\Gamma_i) - \io(Z_i)\deg(\Gamma_i) - 1 < \g(Z_i) - \g(\Gamma_i) + 1.
\end{equation*}
On the other hand, as we saw above,~$\hh(X^\sm) = \hh(Z_i) + \g(\Gamma_i)$.
Taking the sum of the above inequality and this equality, and taking into account the positivity of~$\g(\Gamma_i)$, we obtain
\begin{equation*}
\g(X^\sm) + \hh(X^\sm) < \g(Z_i) + \hh(Z_i),
\qquad 
\hh(X^\sm) > \hh(Z_i).
\end{equation*}
Now, inspecting Table~\ref{table:hodge-genus}, we see that this is not possible when~$\io(X^\sm) = \io(Z_i) = 1$.
\end{proof} 

\begin{lemma}
\label{lem:hhh}
Assume~$f_i$ is a contraction of type~\type{(B_1)}, \type{(C_1)}, or~\type{(D_1)}.
Then 
\begin{equation*}
\label{eq:hhh}
H^2 \cdot H_i = 
\begin{cases}
\frac1{\io(Z_i)}\Big({(\g(Z_i) + \hh(Z_i)) + (\g(X^\sm) - \hh(X^\sm)) - 1}\Big), &
\text{if~$f_i$ is of type~\type{(B_1^1)}},\\
\frac1{\io(Z_i)}\Big({(\g(Z_i^\sm) + \hh(Z_i^\sm)) + (\g(X^\sm) - \hh(X^\sm)) - 2}\Big), &
\text{if~$f_i$ is of type~\type{(B_1^0)}},\\
12 - \deg(\Delta_i), & 
\text{if~$f_i$ is of type~\type{(C_1)}},\\
\deg(X_i/Z_i), & 
\text{if~$f_i$ is of type~\type{(D_1)}}.
\end{cases}
\end{equation*}
Moreover, if~$\g(X) \ge 7$ then~$H^2 \cdot H_i \le 2\g(X) - 4$ and if~$f_i$ is of type~\type{(B_1)} then~$H^2 \cdot H_i \ge 4$.
\end{lemma}

\begin{proof}
If~$f_i$ is of type~\type{(C_1)} or~\type{(D_1)}, we use~\cite[Lemma~4.1.6(C,D)]{IP99}.

Assume~$f_i$ is of type~\type{(B_1^1)}.
Then~$H = \io(Z_i)H_i - E_i$ by the canonical class formula, hence
\begin{equation*}
H^2 \cdot H_i = \io(Z_i)^2H_i^3 - \deg(\Gamma_i)
\end{equation*}
by Lemma~\ref{lem:intersection}.
On the other hand, $\io(Z_i)^3H_i^3 = 2\g(Z_i) - 2$ by definition of the genus;
combining this with the expression of Lemma~\ref{lem:curve-genus-degree} for~$\deg(\Gamma_i)$, we obtain the first formula.

Similarly, if~$f_i$ is of type~\type{(B_1^0)}, we have~$H^2 \cdot H_i = \io(Z_i)^2H_i^3 = \tfrac1{\io(Z_i)}(2\g(Z_i) - 2)$,
and using Proposition~\ref{prop:hh-g-xsm}, and~\eqref{eq:smoothing-invariants}, we obtain the second formula.

Now we prove the second claim; accordingly, we assume~$\g(X) \ge 7$.

If~$f_i$ is of type~\type{(D_1)} then~$2\g(X) - 4 \ge 10 \ge \deg(X_i/Z_i)$.

If~$f_i$ is of type~\type{(C_1)} then~$2\g(X) - 4 \ge 10 > 12 - \deg(\Delta_i)$,
because~$\deg(\Delta_i) \ge 3$.

Finally, if~$f_i$ is of type~\type{(B_1)},
the inequality~$H^2 \cdot H_i \le 2\g(X) - 4$ can be obtained by considering all pairs~$(X^\sm,Z_i^\sm)$ 
such that~$\io(X^\sm) = 1$, $\io(Z_i^\sm) \ge 2$ (we can assume this by Lemma~\ref{lem:curve-genus-degree}), 
and~$\hh(X^\sm) \ge \hh(Z_i^\sm)$,
and using Table~\ref{table:hodge-genus} to compute the invariants.
\end{proof}

\subsection{Nonfactorial threefolds with contractions of large index}
\label{ss:bi}

In this subsection we classify all nonfactorial $1$-nodal threefolds with~$\io(f_1) > 1$ or~$\io(f_2) > 1$.

The case where~$\io(f_1)>2$ is obvious.
Recall from~\S\ref{ss:contractions} that~$\io(f_1) \le 3$.
Also recall the definition of the integer~$\balpha(X)$ from Proposition~\ref{prop:pic}.

\begin{proposition}
\label{prop:if3}
If~$\io(f_1) = 3$ or~$\io(f_2) = 3$ then~$\io(X) = 3$ and~$X$ is a nodal quadric in~$\PP^4$.
In this case~$\balpha(X) = 1$, $X_1$ and~$X_2$ are isomorphic to~$\PP_{\PP^1}(\cO \oplus \cO(-1)^{\oplus 2})$,
$f_1$ and~$f_2$ are the projections to~$\PP^1$,
while~$\hX \cong \PP_{\PP^1 \times \PP^1}(\cO \oplus \cO(-1,-1))$ is a Fano threefold of type~\typemm{3}{31}.
\end{proposition}

\begin{proof}
We may assume that~$\io(f_1) = 3$.
The description of~\S\ref{ss:contractions} shows that~$f_1$ is of type~\type{(D_3)}, 
i.e., it is a $\PP^2$-bundle over~$\PP^1$, hence
\begin{equation*}
X_1 \cong \PP_{\PP^1}(\cO \oplus \cO(-a) \oplus \cO(-b)) \xrightarrow{\ f_1\ } \PP^1,
\end{equation*}
where~$a,b \ge 0$.
It is easy to see that the Mori cone~$\NE(X_1)$ is generated by 
the class~$\Upsilon_1$ of a line in a fiber of~$f_1$,
and by the section of~$f_1$ corresponding to the summand~$\cO$.
Comparing this with the description of~$\NE(X_1)$ in Lemma~\ref{lem:eff-mov}, 
we see that this section coincides with the flopping curve~$C_1$.
Now it follows that
\begin{equation*}
\cN_{C_1/X_1} \cong \cO(-a) \oplus \cO(-b),
\end{equation*}
hence~\eqref{eq:cn-ci} implies that~$a = b = 1$.
The rest of the proposition is obvious.
\end{proof}

The other cases are also well-known, see, e.g., 
\cite[Theorem~1.2]{KP23}, \cite[Theorem~1.2]{P:V22}, and~\cite[Theorem~2.3]{Takeuchi:DP}, respectively;
see also~\cite[(2.3.2), (2.13.1), and~(2.11.2)]{Takeuchi:DP}.
Recall from~\S\ref{ss:conventions} our convention about the relative hyperplane class of a projective bundle.

\begin{proposition}
\label{prop:big-index}
If~$\io(f_1) = 2$ there are three cases:
\begin{enumerate}[wide]
\item 
\label{it:ix2}
$\io(X) = 2$ and~$\dd(X) = 5$.
In this case~$X_1 \cong \PP_{\PP^2}(\cE)$, $f_1$ is the projection to~$\PP^2$, so that~$\io(f_1) = 2$,
and~$\cE$ is the vector bundle defined by an exact sequence
\begin{equation}
\label{eq:i2d5}
0 \longrightarrow \cE \longrightarrow \cO(-1)^{\oplus 2} \longrightarrow \cO_{L_1}(1) \longrightarrow 0,
\end{equation}
where~$L_1 \subset \PP^2$ is a line.
Moreover, $X_2 \subset \PP_{\PP^1}(\cO \oplus \cO(-1)^{\oplus 3})$ is a divisor of class~\mbox{$2H - H_2$}, 
where~$H$ is the relative hyperplane class;
in particular~$f_2$ is a quadric surface bundle over~$\PP^1$ and~\mbox{$\io(f_2) = 2$}.
Finally, $\hX$ is a smooth Fano threefold of type~\typemm{3}{21}.

\item 
\label{it:12d}
$\io(X) = 1$ and~$\g(X) = 12$.
In this case~$X_1 \cong \PP_{\PP^2}(\cE)$, $f_1$ is the projection to~$\PP^2$,
and~$\cE$ is the stable vector bundle defined by an exact sequence
\begin{equation}
\label{eq:ce-12d}
0 \longrightarrow \cE \longrightarrow \cO^{\oplus 2} \longrightarrow \cO_{\Gamma_1}(5) \longrightarrow 0,
\end{equation}
where~$\Gamma_1 \subset \PP^2$ is a conic.
Moreover, $f_2 \colon X_2 \to \PP^1$ is a quintic del Pezzo fibration and therefore~\mbox{$\io(f_2) = 1$}.
Finally, $\hX$ is a smooth Fano threefold of type~~\typemm{3}{5}.

\item 
\label{it:9b}
$\io(X) = 1$ and~$\g(X) = 9$.
In this case
\begin{equation*}
X_1 \subset \PP_{\PP^1}(\cO \oplus \cO \oplus \cO(-1) \oplus \cO(-1))
\end{equation*}
is a divisor of class~$2M_1 + H_1$, where~$M_1$ is the relative hyperplane class
and~$f_1$ is the projection to~$\PP^1$.
Moreover, $f_2 \colon X_2 \to \PP^1$ is a quartic del Pezzo fibration and~\mbox{$\io(f_2) = 1$}.
Finally, $\hX$ is a smooth Fano threefold of type~~\typemm{3}{2}.
\end{enumerate}

In all these cases we have~$\balpha(X) = 1$ and~$H$ is base point free.
\end{proposition}

\begin{proof}
Since~$\io(X)$ divides~$\io(f_1)$ by~\eqref{eq:io-fi}, we have~$\io(X) \in \{1,2\}$.

Assume~$\io(X) = 2$, so that~$X$ is a del Pezzo threefold.
Note that~$\dd(X) \le 5$ by~\eqref{eq:smoothing-invariants}.
Using the classification~\cite[Corollary~0.8]{Shin1989} or~\cite[Theorem~1.2]{KP23} 
and Corollary~\ref{cor:ci-factorial}\ref{it:ci-i2} we see that nonfactoriality of~$X$ implies~$\dd(X) = 5$.
Then we apply~\cite[Theorem~7.1, Proposition~4.14, and Lemma~4.16]{KP23} and obtain part~\ref{it:ix2}.

From now on we assume~$\io(X) = 1$.
Since~$\io(f_1) = 2$, it follows from~\eqref{eq:deg-ci} that
\begin{equation}
\label{eq:9b-h1c1}
H_1 \cdot C_1 = \tfrac{\io(f_1)}{\io(X)} = 2,
\end{equation}
where recall that~$C_1 \subset X_1$ is the flopping curve.
Furthermore, since~$f_1$ cannot be of type~\type{(B_2)} by Lemma~\ref{lem:no-blowup-smooth-point},
it should be of type~\type{(C_2)} or~\type{(D_2)}, i.e., 
a $\PP^1$-bundle over~$\PP^2$ or a quadric surface bundle over~$\PP^1$.
We consider these cases separately.

\subsection*{Type~\type{(C_2)}.}
Assume~$X_1$ is a $\PP^1$-bundle over~$\PP^2$, i.e., $X_1 \cong \PP_{\PP^2}(\cE)$,
where~$\cE$ is a vector bundle of rank~$2$.
Then~$\hh(X_1) = 0$, hence~$\hh(X^\sm) = 0$ by Proposition~\ref{prop:hh-g-xsm},
and since~$\io(X) = 1$, we have~$\g(X) = \g(X^\sm) = 12$, see Table~\ref{table:hodge-genus}. 

Since~$\io(f_1) = 2$ and~$\io(X) = 1$, twisting~$\cE$ appropriately,
we may assume that
\begin{equation}
\label{eq:hmh-c2}
H \sim 2M_1 + H_1,
\end{equation}
where~$M_1$ is the relative hyperplane class for~$\PP_{\PP^2}(\cE)$ and~$H_1$ is the pullback of a line class from~$\PP^2$.
Using the canonical class formula, we conclude that~$\rc_1(\cE) = -2H_1$.

On the other hand, applying~\eqref{eq:9b-h1c1}, we see that the curve
\begin{equation*}
\Gamma_1 \coloneqq f_1(C_1) \subset \PP^2 
\end{equation*}
is a line or a smooth conic.
Moreover, the flopping curve~$C_1 \subset \PP_{\Gamma_1}(\cE\vert_{\Gamma_1})$ 
must be a rigid curve in the scroll~$\PP_{\Gamma_1}(\cE\vert_{\Gamma_1})$, 
hence it is an exceptional section of~$ \PP_{\Gamma_1}(\cE\vert_{\Gamma_1}) \to \Gamma_1$;
in particular, the map~$C_1 \to \Gamma_1$ is an isomorphism, hence~$\Gamma_1$ is a smooth conic.

Furthermore, consider the standard exact sequence of normal bundles
\begin{equation*}
0 \longrightarrow \cN_{C_1/\PP_{\Gamma_1}(\cE\vert_{\Gamma_1})} \longrightarrow \cN_{C_1/X_1} \longrightarrow \cO_{\PP_{\Gamma_1}(\cE\vert_{\Gamma_1})/X_1}\vert_{C_1} \longrightarrow 0.
\end{equation*}
The last term is isomorphic to the pullback of~$\cN_{\Gamma_1/\PP^2} \cong \cO_{\Gamma_1}(4)$, i.e., to~$\cO_{C_1}(4)$,
and the middle term is identified in~\eqref{eq:cn-ci},
hence the first term is~$\cO_{C_1}(-6)$.
On the other hand, a combination of~\eqref{eq:9b-h1c1} with~\eqref{eq:hmh-c2} implies that~$M_1 \cdot C_1 = -1$, hence
\begin{equation*}
\cE\vert_{\Gamma_1} \cong \cO_{\Gamma_1}(1) \oplus \cO_{\Gamma_1}(-5)
\end{equation*}
and~$C_1$ is the exceptional section of~$\PP_{\Gamma_1}(\cE\vert_{\Gamma_1})$. 

Now consider the composition of the epimorphisms~$\cE^\vee \twoheadrightarrow \cE^\vee\vert_{\Gamma_1} \twoheadrightarrow \cO_{\Gamma_1}(-1)$.
Its kernel is a vector bundle~$\cF$ on~$\PP^2$ and we have an exact sequence
\begin{equation}
\label{eq:cf-cevee-c2}
0 \longrightarrow \cF \longrightarrow \cE^\vee \longrightarrow \cO_{\Gamma_1}(-1) \longrightarrow 0.
\end{equation}
Using the computation of Chern classes of~$\cE$ in~\cite[Lemma~5.7]{P:V22}
(note, however, that the twist of the bundle is different), 
we obtain~$\rc_1(\cF) = \rc_2(\cF) = 0$,
and using the stability of~$\cE$ (also proved in~\cite[Lemma~5.7]{P:V22}), 
we obtain~$h^0(\cE(-3)) = 0$. 
Therefore~ by Serre duality~$h^2(\cE^\vee) = 0$, hence~$h^2(\cF) = 0$ by~\eqref{eq:cf-cevee-c2}.
Finally, applying to~$\cF$ the Riemann--Roch Theorem we deduce~$h^0(\cF) \ge 2$. 
Consider a general morphism
\begin{equation*}
\cO_{\PP^2}^{\oplus 2} \longrightarrow \cF.
\end{equation*}
The composition~$\cO_{\PP^2}^{\oplus 2} \longrightarrow \cF \longrightarrow \cE^\vee$ 
is injective by stability of~$\cE^\vee$, hence the morphism~$\cO_{\PP^2}^{\oplus 2} \longrightarrow \cF$ is also injective.
Since the Chern classes of the source and target are the same, it must be an isomorphism.
Thus, $\cF \cong \cO_{\PP^2}^{\oplus 2}$, and dualizing~\eqref{eq:cf-cevee-c2}, we obtain~\eqref{eq:ce-12d}.

Now we describe~$\hX$ and~$X_2$.
Sequence~\eqref{eq:ce-12d} implies that the~$\PP^1$-bundles~$X_1 \cong \PP_{\PP^2}(\cE)$ and~$\PP_{\PP^2}(\cO^{\oplus 2}) \cong \PP^2 \times \PP^1$
are related by an elementary transformation.
More precisely,
\begin{equation*}
\hX \coloneqq \Bl_{C_1}(X_1) \cong \Bl_{C'_1}(\PP^2 \times \PP^1),
\end{equation*}
where~$C'_1 \subset \PP^2 \times \PP^1$ is a smooth rational curve of bidegree~$(2,5)$ 
that corresponds to the morphism~$\cO^{\oplus 2} \to \cO_{\Gamma_1}(5)$ in~\eqref{eq:ce-12d}.
Therefore, $\hX$ is the smooth Fano threefold of type~\typemm{3}{5}.

Next we note that the composition
\begin{equation*}
\hat{f}_2 \colon \hX \cong \Bl_{C'_1}(\PP^2 \times \PP^1) \longrightarrow \PP^2 \times \PP^1 \longrightarrow \PP^1,
\end{equation*}
where the last arrow is the second projection, 
factors as~\mbox{$\hX \xrightarrow{\ \sigma_2\ } X_2 \xrightarrow{\ f_2\ } \PP^1$}.
Indeed, the definition of elementary transformation implies that the exceptional divisor~$E$ of~$\sigma_2$ 
is equal to the strict transform of~$\Gamma_1 \times \PP^1 \subset \PP^2 \times \PP^1$,
and the map~$\hat{f}_2$ restricts to it as the second projection~$\Gamma_1 \times \PP^1 \to \PP^1$, 
i.e., it is equal to the restriction of~$\sigma_2$, 
hence the required factorization. 

Furthermore, note that the fibers of~$\hat{f}_2 \colon \hX \to \PP^1$ 
are isomorphic to~$\PP^2$ blown up in~$5$ points (the fiber of~$C'_1$),
i.e., the general fiber is a del Pezzo surface of degree~4.
Since~$\io(f_2) \in \{1,2\}$, the degree of the flopping curve~$C_2 \subset X_2$ over~$\PP^1$ is at most~2 by~\eqref{eq:deg-ci},
hence~$f_2$ is a del Pezzo fibration of degree~$5$ or~$6$.
This means that in fact~$\io(f_2) = 1$, hence the degree of~$C_2$ over~$\PP^1$ is~1, again by~\eqref{eq:deg-ci},
and therefore~$f_2$ is a quintic del Pezzo fibration.

Finally, the rational map~$X_1 \dashrightarrow X_2 \xrightarrow{\ f_2\ } \PP^1$ is equal to the composition
\begin{equation*}
\PP_{\PP^2}(\cE) \dashrightarrow \PP_{\PP^2}(\cO^{\oplus 2}) \cong \PP^2 \times \PP^1 \longrightarrow \PP^1 
\end{equation*}
hence we have a relation~$H_2 \sim M_1 - E$ in~$\Pic(\hX)$.
It can be rewritten as
\begin{equation*}
(2M_1 + H_1) - 2E \sim H_1 + 2H_2
\end{equation*}
and since~$2M_1 + H_1 \sim H = -K_{X_1}$ by~\eqref{eq:hmh-c2}, it follows from~\eqref{eq:mkx} that~$\balpha(X) = 1$.
It is also clear that the anticanonical class~$H$ of~$X$ is base point free,
because so is the anticanonical class of~$\hX$ 
(see~\cite[Ch.~I, \S6]{Isk:Anti} or~\cite[Theorem~2.4.5]{IP99}).

\subsection*{Type~\type{(D_2)}} 
Assume~$f_1$ is 
a quadric surface bundle over~$\PP^1$, i.e.,
\begin{equation*}
X_1 \subset \PP_{\PP^1}(\cE)
\end{equation*}
is a divisor of relative degree~$2$ and~$\cE$ is a vector bundle of rank~4.
By~\eqref{eq:9b-h1c1} the flopping curve~$C_1 \subset X_1$ is a bisection of the projection~$f_1 \colon X_1 \to \PP^1$.
Moreover, the argument of Proposition~\ref{prop:pic} shows that
if~$M_1$ is an ample generator of~$\Pic(X_1/\PP^1)$, then~$M_1 \cdot C_1$ is coprime to~$H_1 \cdot C_1 = 2$, hence odd, 
so we can choose~$M_1$ in such a way that
\begin{equation}
\label{eq:9b-m1c1}
M_1 \cdot C_1 = 1.
\end{equation}
Since~$H^\bullet(C_1,\cO_{C_1}(-M_1)) = 0$, it follows that~$f_{1*}\cO_{C_1}(-M_1)$ 
is an acyclic vector bundle of rank~$2$, hence isomorphic to~$\cO(-1)^{\oplus 2}$
and the projection formula implies~$f_{1*}\cO_{C_1}(M_1) \cong \cO^{\oplus 2}$.
Note that it is a rank~$2$ quotient bundle of the rank~$4$ bundle~$f_{1*}\cO_{X_1}(M_1) \cong \cE^\vee$.
Dualizing this epimorphism we obtain an embedding~$\cO \oplus \cO \hookrightarrow \cE$
such that
\begin{equation*}
C_1 \subset \PP_{\PP^1}(\cO \oplus \cO) = \PP^1 \times \PP^1 \subset \PP_{\PP^1}(\cE)
\end{equation*}
is a curve of type~$2M_1 + H_1$.
Since the curve~$C_1$ is rigid in~$X_1$, we see that~$\PP^1 \times \PP^1 \not\subset X_1$, 
hence~$X_1 \cap (\PP^1 \times \PP^1)$ is a curve of type~$2M_1 + tH_1$ with~$t \ge 1$,
and hence~$X_1$ itself is a hypersurface of type~$2M_1 + tH_1$ in~$\PP_{\PP^1}(\cE)$.
We show that~$t = 1$.

Indeed, on the one hand, the equality~$K_{X_1} \cdot C_1 = 0$ implies that
\begin{equation*}
-K_{X_1} = 2M_1 - H_1.
\end{equation*}
This allows us to compute~$2\g(X) - 2 = (-K_{X_1})^3$ by using the intersection theory of~$\PP_{\PP^1}(\cE)$: 
\begin{equation*}
2\g(X) - 2 = 
(2M_1 - H_1)^3 \cdot (2M_1 + tH_1) = 
16 M_1^4 + 8(t - 3)M_1^3\cdot H_1 =
8(t - 3) - 16\deg(\cE).
\end{equation*}
It follows that~$\g(X) = 4(t - 3 - 2\deg(\cE)) + 1$.
In particular, $\g(X) \equiv 1 \pmod 4$, 
and recalling from Table~\ref{table:hodge-genus} the possible values of~$\g(X)$ 
we see that~$\g(X) \in \{5,9\}$.

On the other hand, the discriminant divisor of~$f_1$ 
is equal to the degeneracy locus of the morphism~$\cE \to \cE^\vee(t)$ given by the equation of~$X_1$,
hence its degree is equal to~$4t - 2\deg(\cE)$.
Assuming~$\Bbbk = \CC$, we compute the topological Euler characteristic
\begin{equation*}
\upchi_{\mathrm{top}}(X_1) = 4 \cdot 2 - (4t - 2\deg(\cE)) = 2\deg(\cE) - 4t + 8.
\end{equation*}
Since~$\uprho(X_1) = 2$, we have~$\upchi_{\mathrm{top}}(X_1) = 6 - 2\hh(X_1)$, hence
\begin{equation*}
\hh(X^\sm) = \hh(X_1) = 2t - \deg(\cE) - 1.
\end{equation*}
Looking at Table~\ref{table:hodge-genus} and using~$\g(X^\sm) = \g(X) \in \{5,9\}$, we see that~$\hh(X^\sm) \in \{14,3\}$.
Now the system of linear equations
\begin{equation*}
\left\{
\begin{aligned}
4(t - 3 - 2\deg(\cE)) + 1 &= \g(X^\sm),\\
2t - \deg(\cE) - 1 &= \hh(X^\sm)
\end{aligned}
\right.
\end{equation*}
does not have integral solutions when~$\g(X^\sm) = 5$ and~$\hh(X^\sm) = 14$, 
hence~$\g(X^\sm) = 9$ and~$\hh(X^\sm) = 3$, and solving the system we obtain~$t = 1$
and conclude that
\begin{equation*}
X_1 \sim 2M_1 + H_1.
\end{equation*}

Now we extend the embedding~$\cO \oplus \cO \hookrightarrow \cE$ to an exact sequence of vector bundles
\begin{equation*}
0 \longrightarrow \cO \oplus \cO \longrightarrow \cE \longrightarrow \cE' \longrightarrow 0.
\end{equation*}
Since~$t = 1$, we see that~$C_1 = X_1 \cap \PP_{\PP^1}(\cO \oplus \cO)$ is a transverse intersection.
Therefore,
\begin{equation*}
\cO_{C_1}(-1)^{\oplus 2} \cong 
\cN_{C_1/X_1} \cong 
\cN_{\PP_{\PP^1}(\cO \oplus \cO)/\PP_{\PP^1}(\cE)}\vert_{C_1} \cong 
f_1^*\cE'(M_1)\vert_{C_1}.
\end{equation*}
Using~\eqref{eq:9b-m1c1}, we obtain~$f_1^*\cE'\vert_{C_1} \cong \cO_{C_1}(-2)^{\oplus 2}$, 
and using~\eqref{eq:9b-h1c1}, we deduce~$\cE' \cong \cO_{\PP^1}(-1)^{\oplus 2}$.
Therefore, the exact sequence splits, and we conclude that
\begin{equation*}
\cE \cong \cO \oplus \cO \oplus \cO(-1) \oplus \cO(-1).
\end{equation*}
This completes the description of~$X_1$ and~$f_1$.

Now we describe~$\hX$ and~$X_2$.
Since the curve~$C_1$ is an intersection of~$X_1$ with the subbundle~$\PP_{\PP^1}(\cO \oplus \cO)$, we have
\begin{equation*}
\hX = \Bl_{C_1}(X_1) \subset \Bl_{\PP_{\PP^1}(\cO \oplus \cO)}(\PP_{\PP^1}(\cE)) \cong \PP_{\PP^1 \times \PP^1}(\cO \oplus \cO \oplus \cO(-1,-1))
\end{equation*}
is a divisor of class~$2M + H_1$, 
where~$M$ is the relative hyperplane class of the projective bundle in the right side.
Thus, $\hX$ is a smooth Fano variety of type~\typemm{3}{2}.

Next we note that the composition
\begin{equation*}
\hat{f}_2 \colon \hX \hookrightarrow \PP_{\PP^1 \times \PP^1}(\cO \oplus \cO \oplus \cO(-1,-1)) \longrightarrow \PP^1 \times \PP^1 \longrightarrow \PP^1,
\end{equation*}
where the last arrow is the second projection factors 
as~\mbox{$\hX \xrightarrow{\ \sigma_2\ } X_2 \xrightarrow{\ f_2\ } \PP^1$}.
Indeed, the exceptional divisor~$E$ of~$\sigma_2$ 
is the intersection~$\hX \cap \PP_{\PP^1 \times \PP^1}(\cO \oplus \cO)$,
which is a divisor of degree~$(2,1,0)$ in~$(\PP^1)^3$, 
and the restriction of~$\hat{f}_2$ to~$E$ is induced by the projection of~$(\PP^1)^3$ to the last factor,
i.e., it is equal to the restriction of~$\sigma_2$, 
hence the required factorization.

It remains to show that~$f_2 \colon X_2 \to \PP^1$ is a quartic del Pezzo fibration.
For this just note that~$\hX$ is a conic bundle over~$\PP^1 \times \PP^1$ associated with a morphism
\begin{equation*}
\cO \oplus \cO \oplus \cO(-1,-1) \longrightarrow (\cO \oplus \cO \oplus \cO(1,1)) \otimes \cO(1,0),
\end{equation*}
hence its discriminant curve has bidegree~$(5,2)$,
and hence the fibers of~$\hat{f}_2 \colon \hX \to \PP^1$ are conic bundles over~$\PP^1$ with~$5$ degenerate fibers.
Thus, $\hat{f}_2$ is a fibration in cubic del Pezzo surfaces, 
and arguing as in the case of type~\type{(C_2)} we see that~$f_2$ is a quartic del Pezzo fibration.

Finally, 
the rational map~$X_1 \dashrightarrow X_2 \xrightarrow{\ f_2\ } \PP^1$ is equal to the composition
\begin{equation*}
\PP_{\PP^1}(\cE) \dashrightarrow \PP_{\PP^1}(\cE') \cong \PP^1 \times \PP^1 \longrightarrow \PP^1 
\end{equation*}
hence we have a relation~$H_2 \sim M_1 - H_1 - E$ in~$\Pic(\hX)$.
It can be rewritten as
\begin{equation*}
(2M_1 - H_1) - 2E \sim H_1 + 2H_2
\end{equation*}
and since~$2M_1 - H_1 = -K_{X_1} = H$, it follows that~$\balpha(X) = 1$.
It is also clear that the anticanonical class~$H$ of~$X$ is base point free,
because so is the anticanonical class of~$\hX$
(see~\cite[Ch.~I, \S6]{Isk:Anti} or~\cite[Theorem~2.4.5]{IP99}).
\end{proof}

\section{Special factorial and nonfactorial threefolds}
\label{sec:special}

In this section we consider Fano threefolds~$X$ (both factorial and nonfactorial)
such that the anticanonical class~$-K_X$ has base points, 
or not very ample (we will say that~$X$ is {\sf hyperelliptic}),
or the anticanonical embedding of~$X$ is not an intersection of quadrics (we will say that~$X$ is {\sf trigonal}).
These results are well-known in the smooth case
and we show that the picture is the same for factorial threefolds, and quite different in the nonfactorial case.
Besides, we classify all $1$-nodal or $1$-cuspidal prime Fano threefolds~$X$ with~$\g(X) \le 6$.

\subsection{Fano threefolds with non-empty anticanonical base locus}
\label{ss:bs}

Here we identify the unique deformation class of $1$-nodal Fano threefolds with non-empty anticanonical base locus.
Recall the integer~$\balpha(X)$ defined in Proposition~\ref{prop:pic}.

\begin{proposition}[{cf.~\cite[Proposition~2.5]{Jahnke-Peternell-Radloff-II}}]
\label{prop:base-points}
Let~$X$ be a Fano threefold with terminal Gorenstein singularities and~$\uprho(X) = 1$.
If~$\Bs(|-K_X|) \ne \varnothing$ then~$X$ is nonfactorial.

If, moreover, $(X,x_0)$ is $1$-nodal
then~$\Bs(|-K_X|)=\{x_0\}$, $\io(X) = 1$, $\g(X) = 2$, and~$X$ is a complete intersection
\begin{equation*}
X = X_{2,6} \subset \PP(1^4,2,3)
\end{equation*}
of the cone in~$\PP(1^4,2,3)$ over a smooth quadric surface in~$\PP(1^4) = \PP^3$ with vertex~$\PP(2,3)$ and a sextic hypersurface.
In particular, $x_0 = X \cap \PP(2,3)$ and~$\balpha(X) = 1$.
\end{proposition}

\begin{remark}
If~$X$ is the 1-nodal threefold with~$\Bs(|-K_X|) \ne \varnothing$ described in Proposition~\ref{prop:base-points},
the anticanonical linear system of~$\hX = \Bl_{x_0}(X)$ is base point free
and its anticanonical morphism is an elliptic fibration~$\hX \to \PP^1 \times \PP^1 \subset \PP(1^4)= \PP^3$.
\end{remark} 

\begin{proof}
By~\cite[Theorem~0.5]{Shin1989} we have 
\begin{equation*}
\io(X) = 1
\end{equation*}
and
there are two cases: either~$\Bs(|-K_X|) \cong \PP^1$, or~$\Bs(|-K_X|)$ is a reduced point.
Below we show that the first case is impossible (by a rather long though elementary argument),
and after that we identify the second case.
As usual, we set~$H \coloneqq -K_X$ and~$g \coloneqq \g(X)$.

\subsection*{Case where~$\Bs(|-K_X|) \cong \PP^1$.}

We set~$B \coloneqq \Bs(|-K_X|)$.
By~\cite[Theorem~0.5]{Shin1989} and its proof,
the curve~$B$ is contained in the smooth locus of~$X$ 
and a general anticanonical divisor~$S \subset X$ is a smooth K3 surface containing~$B$.
Set~$H_S \coloneqq H\vert_S$.
By Kawamata--Viehweg vanishing theorem~$H^1(X,\cO_X) = 0$, 
hence the schematic base locus of the linear system~$|H_S|$ is also equal to~$B$.
According to~\cite[Proposition~8.1]{SaintDonat1974}, we have
\begin{equation}
\label{base-point-index=1}
H_S \sim gC + B,
\qquad 
\text{where~$B^2 = -2$},
\quad
\text{$B \cdot C = 1$},
\quad
\text{$C^2 = 0$},
\end{equation}
and~$|C|$ is a base point free elliptic pencil on~$S$.
In particular, $H \cdot B = g- 2$, hence
\begin{equation*}
g \ge 3.
\end{equation*}

Consider the anticanonical rational map~$X \dashrightarrow \PP^{g+1}$ and its resolution
\begin{equation*}
\xymatrix{
&
\tX \ar[dl]_{\sigma} \ar[dr]^{\varpi} & 
\\
X \ar@{-->}[rr] &&
W \hbox to 0pt{${}\subset \PP^{g+1}$}
} 
\end{equation*}
where $\sigma \colon \tX \to X$ is the blowup of~$B$ with exceptional divisor~$E_B \subset \tX$,
$\varpi$ is given by the linear system~$|H - E_B|$ on~$\tX$, and~$W \coloneqq \varpi(\tX)$.
Clearly, the strict transform of~$S$ in~$\tX$ is~$\tS \cong S$ and, 
under this identification,~$(H - E_B)\vert_{\tS} \sim gC$,
hence the restriction of the map~$\varpi$ to~$\tS$ is the composition
\begin{equation*}
\tS \longrightarrow \PP^1 \longrightarrow \PP^g \hookrightarrow \PP^{g+1},
\end{equation*}
where the first arrow is the elliptic fibration with fiber~$C$, the second arrow is the Veronese embedding of degree~$g$, 
and the third arrow is the embedding of the hyperplane corresponding to the anticanonical divisor~$S$.
Therefore, $W$ is a surface of minimal degree~$g$
(see~\cite[Theorem~1.10]{SaintDonat1974} or~\cite{EisenbudHarris:VMD}).
Moreover, since~$\uprho(\tX) = \uprho(X) + 1 = 2$, it follows that~$\uprho(W) = 1$, hence~$W$ is not a scroll.
On the other hand, $E_B \cong \PP_B(\cN_{B/X})$ and the class~$H - E_B$ defining the morphism~$\varpi$
restricts to~$E_B$ as a relative hyperplane class over~$B$,
hence the images in~$W$ of the fibers of~$E_B \to B$ are lines on~$W$, hence~$W$ is not a Veronese surface.
Thus, \cite[Theorem~1.10]{SaintDonat1974} shows that~$W$ is a cone over the Veronese curve of degree~$g$. 

Let~$w_0 \in W$ denote the vertex of the cone.
Obviously, in the group~$\Cl(W) = \Pic(W \setminus w_0)$ the class of a hyperplane section of~$W$ 
equals the class of a ruling multiplied by~$g$.
Lifting this relation along the projection~$\tX \to W$, 
we obtain the relation~$H - E_B \sim g\tF + \tD$ in~$\Cl(\tX)$, where~$\tF$ is the strict transform of a ruling
and~$\tD$ is an effective Weil divisor with support at~$\varpi^{-1}(w_0)$.
Pushing forward to~$X$ we obtain the relation
\begin{equation}
\label{eq:hfd}
H \sim gF + D
\qquad\text{in~$\Cl(X)$,}
\end{equation}
where~$F$ is nonzero and movable and~$D$ is effective.

By construction, for a general anticanonical divisor~$S \subset X$ 
we have~\mbox{$F \cap S = C$} and therefore~\mbox{$D \cap S = B$}.
Since~$C$ and~$B$ are linearly independent in~$\Pic(S)$,
it follows that~$F$ and~$D$ are linearly independent in~$\Cl(X)$, hence~$X$ is not factorial.

So, now we assume that~$X$ is nonfactorial and 1-nodal and continue the argument.
Recall from Propositions~\ref{prop:if3} and~\ref{prop:big-index} 
that if~$\io(f_1) > 1$ or~$\io(f_2) > 1$ then~$H$ is base point free.
Thus, in our case we have~$\io(f_1) = \io(f_2) = 1$.

Since~$B$ is an irreducible curve and~$S$ is an ample Cartier divisor, 
the equality~$D \cap S = B$ implies that~$D$ is irreducible (and reduced).
Since~$F$ is movable, the argument of Lemma~\ref{lem:eff-mov} shows that~\mbox{$F \cdot \Upsilon_i \ge 0$} for~$i = 1,2$
and one of these products is positive.
Without loss of generality we may assume~$F \cdot \Upsilon_1 > 0$.
Then a combination of~\eqref{eq:io-fi} and~\eqref{eq:hfd} implies
\begin{equation*}
1 = \tfrac{\io(f_1)}{\io(X)} = H \cdot \Upsilon_1 = g F \cdot \Upsilon_1 + D \cdot \Upsilon_1,
\end{equation*}
hence~$D \cdot \Upsilon_1 = 1 - g F \cdot \Upsilon_1 \le 1 - g$.
Since~$g \ge 3$, we have~$D \cdot \Upsilon_1 \le -2$ and Lemma~\ref{lem:eff-mov} implies 
that~$D = E_1$, the contraction~$f_1$ is of type~\typeb, 
$g = 3$, and~$F \cdot \Upsilon_1 = 1$. 

On the other hand, the inequality~$D \cdot \Upsilon_1 < 0$ 
implies~$D \cdot C_1 > 0$ (because~$C_1$ and~$\Upsilon_1$ generate the Mori cone~$\NE(X_1)$),
therefore, the relation
\begin{equation*}
0 = H \cdot C_1 = g F \cdot C_1 + D \cdot C_1,
\end{equation*}
implies~$F \cdot C_1 < 0$, and hence~$F \cdot C_2 > 0$ because~$X_1 \dashrightarrow X_2$ is the flop of~$C_1$.
If~$F \cdot \Upsilon_2 > 0$, repeating the above argument we obtain a contradiction, hence~$F \cdot \Upsilon_2 = 0$, 
hence~$F$ is proportional to~$H_2$, and since~$F \cdot \Upsilon_1 = 1$, we obtain~$F \sim H_2$, hence~$H_2 \cdot \Upsilon_1 = 1$.
Furthermore, the linear equivalence~\eqref{eq:mkx} implies that
\begin{equation*}
\balpha(X) = \balpha(X) H \cdot \Upsilon_1 = H_1 \cdot \Upsilon_1 + H_2 \cdot \Upsilon_1 = 1.
\end{equation*}
Now, comparing~\eqref{eq:mkx} and~\eqref{eq:hfd} with the linear equivalences~$F \sim H_2$ and~$D \sim E_1$ 
and the equality~$g = 3$,
we obtain~$H_1 + H_2 \sim H \sim 3H_2 + E_1$,
hence~$H_2 \sim \tfrac12(H_1 - E_1)$, hence
\begin{equation*}
H \sim \tfrac12(3H_1 - E_1).
\end{equation*}
Note that~$H_1^2 \cdot E_1 = H_1 \cdot E_1^2 = 0$
and~$E_1^3 = 4$ by Lemma~\ref{lem:intersection}.
Thus, multiplying the above equivalence by~$2$ and taking its top self-intersection, we obtain 
\begin{equation*}
8(2\g(X)-2)= (2H)^3= (3H_1-E_1)^3= 27 H_1^3 -4.
\end{equation*}
It follows that~$4(4\g(X)-3) = 27 H_1^3$, hence~$4\g(X) - 3$ must be divisible by~$27$.
Looking at Table~\ref{table:hodge-genus} we see that this does not happen when~$\io(X) = 1$;
thus, the case where~\mbox{$\Bs(|-K_X|) \cong \PP^1$} is not possible.

\subsection*{Case where~$\Bs(|-K_X|)$ is a reduced point.}
In this case, again by~\cite[Theorem~0.5]{Shin1989} we have~$\Bs(|-K_X|) \eqqcolon x_0$ is a singular point of~$X$
and a general divisor~\mbox{$S\in |-K_X|$}
is a K3 surface with a single ordinary double point at~$x_0$.
Let~$\tS$ be the blowup of~$S$ at~$x_0$, and let~$B \subset \tS$ be the 
exceptional~$(-2)$-curve.
Then the pullback to~$\tS$ of the linear system~$|H|$ has~$B$ as the base locus, 
hence~\eqref{base-point-index=1} still holds, again by~\cite[Proposition~8.1]{SaintDonat1974}.
This time, however, we have~$H \cdot B = 0$ by construction, which means that
\begin{equation*}
g = 2,
\end{equation*}
hence~$H_{\tS} \sim 2C + B$.

Furthermore, applying~\cite[Proposition~5.7(i)]{SaintDonat1974} 
to the linear system~\mbox{$|2H_{\tS}| = |4C + 2B|$} on~$\tS$,
we conclude that the surface~$S$ is a double covering of a cone over a rational normal twisted quartic curve,
and the Riemann--Hurwitz formula implies that the covering is ramified 
over a sextic hypersurface and the vertex of the cone.
In other words, $S$ can be realized as a complete intersection in~$\PP(1^3,2,3)$ 
of a quadric (independent of the coordinates of degree~$2$ and~$3$) and a sextic.
Since~$H^0(X, \cO_X(-nK_X)) \to H^0(X, \cO_S(-nK_X))$ is surjective for~$n\ge 0$, 
the standard extension arguments
(see, e.g., \cite[Theorem~3.6]{Mori:WPS} or~\cite[Lemmas~2.9--2.10]{Iskovskih1977})
show that~$X$ is an intersection of a hypersurface of degree~$2$
(independent of the coordinates of degree~$2$ and~$3$, as before) 
and~$6$ in $\PP(1^4,2,3)$. 
Since~$X$ has terminal and hence isolated singularities,
the quadratic equation of~$X$ considered as a polynomial in variables of degree~$1$, must be nondegenerate.
Thus, $X$ is a complete intersection of the cone in~$\PP(1^4,2,3)$ with vertex~$\PP(2,3)$
over a smooth quadric surface~$\PP^1 \times \PP^1 \subset \PP^3$
and a sextic hypersurface in~$\PP(1^4,2,3)$.

The preimage in~$X$ of a ruling of~$\PP^1 \times \PP^1$ is not proportional to the hyperplane class,
hence~$X$ is not factorial, hence we may assume that~$X$ is 1-nodal and that~$x_0$ is the node of~$X$.
Then the projections~$X \dashrightarrow \PP^1 \times \PP^1 \to \PP^1$ 
lift to $K$-negative contractions~$\hX = \Bl_{x_0}(X) \to \PP^1$, 
hence the divisor classes~$H_1$ and~$H_2$ on~$\hX$ are the pullbacks of the point class in~$\PP^1$ under these maps.
The relation~$H \sim H_1 + H_2$ in the class group of the cone 
implies a relation of the same form in~$\Cl(X)$, hence~$\balpha(X) = 1$.
\end{proof} 

Using the above result, it is easy to classify all $1$-nodal prime Fano threefolds of genus~$2$.

\begin{lemma}
\label{lem:g2}
Let~$X$ be a Fano threefold with a single node or cusp, $\uprho(X) = 1$, $\io(X) = 1$, and~$\g(X) = 2$.
Then either 
\begin{enumerate}
\item 
\label{it:g2-f}
$X$ is factorial and~$X = X_6 \subset \PP(1^4,3)$ is a sextic hypersurface, or
\item 
\label{it:g2-nf}
$X$ is 
is a nonfactorial threefold from Proposition~\xref{prop:base-points}.
\end{enumerate}
\end{lemma}

\begin{proof}
If~$\Bs(|-K_X|) \ne \varnothing$ we apply Proposition~\ref{prop:base-points} and obtain case~\ref{it:g2-nf}.

Otherwise, if~$\Bs(|-K_X|) = \varnothing$, the anticanonical morphism is a double covering~\mbox{$X\to \PP^3$}.
Its branch divisor is a sextic surface in~$\PP^3$ by the Riemann--Hurwitz formula,
hence~$X$ is a sextic hypersurface in~$\PP(1^4,3)$.
In this case~$X$ is factorial by Corollary~\ref{cor:ci-factorial}.
\end{proof}

\subsection{Hyperelliptic Fano threefolds}
\label{ss:he}

Recall that a Fano threefold~$X$ is called {\sf hyperelliptic} if~$\Bs(|-K_X|) = \varnothing$, but~$-K_X$ is not very ample. 

\begin{proposition}[cf.~{\cite[Theorem~4.2]{P:ratF-1}} or~{\cite[Proposition~2.7]{Jahnke-Peternell-Radloff-II}}]
\label{prop:hyperelliptic}
Let~$X$ be a hyperelliptic Fano threefold with terminal Gorenstein singularities and~$\uprho(X) = 1$.

If~$X$ is factorial then
we have one of the following situations:
\begin{enumerate}
\item 
\label{he:i2}
$\io(X) = 2$, $\dd(X) = 1$ and~$X = X_6 \subset \PP(1^3,2,3)$; 
\item 
\label{he:i1-g2}
$\io(X) = 1$, $\g(X) = 2$ and~$X = X_6 \subset \PP(1^4,3)$;
\item 
\label{he:i1-g3}
$\io(X) = 1$, $\g(X) = 3$ and~$X = X_{2,4} \subset \PP(1^5,2)$. 
\end{enumerate}

If~$X$ is nonfactorial and $1$-nodal with node~$x_0$ then
\begin{enumerate}[resume*]
\item 
\label{he:nf}
$\io(X) = 1$, $\g(X) = 4$, and~$\balpha(X) = 1$.
\end{enumerate}
In case~\ref{he:nf} the small resolutions~$X_i$ of~$X$ and the anticanonical model~$\bar{X}$ of~$\hX$ are:
\begin{itemize}[wide]
\item 
$X_1 \to \PP^1$ is a degree~$2$ del Pezzo fibration and the relative anticanonical map
\begin{equation*}
X_1 \to Y_1 \coloneqq \PP_{\PP^1}(\cO \oplus \cO(-1) \oplus \cO(-2)) 
\end{equation*}
is the double covering ramified over a divisor of class~$4H - 2H_1$,
\item 
$X_2 = \Bl_{z_2}(Z_2)$ is the blowup of a del Pezzo threefold~$Z_2$ of degree~$1$
with one node or generalized cusp~$z_2$.
\item 
$\bar{X}$ is a smoothable Fano threefold with~$\uprho(\bar{X}) = 2$, $\io(\bar{X}) = 1$, 
and~\mbox{$\g(\bar{X}) = 3$} of type~\typemm{2}{1}
with non-isolated canonical Gorenstein singularities.
\end{itemize}
\end{proposition}

Cases~\ref{he:i2}--\ref{he:i1-g3} are classical, 
see, e.g., \cite[Theorem~1.5]{Przhiyalkovskij-Cheltsov-Shramov}, types~$\mathbf{H}_1$, $\mathbf{H}_2$, and~$\mathbf{H}_3$. 
Case~\ref{he:nf} can be found in~{\cite[Example~4.3]{P:ratF-1}};
it is of type~$\mathbf{H}_5$ in~\cite[Theorem~1.5]{Przhiyalkovskij-Cheltsov-Shramov}.

\begin{proof}
If~$\io(X) \ge 3$ then~$-K_X$ is very ample by Theorem~\ref{hm:KO}, so we assume~$\io(X) \le 2$.

Since~$-K_X$ is base point free, it defines a morphism
\begin{equation*}
\phi \colon X \longrightarrow Y \subset \PP^{g+1}, 
\end{equation*}
where~$g \coloneqq \g(X)$.
This morphism is finite because~$-K_X$ is ample, hence~$\dim(Y) = 3$, and, since~$-K_X$ is not very ample, 
we have~$\deg(\phi) \ge 2$, hence~$\deg(Y) \le \tfrac12\deg(X) = g - 1$.
Therefore, $Y$ is a variety of minimal degree~$g - 1$ and~$\deg(\phi) = 2$.
According to \cite[Theorem~1.10]{SaintDonat1974} or \cite{EisenbudHarris:VMD}
the variety $Y$ is either $\PP^3$ (if~$g = 2$), or a quadric $Q\subset \PP^4$ (if~$g = 3$), 
or~$\PP(1^3,2)$ embedded into~$\PP^6$ by the ample generator of the Picard group
(if~$g = 5$), 
or~$g \ge 4$ and~$Y$ is either a 3-dimensional smooth rational scroll, 
or the cone over a 2-dimensional smooth rational scroll, 
or the double cone over a rational twisted curve. 
In the first three cases the double covering~$X$ of~$Y$ 
can be realized as a complete intersection in a weighted projective space. 
These cases give varieties of types~\ref{he:i1-g2}, \ref{he:i1-g3}, and~\ref{he:i2}, respectively.

If~$Y$ is a smooth rational scroll then~$X$ has a contraction to a curve
and so~$\uprho(X) \ge 2$ which contradicts our assumption. 

If~$Y$ is a double cone over a rational twisted curve, i.e., $Y = \PP(1,1,g-1,g-1)$, 
then the Riemann--Hurwitz formula implies that~$X \subset \PP(1,1,g-1,g-1,g+1)$ is hypersurface of degree~$2g + 2$.
Since~$g \ge 4$, it has non-isolated, hence not terminal singularities, 
again in contradiction with our assumptions.

Finally, consider the case where~$g \ge 4$ and~$Y$ is the cone over a 2-dimensional smooth rational scroll of degree~$g - 1 \ge 3$. 
In this case~$Y$ has a small resolution 
\begin{equation*}
Y_1 \coloneqq \PP_{\PP^1}(\cO \oplus \cO(-a) \oplus \cO(-b)) \longrightarrow Y,
\qquad
b \ge a \ge 1,
\qquad 
a + b = g - 1 \ge 3,
\end{equation*}
which induces a commutative diagram
\begin{equation*}
\xymatrix{
X_1\ar[r]^{\phi_1}\ar[d] & Y_1\ar[d]\ar[r] & \PP^1
\\
X\ar[r]^{\phi} & Y
} 
\end{equation*}
where~$X_1$ is the normalization of~$X\times_{Y} Y_1$.
Then~$X_1\to X$ is a small birational morphism from a normal variety.
Therefore, $X$ is not factorial.

So, from now on we assume that~$X$ is nonfactorial and 1-nodal and continue the argument.
Note that in this case Remark~\ref{rem:xi-bu} shows that the morphism~$X_1 \to X$ is a small resolution;
in particular, $X_1$ is smooth.

The map~$\phi_1$ in the above diagram is a double covering by construction, 
and by the Riemann--Hurwitz formula it is branched over a divisor~$B_1 \subset Y_1$ of class~$4H_Y - 2(a + b - 2)H_1$,
where~$H_Y$ is the relative hyperplane class of~$Y_1 \to \PP^1$ and~$H_1$ is the pullback of the point class from~$\PP^1$.
Thus, the composition 
\begin{equation*}
f_1 \colon X_1\longrightarrow Y_1\longrightarrow \PP^1 
\end{equation*}
is a del Pezzo fibration of degree $2$.
In particular, $\io(f_1) = 1$, and hence~$\io(X) = 1$ by~\eqref{eq:deg-ci}.

It also follows that the flopping curve~$C_1 \subset X_1$
is contained in the preimage of the exceptional section of~$Y_1 \to \PP^1$.
On the other hand, by~\eqref{eq:deg-ci} we have~\mbox{$H_1 \cdot C_1 = 1$}, 
hence the curve~$C_1$ must be contained in the ramification divisor~$R_1 \subset X_1$ of the double covering~$\phi_1$.
Since~$X_1$ and~$Y_1$ are smooth, $R_1$ is also smooth, hence we have an exact sequence
\begin{equation*}
0 \longrightarrow \cN_{C_1/R_1} \longrightarrow \cN_{C_1/X_1} \longrightarrow \cN_{R_1/X_1}\vert_{C_1} \longrightarrow 0.
\end{equation*} 
By~\eqref{eq:cn-ci} the middle term is~$\cO(-1) \oplus \cO(-1)$,
and since~$2R_1$ is linearly equivalent to the pullback of the branch divisor~$B_1$, the last term is~$\cO(2 - a - b)$.
Therefore, $2 - a - b \ge -1$, hence~$a + b \le 3$.
On the other hand, $a + b = g - 1 \ge 3$. 
Finally, $b \ge a \ge 1$, hence
\begin{itemize}[wide]
\item 
$a = 1$, $b = 2$, $B_1 \sim 4H - 2H_1$, and~$g = 4$.
\end{itemize}
This is precisely case~\ref{he:nf} of the proposition. 

It remains to describe case~\ref{he:nf} in detail.
For this we consider the weighted projective space~$\PP(1^3,2) \subset \PP^6$ 
(i.e., the cone over the Veronese surface in~$\PP^5$) 
and a smooth point~\mbox{$p_0 \in \PP(1^3,2)$}. 
Then the linear projection out of~$p_0$ defines a regular morphism
\begin{equation*}
Y_2 \coloneqq \Bl_{p_0}(\PP(1^3,2)) \longrightarrow Y.
\end{equation*}
Its exceptional locus is the ruling of the Veronese cone passing through~$p_0$; 
in particular, the morphism~$Y_2 \to Y$ is small.
Thus, we have the following toric commutative diagram
\begin{equation*}
\xymatrix{
& Y_1 \ar[dl] \ar[dr] \ar@{<-->}[rr] &&
Y_2 \ar[dl] \ar[dr]
\\
\PP^1 &&
Y &&
\hbox to 2em {$\PP(1^3,2),$\hss}
}
\end{equation*}
where the dashed arrow is the \emph{Francia's flip}
and the inner diagonal arrows are small.

Note that the class group~$\Cl(Y_1)$ is generated by the classes~$H_Y$ and~$H_1$ defined above,
while~$\Cl(Y_2)$ is generated by the pullback~$H_2$ of the ample generator of~$\Cl(\PP(1^3,2))$ 
and the class~$E_2$ of the exceptional divisor of the blowup~$Y_2 \to \PP(1^3,2)$.
Moreover, under the isomorphism~$\Cl(Y_1) \cong \Cl(Y) \cong \Cl(Y_2)$ induced by the flip, we have
\begin{equation}
\label{eq:cl-y-g4}
H_1 \sim H_2 - E_2
\qquad\text{and}\qquad
H_Y \sim 2H_2 - E_2,
\end{equation}
hence the strict transform of the divisor~$B_1 \subset Y_1$ under the flip is a divisor~$B_2 \subset Y_2$ of class
\begin{equation*}
B_2 \sim 4(2H_2 - E_2) - 2(H_2 - E_2) \sim 6H_2 - 2E_2.
\end{equation*}

Now let~$X_2$ be the normalization of~$X \times_Y Y_2$.
As before, Remark~\ref{rem:xi-bu} implies that~$X_2 \to X$ is the second small resolution of~$X$
(in particular, $X_2$ is smooth),
and~$\phi_2 \colon X_2 \to Y_2$ is a double covering. 
Since~$X_2$ is smooth, the ramification locus of~$\phi_2$ is the fixed locus of the covering involution,
and since~$Y_2$ has an isolated singularity of type~$\tfrac12(1,1,1)$, 
the ramification locus is the union of a (reduced) point and a smooth surface,
hence~$\phi_2$ is branched at the vertex of~$\PP(1^3,2)$ and at the divisor~$B_2 \subset Y_2$.
As we checked above, $B_2$ is the strict transform of a sextic hypersurface in~$\PP(1^3,2)$ singular at~$p_0$.
In particular, the Stein factorization of the morphism~$X_2 \to \PP(1^3,2)$
is given by the double covering~\mbox{$Z_2 \to \PP(1^3,2)$} branched at the vertex and at a sextic hypersurface.
Thus, $Z_2$ is a Gorenstein del Pezzo threefold of degree~$1$ 
and~$f_2 \colon X_2 \to Z_2$ is an extremal contraction, which must be of type~\type{(B_1^0)}.

Pulling back relations~\eqref{eq:cl-y-g4} to~$X$ we see that~$H \sim H_1 + H_2$ in~$\Cl(X)$, i.e., $\balpha(X) = 1$.

Finally, using the maps~$f_1 \circ \sigma_1 \colon \hX \to \PP^1$ and~$f_2 \circ \sigma_2 \colon \hX \to Z_2$, 
we obtain the map
\begin{equation*}
\xi \colon \hX \longrightarrow \PP^1 \times Z_2.
\end{equation*}
It is birational onto its image, because its component~$f_2 \circ \sigma_2 \colon \hX \to Z_2$ is.
On the other hand, by the K\"unneth formula~$\dim H^0(\PP^1 \times Z_2, \cO(H_1 + H_2)) = 2 \cdot 3 = 6$,
while on~$\hX$ we have~$H_1 + H_2 \sim H - E$ by~\eqref{eq:mkhx}, and~$\dim H^0(\hX, \cO(H - E)) = \g(X) + 1 = 5$,
therefore 
\begin{equation*}
\bar{X} \coloneqq \xi(\hX) \subset \PP^1 \times Z_2
\end{equation*}
is a hyperplane section of~$\PP^1 \times Z_2$, i.e., it is a smoothable Fano variety of type~\typemm{2}{1}.

Finally, note that the intersection of the branch divisor~$B_1$ of~$X_1 \to Y_1$
with the subscroll~\mbox{$\PP_{\PP^1}(\cO \oplus \cO(-1)) \subset Y_1$} 
has multiplicity~$2$ along the exceptional section, 
hence the preimage of this subscroll in~$X_1$ is a nonnormal surface~$S \subset X_1$ 
singular along the flopping curve~\mbox{$C_1 \subset X_1$},
its strict transform~$\hS \subset \hX \coloneqq \Bl_{x_0}(X) \cong \Bl_{C_1}(X_1)$ is a sextic del Pezzo surface with Du Val singularities, 
and the restriction of the anticanonical contraction~\mbox{$\xi \colon \hX \to \bar{X}$} to~$\hS$ is a conic bundle. 
Thus, $\bar{X}$ is singular along a smooth rational curve.
\end{proof}

\begin{remark}
\label{rem:sing-g4}
As we will see later, the family of varieties~$X$ from Proposition~\textup{\ref{prop:hyperelliptic}}\ref{he:nf}
is the only deformation family of 1-nodal nonfactorial Fano threefolds with~$\uprho(X) = 1$
such that the variety~$\bar{X}$ for every member of the family has non-isolated singularities.
\end{remark}

Now we can classify all $1$-nodal or $1$-cuspidal prime Fano threefolds of genus~$3$ and~$4$.

\begin{lemma}
\label{lem:g34}
Let~$X$ be a Fano threefold with a single node or cusp, $\uprho(X) = 1$, $\io(X) = 1$, and~\mbox{$\g(X) \in \{3,4\}$}.
Then either 
\begin{enumerate}
\item 
\label{it:g3-f}
$X = X_{2,4} \subset \PP(1^5,2)$, $\g(X) = 3$, and~$X$ is factorial, or
\item 
\label{it:g4-f}
$X = X_{2,3} \subset \PP^5$, $\g(X) = 4$, and~$X$ is factorial, or
\item 
\label{it:g4-nf}
$X$ is a nonfactorial threefold from Proposition~\textup{\ref{prop:hyperelliptic}}\ref{he:nf} and~$\g(X) = 4$.
\end{enumerate}
\end{lemma}

\begin{proof}
Indeed, if~$X$ is hyperelliptic we apply Proposition~\ref{prop:hyperelliptic} and obtain cases~\ref{it:g3-f} and~\ref{it:g4-nf}.

Otherwise, if~$-K_X$ is very ample, then~$X \subset \PP^{g + 1}$ is a threefold of degree~$(-K_X)^3 = 2g - 2$.
If~$g = 3$, it follows that~$X$ is a quartic threefold, this is again case~\ref{it:g3-f}.

Similarly, if~$g = 4$ we note that~$h^0(\cO_X(-2K_X))=20$, while~$h^0(\cO_{\PP^5}(2))=21$,
hence~$X$ is contained in an irreducible 4-dimensional quadric~$Q \subset \PP^5$.
A similar computation shows that~$X$ is also contained in a cubic hypersurface~$Y \subset \PP^5$ such that~$\dim(Q \cap Y) = 3$.
Since~$\deg(Q \cap Y) = 6 = \deg(X)$, we conclude that~$X = Q \cap Y$; 
this is case~\ref{it:g4-f}.

Finally, in both cases~\ref{it:g3-f} and~\ref{it:g4-f} 
it follows from Remark~\ref{rem:blowup-nc} and Corollary~\ref{cor:ci-factorial}\ref{it:ci-i1} that~$X$ is factorial.
\end{proof}

The following corollary is quite useful.

\begin{corollary}
\label{cor:very-ample}
Let~$X$ be a nonfactorial $1$-nodal Fano threefold with~$\uprho(X) = 1$ and~$\io(X) = 1$.
Then~$-K_X$ is very ample if and only if~$\g(X) \ge 5$.
\end{corollary}

\subsection{Trigonal Fano threefolds}
\label{ss:tr}

Recall that a Fano threefold~$X$ is called {\sf trigonal} if~$-K_X$ is very ample, 
but the anticanonical model~$X \subset \PP^{g + 1}$ is not an intersection of quadrics. 

\begin{proposition}[cf.~{\cite[Theorem~4.5]{P:ratF-1}}]
\label{prop:trigonal}
Let~$X$ be a trigonal Fano threefold with terminal Gorenstein singularities and~$\uprho(X) = 1$.

If~$X$ is factorial then we have one of the following situations:
\begin{enumerate}
\item 
\label{it:tri:g3}
$\io(X) = 1$, $\g(X) = 3$ and~$X = X_4 \subset \PP^4$;
\item 
\label{it:tri:g4}
$\io(X) = 1$, $\g(X) = 4$ and~$X = X_{2,3} \subset \PP^5$. 
\end{enumerate}

If~$X$ is nonfactorial and $1$-nodal with node~$x_0$ then
\begin{enumerate}[resume*]
\item 
\label{it:tri:g5}
$\io(X) = 1$, $\g(X) = 5$, and~$\balpha(X) = 1$, or

\item 
\label{it:tri:g6}
$\io(X) = 1$, $\g(X) = 6$, and~$\balpha(X) = 1$.
\end{enumerate}

In case~\ref{it:tri:g5} the small resolutions~$X_i$ of~$X$ and the anticanonical model~$\bar{X}$ of~$\hX$ are:
\begin{itemize}[wide]
\item 
$X_1 \to \PP^1$ is a cubic del Pezzo fibration such that
\begin{equation*}
X_1 \subset Y_1 \coloneqq \PP_{\PP^1}(\cO \oplus \cO(-1)^{\oplus 3})
\end{equation*}
is a divisor of class~$3H - H_1$,
\item 
$X_2 \to \PP^2$ is a conic bundle with~$\deg(\Delta_2) = 7$,
\item 
$\bar{X}$ is a smoothable Fano variety with~$\uprho(\bar{X}) = 2$, $\io(\bar{X}) = 1$, 
and~\mbox{$\g(\bar{X}) = 4$} of type~\typemm{2}{2};
it has non-isolated canonical singularities if the flopping curve~$C_1 \subset X_1$ 
corresponds to an Eckardt point on the general fiber of~$X_1 \to \PP^1$
and terminal singularities otherwise.
\end{itemize}

In case~\ref{it:tri:g6} the small resolutions~$X_i$ of~$X$ and the anticanonical model~$\bar{X}$ of~$\hX$ are:
\begin{itemize}[wide]
\item 
$X_1 \to \PP^1$ is a cubic del Pezzo fibration such that
\begin{equation*}
X_1 \subset Y_1 \coloneqq \PP_{\PP^1}(\cO \oplus \cO(-1)^{\oplus 2} \oplus \cO(-2))
\end{equation*}
is a divisor of class~$3H - 2H_1$,
\item 
$X_2 = \Bl_{\Gamma_2}(Z_2)$ is the blowup of a smooth del Pezzo threefold~$Z_2$ of degree~$2$ in a line~\mbox{$\Gamma_2$},
\item 
$\bar{X}$ is a smoothable Fano variety with~$\uprho(\bar{X}) = 2$, $\io(\bar{X}) = 1$, 
and~\mbox{$\g(\bar{X}) = 5$} of type~\typemm{2}{3};
it has non-isolated canonical singularities if the flopping curve~$C_1 \subset X_1$ corresponds to an Eckardt point 
on the general fiber of~$X_1 \to \PP^1$,
and terminal singularities otherwise.
\end{itemize}
\end{proposition}

Cases~\ref{it:tri:g3} and~\ref{it:tri:g4} are classical, 
see, e.g., \cite[Theorem~1.6]{Przhiyalkovskij-Cheltsov-Shramov}, types~$\mathbf{T}_1$ and~$\mathbf{T}_2$. 
Cases~\ref{it:tri:g5} and~\ref{it:tri:g6} can be found in~{\cite[Examples~4.6, 4.7]{P:ratF-1}}
or~\cite[Theorem~1.6]{Przhiyalkovskij-Cheltsov-Shramov}, types~$\mathbf{T}_4$, $\mathbf{T}_7$.

\begin{remark}
\label{rem:sing-g56}
In case~\ref{it:tri:g5} the condition that~$\bar{X}$ is not terminal 
is equivalent to the condition that the line~$f_2(C_2) \subset \PP^2$ 
is a component of the discriminant of the conic bundle~$f_2$.
In case~\ref{it:tri:g6} the condition that~$\bar{X}$ is not terminal 
is equivalent to the condition 
that the line~\mbox{$\Gamma_2$} 
is contained in the ramification divisor on the anticanonical double covering~$Z_2 \to \PP^3$.
\end{remark}

\begin{proof}
The intersection of quadrics passing through~$X=X_{2g-2}\subset \PP^{g+1}$ 
is a four-dimensional variety~$Y=Y_{g-2}\subset \PP^{g+1}$ of minimal degree
(see~\cite[\S2]{Isk:Fano2e} or~\cite[Theorem~1.6]{Przhiyalkovskij-Cheltsov-Shramov}).
According to \cite[Theorem~1.10]{SaintDonat1974} or \cite{EisenbudHarris:VMD}
the variety~$Y$ is either~$\PP^4$, 
or a quadric~$Q\subset \PP^5$,
or a $4$-dimensional smooth rational scroll, 
or an iterated cone of degree~$g - 2 \ge 3$ over a smooth rational scroll or over a twisted rational curve,
or over the Veronese surface.

If~$Y\cong \PP^4$ or~$Y\cong Q\subset \PP^5$, 
the argument of Lemma~\ref{lem:g34} shows that~$X$ is a quartic hypersurface 
or a complete intersection of a quadric and cubic hypersurfaces, respectively.
By Corollary~\ref{cor:ci-factorial} these complete intersections cannot be nonfactorial and 1-nodal.
Therefore, these cases give varieties of types~\ref{it:tri:g3} and~\ref{it:tri:g4}.

If~$Y$ is a smooth rational scroll, then~$X$ has a contraction to a curve and so~$\uprho(X) \ge 2$ 
which contradicts our assumption. 

If~$Y$ is a double cone of degree~$g - 2 \ge 3$ over a surface scroll or a Veronese surface,
then~$Y$ is not a local complete intersection at every point of its vertex, 
hence~$X$ is not a local complete intersection at every point of its intersection with the vertex,
in contradiction with the assumption that~$X$ has terminal Gorenstein singularities.

If~$Y$ is a triple cone over a rational twisted curve, i.e., $Y = \PP(1,1,g-2,g-2,g-2)$, 
then the adjunction formula implies that~$X \subset \PP(1,1,g-2,g-2,g-2)$ is hypersurface of degree~$2g - 2$,
so it has non-isolated, hence not terminal singularities, 
again in contradiction with our assumptions.

Finally, consider the case where~$g \ge 5$ and~$Y$ is the cone over a 3-dimensional smooth rational scroll of degree~$g - 2 \ge 3$.
In this case~$Y$ has a small resolution
\begin{equation*}
Y_1 \coloneqq \PP_{\PP^1}(\cO \oplus \cO(-a) \oplus \cO(-b) \oplus \cO(-c)) \longrightarrow Y,
\quad 
c \ge b \ge a \ge 1,
\quad 
a + b + c = g - 2,
\end{equation*}
which induces a commutative diagram 
\begin{equation*}
\xymatrix{
X_1\ar@{^{(}->}[r]\ar[d] & Y_1\ar[d]\ar[r] & \PP^1
\\
X\ar@{^{(}->}[r] & Y
} 
\end{equation*}
where~$X_1$ is the strict transform of~$X$. 
Note that~$X_1$ is normal, because it is Cohen--Macaulay (as a divisor in the smooth variety~$Y_1$) 
and normal away from the exceptional curve~$C_1 \subset Y_1$ (because~$X$ is normal).
Moreover, $X_1 \to X$ is a small birational morphism.
If it is an isomorphism then~$X \subset Y_1$ has a contraction to a curve in contradiction with the assumption~$\uprho(X) = 1$.
Therefore, $X$ is not factorial.

So, from now on we assume that~$X$ is nonfactorial and 1-nodal and continue the argument.
Note that in this case Remark~\ref{rem:xi-bu} shows that the morphism~$X_1 \to X$ is a small resolution;
in particular, $X_1$ is smooth.

By the adjunction formula, $X_1$ is a divisor in~$Y_1$ of class~$3H_Y - (a + b + c - 2)H_1$,
where~$H_Y$ is the relative hyperplane class of~$Y_1 \to \PP^1$ and~$H_1$ is the pullback of the point class from~$\PP^1$.
Thus, the composition 
\begin{equation*}
f_1 \colon X_1\hookrightarrow Y_1\longrightarrow \PP^1
\end{equation*}
is a del Pezzo fibration of degree $3$.
In particular, we have~$\io(f_1) = \io(X) = 1$,
and~$\Pic(X_1/\PP^1)$ is generated by the restriction of~$H_Y$.

It also follows that the flopping curve~$C_1 \subset X_1$
is contained in the preimage of the exceptional section of~$Y_1 \to \PP^1$,
hence coincides with this section.
Consider the normal bundle sequence
\begin{equation*}
0 \longrightarrow \cN_{C_1/X_1} \longrightarrow 
\cN_{C_1/Y_1} \longrightarrow 
\cN_{X_1/Y_1}\vert_{C_1} \longrightarrow 0.
\end{equation*}
Note that the middle term is~$\cO(-a) \oplus \cO(-b) \oplus \cO(-c)$.
On the other hand, by~\eqref{eq:cn-ci} the first term is~$\cO(-1) \oplus \cO(-1)$, 
hence
the last term is~$\cO(2 - a - b - c)$.
But~$a + b + c = g - 2 \ge 3$,
hence the normal bundle sequence splits, 
and we conclude 
that~$a = b = 1$, $c \ge 1$, and the class of~$X_1$ is~$3H_Y - cH_1$.

Assume~$c \ge 3$.
Consider the intersection of~$X_1$ with the subscroll
\begin{equation*}
X_1 \cap \PP_{\PP^1}(\cO \oplus \cO(-1) \oplus \cO(-1)) \subset \PP_{\PP^1}(\cO \oplus \cO(-1) \oplus \cO(-1) \oplus \cO(-c)) = Y_1.
\end{equation*}
By definition, it is a divisor in the linear system~$|3H_Y - cH_1|$ on~$\PP_{\PP^1}(\cO \oplus \cO(-1) \oplus \cO(-1))$.
If~$c \ge 4$ this linear system is empty, hence~$\PP_{\PP^1}(\cO \oplus \cO(-1) \oplus \cO(-1)) \subset X_1$, which is absurd.
If~$c = 3$, any divisor in this linear system is the union~$D_1 \cup D_2 \cup D_3$ of three divisors of class~$H_Y - H_1$.
Thus, $D_i \subset X_1$, and since~$D_i$ has degree~$1$ over~$\PP^1$, 
this contradicts the fact that~$\Pic(X_1/\PP^1)$ is generated by the restriction of~$H_Y$.

Therefore $c \in \{1,2\}$, so that we have two cases:
\begin{aenumerate}
\item 
\label{it:111}
$a = 1$, $b = 1$, $c = 1$, $X_1 \sim 3H_Y - H_1$, and~$\g(X) = 5$, or
\item 
\label{it:112}
$a = 1$, $b = 1$, $c = 2$, $X_1 \sim 3H_Y - 2H_1$, and~$\g(X) = 6$.
\end{aenumerate}
These are precisely cases~\ref{it:tri:g5} and~\ref{it:tri:g6} of the proposition.

It remains to describe cases~\ref{it:tri:g5} and~\ref{it:tri:g6} in detail.
For this we consider the toric diagram
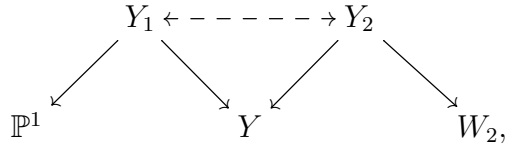
\begin{equation*}
\xymatrix{
& Y_1 \ar[dl] \ar[dr] \ar@{<-->}[rr] &&
Y_2 \ar[dl] \ar[dr]
\\
\PP^1 &&
Y &&
W_2,
}
\end{equation*}
where this time
\begin{itemize}[wide]
\item 
in case~\ref{it:111} we take~$Y_2 = \PP_{\PP^2}(\cO \oplus \cO(-1) \oplus \cO(-1)) \to \PP^2 \eqqcolon W_2$, and
\item 
in case~\ref{it:112} we take~$Y_2 = \Bl_{\PP^1}(\PP(1^4,2)) \to \PP(1^4,2) \eqqcolon W_2$.
\end{itemize}
In both cases the morphism~$Y_2 \to \PP^{g+1}$ given by the hyperplane class of the projective bundle 
or by the linear projection of~$\PP(1^4,2) \subset \PP^{10}$ 
out of the linear span~$\PP^2$ of the conic~$\Gamma_2\coloneqq \PP^1 \subset \PP(1^4,2)$
provides a small resolution~$Y_2 \to Y$ 
(in the first case it contracts the exceptional section~$\PP^2$ of the projective bundle, 
and in the second case it contracts a quadratic cone~$\PP(1,1,2) \subset \PP(1^4,2)$).
In particular, we obtain isomorphisms of the class groups~$\Cl(Y_1) \cong \Cl(Y) \cong \Cl(Y_2)$ such that
\begin{equation*}
\begin{aligned}[t]
H_1 & \sim H_Y - H_2, && \text{in case~\ref{it:111}, and}
\\
H_1 &\sim H_2 - E_2,
\quad
H_Y \sim 2H_2 - E_2, && \text{in case~\ref{it:112},}
\end{aligned}
\end{equation*} 
where~$H_2$ denotes the ample generator of~$\Cl(W_2)$ and~$E_2$ is the exceptional divisor in case~\ref{it:112}.

Let~$X_2$ be the strict transform of~$X$ in~$Y_2$; the above formulas imply that
\begin{equation*}
\begin{aligned}[t]
X_2 &\sim 3H_Y - (H_Y - H_2) = 2H_Y + H_2, && \text{in case~\ref{it:111}, and}
\\
X_2 &\sim 3(2H_2 - E_2) - 2(H_2 - E_2) = 4H_2 - E_2, && \text{in case~\ref{it:112},}
\end{aligned}
\end{equation*}
respectively.
Note that~$X_2$ does not contain the fiber of~$Y_2$ over the vertex of~$Y$ 
(if it does, then the tangent space to~$X$ at~$x_0$ has dimension larger than~$4$),
hence the morphism~$X_2 \to X$ is small.
As before, it follows that~$X_2$ is normal, hence it is the second small resolution of~$X$ by Remark~\ref{rem:xi-bu}
(in particular, $X_2$ is smooth).
Moreover, it follows than the composition~\mbox{$X_2 \hookrightarrow Y_2 \longrightarrow W_2$} is
\begin{itemize}[wide]
\item 
in case~\ref{it:111}, 
a conic bundle over~$Z_2 = W_2 = \PP^2$ with discriminant of degree~$7$, 
\item 
in case~\ref{it:112}, 
a birational morphism onto a quartic divisor~$Z_2 \subset W_2$ 
(i.e., onto a del Pezzo threefold of degree~$2$) 
contracting the divisor~$X_2 \cap E_2$ onto the line~$\Gamma_2 \subset Z_2$.
\end{itemize}
In particular, we see that in both cases the morphism~$X_2 \to Z_2$ is an extremal contraction.
In case~\ref{it:112} the classification of extremal contractions from~\S\ref{ss:contractions} shows that~$Z_2$ is smooth.

Restricting the relations in~$\Cl(Y)$ to~$X$ it is easy to see that in both cases~$\balpha(X) = 1$.

Finally, in case~\ref{it:111} it is easy to see that
\begin{equation*}
\hX \subset \hY \coloneqq \PP_{\PP^1 \times \PP^2}(\cO \oplus \cO(-1,-1))
\end{equation*}
is a divisor of class~$2H + H_2$, hence~$\hX \to \PP^1 \times \PP^2$ is generically finite of degree~$2$
with branch divisor of class~$2H_1 + 4H_2$, hence the Stein factorization of this map 
is a birational contraction~$\xi \colon \hX \to \bar{X}$ onto a smoothable Fano threefold of type~\typemm{2}{2}.

On the other hand, in case~\ref{it:112} we can use the maps~$f_1 \circ \sigma_1 \colon \hX \to \PP^1$ and~$f_2 \circ \sigma_2 \colon \hX \to Z_2$, 
to construct a morphism~$\xi \colon \hX \longrightarrow \PP^1 \times Z_2$,
which is birational onto its image, because its component~$f_2 \circ \sigma_2\colon \hX \to Z_2$ is,
and with the argument of Proposition~\ref{prop:hyperelliptic} it is easy to see that the image of~$\xi$ 
is a divisor~$\bar{X} \subset \PP^1 \times Z_2$ of class~$H_1 + H_2$, 
i.e., a smoothable Fano variety of type~\typemm{2}{3}.

In both cases~\ref{it:tri:g5} and~\ref{it:tri:g6} 
the restriction of the anticanonical contraction~$\xi \colon \hX \to \bar{X}$
to the strict transform of a fiber of the cubic del Pezzo fibration~$f_1 \colon X_1 \to \PP^1$
is the projection out of the intersection point of the section~$C_1$ with that fiber.
In particular, the nontrivial fibers of~$\xi$ are the lines on the fibers of~$f_1$ intersecting~$C_1$.
As there are at most finitely many lines on a general fiber of~$f_1$,
it follows that~$\bar{X}$ has non-isolated singularities if an only if
the general point of~$C_1$ lies on a line on the general fiber of~$f_1$.
Since any curve on a cubic del Pezzo surface with Picard number~1 is a multiple of a hyperplane class,
it follows that the general point of~$C_1$ is an Eckardt point of the general fiber of~$f_1$.
\end{proof}

Now we can classify all $1$-nodal or $1$-cuspidal prime Fano threefolds of genus~$5$ and~$6$.
We denote by~$\CGr(2,5)\subset \PP^{10}$ the cone over the Grassmannian~$\Gr(2,5) \subset \PP^9$.

\begin{lemma}
\label{lem:g56}
Let~$X$ be a Fano threefold with a single node or cusp, $\uprho(X) = 1$, $\io(X) = 1$, and~\mbox{$\g(X) \in \{5,6\}$}.
Then either 
\begin{enumerate}
\item 
\label{it:g5-f}
$X = X_{2,2,2} \subset \PP^6$, $\g(X) = 5$, and~$X$ is factorial, or
\item 
\label{it:g6-f}
$X = X_{1,1,1,2} \subset \CGr(2,5)$, $\g(X) = 6$, and~$X$ is factorial, or
\item 
\label{it:g5-nf}
$X$ is a nonfactorial threefold from Proposition~\textup{\ref{prop:trigonal}}\ref{it:tri:g5} and~$\g(X) = 5$.
\item 
\label{it:g6-nf}
$X$ is a nonfactorial threefold from Proposition~\textup{\ref{prop:trigonal}}\ref{it:tri:g6} and~$\g(X) = 6$.
\end{enumerate}
\end{lemma}

\begin{proof}
By Propositions~\ref{prop:base-points} and~\ref{prop:hyperelliptic} the class~$-K_X$ is very ample.
If~$X$ is trigonal we apply Proposition~\ref{prop:trigonal} and obtain cases~\ref{it:g5-nf} and~\ref{it:g6-nf}.
Otherwise, $X \subset \PP^{g + 1}$ is an intersection of quadrics.
If~$g = 5$ the codimension of~$X$ is~3, hence~$X$ is
an intersection of at least three quadrics, 
and since~$\deg(X) = 2g - 2 = 8 = 2^3$, it follows that~$X = X_{2,2,2}$;
this is case~\ref{it:g5-f}.

Similarly, if~$g = 6$ then a general hyperplane section of~$X$ is a Brill--Noether general K3 surface of genus~$6$
and~$X$ is a Gushel--Mukai threefold by~\cite[Theorem~2.16 and Proposition~2.15]{DK1};
this is case~\ref{it:g6-f}.

Finally, in both cases~\ref{it:g5-f} and~\ref{it:g6-f} 
it follows from Corollary~\ref{cor:ci-factorial} that~$X$ is factorial.
\end{proof} 

\begin{corollary}
\label{cor:intersection-of-quadrics}
Let~$X$ be a $1$-nodal Fano threefold with~$\uprho(X) = 1$, $\io(X) = 1$, and~$-K_X$ very ample.
Then~$X \subset \PP^{\g(X)+1}$ is an intersection of quadrics 
if and only if~$\g(X) \ge 5$ in the factorial case and~$\g(X) \ge 7$ in the nonfactorial case.
\end{corollary}

\section{Classification of nonfactorial threefolds}
\label{sec:nonfactorial} 

In this section we discuss the structure of nonfactorial Fano threefolds of genus~$g \ge 7$ 
and prove Theorems~\ref{thm:intro-nf} and~\ref{thm:intro-nf-contractions} from the Introduction. 

\subsection{Intersection with the embedded tangent space}
\label{ss:int-ets}

In this subsection we assume that~$X$ is a nonfactorial $1$-nodal Fano threefold 
such that~$\uprho(X) = 1$, $\io(X) = 1$, and~\mbox{$\g(X) \ge 7$},
so that by Corollary~\ref{cor:intersection-of-quadrics} the anticanonical class~$H = -K_X$ is very ample 
and the anticanonical model~$X \subset \PP^{\g(X) + 1}$ is an intersection of quadrics.

We denote by~$T_{x_0}(X) \cong \Bbbk^4$ the Zariski tangent space to~$X$ at the node~$x_0 \in X$
and by~$\rT_{x_0}(X) \subset \PP^{\g(X) + 1}$ the embedded tangent space.
One can think of~$\rT_{x_0}(X) \cong \PP^4$ as the cone with vertex~$x_0$ over~$\PP(T_{x_0}(X)) \cong \PP^3$.

\begin{lemma}
\label{lem:intersection-tangent}
Let~$X \subset \PP^n$ be a $3$-dimensional intersection of quadrics with a node~$x_0 \in X$.
Then~\mbox{$X \cap \rT_{x_0}(X)$} is a cone with vertex~$x_0$ over a scheme~$B$, where
\begin{equation*}
B \subset E \subset \PP(T_{x_0}(X)) \cong \PP^3,
\end{equation*}
and~$E \cong \PP^1 \times \PP^1$ is the exceptional divisor of~$\Bl_{x_0}(X)$.
Moreover, $B$ is an intersection of quadrics and~$\dim(B) \le 1$ unless~$X$ is a quadratic cone.
\end{lemma}

\begin{proof}
For each quadric~$Q \subset \PP^n$ containing~$X$ we have~$\rT_{x_0}(X) \subset \rT_{x_0}(Q)$
and~$Q \cap \rT_{x_0}(Q)$ is a quadratic cone with vertex~$x_0$.
Therefore, $X \cap \rT_{x_0}(X)$ is an intersection of quadratic cones, 
hence a cone with vertex~$x_0$ over an intersection of quadrics in~$\PP(T_{x_0}(X))$ which we denote by~$B$.
The embedding~$B \subset E$ is obvious and if~$B = E$ then~$X$ contains a 3-dimensional quadratic cone, hence coincides with it.
\end{proof}

Recall that~$f_1 \colon X_1 \to Z_1$ and~$f_2 \colon X_2 \to Z_2$ denote the extremal contractions 
of the small resolutions~$X_1$ and~$X_2$ of~$X$, see~\eqref{eq:intro-sl},
and the Fano indices~$\io(f_i)$ are defined in~\eqref{eq:io-fi}.

\begin{proposition}
\label{prop:no-cones}
Let~$(X,x_0)$ be a nonfactorial $1$-nodal Fano threefold such that~$\uprho(X) = 1$, 
$X \subset \PP^{\g(X) + 1}$ is an intersection of quadrics, and~$\io(f_1) = \io(f_2) = 1$.
Then~$X \cap \rT_{x_0}(X)$ cannot contain a surface.
\end{proposition}

\begin{proof}
Note that the assumption~$\io(f_1) = \io(f_2) = 1$ implies~$\io(X) = 1$ by~\eqref{eq:io-fi}, hence any Cartier divisor on~$X$
is linearly equivalent to a multiple of~$H = - K_X$.
Recall also that~\mbox{$\g(X) \ge 7$} by Corollary~\ref{cor:intersection-of-quadrics}.

If the intersection~$X \cap \rT_{x_0}(X)$ contains an irreducible surface~$S$ of degree~$d$, 
then, by Lemma~\ref{lem:intersection-tangent}, $S$ is the cone over an irreducible curve~$B \subset E$ of degree~$d$.
Let
\begin{equation*}
\tS \coloneqq \Bl_{x_0}(S) \subset \Bl_{x_0}(X)
\end{equation*}
be the strict transform of~$S$; then~$\tS \cap E = B$.
Since~$B$ is contained in the intersection of~$E$ with another quadric, 
the curve~$B$ has one of the following bidegrees
\begin{equation*}
(1,0),
\qquad\text{or}\qquad 
(1,1),
\qquad\text{or}\qquad 
(2,1),
\qquad\text{or}\qquad 
(2,2)
\end{equation*}
(up to transposition).
We will show that all these cases are impossible. 

\subsection*{Case 1.} 

First, assume the curve~$B$ has bidegree~$(k,k)$, $k = 1,2$; then~$\deg(S) = 2k$.
Since~$E\vert_E$ is a divisor of bidegree~$(-1,-1)$,
we have~$(\tS + kE)\vert_E \sim 0$, hence~$\tS + kE$ is the pullback of a Cartier divisor on~$X$,
hence
\begin{equation*}
\tS + kE \sim aH 
\end{equation*}
for some~$a \in \ZZ$. 
Obviously, $a \ge 1$.
Now, intersecting this with~$H^2$, we obtain
\begin{equation*}
2k = \deg(S) = H^2 \cdot \tS + k H^2 \cdot E = a H^3 = a(2\g(X) - 2).
\end{equation*}
But as~$k \in \{1,2\}$ while~$\g(X) \ge 7$ and~$a \ge 1$, this is impossible.

\subsection*{Case 2.} 

Now assume~$B$ has bidegree~$(1,0)$; then~$\deg(S) = 1$.
Note that by~\eqref{eq:deg-ci} the divisor classes~$H_1\vert_E$ and~$H_2\vert_E$ have bidegree~$(1,0)$ and~$(0,1)$, respectively.
Therefore, $(\tS + H_2)\vert_E$ has bidegree~$(1,1)$, hence
\begin{equation*}
\tS + E + H_2 \sim aH,
\end{equation*}
and again~$a \ge 1$.
Intersecting with~$H^2$ we obtain
\begin{equation*}
1 + H^2 \cdot H_2 = a(2\g(X) - 2).
\end{equation*}
If the contraction~$f_2$ is of type~\type{(B_1)}, or~\type{(C_1)}, or~\type{(D_1)},
the left-hand side is positive and does not exceed~$2\g(X) -3$ by Lemma~\ref{lem:hhh}, a contradiction.
Since~$\io(f_2) = 1$, we conclude that~$f_2$ is of type~\typeb.
By Remark~\ref{rem:contr} the image in~$X$ of the $f_2$-exceptional divisor~$E_2 \subset X_2$ is a plane~$S' \subset X$.
It is easy to see that~$S' \subset X \cap \rT_{x_0}(X)$, and the corresponding curve~$B' \subset E$ has bidegree~$(0,1)$.
Then~$B \cup B'$ is a curve of bidegree~$(1,1)$, hence~$\tS + \tS' + E \sim a' H$ 
and the computation of Case~1 gives a contradiction.

\subsection*{Case 3.} 

Finally, assume~$B$ has bidegree~$(2,1)$; then~$\deg(S) = 3$.
In this case~$(\tS - H_1)\vert_E$ has bidegree~$(1,1)$, hence
\begin{equation*}
\tS + E - H_1 \sim -aH,
\end{equation*}
and the same computation as before gives 
\begin{equation*}
a(2\g(X) - 2) = H^2 \cdot H_1 - 3.
\end{equation*}
Now, if~$f_1$ is of type~\type{(B_1)} then Lemma~\ref{lem:hhh} proves that 
the right hand side is between~$1$ and~$2\g(X) - 7$ in contradiction with~$a \in \ZZ$.
If it is of type~\type{(C_1)} or~\type{(D_1)} then the right hand-side is between~$-2$ and~$2\g(X) - 7$, 
hence integrality of~$a$ implies that it is~$0$, hence~$H^2 \cdot H_1 = 3$.
This means that~$X$ is covered by cubic surfaces,
and it cannot be an intersection of quadrics.

This means that~$f_1$ is of type~\typeb.
In this case by Remark~\ref{rem:contr} the image in~$X$ 
of the exceptional divisor~$E_1 \subset X_1$ of this contraction is a plane~$S' \subset X \cap \rT_{x_0}(X)$,
and the computation of Case~2 gives a contradiction.
\end{proof} 

\begin{corollary}
\label{cor:ng}
Let~$X$ be a nonfactorial $1$-nodal Fano threefold such that~$\uprho(X) = 1$.
Then neither~$f_1$ nor~$f_2$ can be of type~\typeb.
\end{corollary}

\begin{proof}
If~$\io(f_1) > 1$ or~$\io(f_2) > 1$ we use the descriptions of Proposition~\ref{prop:if3} and Proposition~\ref{prop:big-index},
and if~$\io(f_1) = \io(f_2) = 1$ and~$\g(X) \le 6$ we apply classification results of Lemmas~\ref{lem:g2}, \ref{lem:g34}, and~\ref{lem:g56}
together with the descriptions of Propositions~\ref{prop:base-points}, \ref{prop:hyperelliptic}, and~\ref{prop:trigonal}.
In all these cases we see that contractions of type~\typeb{} do not show up.

Finally, assume~$\io(f_1) = \io(f_2) = 1$ and~$\g(X) \ge 7$.
If either of~$f_i$ is of type~\typeb{} then by Remark~\ref{rem:contr} the image in~$X$ 
of the exceptional divisor~$E_i \subset X_i$ of this contraction is a plane~$S' \subset X \cap \rT_{x_0}(X)$
in contradiction with Proposition~\ref{prop:no-cones}.
\end{proof} 

Recall the integer~$\balpha(X)$ defined in Proposition~\ref{prop:pic}. 
As we already mentioned, it would be nice to find an a priori reason for the following crucial result.

\begin{corollary}
\label{cor:alpha}
If~$X$ is a nonfactorial $1$-nodal Fano threefold with~$\uprho(X) = 1$ then~$\balpha(X) = 1$.
\end{corollary}

\begin{proof}
If~$\io(f_1) > 1$ or~$\io(f_2) > 1$, we apply Propositions~\ref{prop:if3} and~\ref{prop:big-index}.
If~$\io(f_1) = \io(f_2) = 1$ and~$\g(X) \le 6$ we apply classification results of Lemmas~\ref{lem:g2}, \ref{lem:g34}, and~\ref{lem:g56}
together with the descriptions of Propositions~\ref{prop:base-points}, \ref{prop:hyperelliptic}, and~\ref{prop:trigonal}.
In all these cases we see that~$\balpha(X) = 1$.

Finally, assume~$\io(f_1) = \io(f_2) = 1$ and~$g \coloneqq \g(X) \ge 7$.
Multiplying~\eqref{eq:mkx} by~$H^2$, we obtain
\begin{equation}
\label{eq:a2g2}
\balpha(X)(2g - 2) = H^2\cdot H_1 + H^2 \cdot H_2.
\end{equation}
Now we note that by Corollary~\ref{cor:ng} the contractions~$f_1$ and~$f_2$ are of type~\type{(B_1)}, \type{(C_1)}, or~\type{(D_1)}, 
hence Lemma~\ref{lem:hhh} implies that both summands in the right-hand side are less than~$2g - 2$.
Therefore, their sum is less than~$2(2g - 2)$, hence~$\balpha(X) < 2$, and therefore~$\balpha(X) = 1$.
\end{proof}

\subsection{The anticanonical model of~$\hX$}
\label{ss:iq} 

In this subsection we keep the assumptions of~\S\ref{ss:int-ets}, i.e., 
$X$ is a nonfactorial $1$-nodal Fano threefold 
with~$\uprho(X) = 1$, $\io(X) = 1$, and~\mbox{$\g(X) \ge 7$},
hence~$-K_X$ is very ample and~$X \subset \PP^{\g(X) + 1}$ is an intersection of quadrics. 
We also assume~$\io(f_1) = \io(f_2) = 1$.

Recall the blowup diagram~\eqref{eq:diag-pi-sigma} and the Sarkisov link diagram~\eqref{eq:intro-sl}.

\begin{lemma}
\label{lem:fi-zi}
Let~$X$ be a nonfactorial $1$-nodal Fano threefold with~$\uprho(X) = 1$ and~\mbox{$\g(X) \ge 7$}.
If the $K$-negative contraction~$f_i \colon X_i \to Z_i$ of a small resolutions~$X_i$ of~$X$ satisfies~$\io(f_i) = 1$
then~$f_i$ can be of type~\type{(B_1^1)}, \type{(C_1)}, and~\type{(D_1)} only.
In particular, $Z_i$ is smooth.

Moreover, the case where both~$f_1$ and~$f_2$ are of type~\type{(D_1)} is impossible.
\end{lemma}

\begin{proof}
By Corollary~\ref{cor:ng} the contraction~$f_i$ cannot be of type~\typeb.
Since~$\io(f_i) = 1$, it follows that this contraction is of type~\type{(B_1)}, \type{(C_1)}, or~\type{(D_1)}.
Moreover, if~$f_i$ is of type~\type{(B_1)} then
\begin{equation}
\label{eq:h-b1}
H \sim \io(Z_i)H_i - E_i.
\end{equation} 
It follows that~$E_i \sim \io(Z_i)H_i - H$ and using~\eqref{eq:deg-ci} we obtain
\begin{equation}
\label{eq:ei-ci}
E_i \cdot C_i = 
\big(\io(Z_i)H_i - H\big) \cdot C_i = 
\io(Z_i)H_i \cdot C_i - H \cdot C_i = 
\io(Z_i).
\end{equation}

Assume~$f_i \colon X_i \to Z_i$ is of type~\type{(B_1^0)}.
Then~$E_i$ is a quadric surface.
Since~$E_i$ is contracted by~$f_i$ to a point,
$C_i \not\subset E_i$,
and therefore~\eqref{eq:ei-ci} implies that~$E_i \cap C_i$ is a finite scheme of length~$\io(Z_i)$.
Moreover, $H\vert_{E_i}$ is the hyperplane class by~\eqref{eq:h-b1}, hence the morphism
\begin{equation*}
\Bl_{E_i \cap C_i}(E_i) \hookrightarrow \Bl_{C_i}(X_i) = \hX \xrightarrow{\ \xi\ } \bar{X}
\end{equation*}
is given by the linear system of hyperplane sections of the quadric~$E_i \subset \PP^3$ 
that contain the scheme~$E_i \cap C_i$ whose length is~$\io(Z_i) \ge 2$ by (see Lemma~\ref{lem:curve-genus-degree}).
Therefore, this linear system is at most a pencil, hence the surface~$\Bl_{E_i \cap C_i}(E_i)$ is contracted by~$\xi$,
which contradicts the combination of Proposition~\ref{prop:no-cones} with Lemma~\ref{lem:barx}\ref{it:barx-almost} 
saying that~$\xi$ is small.

It remains to show that the case where both~$f_1$ and~$f_2$ are of type~\type{(D_1)} is impossible.
Since~$\balpha(X) = 1$ by Corollary~\ref{cor:alpha}, the relation~\eqref{eq:mkhx} takes the form~$H_1 + H_2 \sim H - E$.
Computing on~$\hX$ the top self-intersection of each side, we obtain
\begin{equation*}
H_1^3 + 3H_1^2 \cdot H_2 + 3H_1 \cdot H_2^2 + H_2^3 = H^3 - 3H^2 \cdot E +3H \cdot E^2 - E^3.
\end{equation*}
Now every summand in the left side is zero, while the right side is equal to~$2g - 4$ by Lemma~\ref{lem:intersection}, 
so the equality is impossible when~$g \ge 7$.
\end{proof}

Now we establish a relation between the Sarkisov link~\eqref{eq:intro-sl} and the contractions of~$\bar{X}$.
This relation will be crucial for the arguments of the next subsection.

\begin{proposition}
\label{prop:cones}
Let~$(X,x_0)$ be a nonfactorial $1$-nodal Fano threefold with~$\uprho(X) = 1$, $\io(f_1) = \io(f_2) = 1$, and~$g \coloneqq \g(X) \ge 7$.
Then~$\hX \coloneqq \Bl_{x_0}(X)$ is an almost Fano but not a Fano variety,
and~$\bar{X} \coloneqq \hX_\can$ is a Fano variety with terminal Gorenstein singularities,
and~\mbox{$\uprho(\bar{X}) = 2$}, $\io(\bar{X}) = 1$, \mbox{$\g(\bar{X}) = g - 1$}.
Moreover, there is a commutative diagram
\begin{equation}
\label{eq:barx-diagram}
\vcenter{\xymatrix{
X_1 \ar[d]_{f_1} & 
\hX \ar[l]_{\sigma_1} \ar[r]^{\sigma_2} \ar[d]^\xi &
X_2 \ar[d]^{f_2} 
\\
Z_1 &
\bar{X} \ar[l]_{\bar{f}_1} \ar[r]^{\bar{f}_2} &
Z_2,
}}
\end{equation}
where~$\xi$ is the anticanonical contraction.
Finally, $L_i \coloneqq f_i(C_i) \subset Z_i$ is a line and
\begin{itemize}[wide]
\item 
if~$f_i$ is the blowup of a smooth curve~$\Gamma_i$ then~$\bar{f}_i$ is the blowup of a subscheme~$\bar{\Gamma}_i \subset Z_i$
such that there is an exact sequence~$0 \to \cO_{L_i}(-\io(Z_i)) \to \cO_{\bar{\Gamma}_i} \to \cO_{\Gamma_i} \to 0$;
\item 
if~$f_i$ is a conic bundle over~$Z_i = \PP^2$ with the discriminant~$\Delta_i \subset \PP^2$
then~$\bar{f}_i$ is a conic bundle with the discriminant~$\bar{\Delta}_i = \Delta_i \cup L_i \subset \PP^2$;
\item 
if~$f_i$ is a del Pezzo fibration 
of degree~$d$
then~$\bar{f}_i$ is a del Pezzo fibration of degree~$d - 1$.
\end{itemize}
\end{proposition} 

The morphisms~$\bar{f}_i$ in~\eqref{eq:barx-diagram} are $K$-negative with connected fibers and relative Picard number~$1$,
but as~$\bar{X}$ is not $\QQ$-factorial, they are not extremal contractions.
As we will see in the proof, in the cases where~$f_i$ is a blowup, the curve~$\bar{\Gamma}_i = \Gamma_i \cup L_i$ is reducible,
except in the case where~$\Gamma_i = L_i$ and~$\bar{\Gamma}_i$ is nonreduced.

\begin{proof}
A combination of Proposition~\ref{prop:no-cones} and Lemma~\ref{lem:barx}\ref{it:barx-almost} proves that~$\hX$ is almost Fano
and~$\bar{X}$ is a Fano variety with terminal Gorenstein singularities;
the index~$\io(\bar{X})$ and genus~$\g(\bar{X})$ of~$\bar{X}$ are computed in Lemma~\ref{lem:barx}\ref{it:barx-weak}.

Next, we check that~$\hX$ is not a Fano variety, i.e., $\hX$ contains a $K$-trivial curve.
By Lemma~\ref{lem:barx} it is enough to check that there are lines in~$X$ passing through the node~$x_0$.
By Lemma~\ref{lem:fi-zi} one of the~$f_i$ has type~\type{(B_1^1)} or~\type{(C_1)};
we consider these two cases separately.
Recall that~\mbox{$C_1 \subset X_1$} and~\mbox{$C_2 \subset X_2$} denote the flopping curves of the small resolutions of~$X$. 

First, if~$f_i$ is of type~\type{(B_1^1)}, 
we have~$E_i \cdot C_i > 0$ by~\eqref{eq:ei-ci}.
Therefore, $E_i \cap C_i \ne \varnothing$ 
and~\eqref{eq:h-b1} implies that the image in~$X$ of any ruling~$\ell$ of~$E_i$ intersecting~$C_i$ is a line passing through~$x_0$.
It also follows that the length of intersection~$\ell \cap C_i$ cannot be larger than~$1$,
hence the morphism~$f_i$ induces an isomorphism~$C_i \xrightiso{} L_i \subset Z_i$, hence~$L_i \cong \PP^1$.

Similarly, if~$f_i$ is of type~\type{(C_1)}, the discriminant~$\Delta_i \subset Z_i = \PP^2$ of~$f_i$ is a nonempty curve,
and by~\eqref{eq:deg-ci} the flopping curve~$C_i$ is not contracted by~$f_i$.
Therefore
\begin{equation*}
C_i \cdot f_i^{*}(\Delta_i)= {f_i}_*(C_i)\cdot \Delta_i > 0
\end{equation*}
and the image in~$X$ of any half-fiber~$\ell$ of~$f_i$ intersecting~$C_i$ is a line passing through~$x_0$.
Since~$H_i \cdot C_i = 1$ and~$Z_i = \PP^2$, the map~$C_i \to L_i \subset Z_i$ is an isomorphism and~$L_i$ is a line.

We see that~$\hX$ is, indeed, not a Fano variety and~$\uprho(\bar{X}) \le \uprho(\hX) - 1 = 2$.
This proves the first claim of the proposition except for the equality~$\uprho(\bar{X}) = 2$ that will be proved later.

Next, we construct the diagram~\eqref{eq:barx-diagram}.
Since~$\balpha(X) = 1$ by Corollary~\ref{cor:alpha}, the relation~\eqref{eq:mkhx} takes the form~$H_1 + H_2 \sim H - E$.
Therefore, if~$\Upsilon$ is a curve class contracted by~$\xi$, we have
\begin{equation*}
H_1 \cdot \Upsilon + H_2 \cdot \Upsilon = (H_1 + H_2) \cdot \Upsilon = (H - E) \cdot \Upsilon = 0.
\end{equation*}
Since both~$H_1$ and~$H_2$ are nef classes on~$\hX$, we deduce that~$H_1 \cdot \Upsilon = H_2 \cdot \Upsilon = 0$.
This means that there are classes~$\bar{H}_1,\bar{H}_2 \in \Pic(\bar{X})$ such that~$H_i = \xi^*\bar{H}_i$.
Moreover, it follows that~$\uprho(\bar{X}) \ge 2$ 
and together with the opposite inequality proved above, this proves that~$\uprho(\bar{X}) = 2$.
We also conclude that the maps~$f_i \circ \sigma_i \colon \hX \to Z_i$ factor through~$\bar{X}$.
This defines the morphisms~$\bar{f}_i$ and proves the existence of commutative diagram~\eqref{eq:barx-diagram}. 

It remains to relate~$\bar{f}_i$ to~$f_i$.
Let~$L_i \coloneqq f_i(C_i) \subset Z_i$;
then~$H_i \cdot L_i = 1$ by~\eqref{eq:deg-ci}. 

First, assume~$f_i$ is the blowup of a smooth curve~$\Gamma_i \subset Z_i$ on a smooth Fano threefold~$Z_i$.
Recall also that~$L_i = f_i(C_i) \cong C_i$ is a smooth rational curve on~$Z_i$.

Assume~$L_i \ne \Gamma_i$.
Then~$\bar{\Gamma}_i \coloneqq \Gamma_i \cup L_i$ is a planar (hence local complete intersection) curve,
$C_i$ is the strict transform of~$L_i$, and a local computation 
shows that~$\Bl_{\bar{\Gamma}_i}(Z_i)$ is normal and
the natural morphism
\begin{equation*}
\xi' \colon \hX = \Bl_{C_i}(X_i) = \Bl_{C_i}(\Bl_{\Gamma_i}(Z_i)) \longrightarrow \Bl_{\Gamma_i \cup L_i}(Z_i) = \Bl_{\bar{\Gamma}_i}(Z_i)
\end{equation*}
is small, hence crepant. 
Therefore, the anticanonical contraction~$\xi$ factors through~$\xi'$.
Since the target~$\Bl_{\bar{\Gamma}_i}(Z_i)$ of~$\xi'$ is singular 
(because~$\Gamma_i \cap L_i \ne \varnothing$ by~\eqref{eq:ei-ci})
and the source~$\hX$ is smooth with~$\uprho(\hX) = 3$,
we have~$\uprho(\Bl_{\bar{\Gamma}_i}(Z_i)) \le 2 = \uprho(\bar{X})$, hence
\begin{equation}
\label{eq:barx-barg}
\bar{X} = \Bl_{\bar{\Gamma}_i}(Z_i)
\qquad\text{and}\qquad 
\xi = \xi'.
\end{equation}

Similarly, if~$L_i = \Gamma_i$ we let~$\bar{\Gamma}_i$ be the nonreduced subscheme of~$Z_i$ 
such that~$I^2_{\Gamma_i} \subset I_{\bar{\Gamma}_i} \subset I_{\Gamma_i}$
and the sheaf~$I_{\bar{\Gamma}_i} / I^2_{\Gamma_i} \subset I_{\Gamma_i} / I^2_{\Gamma_i} = \cN^\vee_{\Gamma_i/Z_i}$
corresponds to the section of~$E_i = \PP_{\Gamma_i}(\cN^\vee_{\Gamma_i/Z_i}) \to \Gamma_i$ given by~$C_i$.
Then~$\bar{\Gamma}_i$ is again planar, $\Bl_{\bar{\Gamma}_i}(Z_i)$ is normal,
the natural morphism
\begin{equation*}
\xi' \colon \hX = \Bl_{C_i}(X_i) = \Bl_{C_i}(\Bl_{\Gamma_i}(Z_i)) \longrightarrow \Bl_{\bar{\Gamma}_i}(Z_i)
\end{equation*}
is crepant (although this time~$\Bl_{\bar{\Gamma}_i}(Z_i)$ has non-isolated canonical singularities and~$\xi$ is not small), 
and the same argument as above proves~\eqref{eq:barx-barg}.

In either case we see that~$\bar{f}_i$ is the blowup of the curve~$\bar{\Gamma}_i$
and obtain the required relation between the structure sheaves of~$\Gamma_i$ and~$\bar{\Gamma}_i$.

Next, assume~$f_i \colon X_i \to Z_i = \PP^2$ is a conic bundle with discriminant~$\Delta_i \subset \PP^2$.
Consider the vector bundle~$\cE \coloneqq f_{i*}(\cO_{X_i}(H))^\vee$, 
so that~$X_i \subset \PP_{\PP^2}(\cE)$ is a divisor of relative degree~$2$.
Since~$C_i$ is an $H$-trivial section of~$f_i$ over~$L_i$,
it corresponds to an embedding~$\cO_{L_i} \hookrightarrow \cE\vert_{L_i}$.
Let~$\bar{\cE}$ be the dual bundle to the kernel of the corresponding epimorphism~$\cE^\vee \to \cO_{L_i}$,
so that we have an exact sequence~$0 \to \bar{\cE}^\vee \to \cE^\vee \to \cO_{L_i} \to 0$ and its dual sequence
\begin{equation*}
0 \longrightarrow \cE \xrightarrow{\ \varphi\ } \bar{\cE} \longrightarrow \cO_{L_i}(1) \longrightarrow 0.
\end{equation*}
Moreover, if~$q \colon \cE \to \cE^\vee(k)$ is the morphism corresponding to the equation of~$X_i \subset \PP_{\PP^2}(\cE)$
and~$\varphi' \colon \bar{\cE}(-1) \longrightarrow \cE$ is the unique morphism 
such that~$\varphi \circ \varphi'$ is the multiplication by the equation of the line~$L_i$
then there is a unique commutative diagram
\begin{equation*}
\xymatrix{
\bar{\cE}(-1) \ar[r]^{\bar{q}} \ar[d]_{\varphi'} &
\bar{\cE}^\vee(k) \ar[d]^{\varphi^\vee}
\\
\cE \ar[r]^q &
\cE^\vee(k)
}
\end{equation*}
It defines a conic bundle~$\bar{X}' \subset \PP_{\PP^2}(\bar{\cE})$ with the discriminant~$\bar{\Delta}_i = \Delta_i \cup L_i$
and the morphism
\begin{equation*}
\xi' \colon \hX = \Bl_{C_i}(X_i) \hookrightarrow \Bl_{C_i}(\PP_{\PP^2}(\cE)) \xrightarrow{\ \varphi\ } \PP_{\PP^2}(\bar{\cE})
\end{equation*}
factors through~$\bar{X}'$.
Again, it is easy to see that~$\bar{X}'$ is normal, $\xi'$ is crepant, and~$\bar{X}'$ is singular,
so arguing as above we conclude that~$\bar{X} = \bar{X}'$ and~$\xi = \xi'$.

Finally, assume~$f_i \colon X_i \to Z_i = \PP^1$ is a del Pezzo fibration of degree~$d \coloneqq \deg(X_i/Z_i)$.
Consider the vector bundle~$\cE \coloneqq f_{i*}(\cO_{X_i}(H))^\vee$ of rank~$d + 1$, 
so that~$X_i \subset \PP_{\PP^1}(\cE)$ is the relative anticanonical embedding.
Clearly, $f_i\vert_{C_i} \colon C_i \to L_i = Z_i$ is an isomorphism, hence~$C_i$ defines a section of~$\cE$.
Then we have an exact sequence
\begin{equation*}
0 \longrightarrow \cO_{\PP^1} \xrightarrow{\ C_i\ } \cE \xrightarrow{\ \varphi\ } \bar{\cE} \longrightarrow 0,
\end{equation*}
where~$\bar{\cE}$ is a vector bundle of rank~$d$.
It induces a morphism
\begin{equation*}
\xi' \colon \hX = \Bl_{C_i}(X_i) \hookrightarrow \Bl_{C_i}(\PP_{\PP^1}(\cE)) \xrightarrow{\ \varphi\ } \PP_{\PP^1}(\bar{\cE}),
\end{equation*}
and if~$\bar{X}' \subset \PP_{\PP^1}(\bar{\cE})$ is its image, 
then~$\bar{X}'$ is normal and~$\xi'$ is crepant
and the same argument as above shows that~$\bar{X} = \bar{X}'$ and~$\xi = \xi'$.
It remains to note that the general fiber of~$\bar{X}$ over~$\PP^1$ 
is the anticanonical model of the blowup of a del Pezzo surface of degree~$d$ at point,
hence it is a (possibly singular) del Pezzo surface of degree~$d - 1$.
\end{proof} 

\subsection{Proof of Theorems~\textup{\ref{thm:intro-nf}} and~\textup{\ref{thm:intro-nf-contractions}}}
\label{ss:proofs-nf}

In this subsection~$X$ is any nonfactorial 1-nodal Fano threefold with~$\uprho(X) = 1$.
We start by interpreting the cases~\ref{it:intro-nf-ample}--\ref{it:intro-nf-g2} of Theorem~\ref{thm:intro-nf}
in terms of invariants of~$X$.

\begin{lemma}
\label{lem:io-k}
Let~$(X,x_0)$ be a nonfactorial $1$-nodal Fano threefold with~$\uprho(X) = 1$.
Let~$f_1$ and~$f_2$ be the $K$-negative contractions of the small resolutions of~$X$ and~\mbox{$\hX \coloneqq \Bl_{x_0}(X)$}. 
Then
\begin{enumerate}
\item 
$-K_{\hX}$ is ample if and only if~$\io(f_1) > 1$ or~$\io(f_2) > 1$;
\item 
$-K_{\hX}$ is nef and big but not ample if and only if~$\io(f_1) = \io(f_2) = 1$ and~$\g(X) > 2$;
\item 
$-K_{\hX}$ is nef but not big if and only if~$\io(f_1) = \io(f_2) = 1$ and~$\g(X) = 2$.
\end{enumerate}
\end{lemma}

\begin{proof}
First, assume that~$\io(f_1) > 1$ or~$\io(f_2) > 1$.
Then we apply Propositions~\ref{prop:if3} and~\ref{prop:big-index}.
In each case we see that~$\hX$ is a smooth Fano variety, in particular~$-K_{\hX}$ is ample.

Next, assume~$\io(f_1) = \io(f_2) = 1$ and~$\g(X) > 2$.
If~$\g(X) \in \{3,4\}$ we apply Lemma~\ref{lem:g34} and Proposition~\ref{prop:hyperelliptic},
if~$\g(X) \in \{5,6\}$ we apply Lemma~\ref{lem:g56} and Proposition~\ref{prop:trigonal},
and if~$\g(X) \ge 7$ we apply Proposition~\ref{prop:cones}.
In all cases we see that~$-K_{\hX}$ is nef and big but not ample.

Finally, assume~$\io(f_1) = \io(f_2) = 1$ and~$\g(X) = 2$.
Then we apply Lemma~\ref{lem:g2} and Proposition~\ref{prop:base-points} and see that~$-K_{\hX}$ is nef but not big.

Since we exhausted all possibilities, it follows that the implications are equivalences.
\end{proof} 

Now we can prove the main results of this section.

\begin{proof}[Proof of Theorems~\xref{thm:intro-nf} and~\xref{thm:intro-nf-contractions}]
Let~$(X,x_0)$ be a nonfactorial $1$-nodal Fano threefold.

If~$\io(f_1) > 1$ or~$\io(f_2) > 1$, 
so that by Lemma~\ref{lem:io-k} we are in case~\ref{it:intro-nf-ample} of Theorem~\ref{thm:intro-nf},
we apply Propositions~\ref{prop:if3} and~\ref{prop:big-index}.

Similarly, if~$\io(f_1) = \io(f_2) = 1$ and~$\g(X) = 2$,
so that by Lemma~\ref{lem:io-k} we are in case~\ref{it:intro-nf-g2} of Theorem~\ref{thm:intro-nf},
we apply Lemma~\ref{lem:g2} and Proposition~\ref{prop:base-points}.

Finally, assume~$\io(f_1) = \io(f_2) = 1$ and~$\g(X) > 2$,
so that by Lemma~\ref{lem:io-k} we are in case~\ref{it:intro-nf-bpf} of Theorem~\ref{thm:intro-nf}.

If~$\g(X) \in \{3,4\}$ we apply Lemma~\ref{lem:g34} and Proposition~\ref{prop:hyperelliptic},
and see that~$\g(X) = 4$ and~$\bar{X}$ is a smoothable Fano threefold with canonical Gorenstein singularities of type~\typemm{2}{1}.

If~$\g(X) \in \{5,6\}$ we apply Lemma~\ref{lem:g56} and Proposition~\ref{prop:trigonal}.
and see that~$\bar{X}$ is a smoothable Fano threefold with terminal or canonical Gorenstein singularities 
of type~\typemm{2}{2} (if~$\g(X) = 5$) or type~\typemm{2}{3} (if~$\g(X) = 6$).
Note that the claims of Theorem~\ref{thm:intro-nf-contractions} in these cases also follow. 

Finally, assume~$\g(X) \ge 7$.
We will prove the remaining claims of Theorems~\xref{thm:intro-nf} and~\xref{thm:intro-nf-contractions}. 
Note that now we are in the situation of Propositions~\ref{prop:cones};
we use below the notation introduced therein.
In particular, we denote by~$\bar{X}$ the anticanonical model of~$\hX = \Bl_{x_0}(X)$, 
which is a Fano threefold with terminal Gorenstein singularities fitting into the diagram~\eqref{eq:barx-diagram},
and by~\mbox{$\bar{H}_i \in \Pic(\bar{X})$} the pullbacks of the ample generators~$H_i$ of~$\Pic(Z_i)$ along the morphisms~$\bar{f}_i$.

First, note that~$\bar{X}$ is smoothable by~\cite[Theorem~11]{Na97}.
Let~$p \colon \cX \to B$ be a smoothing of~$\bar{X}$ over a smooth pointed curve~$(B,o)$, 
so that~$\cX_o \cong \bar{X}$ and~$\cX$ is smooth over~$B \setminus o$.
By Proposition~\ref{prop:picard-sheaf} the fourfold~$\cX$ is factorial and has terminal Gorenstein singularities.
Moreover, we may assume that~$-K_{\cX}$ is $p$-ample, $\uprho(\cX/B) = 2$, the \'etale Picard sheaf~$\Pic_{\cX/B}$ is locally constant,
and the fibers~$\cX_b$ with~$b \ne o$ are smooth Fano varieties with
\begin{equation*}
\uprho(\cX_b) = \uprho(\bar{X}) = 2,
\qquad 
\io(\cX_b) = \io(\bar{X}) = 1,
\qquad\text{and}\qquad 
\g(\cX_b) = \g(\bar{X}) = \g(X) - 1.
\end{equation*}
The last equality implies that~$6 \le \g(\cX_b) \le 11$, $\g(\cX_b) \ne 10$.
Inspecting~\cite[Table~2]{Mori-Mukai:MM} or~\cite{fanography}, we conclude that~$\cX_b$ must be of type~\typemm{2}{m} with
\begin{equation*}
4 \le m \le 14,
\qquad 
m \ne 11.
\end{equation*}

Furthermore, since~$\Pic_{\cX/B}$ is \'etale-locally constant, 
after a base change to an appropriate \'etale neighborhood of~$o \in B$
the classes~\mbox{$\bar{H}_i \in \Pic(\bar{X})$} extend to classes~\mbox{$\cH_i \in \Pic(\cX)$}.
Since, moreover, the canonical class of~$\cX$ provides a section of~$\Pic_{\cX/B}$, 
the relation~\eqref{eq:mkx} in the central fiber implies the relation
\begin{equation}
\label{eq:barh-relation}
\cH_1+ \cH_2 \sim -K_{\cX}
\qquad\text{in~$\Pic(\cX)$}
\end{equation}
(we may need to shrink the base to get rid of an extra divisor class pulled back from~$B$).

The classes~$\bar{H}_i$ are nef but not ample by Proposition~\ref{prop:cones},
hence, shrinking~$B$ if necessary we may assume that the classes~$\cH_i$ are nef but not ample over~$B$.
Therefore, these classes define extremal contractions over~$B$,
so that the morphism~$p$ factors as
\begin{equation*}
\cX \xrightarrow{\ \varphi_i\ } \cZ_i \xrightarrow{\quad} B,
\end{equation*}
where~$\varphi_i$
is the contraction associated to the class~$\cH_i$; 
note that~$\varphi_i$ is $K$-negative because~$-K_\cX$ is ample over~$B$ and that~$\uprho(\cX/\cZ_i) = 1$ because~$\uprho(\cX/B) = 2$.
The scheme~$\cZ_i$ is integral by construction,
and since~$B$ is a smooth curve, it follows that~$\cZ_i$ is flat over~$B$.
Moreover, by construction we have
\begin{equation*}
\cZ_i \cong \Proj_B \left( \moplus p_*\cO_\cX(n\cH_i) \right).
\end{equation*}
Since~$\mathbf{R}^{>0}p_*\cO_\cX(n\cH_i) = 0$ by Kawamata--Viehweg vanishing, 
the base change isomorphism shows that the scheme fiber~$\cZ_{i,o}$ of~$\cZ_i$ over the point~$o \in B$ can be written as
\begin{equation*}
\cZ_{i,o} \cong \Proj \left( \moplus H^0(\bar{X}, \cO_{\bar{X}}(n\bar{H}_i)) \right).
\end{equation*}
Since, moreover, $H^0(\bar{X}, \cO_{ \bar{X}}(n\bar{H}_i)) \cong H^0(\hX, \cO_{\hX}(n H_i)) \cong H^0(X_i, \cO_{X_i}(n H_i))$ 
by~\eqref{eq:barx-diagram},
we see that
\begin{equation*}
\cZ_{i,o} \cong Z_i.
\end{equation*}
Since~$Z_i$ is smooth (Lemma~\ref{lem:fi-zi}), shrinking~$B$ if necessary, we may assume that~$\cZ_i$ is smooth over~$B$.
It also follows that the morphism~$\varphi_{i,o} \colon \cX_o \to \cZ_{i,o}$ coincides with~$\bar{f}_i \colon \bar{X} \to Z_i$.

On the other hand, relation~\eqref{eq:barh-relation} shows that for~$b \ne o$ the classes~$\cH_1\vert_{\cX_b}$ and~$\cH_2\vert_{\cX_b}$ 
provide a nef decomposition for the anticanonical class of the smooth Fano threefold~$\cX_b$
of type~\typemm{2}{m} with~$m \le 14$.
Looking at the descriptions of these threefolds in~\cite{Mori-Mukai:MM} 
we see that the indices of their extremal contractions both equal~$1$, 
hence~\cite[Theorem~5.1(1)]{Mori-Mukai1983} (where these indices are denoted by~$\mu_1$ and~$\mu_2$)
implies that the unique nef decomposition of~$-K_{\cX_b}$ is given 
by the pullbacks of the ample generators of the Picard groups of the targets of the contractions.
This means that the morphisms~\mbox{$\varphi_{i,b} \colon \cX_b \to \cZ_{i,b}$} are the extremal contractions.

We already checked that~$4 \le m \le 14$ and~$m \ne 11$. 
Let us also show that
\begin{equation*}
m \ne 8.
\end{equation*}
Indeed, if~$m = 8$ the description of~\cite[Table~2]{Mori-Mukai:MM} shows that 
for any~$b \ne o$ in~$B$ one of the extremal contractions~$\varphi_{i,b} \colon \cX_b \to \cZ_{i,b}$ is of type~\type{(B_1^0)}
and~$\cZ_{i,b}$ is singular
in contradiction with the smoothness of~$\cZ_i$ over~$B$ explained above.

This completes the proof of Theorem~\ref{thm:intro-nf}.
From now on we assume that~$m \ge 4$ and satisfies~\eqref{eq:m-list}.
It remains to prove claims~\ref{it:intro-b1}--\ref{it:intro-d1} of Theorem~\ref{thm:intro-nf-contractions}. 

First, assume~$\dim \cZ_i = 4$, i.e., the contraction~$\varphi_{i} \colon \cX \to \cZ_i$ is birational.
Since~$\cZ_i$ is flat over~$B$, we have~$\dim\cZ_{i,b} = 3$ for any~$b\in B$ (including~$b = o$),
hence the contraction~$\varphi_{i,b} \colon \cX_b \to \cZ_{i,b}$ is also birational.
By Lemma~\ref{lem:fi-zi} and Proposition~\ref{prop:cones} the map~$\varphi_{i,o} = \bar{f}_i$ 
is the blowup of a singular curve~$\bar\Gamma_i \subset \cZ_{i,o}$
and the constraints on~$m$ explained above together with~\cite[Table~2]{Mori-Mukai:MM}
imply that~$\varphi_{i,b}$ is the blowup of a smooth curve~$\cG_{i,b} \subset \cZ_{i,b}$.
It remains to show that there is a family of curves~$\cG_i \subset \cZ_i$ flat over~$B$ 
with fibers~$\cG_{i,b}$ and~$\bar\Gamma_i$ over~$b \ne o$ and~$b = o$, respectively.
For this just note that
\begin{equation*}
(\varphi_{i,b})_*\cO_{\cX_b}(-K_{\cX_b/\cZ_{i,b}}) \cong I_{\cG_{i,b}}
\qquad\text{and}\qquad
(\bar{f}_i)_*\cO_{\bar{X}}(-K_{\bar{X}/Z_i}) \cong I_{\bar{\Gamma}_i},
\end{equation*}
so we can define the family of curves~$\cG_i$ from the equality~$(\varphi_i)_*\cO_{\cX}(-K_{\cX/\cZ_i}) \cong I_{\cG_i}$
and use base change isomorphisms to identify the fibers of~$\cG_i$ with~$\cG_{i,b}$ and~$\bar\Gamma_i$.
Flatness of~$\cG_i$ over~$B$ follows from flatness of~$I_{\cG_i}$, 
which in its turn follows from flatness of the line bundle~$\cO_{\cX}(-K_{\cX/\cZ_i})$.
Finally, using flatness of~$\cG_i$ and Proposition~\ref{prop:cones} we deduce
\begin{equation*}
\deg(\Gamma_i) = \deg(\bar{\Gamma}_i) - 1 = \deg(\cG_{i,b}) - 1,
\qquad
\g(\Gamma_i) = \p(\bar{\Gamma}_i) - \io(Z_i) + 1 =\g(\cG_{i,b}) - \io(Z_i) + 1,
\end{equation*}
which proves part~\ref{it:intro-b1} of Theorem~\ref{thm:intro-nf-contractions}.

Next, assume~$\dim (\cZ_i)=3$.
Since~$\cZ_i$ is flat over~$B$, we have~$\dim\cZ_{i,b} = 2$ for any~\mbox{$b\in B$} (including~$b = o$),
hence the contractions~$\varphi_{i,b} \colon \cX_b \to \cZ_{i,b}$ 
are conic bundles.
It also follows that~$\cZ_{i,b} \cong \PP^2$ for all~$b\in B$ (including~$b = o$).
Therefore, $\cZ_i$ is a $\PP^2$-bundle over~$B$
and~\mbox{$\cX \to \cZ_i$} is a conic bundle.
Let~$\cD_i\subset \cZ_i$ be its discriminant divisor;
note that~\mbox{$\cD_{i,o} = \bar{\Delta}_i$}.
Since~$\cD_i$ is a Cartier divisor without vertical components, it is flat over~$B$.
Using Proposition~\ref{prop:cones}, we deduce
\begin{equation*}
\deg(\Delta_i) = \deg(\bar{\Delta}_i) - 1 = \deg(\cD_{i,o}) - 1 = \deg(\cD_{i,b}) - 1,
\end{equation*}
which proves part~\ref{it:intro-c1} of Theorem~\ref{thm:intro-nf-contractions}.

Finally, assume~$\dim (\cZ_i)=2$.
Since~$\cZ_i$ is flat over~$B$, we have~$\dim\cZ_{i,b} = 1$ for any~$b\in B$ (including~$b = o$),
hence the contractions~$\varphi_{i,b} \colon \cX_b \to \cZ_{i,b}$ 
are del Pezzo fibrations.
It also follows that~$\cZ_{i,b} \cong \PP^1$ for all~$b\in B$ (including~$b = o$).
Therefore, $\cZ_i$ is a $\PP^1$-bundle over~$B$
and~$\cX \to \cZ_i$ is a del Pezzo fibration;
in particular, it is flat.
Using Proposition~\ref{prop:cones}, we deduce
\begin{equation*}
\deg(X_i/Z_i) = \deg(\bar{X}/Z_i) + 1 = \deg(\cX_o/\cZ_{i,o}) + 1 = \deg(\cX_b/\cX_{i,b}) + 1
\end{equation*}
which proves part~\ref{it:intro-d1} of Theorem~\ref{thm:intro-nf-contractions}.
\end{proof} 

\begin{remark}
\label{rem:other-m}
Let~$X'$ be a smooth Fano threefold of type~\typemm{2}{m} with~$m = 11$ or~$m \ge 15$
and let~$f'_i \colon X' \to Z'_i$, $i = 1, 2$, be its extremal contractions.
Then from the explicit description of~$X'$ in~\cite[Table~2]{Mori-Mukai:MM}
it is easy to see that for one of the contractions~$f'_i \colon X' \to Z'_i$ we have either
\begin{itemize}
\item 
$\io(f'_i) > 1$, or
\item 
$Z'_i$ is singular, or
\item 
$f'_i$ is the blowup of a curve~$\Gamma'_i$ on a threefold~$Z'_i$ such that~$\g(\Gamma'_i) < \io(Z'_i) - 1$.
\end{itemize}
In either case the above proof of Theorem~\ref{thm:intro-nf-contractions} shows that~$X'$ cannot be a smoothing
of the anticanonical model~$\bar{X}$ of the blowup~$\hX$ of a nonfactorial $1$-nodal Fano threefold~$X$ with~$\g(X) \ge 7$.
This gives a ``proof'' of the famous genus bound~$g \le 12$ and~$g \ne 11$ 
for prime Fano threefolds (or rather for their nonfactorial $1$-nodal degenerations).
\end{remark}

To conclude this section we explain how Table~\ref{table:nf} was produced.
Each row corresponds to one of the cases from Theorem~\ref{thm:intro-nf}:
\begin{itemize}
\item 
The light-gray rows correspond to case~\ref{it:intro-nf-ample} of the theorem;
all information in these rows can be extracted from Propositions~\ref{prop:if3} and~\ref{prop:big-index}.
\item 
The dark-gray row corresponds to case~\ref{it:intro-nf-g2}, see Proposition~\ref{prop:base-points} for details.
\item 
The white rows correspond to case~\ref{it:intro-nf-bpf}.
\end{itemize}
The white rows with~$\g(X) \in \{4,5,6\}$ are described in Propositions~\ref{prop:hyperelliptic} and~\ref{prop:trigonal}.
Finally, if~$\g(X) \ge 7$, the targets~$Z_i$ of the extremal contractions~$f_i \colon X_i \to Z_i$
as well as the degree and genus of the curve~$\Gamma_i$ (if~$f_i$ is birational) 
or~$\Delta_i$ (if~$f_i$ is a conic bundle),
or~$\deg(X_i/Z_i)$ (if~$f_i$ is a del Pezzo fibration) 
are identified by a combination of~\cite{Mori-Mukai:MM} with Theorem~\ref{thm:intro-nf-contractions} 
(see parts~\ref{it:intro-b1}, \ref{it:intro-c1}, and~\ref{it:intro-d1} for the respective contraction types).
Finally the degree of the discriminant for del Pezzo fibrations is computed 
as the difference of the topological Euler characteristic of~$X_i$ and twice the Euler characteristic of the general fiber of~$f_i$.

\section{Nonfactorial threefolds and complete intersections}
\label{sec:nf-ci}

In this section we 
prove Theorem~\ref{thm:intro-nf-ci} and finish the classification of nonfactorial $1$-nodal threefolds;
in particular we explain how all varieties from Table~\ref{table:nf} can be constructed.

\subsection{Almost Fano blowups}
\label{ss:af-blowups}

Here we discuss nonfactorial $1$-nodal Fano threefolds~$X$ with a small resolution 
that admit a $K$-negative birational contraction.
As we explained in Theorem~\ref{thm:intro-nf-contractions},
the anticanonical models~$\bar{X}$ of their blowups~$\hX = \Bl_{x_0}(X)$ have smoothings 
by Fano threefolds of type~\typemm{2}{m} with~$m$ in~\eqref{eq:m-list} that admit a birational contraction. 
By~\cite[Table~2]{Mori-Mukai:MM} these are Fano threefolds of type~\typemm{2}{m} with
\begin{equation*}
m \in \{1,3,4,5,7,9,10,12,13,14\}
\end{equation*}
(these correspond to types~\hyperlink{4n}{\ntype{1}{4}{n}}, 
\hyperlink{6n}{\ntype{1}{6}{n}}, 
\hyperlink{7n}{\ntype{1}{7}{n}}, 
\hyperlink{8na}{\ntype{1}{8}{na}}, 
\hyperlink{9na}{\ntype{1}{9}{na}}, 
\hyperlink{10nb}{\ntype{1}{10}{nb}}, 
\hyperlink{10na}{\ntype{1}{10}{na}}, 
\hyperlink{12nc}{\ntype{1}{12}{nc}}, 
\hyperlink{12nb}{\ntype{1}{12}{nb}}, 
\hyperlink{12na}{\ntype{1}{12}{na}}, 
respectively, in Table~\ref{table:nf}).
Since threefolds of type~\hyperlink{4n}{\ntype{1}{4}{n}} were already described in Proposition~\ref{prop:hyperelliptic}\ref{he:nf}
we concentrate on the remaining nine types.

So, let~$X$ be a nonfactorial $1$-nodal Fano threefold of one of types
\begin{equation*}
\hyperlink{6n}{\ntype{1}{6}{n}}, 
\hyperlink{7n}{\ntype{1}{7}{n}}, 
\hyperlink{8na}{\ntype{1}{8}{na}}, 
\hyperlink{9na}{\ntype{1}{9}{na}}, 
\hyperlink{10nb}{\ntype{1}{10}{nb}}, 
\hyperlink{10na}{\ntype{1}{10}{na}}, 
\hyperlink{12nc}{\ntype{1}{12}{nc}}, 
\hyperlink{12nb}{\ntype{1}{12}{nb}}, 
\hyperlink{12na}{\ntype{1}{12}{na}}, 
\end{equation*}
By Theorem~\ref{thm:intro-nf-contractions}, in each of these cases
there is a small resolution of~$X$ (to fix the notation, we denote it by~$X_1$) 
such that the extremal contraction~$f_1 \colon X_1 \to Z_1$ is the blowup of a curve~$\Gamma \subset Z_1$, 
so that~$X_1 \cong \Bl_{\Gamma}(Z_1)$, where
\begin{itemize}[wide]
\item 
$\Gamma$ is a smooth connected curve, and 
\item 
$Z_1 \cong \PP^3$, or~$Z_1 \cong Q^3$, 
or~$Z_1$ is a smooth del Pezzo threefold of degree~$d \in \{2,3,4,5\}$.
\end{itemize}
Moreover, according to Table~\ref{table:nf} in each case
the target of the other extremal contraction~$f_2 \colon X_2 \to Z_2$ is a projective space~$Z_2 \cong \PP^k$ with~$k \in \{1,2,3\}$.
We note that the pair~$(Z_1,Z_2)$ determines the type of~$X$ uniquely,
and using this observation, we list the degree and genus of the blowup centers~$\Gamma$ in the following table:
\begin{table}[H]
\begin{equation*}
\renewcommand{\arraystretch}{1.2}
\renewcommand{\arraycolsep}{0.02\textwidth} 
\begin{array}{c|c!{\vrule width 0.1em}c|c|c}
\io(Z_1) & Z_1 & 
\text{$Z_2 \cong \PP^3$} & 
\text{$Z_2 \cong \PP^2$} & 
\text{$Z_2 \cong \PP^1$}
\\
\noalign{\hrule height 0.1em}
4 &\PP^3 & \Gamma_5^0 & \Gamma_6^2 & \Gamma_8^7
\\
\hline 
3 & Q^3 & \cellcolor{gray!25} & \Gamma_5^0 & \Gamma_7^3
\\
\hline 
2 & \rY_d& \cellcolor{gray!25} & \cellcolor{gray!25} & \Gamma_{d-1}^0
\end{array}
\end{equation*}
\caption{Blowup centers}
\label{table:blowups}
\end{table}
\noindent
Here notation~$\Gamma_d^g$ means that~$\deg(\Gamma) = d$ and~$\g(\Gamma) = g$.

As usual, we write~$H_1$ for the ample generator of~$\Pic(Z_1)$ 
and~$E_1$ for the exceptional divisor of the blowup~$f_1 \colon X_1 = \Bl_\Gamma(Z_1) \to Z_1$. 

\begin{proposition}
\label{prop:base-gl}
Let~$(Z_1,\Gamma)$ and~$Z_2$
be as in Table~\textup{\ref{table:blowups}}
and in the case where~$\Gamma$ is a line on a smooth del Pezzo threefold~$Z_1$ of degree~$2$,
assume additionally that~$\Gamma$ does not lie in the ramification divisor of the anticanonical double covering~$Z_1 \to \PP^3$.

\begin{thmenumerate}
\item 
The blowup~$X_1 \coloneqq \Bl_\Gamma(Z_1)$ is a smooth almost Fano variety if and only if
\begin{equation}
\label{eq:af-condition}
| (\io(Z_1) - 2)H_1 - \Gamma | = \varnothing.
\end{equation} 

\item 
If~\eqref{eq:af-condition} holds, 
then~$\dim(| (\io(Z_1) - 1)H_1 - \Gamma |) = k$, 
where~$k = \dim(Z_2)$, with
\begin{equation}
\label{eq:bs-z1}
\Bs(| (\io(Z_1) - 1)H_1 - \Gamma |) = \Gamma \cup L,
\end{equation} 
where~$L \subset Z_1$ is the unique line such that~$\Gamma \cap L$ is a scheme of length~$\io(Z_1)$,
and the strict transform of~$L$ is the unique $K$-trivial curve on~$X_1$.
Moreover, the twisted ideal sheaf of the curve~$\Gamma \cup L$
on~$Z_1$ has a locally free resolution of the form
\begin{equation}
\label{eq:resolution-ci-gl}
0 \longrightarrow 
\moplus_{i=1}^{k} \cO_{Z_1}(-a_iH_1) \longrightarrow 
\cO_{Z_1}^{\oplus (k + 1)} \longrightarrow 
\cI_{\Gamma \cup L}((\io(Z_1) - 1)H_1) \longrightarrow 
0,
\end{equation} 
where~$a_1 \ge \dots \ge a_k \ge 1$ is the unique collection of~$k$ integers such that~$\sum a_i = \io(Z_1) - 1$.
\end{thmenumerate}
\end{proposition}

\begin{remark}
In the case where~$\Gamma$ is a line on a smooth del Pezzo threefold~$Z_1$ of degree~$2$ lying
in the ramification divisor of the double covering~$Z_1\to \PP^3$,
the blowup~\mbox{$X_1 = \Bl_\Gamma(Z_1)$} is still an almost Fano variety (and its anticanonical model is~$1$-nodal),
but~$L = \Gamma$ and~$\Gamma \cup L$ should be replaced 
by a non-reduced scheme supported on~$\Gamma$ (cf.\ the proof of Proposition~\ref{prop:cones}).
Since this case is discussed in detail in~\cite[\S2]{KS23}, we omit it here.
\end{remark}

\begin{proof}
In the case where~$\io(Z_1) = 2$ the condition~\eqref{eq:af-condition} is void
and all claims of the proposition are proved in~\cite[\S2]{KS23}.
The cases where~$\io(Z_1) \in \{3,4\}$ but~$k = 1$ 
(i.e., $Z_1 = \PP^3$ and~$\Gamma = \Gamma_8^7$ or~$Z_1 = Q^3$ and~$\Gamma = \Gamma_7^3$)
are completely analogous and we leave these cases as an exercise.
So, from now on we concentrate on the remaining three cases:
\begin{aenumerate}
\item 
\label{it:p3g50}
$Z_1 = \:\PP^3$, $\Gamma = \Gamma_5^0$ is a smooth quintic curve of genus~$0$,
\item 
\label{it:p3g62}
$Z_1 = \:\PP^3$, $\Gamma = \Gamma_6^2$ is a smooth sextic curve of genus~$2$,
\item 
\label{it:q3g50}
$Z_1 = Q^3$, $\Gamma = \Gamma_5^0$ is a smooth quintic curve of genus~$0$.
\end{aenumerate} 

First, assume~\eqref{eq:af-condition} fails and let~$D \subset Z_1$ be a divisor in the linear system~$| (\io(Z_1) - 2)H_1 - \Gamma |$.
If~$Z_1 = \PP^3$ then~$D \in |2H - \Gamma|$ and if~$Z_2 = Q^3$ then~$D \in |H - \Gamma|$;
thus, in all three cases~$D$ is a quadric surface.

If~$D$ is smooth, i.e., $D \cong \PP^1 \times \PP^1$, 
case~\ref{it:p3g62} is impossible
(a smooth curve of genus~$2$ must have bidegree~$(2,3)$ or~$(3,2)$, but then its degree is~$5$, not~$6$), 
and in cases~\ref{it:p3g50} and~\ref{it:q3g50} the curve~$\Gamma$ must have bidegree~$(1,4)$ or~$(4,1)$.
In both cases~$D$ is covered by lines which are~$4$-secant to~$\Gamma$,
hence the anticanonical linear system~$|4H_1 - E_1|$ or~$|3H_1 - E_1|$ of~$X_1$ contracts~$D$ or contains it in the base locus;
in either case~$X_1$ is not an almost Fano variety.

Similarly, if~$D$ is a quadratic cone and~$\tD$ is the blowup of its vertex 
with the exceptional curve class~$e$ and ruling~$f$
(so that~$e^2 = -2$, $e \cdot f = 1$, and~$f^2 = 0$),
and if~$\tilde\Gamma \subset \tD$ is the strict transform of~$\Gamma$
then~$\tilde\Gamma \cdot (e + 2f) = \deg(\Gamma) \in \{5,6\}$ and~$\tilde\Gamma \cdot e \in \{0,1\}$ because~$\Gamma$ is smooth.
It follows that in cases~\ref{it:p3g50} and~\ref{it:q3g50} we have~$\tilde\Gamma \sim 2e + 5f$, 
and in case~\ref{it:p3g62} we have~$\tilde\Gamma \sim 3e + 6f$;
in both cases this contradicts the assumption that~$\g(\Gamma) = 0$ or~$\g(\Gamma) = 2$, respectively.

Finally, if~$D$ is reducible or nonreduced, it follows that~$\Gamma$ is contained in a plane,
and we easily get a contradiction between the degree and genus assumptions.

Thus, if~\eqref{eq:af-condition} fails, $X_1$ is not an almost Fano variety.

\medskip

From now on assume~\eqref{eq:af-condition} holds.
Note that~$| (\io(Z_1) - 1)H_1 - 2E_1 | = \varnothing$ ---
when~$Z = \PP^3$ this is
because the singular locus of any irreducible cubic surface is either a finite union of points or a line
(see, e.g., \cite[Theorem~1.1]{Reid:dP94} or~\cite[Theorem~1.5]{Abe2003}),
and when~$Z_1 = Q^3$ the same is true for the singular locus of any irreducible intersection of~$Z_1$ with a quadric.
Combining this with~\eqref{eq:af-condition} we see that the linear system
\begin{equation}
\label{eq:special-ls}
| (\io(Z_1) - 1)H_1 - E_1 |
\end{equation} 
has no fixed components.
Thus, the base locus of~\eqref{eq:special-ls} is at most one-dimensional.
We will show that it consists of a unique irreducible curve.

For this note that~\cite[Corollary]{GLP} implies that 
the curve~$\Gamma$ is cut out by hypersurfaces of degree~$\io(Z_1)$;
in other words, the linear system~$|-K_{X_1}| = |\io(Z_1)H_1 - E_1|$ is base point free, hence nef.
A simple computation also shows that~$(-K_{X_1})^3 > 0$, i.e., $-K_{X_1}$ is big, 
hence~$X_1$ is a weak Fano variety.
On the other hand, $X_1$ is not a Fano variety by the Mori--Mukai classification~\cite{Mori-Mukai:MM},
therefore there is a reduced and irreducible curve~$C \subset X_1$ such that
\begin{equation}
\label{eq:k-cdot-c}
\big(\io(Z_1)H_1 - E_1\big) \cdot C = 0.
\end{equation}
Since~$H_1$ is nef, we have~$H_1 \cdot C \ge 0$.
If~$H_1 \cdot C = 0$ then~\eqref{eq:k-cdot-c} implies~$E_1 \cdot C = 0$, which is absurd.
Therefore, $H_1 \cdot C > 0$, and it follows that
\begin{equation*}
\big((\io(Z_1) - 1)H_1 - E_1\big) \cdot C < 0.
\end{equation*}
Thus, $C$ belongs to the base locus of~\eqref{eq:special-ls}.

Next, we show that~$H_1 \cdot C = \deg(f_1(C)) \le 2$.
Indeed, 
in case~\ref{it:p3g50} a parameter count shows that~$\dim|3H_1 - \Gamma| \ge 3$,
hence the degree of the base locus of this system is at most~\mbox{$9 - 2 = 7$}, 
and since~$\deg(\Gamma) = 5$, we have~$\deg(f_1(C))\le 7 - 5 = 2$.
Analogously, in case~\ref{it:p3g62} we have~\mbox{$\dim|3H_1 - \Gamma| \ge 2$},
the degree of the base locus of this system is at most~\mbox{$9 - 1 = 8$}, 
and since~\mbox{$\deg(\Gamma) = 6$}, we have~\mbox{$\deg(f_1(C)) \le 8 - 6 = 2$}.
Finally, 
in case~\ref{it:q3g50} we have~\mbox{$\dim|2H_1 - \Gamma| \ge 2$},
hence the degree of the base locus is at most~$8 - 1 = 7$, 
and since~$\deg(\Gamma) = 5$, we have~$\deg(f_1(C)) \le 7 - 5 = 2$.

On the other hand, if~$H_1 \cdot C = 2$, hence~$f_1(C)$ is a smooth conic, 
we have~$E_1 \cdot C = 2\io(Z_1)$ by~\eqref{eq:k-cdot-c}, 
hence the curve~$\Gamma$ intersects the plane~$\langle f_1(C) \rangle$ in at least~$2\io(Z_1)$ points (counted with multiplicities);
as this number is larger than~$\deg(\Gamma)$, we conclude that~$\Gamma$ must be contained in the plane~$\langle f_1(C) \rangle$, 
in contradiction with~\eqref{eq:af-condition}.
Thus, we have
\begin{equation}
\label{eq:he-cdot-c}
H_1 \cdot C = 1
\qquad\text{and}\qquad 
E_1 \cdot C = \io(Z_1), 
\end{equation} 
where the second equality again follows from~\eqref{eq:k-cdot-c}.

Now we fix a curve~$C \subset X_1$ satisfying~\eqref{eq:he-cdot-c} and set 
\begin{equation*}
L \coloneqq f_1(C) \subset Z_1.
\end{equation*}
Note that~$H_1 \cdot L = 1$ 
implies~$L \cong \PP^1$ because~$Z_1 = \PP^3$ or~$Z_1 = Q^3$.
Moreover, $\Gamma \ne L$, hence~$C$ is the strict transform of~$L$; in particular~$C \cong \PP^1$ 
and the morphism~$f\vert_C$ is an isomorphism.
Thus, $L \subset \Bs(|(\io(Z_1) - 1)H_1 - \Gamma|)$
and therefore~\eqref{eq:bs-z1} follows from~\eqref{eq:resolution-ci-gl}.
Finally, 
if~\eqref{eq:resolution-ci-gl} is proved, 
it would follow that~the strict transform of~$L$ is the unique $K$-trivial curve on~$X_1$;
in particular, $X_1$ is an almost Fano variety.
So, to finish the proof of the proposition, we prove ~\eqref{eq:resolution-ci-gl} by a case-by-case analysis. 

\subsection*{Case~\ref{it:p3g50}}

In this case we show that there is an exact sequence
\begin{equation}
\label{eq:res-g05-p3}
0 \longrightarrow \cO_{\PP^3}(-4)^{\oplus 3} \longrightarrow \cO_{\PP^3}(-3)^{\oplus 4} \longrightarrow \cI_{\Gamma \cup L} \longrightarrow 0.
\end{equation}
For this we check that the sheaf~$\cI_{\Gamma \cup L}(3)$ is Castelnuovo--Mumford regular.

Indeed, for any~$d \in \ZZ$ consider the exact sequences
\begin{equation*}
0 \longrightarrow \cI_{\Gamma \cup L}(d) \longrightarrow \cO_{\PP^3}(d) \longrightarrow \cO_{\PP^3}(d)\vert_{\Gamma \cup L} \longrightarrow 0
\end{equation*}
and
\begin{equation*}
0 \longrightarrow \cO_\Gamma(5d - 4) \longrightarrow \cO_{\PP^3}(d)\vert_{\Gamma \cup L} \longrightarrow \cO_L(d) \longrightarrow 0,
\end{equation*}
where the second follows from the fact that~$\deg(\Gamma) = 5$ and~$\Gamma \cap L$ is a scheme of length~$4$.
For~$d \in \{1,2\}$ the second sequence implies that~$H^{>0}(\PP^3, \cO_{\PP^3}(d)\vert_{\Gamma \cup L}) = 0$,
and then the first sequence implies that~$H^{>1}(\PP^3, \cI_{\Gamma \cup L}(d)) = 0$.
Moreover, it easily follows that for~$d = 1,2$ the Euler characteristic of the sheaf~$\cI_{\Gamma \cup L}(d)$ is zero, 
and since the space of global sections of this sheaf vanishes by~\eqref{eq:af-condition}, we conclude that
\begin{equation}
\label{eq:van-g05-p3}
H^\bullet(\PP^3, \cI_{\Gamma \cup L}(1)) = H^\bullet(\PP^3, \cI_{\Gamma \cup L}(2)) = 0.
\end{equation}
Finally, the sequences also easily imply that~$H^{3}(\PP^3, \cI_{\Gamma \cup L}) = 0$, 
so the Castelnuovo--Mumford regularity of~$\cI_{\Gamma \cup L}(3)$ follows.

Now, combining the Castelnuovo--Mumford regularity with the cohomology vanishings~\eqref{eq:van-g05-p3} 
we see that~$\cI_{\Gamma \cup L}(3)$ is globally generated and has a resolution of the form
\begin{equation*}
0 \longrightarrow \cO_{\PP^3}(-1)^{\oplus a} \longrightarrow \cO_{\PP^3}^{\oplus b} \longrightarrow \cI_{\Gamma \cup L}(3) \longrightarrow 0,
\end{equation*}
and computing the Euler characteristic~$\upchi(\cI_{\Gamma \cup L}(d))$ for~$d \in \{3,4\}$ 
it is easy to see that~$a = 3$ and~$b = 4$,
hence, after a twist, the resolution takes the form~\eqref{eq:res-g05-p3}.

\subsection*{Case~\ref{it:p3g62}}

We use a similar idea to show that in this case there is an exact sequence
\begin{equation}
\label{eq:res-g26-p3}
0 \longrightarrow \cO_{\PP^3}(-5) \oplus \cO_{\PP^3}(-4) \longrightarrow \cO_{\PP^3}(-3)^{\oplus 3} \longrightarrow \cI_{\Gamma \cup L} \longrightarrow 0.
\end{equation}
Indeed, consider the exact sequences
\begin{equation*}
0 \longrightarrow \cI_{\Gamma \cup L}(d) \longrightarrow \cO_{\PP^3}(d) \longrightarrow \cO_{\PP^3}(d)\vert_{\Gamma \cup L} \longrightarrow 0
\end{equation*}
and
\begin{equation*}
0 \longrightarrow \cO_\Gamma(dH_1\vert_\Gamma - D) \longrightarrow \cO_{\PP^3}(d)\vert_{\Gamma \cup L} \longrightarrow \cO_L(d) \longrightarrow 0,
\end{equation*}
where~$D = \Gamma \cap L$ is a divisor of degree~$4$ on~$\Gamma$.
Condition~\eqref{eq:af-condition} implies that~$h^0(\cI_{\Gamma \cup L}(1)) = 0$, 
hence~$h^0(\cO_{\PP^3}(1)\vert_{\Gamma \cup L}) \ge 4$,
and then the second sequence implies that~$h^0(\cO_\Gamma(H_1\vert_\Gamma - D)) \ge 2$.
But~$H_1\vert_\Gamma - D$ is a divisor of degree~$6 - 4 = 2$ on a curve of genus~$2$, hence 
\begin{equation*}
H_1\vert_\Gamma \sim K_\Gamma + D.
\end{equation*}
It also follows that~$h^1(\cO_{\PP^3}(1)\vert_{\Gamma \cup L}) = 1$, hence~$h^2(\cI_{\Gamma \cup L}(1)) = 1$,
and by Serre duality this gives an extension
\begin{equation*}
0 \longrightarrow \cO_{\PP^3}(-5) \longrightarrow \cF \longrightarrow \cI_{\Gamma \cup L} \longrightarrow 0
\end{equation*}
such that~$H^\bullet(\PP^3, \cF(1)) = 0$.
Moreover, $H^\bullet(\PP^3, \cF(2)) = H^\bullet(\PP^3, \cI_{\Gamma \cup L}(2)) = 0$
by the same argument as in case~\ref{it:p3g50}.
Finally, we show that~$H^3(\PP^3,\cF) = 0$.
Indeed, by the above exact sequences it is enough to check that the induced map 
\begin{equation*}
H^1(\Gamma, \cO_\Gamma(-D)) = 
H^2(\PP^3, \cI_{\Gamma \cup L}) \longrightarrow H^3(\PP^3, \cO(-5))
\end{equation*}
is surjective. 
By Serre duality the dual of this map coincides with the restriction map 
\begin{equation*}
H^0(\PP^3, \cO(1))\longrightarrow H^0(\Gamma, \cO_\Gamma(K_\Gamma + D)).
\end{equation*}
This map is injective, because~$\Gamma$ is not contained in a hyperplane, hence the dual map is surjective, 
and the equality~$H^3(\PP^3,\cF) = 0$ follows.

Therefore, $\cF(3)$ is Castelnuovo--Mumford regular, 
and we conclude that~$\cF(3)$ is globally generated and has a resolution of the form
\begin{equation*}
0 \longrightarrow \cO_{\PP^3}(-1)^{\oplus a} \longrightarrow \cO_{\PP^3}^{\oplus b} \longrightarrow \cF(3) \longrightarrow 0.
\end{equation*}
Finally, computing the Euler characteristic~$\upchi(\cF(d))$ for~$d \in \{3,4\}$ it is easy to see that~$a = 1$ and~$b = 3$,
and the resolution~\eqref{eq:res-g26-p3} easily follows.

\subsection*{Case~\ref{it:q3g50}}

In this case we show that there is an exact sequence
\begin{equation}
\label{eq:res-g05-q3}
0 \longrightarrow \cO_{Q^3}(-3)^{\oplus 2} \longrightarrow \cO_{Q^3}(-2)^{\oplus 3} \longrightarrow \cI_{\Gamma \cup L} \longrightarrow 0.
\end{equation}
This time we have the exact sequences
\begin{equation*}
0 \longrightarrow \cI_{\Gamma \cup L}(d) \longrightarrow \cO_{Q^3}(d) \longrightarrow \cO_{Q^3}(d)\vert_{\Gamma \cup L} \longrightarrow 0
\end{equation*}
and
\begin{equation*}
0 \longrightarrow \cO_\Gamma(5d - 3) \longrightarrow \cO_{Q^3}(d)\vert_{\Gamma \cup L} \longrightarrow \cO_L(d) \longrightarrow 0,
\end{equation*}
where the second follows from the fact that~$\deg(\Gamma) = 5$ and~$\Gamma \cap L$ is a scheme of length~$3$.
Let~$\cS$ be the spinor bundle of~$Q^3$, see~\cite{Ottaviani}.
Note that
\begin{equation}
\label{eq:cs-gamma}
\cS\vert_\Gamma \cong \cO_\Gamma(-2) \oplus \cO_\Gamma(-3).
\end{equation}
Indeed, since~$\rc_1(\cS) = -H_1$, we have~$\cS\vert_\Gamma \cong \cO_\Gamma(-k) \oplus \cO_\Gamma(k-5)$,
where we may assume~$k \le 2$.
Moreover, $k \ge 0$, because~$\cS^\vee$ is globally generated.
Now, if~$k = 0$ or~$k = 1$, then there is a lift of~$\Gamma$ in~$\PP_{Q^3}(\cS)$ 
whose image under the natural projection~$\PP_{Q^3}(\cS) \to \PP^3$ is a point or line, 
and then it follows that~$\Gamma$ is contained in a hyperplane section of~$Q^3$, contradicting~\eqref{eq:af-condition}.
Thus, $k = 2$, which proves~\eqref{eq:cs-gamma}.

Next, using the second exact sequence we show that~$H^{> 0}(Q^3, \cS^\vee \otimes \cO_{\Gamma \cup L}) = 0$,
and then using the first sequence we check that~$H^{> 1}(Q^3, \cS^\vee \otimes \cI_{\Gamma \cup L}) = 0$.
On the other hand, we have~$H^0(Q^3, \cS^\vee \otimes \cI_{\Gamma \cup L}) = 0$, 
because the zero locus of any global section of~$\cS^\vee$ is a line on~$Q^3$, hence it cannot contain~$\Gamma \cup L$.
Further, a simple computation shows that~$\upchi(\cS^\vee \otimes \cI_{\Gamma \cup L}) = 0$, 
hence~$H^\bullet(Q^3, \cS^\vee \otimes \cI_{\Gamma \cup L}) = 0$.
Moreover, $H^\bullet(Q^3, \cI_{\Gamma \cup L}(1)) = 0$
by the argument of case~\ref{it:p3g50}. 

Now consider the decomposition of the sheaf~$\cI_{\Gamma \cup L}$ 
with respect to the exceptional collection~$(\cO_{Q^3}(-3),\cO_{Q^3}(-2),\cO_{Q^3}(-1),\cS)$ on~$Q^3$.
The vanishings proved above imply that the decomposition takes the form of an exact sequence
\begin{equation*}
0 \to \cO_{Q^3}(-3)^{\oplus a} \to \cO_{Q^3}(-2)^{\oplus b} \to \cI_{\Gamma \cup L} \to 
\cO_{Q^3}(-3)^{\oplus c} \to \cO_{Q^3}(-2)^{\oplus d} \to 0.
\end{equation*}
On the other hand, it is easy to check that~$H^3(Q^3, \cI_{\Gamma \cup L}) = 0$, which implies~$c = 0$, and hence~$d = 0$.
Finally, computing the Euler characteristic~$\upchi(\cI_{\Gamma \cup L}(d))$ for~$d \in \{2,3\}$ 
it is easy to see that~$a = 2$ and~$b = 3$,
and the resolution~\eqref{eq:res-g05-q3} easily follows.
\end{proof} 

\begin{corollary}
\label{cor:barx-bu-gl}
Under the assumptions of Proposition~\textup{\ref{prop:base-gl}} 
if the condition~\eqref{eq:af-condition} holds 
then~$\Bl_{\Gamma \cup L}(Z_1)$ is a complete intersection in~$Z_1 \times \PP^k$ 
of ample hypersurfaces listed in Table~\textup{\ref{table:ci-nf}}.
Moreover, the preimage of~$L \subset Z_1$ under the morphism~$\bar{f}_1 \colon \Bl_{\Gamma \cup L}(Z_1) \to Z_1$
is the surface
\begin{equation*}
\Sigma = L \times L_2 \subset Z_1 \times \PP^k,
\end{equation*}
where~$L_2 \subset \PP^k$ is a line.
\end{corollary}

\begin{proof}
We consider the first morphism in~\eqref{eq:resolution-ci-gl} and apply~\cite[Lemma~2.1]{K16} to its dual
\begin{equation}
\label{eq:dual-morphism}
\cO_{Z_1}^{\oplus (k + 1)} \longrightarrow \moplus_{i=1}^{k} \cO_{Z_1}(a_iH_1).
\end{equation}
Since the curves~$\Gamma$ and~$L$ are smooth, their union is a planar curve, hence it is a local complete intersection,
hence its ideal sheaf has fibers of dimension~$2$ at every point of~$\Gamma \cup L$.
Thus, the first degeneration scheme of~\eqref{eq:dual-morphism} is~$\Gamma \cup L$, 
while the second degeneration scheme is empty.
Therefore, \cite[Lemma~2.1]{K16} proves that
\begin{equation*}
\Bl_{\Gamma \cup L}(Z_1) \subset \PP_{Z_1}(\cO_{Z_1}^{\oplus(k+1)}) \cong Z_1 \times \PP^k 
\end{equation*}
is equal to the zero locus of a global section of the vector bundle~$\bigoplus \cO_{Z_1 \times \PP^k}(a_iH_1 + H_2)$,
where~$H_2$ stands for the hyperplane class of~$Z_2 = \PP^k$; 
this proves that~$\Bl_{\Gamma \cup L}(Z_1) \subset Z_1 \times \PP^k$ 
is a complete intersection of ample divisors listed in Table~\ref{table:ci-nf}.

To identify the preimage of~$L$ in~$\bar{X}$ we recall that~$\cN_{\Gamma \cup L / Z_1}$ 
is a locally free sheaf of rank~$2$ (because~$\Gamma \cup L$ is a local complete intersection) and
\begin{equation*}
\deg(\cN_{\Gamma \cup L / Z_1}\vert_L) = \deg(\cN_{L/Z_1}) + \operatorname{length}(\Gamma \cap L) = (\io(Z_1) - 2) + \io(Z_1) = 2\io(Z_1) - 2,
\end{equation*}
where the first equality uses the argument of~\cite[Lemma~A.2.1]{KPS}.
On the other hand, exact sequence~\eqref{eq:resolution-ci-gl} 
implies that the bundle~$\cN_{\Gamma \cup L / Z_1}^\vee(\io(Z_1) - 1)\vert_L$ is globally generated.
Therefore it is trivial, hence
\begin{equation*}
\Sigma \coloneqq \bar{f}_1^{-1}(L) = \PP_L(\cN_{\Gamma \cup L / Z_1}\vert_L) \cong L \times \PP^1,
\end{equation*}
where~$\bar{f}_1 \colon \bar{X} = \Bl_{\Gamma \cup L}(Z_1) \to Z_1$ is the blowup morphism.
Since, moreover, the morphism~$\Sigma \to Z_1 \times \PP^k$ 
is induced by the surjection~$\cO_L^{\oplus (k+1)} \to \cN_{\Gamma \cup L / Z_1}^\vee(\io(Z_1) - 1)\vert_L \cong \cO_L^{\oplus 2}$,
the image of~$\Sigma$ in~$Z_1 \times \PP^k$ is equal to the product of~$L$ and a line in~$\PP^k$.
\end{proof}

\begin{corollary}
\label{cor:ci-bir}
Let~$(Z_1,\Gamma)$ be as in Table~\textup{\ref{table:blowups}} and~$X_1 \coloneqq \Bl_\Gamma(Z_1)$.
If~\eqref{eq:af-condition} holds then the anticanonical contraction~$\pi_1 \colon X_1 \to X$ 
contracts a single smooth rational curve~$C \subset X_1$
and the anticanonical model~$X$ of~$X_1$ is a nonfactorial $1$-nodal Fano threefold with the node~$x_0 = \pi_1(C)$. 
Moreover, the anticanonical model~$\bar{X} \coloneqq \hX_\can$ of~$\hX \coloneqq \Bl_{x_0}(X)$ 
is isomorphic to~$\Bl_{\Gamma \cup L}(Z_1)$, where~$L$ is the line described in Proposition~\textup{\ref{prop:base-gl}};
in particular, $\bar{X}$ is a complete intersection in~$Z_1 \times \PP^k$ 
of hypersurfaces listed in Table~\textup{\ref{table:ci-nf}}.
\end{corollary}

\begin{proof}
Since the curve~$\Gamma \cup L$ is planar, 
the argument of Proposition~\ref{prop:cones} shows that there is a small (hence crepant) birational contraction
\begin{equation*}
\xi' \colon \hX \coloneqq \Bl_C(\Bl_\Gamma(Z_1)) \to \Bl_{\Gamma \cup L}(Z_1) 
\end{equation*}
and the target is normal.
On the other hand, Corollary~\ref{cor:barx-bu-gl} implies that 
\begin{equation*}
-K_{\Bl_{\Gamma \cup L}(Z_1)} = (\io(Z_1)H_1 + (k+1)H_2) - \sum_{i=1}^k(a_iH_1 + H_2) = H_1 + H_2
\end{equation*}
is ample.
Therefore, $\xi' = \xi$ is the anticanonical contraction and~$\bar{X} \cong \Bl_{\Gamma \cup L}(Z_1)$.

Moreover, we proved in Proposition~\ref{prop:base-gl} that~$X_1 = \Bl_{\Gamma}(Z_1)$ is an almost Fano variety
and the unique $K$-trivial curve on~$X_1$ is the strict transform~$C$ of the line~$L$.
To show that the anticanonical model~$X$ of~$X_1$ is $1$-nodal it remains to compute the normal bundle of~$C$.
For this we note that if~$E \subset \hX$ is the exceptional divisor of~$\hX \to X_1$ then the morphism
\begin{equation*}
\xi\vert_E \colon E \cong \PP_{C}(\cN_{C/X_1}) \to \Sigma \cong L \times\PP^1
\end{equation*}
is birational, and since the source and target have the same Picard number, it is an isomorphism.
Since, on the other hand, $\deg(\cN_{C/X_1}) = -2$ because~$C$ is $K$-trivial, 
we conclude that~$\cN_{C/X_1} \cong \cO_C(-1)^{\oplus 2}$, hence~$X$ is indeed $1$-nodal.
It is nonfactorial, because~$X_1$ provides for it a small resolution.
\end{proof}

\subsection{Conic bundles}
\label{ss:af-conic}

In this subsection we discuss nonfactorial $1$-nodal Fano threefolds with a small resolution that has a conic bundle structure.
By Theorem~\ref{thm:intro-nf-contractions} and~\cite{Mori-Mukai:MM} their smoothings 
are Fano threefolds of type~\typemm{2}{m} with~$m \in \{2,\, 6,\, 9,\, 13\}$
(these correspond to types~\hyperlink{5n}{\ntype{1}{5}{n}}, 
\hyperlink{8nb}{\ntype{1}{8}{nb}}, 
\hyperlink{10na}{\ntype{1}{10}{na}}, 
\hyperlink{12nb}{\ntype{1}{12}{nb}}, 
respectively, in Table~\ref{table:nf}).
The last two of them have a birational contraction on the other side of the Sarkisov link,
and therefore the corresponding varieties~$\bar{X}$ were described in~\S\ref{ss:af-blowups},
while the first threefold is described in detail in Proposition~\ref{prop:trigonal}.
For this reason we concentrate on type~\hyperlink{8nb}{\ntype{1}{8}{nb}} (that corresponds to~$m = 6$).
In this case both extremal contractions~$f_1$ and~$f_2$ are conic bundles,
and to fix the notation we consider~$f_2$.

\begin{proposition}
\label{prop:resolution-ce}
Let~$X$ be a nonfactorial $1$-nodal Fano threefold of type~\hyperlink{8nb}{\ntype{1}{8}{nb}} in the notation of Table~\xref{table:nf}
and let~$\pi_2 \colon X_2 \to X$ be a small resolution.
Then~$X_2 \subset \PP_{\PP^2}(\cE)$ is a divisor of class~$2H - H_2$ 
in the projectivization of the vector bundle~$\cE$ defined by an exact sequence
\begin{equation}
\label{eq:resolution-ce}
0 \longrightarrow \cE \longrightarrow \cO_{\PP^2}(-1)^{\oplus 3} \longrightarrow \cO_L(1) \longrightarrow 0,
\end{equation}
where~$L \subset \PP^2$ is a line, equal to the image~$f_2(C_2)$ of the flopping curve~$C_2 \subset X_2$.
\end{proposition}

\begin{remark}
\label{rem:-CB}
A similar (actually, slightly more complicated) computation allows one to show that
for Fano threefolds~$X$ of types~$\hyperlink{12nb}{\ntype{1}{12}{nb}}$, 
\hyperlink{10na}{\ntype{1}{10}{na}}, 
and~\hyperlink{5n}{\ntype{1}{5}{n}}
the small resolution~$X_2$ of~$X$ that has a conic bundle structure
is a divisor of class~$2H - 3H_2$, $2H -2H_2$ and~$2H + H_2$, respectively,
in the projective bundle~$\PP_{\PP^2}(\cE)$, where~$\cE^\vee$ fits into an exact sequence
\begin{align*}
0 \longrightarrow \cO_{\PP^2}^{\oplus 2} \longrightarrow \cO_{\PP^2}(1)^{\oplus 5} \longrightarrow \cE^\vee \longrightarrow \cO_L \longrightarrow 0&,
&&\text{for type~\hyperlink{12nb}{\ntype{1}{12}{nb}},}
\\
0 \longrightarrow \cO_{\PP^2} \longrightarrow \cO_{\PP^2}(1)^{\oplus 4} \longrightarrow \cE^\vee \longrightarrow \cO_L \longrightarrow 0&,
&&\text{for type~\hyperlink{10na}{\ntype{1}{10}{na}},}
\\
0 \longrightarrow \cO_{\PP^2}(-1) \oplus \cO_{\PP^2}(1)^{\oplus 2} \longrightarrow \cE^\vee \longrightarrow \cO_L \longrightarrow 0&,
&&\text{for type~\hyperlink{5n}{\ntype{1}{5}{n}}.}
\end{align*}
Alternatively, one can deduce these exact sequences using the representation of~$\bar{X}$ 
as a complete intersection (for types~\hyperlink{12nb}{\ntype{1}{12}{nb}} and~\hyperlink{10na}{\ntype{1}{10}{na}}, see Corollary~\ref{cor:barx-bu-gl})
or as a double covering (for type~\hyperlink{5n}{\ntype{1}{5}{n}}, see Remark~\ref{rem:barx-g456})
and the argument of Lemma~\ref{lem:dp-bundle}.
\end{remark} 

\begin{proof}
First, according to Table~\ref{table:nf} the extremal contractions of both small resolutions~$X_1$ and~$X_2$ of~$X$ are conic bundles.
Moreover, since~$\balpha(X) = 1$ by Corollary~\ref{cor:alpha}, 
the relations~\eqref{eq:mkx} and~\eqref{eq:mkhx} take the form~$H \sim H_1 + H_2$ in~$\Cl(X)$ 
and~$H - E \sim H_1 + H_2$ in~$\Pic(\hX)$.
In particular, we have~$\dim |H - H_2| = \dim |H_1| = 2$.

Now let~$\cE \coloneqq (f_{2*}\cO_{X_2}(H))^\vee$, so that~$X_2 \subset \PP_{\PP^2}(\cE)$.
Then~$\cE$ is a vector bundle and
\begin{align*}
&\upchi(\cE^\vee) = \upchi(\cO_{X_2}(H)) = \g(X) + 2 = 10,
\\
&\upchi(\cE^\vee(-H_2)) = \upchi(\cO_{X_2}(H - H_2)) = \upchi(\cO_{\hX}(H_1 + E)) = \upchi(\cO_{X_1}(H_1)) = 3,
\end{align*}
where the first line follows from~\eqref{eq:h0-mk}.
Applying the Riemann--Roch Theorem, we find 
\begin{equation}
\label{eq:chern-cevee}
\rank(\cE^\vee) = 3,
\qquad
\rc_1(\cE^\vee) = 4H_2,
\qquad 
\rc_2(\cE^\vee) = 7.
\end{equation} 
Therefore, $K_{\PP_{\PP^2}(\cE)} = H_2 - 3H$, and since~$K_{X_2} = -H$, 
we conclude that~$X_2 \subset \PP_{\PP^2}(\cE)$ is a divisor of class~$2H - H_2$.

Let~$C_2 \subset X_2$ be the exceptional curve of the anticanonical contraction~$\pi_2 \colon X_2 \to X$.
Then~$H_2 \cdot C_2 = 1$ by~\eqref{eq:deg-ci}, hence~$L \coloneqq f_2(C_2) \subset \PP^2$ is a line.
Since~$X_1$ is obtained from~$X_2$ by flopping the curve~$C_2$ and~$|H_1|$ is base point free on~$X_1$, 
we conclude that on~$X_2$ we have
\begin{equation*}
\Bs(|H - H_2|) = C_2.
\end{equation*}
Therefore, we have a right-exact sequence
\begin{equation*}
\cO_{X_2}(H_2 - H)^{\oplus 3} \longrightarrow \cO_{X_2} \longrightarrow \cO_{C_2} \longrightarrow 0.
\end{equation*}
Twisting it by~$\cO_{X_2}(H)$ and using the fact that~$H\vert_{C_2} \sim0$, we obtain and exact sequence
\begin{equation}
\label{eq:def-ck}
0 \longrightarrow \cK \longrightarrow \cO_{X_2}(H_2)^{\oplus 3} \longrightarrow \cO_{X_2}(H) \longrightarrow \cO_{C_2} \longrightarrow 0,
\end{equation}
where~$\cK$ is a vector bundle of rank~$2$ on~$X_2$.
Pushing it forward along the map~$f_2 \colon X_2 \to \PP^2$, we obtain a complex
\begin{equation}
\label{eq:complex-cevee}
0 \longrightarrow \cO_{\PP^2}(1)^{\oplus 3} \xrightarrow{\ \varphi\ } \cE^\vee \xrightarrow{\quad} \cO_L \xrightarrow{\quad} 0,
\end{equation}
whose cohomology sheaves are isomorphic to~$\mathbf{R}^if_{2*}(\cK)$.

On the one hand, all fibers of~$f_2$ are 1-dimensional, hence~$\mathbf{R}^2f_{2*}(\cK) = 0$, 
and therefore~\eqref{eq:complex-cevee} is exact in the right term.
On the other hand, if~$z \in Z_2$ is a general point, then~$f_2^{-1}(z)$ is a smooth conic, i.e., 
$f_2^{-1}(z) \cong \PP^1$, 
$\cO_{X_2}(H)\vert_{f_2^{-1}(z)} \cong \cO_{\PP^1}(2)$, 
and the restriction of the middle arrow of~\eqref{eq:def-ck} to~$f_2^{-1}(z)$ is the evaluation morphism
\begin{equation*}
H^0(f_2^{-1}(z), \cO_{f_2^{-1}(z)}(2)) \otimes \cO_{f_2^{-1}(z)} \to \cO_{f_2^{-1}(z)}(2).
\end{equation*}
It follows that~$\cK\vert_{f_2^{-1}(z)} \cong \cO_{f_2^{-1}(z)}(-1)^{\oplus 2}$ is acyclic, 
hence~$\mathbf{R}^0f_{2*}(\cK) = 0$,
and therefore~\eqref{eq:complex-cevee} is exact in the left term.
Finally, comparing the total Chern class of~$\cE^\vee$ computed in~\eqref{eq:chern-cevee} 
with the product of total Chern classes of the other terms of~\eqref{eq:complex-cevee}, 
it is easy to see that the complex~\eqref{eq:complex-cevee} is exact everywhere.
It remains to dualize~\eqref{eq:complex-cevee} and obtain~\eqref{eq:resolution-ce}.
\end{proof}

This result has an immediate corollary. 

\begin{corollary}[{cf.~\cite[\S3.4]{P:ratFano2:22}}]
\label{cor:barx-cb}
Let~$(X,x_0)$ be a nonfactorial $1$-nodal Fano threefold of type~\hyperlink{8nb}{\ntype{1}{8}{nb}}. 
Then the anticanonical model~$\bar{X}$ of~$\hX \coloneqq \Bl_{x_0}(X)$ 
is a divisor
\begin{equation*}
\bar{X} \subset Z_1 \times Z_2 = \PP^2 \times \PP^2 
\end{equation*}
of bidegree~$(2,2)$ containing a surface~$\Sigma = \PP^1 \times \PP^1 \subset  \PP^2 \times \PP^2$.
\end{corollary}

\begin{proof}
We recall the argument of Proposition~\ref{prop:cones},
where a vector bundle~$\bar{\cE}$ on~$\PP^2$ was constructed 
and~$\bar{X}$ was identified with a divisor in~$\PP_{\PP^2}(\bar{\cE})$ of relative degree~2.
More precisely, as~$X \subset \PP_{\PP^2}(\cE)$ is a divisor of type~$2H - H_2$, 
the argument shows that~$\bar{X} \subset \PP_{\PP^2}(\bar{\cE})$ is a divisor of type~$2\bar{H}$,
where~$\bar{H}$ is the relative hyperplane class.
Now, comparing the construction of~$\bar{\cE}$ in Proposition~\ref{prop:cones} with~\eqref{eq:resolution-ce} we see that
\begin{equation*}
\bar{\cE} \cong \cO_{\PP^2}(-1)^{\oplus 3},
\end{equation*}
hence~$\bar{X}$ is a divisor in~$\PP^2 \times \PP^2$ of bidegree~$(2,2)$.
As usual, the anticanonical contraction~\mbox{$\xi \colon \hX = \Bl_{C_2}(X_2) \to \bar{X}$} 
is birational on the exceptional divisor~$E \cong \PP^1 \times \PP^1$ of~$\hX$, 
hence it induces an isomorphism of~$E$ onto its image~$\Sigma \subset \bar{X} \subset  \PP^2 \times \PP^2$.
\end{proof} 

\subsection{Proof of Theorem~\ref{thm:intro-nf-ci}}

Now, we combine the results of~\S\S\ref{ss:af-blowups}--\ref{ss:af-conic}.

\begin{proof}[Proof of Theorem~\textup{\ref{thm:intro-nf-ci}}]
Since~$\g(X) \ge 6$, it follows from Table~\ref{table:nf} that either
\begin{itemize}
\item 
the contraction~$f_1 \colon X_1 \to Z_1$ is birational, hence we are in the situation of~\S\ref{ss:af-blowups},
\item 
$X$ has type~\hyperlink{8nb}{\ntype{1}{8}{nb}} and we are in the situation of~\S\ref{ss:af-conic}.
\end{itemize}
Therefore, part~\ref{thm:nf-ci-x-barx} of the theorem follows 
from Corollaries~\ref{cor:barx-bu-gl}, \ref{cor:ci-bir}, and~\ref{cor:barx-cb}.

To prove part~\ref{thm:nf-ci-barx-x} assume that~$\bar{X} \subset Z \times \PP^k$ is a complete intersection of ample divisors listed in Table~\ref{table:ci-nf}
such that~\eqref{eq:barx-chain} holds, $\bar{X}$ has terminal singularities,
and the blowup~$\Blw{\Sigma}(\bar{X})$ of the Weil divisor class~$\Sigma$ on~$\bar{X}$ (see Appendix~\ref{App:blowup}) is smooth.
Let
\begin{equation*}
\xi \colon \tX \coloneqq \Blw{-\Sigma}(\bar{X}) \to \bar{X}
\end{equation*}
be the blowup of the opposite Weil divisor class~$-\Sigma$.
Then~$\tX$ is obtained from~$\Blw{\Sigma}(\bar{X})$ by a flop, see Corollary~\ref{cor:bl-weil-flop}, 
hence it is also smooth by~\cite[Theorem~2.4]{Kollar:flops}.
Moreover, if~$E \subset \tX$ is the strict transform of~$\Sigma$ 
then Lemma~\ref{lem:ddprime} proves that~$E$ is $\xi$-ample and~$E \cong \Sigma \cong \PP^1 \times \PP^1$. 

On the other hand, the computation of Corollary~\ref{cor:ci-bir} shows that~$-K_{\bar{X}} = H_1 + H_2$. 
Let~$\tilde{H} \coloneqq \xi^*(H_1 + H_2)$.
Note that~$H_1$ and~$H_2$ are base point free, hence the same is true for~$\tilde{H}$.
Besides, $H^1(\tX, \cO_{\tX}(\tilde{H})) = H^1(\bar{X}, \cO_{\bar{X}}(\bar{H})) = 0$ by Kawamata--Viehweg vanishing.
We will use this to show that the class
\begin{equation*}
H \coloneqq \tilde{H} + E.
\end{equation*}
is base point free.
Indeed, since~$\xi$ is small and~$\bar{X}$ is Gorenstein, we have~$K_{\tX} \sim \xi^*K_{\bar{X}} \sim - \tilde{H}$,
and since~$E \cong \PP^1 \times \PP^1$, the adjunction formula for~$E$ gives
\begin{equation*}
\cO_E(-2,-2) \cong \cO_E(K_E) \cong \cO_E(K_{\tX} + E) \cong \cO_E(- \tilde{H}) \otimes \cO_E(E).
\end{equation*}
As~$\cO_E(-\tilde{H}) \cong \cO_\Sigma(-H_1-H_2) \cong \cO_E(-1,-1)$, 
it follows that~$\cO_E(E) \cong \cO_E(-1,-1)$, 
hence
\begin{equation*}
\cO(H)\vert_E \cong \cO(\tilde{H} + E)\vert_E \cong \cO_E(1,1) \otimes \cO_E(-1,-1) \cong \cO_E.
\end{equation*}
Using this, we obtain an exact sequence
\begin{equation*}
0 \longrightarrow \cO_{\tX}(\tilde{H}) \longrightarrow \cO_{\tX}(H) \longrightarrow \cO_E \longrightarrow 0.
\end{equation*}
Since the first and last terms are globally generated and~$H^1(\tX, \cO_{\tX}(\tilde{H})) = 0$,
the middle term is also globally generated, i.e., $H$ is base point free.

Now, consider the contraction~$\pi \colon \tX \to X$ given by (an appropriate multiple of) the class~$H$.
If~$\Upsilon \subset \tX$ is a curve contracted by~$\pi$, we have~$(\tilde{H} + E) \cdot \Upsilon = 0$.

Assume~$\Upsilon$ is not contained in~$E$; then~$E \cdot \Upsilon \ge 0$.
On the other hand, $\tilde{H} \cdot \Upsilon \ge 0$, because~$\tilde{H}$ is nef.
Therefore, we must have~$\tilde{H} \cdot \Upsilon = 0$ and~$E \cdot \Upsilon = 0$.
The first equality means that~$\Upsilon$ is contracted by the morphism~$\xi \colon \tX \to \bar{X}$, 
and therefore contradicts the second equality, because~$E$ is $\xi$-ample. 

Thus, all curves contracted by~$\pi$ are contained in~$E$.
Since on the other hand, $H\vert_E \sim 0$, it follows that~$\pi$ is birational and contracts only~$E$.
Since the normal bundle of~$E$ is~$\cO(-1,-1)$, the point~$x_0 \coloneqq \pi(E) \in X$ is a node,
and since~$\tX$ is smooth, it follows that~$X$ is $1$-nodal.

Furthermore, $\uprho(\tX) = 3$ by Proposition~\ref{prop:rho-3}
and the rulings of~$E \cong \PP^1 \times \PP^1$ are not numerically equivalent 
(because, $\cO_E(\bar{H}_1) \cong \cO_E(1,0)$ and~$\cO_E(\bar{H}_2) \cong \cO_E(0,1)$), hence
\begin{equation*}
\uprho(X) = \uprho(\tX) - 2 = 3 - 2 = 1,
\end{equation*}
and~$X$ is not factorial.
\end{proof} 

\begin{remark}
In Section~\ref{sec:special} we studied properties of the anticanonical linear system of nonfactorial Fano threefolds with 1-node.
On the other hand, the anticanonical models of almost Fano threefolds~$\hX$ from Theorem~\ref{thm:intro-nf-contractions}
also provide examples of nonfactorial Fano threefolds;
from this perspective it is interesting to look at the respective properties of their anticanonical linear systems.
So, we observe that if~$(X,x_0)$ is a nonfactorial $1$-nodal Fano threefold 
with~$\uprho(X) = 1$, $\io(f_1) = \io(f_2) = 1$, and~$\g(X) \ge 7$ 
then the anticanonical model~\mbox{$\bar{X} = \Bl_{x_0}(X)_\can$} has very ample canonical class 
by Corollary~\ref{cor:intersection-of-quadrics} and Lemma~\ref{lem:barx}\ref{it:barx-iq}.
Furthermore, $\bar{X}$ is hyperelliptic when~$\g(X) \in \{5,6\}$ 
(with the image in~$\PP^{\g(X)}$ equal to~$\PP^2 \times \PP^1$ and a hyperplane section of~$\PP^3 \times \PP^1$, respectively),
and when~$\g(X) = 4$ the base locus of the anticanonical linear system on~$\bar{X}$ is not empty.
\end{remark}

\subsection{Del Pezzo fibrations}

We prove the following fact for completeness.

\begin{lemma}
\label{lem:dp-bundle}
Let~$X$ be a nonfactorial $1$-nodal Fano threefold with~$\uprho(X) = 1$ and~\mbox{$\g(X) \ne 2$}.
Assume~$\io(f_1) = \io(f_2) = 1$.
If~$X_2 \to X$ is a small resolution such that the extremal contraction~$f_2 \colon X_2 \to Z_2 = \PP^1$
is a del Pezzo fibration 
and~$\cE \coloneqq ({f_2}_*\cO_{X_2}(H))^\vee$ 
then
\begin{equation*}
\cE \cong \cO \oplus \cO(-1)^{\oplus a} \oplus \cO(-2)^{\oplus b},
\end{equation*}
where~$a = 3\deg(X_2/\PP^1) - \g(X) - 1$ and~$b = \g(X) - 2\deg(X_2/\PP^1) + 1$.
\end{lemma}

\begin{proof}
If~$3 \le \g(X) \le 6$, the result follows 
from Lemmas~\ref{lem:g34}, \ref{lem:g56} and Propositions~\ref{prop:hyperelliptic}, \ref{prop:trigonal}.

Assume~$\g(X) \ge 7$.
Consider on~$\hX$ the natural exact sequence
\begin{equation*}
0 \longrightarrow 
\cO_{\hX}(H - E) \longrightarrow 
\cO_{\hX}(H) \longrightarrow 
\cO_E \longrightarrow 
0.
\end{equation*}
Pushing it forward along the composition~$\hX \xrightarrow{\ \sigma_2\ } X_2 \xrightarrow{\ f_2\ } \PP^1$, we obtain an exact sequence
\begin{equation}
\label{eq:cevee-extension}
0 \to f_{2*}\sigma_{2*}\cO_{\hX}(H - E) \longrightarrow 
\cE^\vee \longrightarrow 
\cO_{\PP^1} \longrightarrow 
0,
\end{equation}
where we use that~$\sigma_2 \colon E \to C_2$ is a $\PP^1$-bundle over the flopping curve
and~\mbox{$f_2 \colon C_2 \to Z_2 = \PP^1$} is an isomorphism by~\eqref{eq:deg-ci}.
So, it remains to compute the first term.

For this we use diagram~\eqref{eq:barx-diagram} of Proposition~\ref{prop:cones}:
and the linear equivalences
\begin{equation*}
H - E \sim H_1 + H_2 \sim \xi^*\bar{H}_1 + \xi^*\bar{H}_2,
\end{equation*}
where the first is a combination of~\eqref{eq:mkhx} with Corollary~\ref{cor:alpha}
and the second follows from the definition~$\bar{H}_i = \bar{f}_i^*H_i$ and commutativity of the diagram. 
Now, since~$\bar{X}$ has rational singularities, we have
\begin{equation*}
f_{2*}\sigma_{2*}\cO_{\hX}(H - E) \cong
\bar{f}_{2*}\xi_*\cO_{\hX}(\xi^*\bar{H}_1 + \xi^*\bar{H}_2) \cong
\bar{f}_{2*}\cO_{\bar{X}}(\bar{H}_1 + \bar{H}_2).
\end{equation*}
On the other hand, Theorem~\ref{thm:intro-nf-ci} proves that~$\bar{X} \subset Z_1 \times Z_2 = Z_1 \times \PP^1$ 
is a divisor of class~$(\io(Z_1) - 1)H_1 + H_2$.
Twisting the standard resolution of~$\bar{X}$ by~$H_1 + H_2$, we obtain an exact sequence
\begin{equation*}
0 \longrightarrow \cO_{Z_1 \times \PP^1}((2 - \io(Z_1))H_1) \longrightarrow \cO_{Z_1 \times \PP^1}(H_1 + H_2) \longrightarrow \cO_{\bar{X}}(\bar{H}_1 + \bar{H}_2) \longrightarrow 0,
\end{equation*}
and pushing it forward to~$\PP^1$, we see that
\begin{equation*}
\bar{f}_{2*}\cO_{\bar{X}}(\bar{H_1} + \bar{H}_2) \cong 
\begin{cases}
\hphantom{\Big(}H^0(Z_1,\cO_{Z_1}(H_1)) \otimes \cO(1), & \text{if~$\io(Z_1) \ge 3$},\\
\Big(H^0(Z_1,\cO_{Z_1}(H_1)) \otimes \cO(1)\Big) / \cO, & \text{if~$\io(Z_1) = 2$},
\end{cases}
\end{equation*}
In any case, it follows that this bundle is a direct sum of~$\cO(1)$ and~$\cO(2)$, hence exact sequence~\eqref{eq:cevee-extension} splits,
and~$\cE^\vee \cong \cO \oplus \cO(1)^{\oplus a} \oplus \cO(2)^{\oplus b}$ for some~$a, b \ge 0$.
Finally, computing the rank and the Euler characteristic of~$\cE^\vee$ in terms of~$X_2$, 
we obtain the equalities
\begin{equation*}
1 + a + b = \deg(X_2/\PP^1) + 1 
\qquad\text{and}\qquad 
1 + 2a + 3b = \g(X) + 2.
\end{equation*}
Solving these equations, we obtain the required expressions for~$a$ and~$b$.
\end{proof} 

\section{Classification of factorial threefolds}

In this section we prove Theorems~\ref{thm:intro-factorial-ci}, \ref{thm:intro-fi2} and~\ref{thm:intro-fi1}.
Consequently, from now on we denote by~$(X,x_0)$ a factorial Fano threefold with a single node or cusp~$x_0$ and~$\uprho(X) = 1$, 
and its blowup by~$\hX \coloneqq \Bl_{x_0}(X)$.
Furthermore, we denote by~$H$ the ample generator of~$\Pic(X)$ (as well as its pullback to~$\hX$)
and by~$E \subset \hX$ the exceptional divisor of the blowup.
Note that 
\begin{equation*}
-K_{\hX} = \io(X)H - E,
\qquad\text{and}\qquad 
\uprho(\hX) = 2
\end{equation*}
where as usual~$\io(X)$ is the Fano index of~$X$.

\subsection{Complete intersections}
\label{ss:f-ci}

We start with a simple observation.

\begin{lemma}
\label{lem:factorial-criteria}
If~$X$ is a factorial Fano threefold with a single node or cusp and~$\uprho(X) = 1$ then~$\io(X) \in \{1,2\}$.
Moreover, if~$\io(X) = 2$ then~$\dd(X) \le 4$ and if~$\io(X) = 1$ then~$\g(X) \le 10$.
\end{lemma}

\begin{proof}
If~$X$ is a factorial $1$-nodal or $1$-cuspidal Fano threefold, we have 
\begin{equation*}
\hh(\hX) = \hh(X^\sm) - 1
\end{equation*}
by Proposition~\ref{prop:hh-g-xsm}, hence~$\hh(X^\sm) \ge 1$.
Now lemma follows from~\eqref{eq:smoothing-invariants} and Table~\ref{table:hodge-genus}.
\end{proof}

In what follows we distinguish three cases:
\begin{itemize}[wide]
\item $\io(X) = 2$ and~$1 \le \dd(X) \le 4$, 
\item $\io(X) = 1$ and~$2 \le \g(X) \le 6$, 
\item $\io(X) = 1$ and~$7 \le \g(X) \le 10$.
\end{itemize}
The first two cases are considered in this subsection, and the last case is considered in~\S\ref{ss:proofs-fi1}.

\begin{proof}[Proof of Theorem~\textup{\ref{thm:intro-factorial-ci}}\ref{it:fact-i2} and~\ref{it:fact-i1-gsmall}]
Part~\ref{it:fact-i2} is well-known (see, e.g., \cite[Theorem~1.2]{KP23})
and part~\ref{it:fact-i1-gsmall} is a combination of Lemmas~\ref{lem:g2}, \ref{lem:g34} and~\ref{lem:g56}.
\end{proof}

\begin{proof}[Proof of Theorem~\textup{\ref{thm:intro-fi2}}]
If~$\dd(X) \in \{2,3,4\}$ then~$-K_X$ is very ample by Proposition~\ref{prop:hyperelliptic}.
Furthermore, Lemma~\ref{lem:barx} implies that~$\hX$ is a weak Fano variety.
Finally, $X$ does not have anticanonical lines, because~$\io(X) = 2$, 
hence the argument of Lemma~\ref{lem:barx}\ref{it:barx-almost} shows that~$\hX$ is a Fano variety.
We also have~$\uprho(\hX) = 2$, because~$X$ is factorial, 
so inspecting~\cite[Table~2]{Mori-Mukai:MM} we see that~$\hX$ is of type~\typemm{2}{m} with~$m \in \{8,15,23\}$;
this proves part~\ref{it:intro-fi2-ample}.

Similarly, if~$\dd(X) = 1$, we apply~\cite[Proposition~2.6(iii)]{KS23};
This proves part~\ref{it:intro-fi2-bpf}.
\end{proof} 

The following extension of Lemma~\ref{lem:barx} will be useful for the proof of Theorem~\ref{thm:intro-fi1}.

\begin{lemma}
\label{lem:barx-factorial}
Let~$(X,x_0)$ be a factorial Fano threefold with a single node or cusp 
such that~$\uprho(X) = 1$, $\io(X) = 1$, and~$-K_X$ is very ample.

\begin{thmenumerate}
\item
\label{it:bxf-nhe}
The anticanonical contraction~$\xi \colon \hX \coloneqq \Bl_{x_0}(X) \to \bar{X}$ is nontrivial, 
and~$\bar{X}$ is a Fano threefold with canonical Gorenstein singularities,
$\uprho(\bar{X}) = 1$, $\io(\bar{X}) = 1$, and~\mbox{$\g(\bar{X}) = \g(X) - 1$}.
Moreover, $\xi$ induces an isomorphism~$E \xrightiso{} \Sigma$ 
of the exceptional divisor~$E$ of~$\hX$ onto an irreducible quadric surface~$\Sigma \subset \bar{X}$.
\item
\label{it:bxf-ntr}
If~$\g(X) \ge 5$ then~$\xi$ is small, $\hX$ is almost Fano,
its anticanonical model~$\bar{X}$ is a Fano threefold with terminal Gorenstein singularities, 
and
\begin{equation*}
\hX \cong \Blw{-\Sigma}(\bar{X}).
\end{equation*}
\end{thmenumerate}
\end{lemma}

\begin{proof}
All claims in part~\ref{it:bxf-nhe}, except for non-triviality of~$\xi$ follow from Lemma~\ref{lem:barx}\ref{it:barx-weak}.
So, assume~$\xi$ is an isomorphism.
Then Lemma~\ref{lem:barx}\ref{it:barx-weak} implies that~$\hX$ is a smooth Fano threefold with~$\uprho(\hX) = 2$ 
and a $K$-negative extremal contraction of type~\type{(B_1^0)} onto a Fano threefold of index~1.
But the classification of~\cite[Table~2]{Mori-Mukai:MM} shows that there are no such threefolds,
hence~$\xi$ is indeed nontrivial and~$\uprho(\bar{X}) = 1$.

To prove part~\ref{it:bxf-ntr} note that~$X \cap \rT_{x_0}(X)$ 
is contained in the intersection of~$X\subset \PP^{\g(X) + 1}$ with at least two hyperplanes
because~$\rT_{x_0}(X) = \PP^4$ and~$\g(X) + 1 \ge 6$.
Since~$X$ is factorial, it follows that~$\dim(X \cap \rT_0(X)) \le 1$.
On the other hand, $X \subset \PP^{\g(X)+1}$ is an intersection of quadrics by Corollary~\ref{cor:intersection-of-quadrics}.
Therefore, parts~\ref{it:barx-almost} and~\ref{it:barx-iq} of Lemma~\ref{lem:barx} apply.
\end{proof}

To prove Theorem~\ref{thm:intro-fi1} in the case where~$X$ is hyperelliptic we will use the following observation.
As we explained in the proof of Proposition~\ref{prop:hyperelliptic}, 
the anticanonical morphism of any hyperelliptic Fano threefold~$X$  factors through a double covering~$\phi \colon X \to Y$.
We denote by~$H_Y$ the ample generator of~$\Pic(Y)$ and by~$B \subset Y$ the branch divisor of~$\phi$.

\begin{lemma}
\label{lem:blowup-covering}
Let~$(X,x_0)$ be a factorial hyperelliptic Fano threefold with a single node or cusp and~$\uprho(X) = 1$,
and let~$\phi \colon X \to Y$ be its anticanonical double covering.
Then either
\begin{aenumerate}
\item 
\label{it:dp3}
$\g(X) =2$, $X = X_6 \subset \PP(1^4,3)$, $Y = \PP^3$, and~$B \in |6H_Y|$, or
\item 
\label{it:dq3}
$\g(X) =3$, $X = X_{2,4} \subset \PP(1^5,2)$, 
$Y = Q^3$, a smooth quadric, and~$B \in |4H_Y|$.
\end{aenumerate}
Moreover, $B$ is nodal or cuspidal at~$y_0 \coloneqq \phi(x_0)$ and there is a commutative diagram
\begin{equation}
\label{eq:diag:fac}
\vcenter{
\xymatrix@C=3em@!C{
X \ar[d]_{\phi} &
\hX \ar[d]_{{\widehat\phi}} \ar[l]_{\pi} \ar[r]^{\xi} &
\bar{X} \ar[d]^{\bar{\phi}}
\\
Y &
\hY \ar[l]_{\pi_Y} \ar[r]^{\xi_Y} &
\PP^{\g(X)}
}}
\end{equation}
where~$\hY \coloneqq \Bl_{y_0}(Y)$, $\pi_Y$ is the blowup of~$y_0$, 
$\xi_Y \colon \hY \to \PP^{\g(X)}$ is the morphism induced by the linear projection of~$Y \subset \PP^{\g(X) + 1}$ out of~$y_0$,
$\widehat{\phi}$ is the double covering branched at the strict transform~$\hB \subset \hY$ of~$B$,
$\bar{\phi}$ is a finite morphism, 
and~$\xi$ is the anticanonical contraction of~$\hX$.
\end{lemma}

\begin{proof}
The description of~$X$ is given in Proposition~\ref{prop:hyperelliptic},
and the description of~$Y$ easily follows.
Moreover, the argument of Proposition~\ref{prop:hyperelliptic} shows that in case~\ref{it:dq3} the quadric~\mbox{$Y = Q^3$} is smooth.
Indeed, if it is a double cone over a conic
or a cone over a smooth quadric surface, then its hyperplane class is a sum of two effective classes;
hence the same is true for the anticanonical class of~$X$, contradicting the equality~$\Cl(X) = \Pic(X) = \ZZ \cdot H$.
Finally, the description of~$B$ follows from the Riemann--Hurwitz formula.

To describe the singularity of~$B$ note that the local equation of~$X$ at point~$x_0$ has the form~\mbox{$u^2 = F$},
where~$F$ is the equation of~$B \subset Y$, 
hence the corank of its Hesse matrix at~$x_0$ 
is equal to the corank of the Hesse matrix of~$F$ at~$y_0$;
thus~$x_0$ is a node or cusp of~$X$ if and only if~$y_0$ is a node or cusp of~$B$, respectively.

Furthermore, let~$\pi_Y \colon \hY \coloneqq \Bl_{y_0}(Y) \to Y$ be the blowup, let~$\hB$ be the strict transform of~$B$, 
and let~$\widehat{\phi} \colon \hX \to \hY$ be the double covering of~$\hY$ branched at~$\hB$
(note that~$\hB \sim \pi_Y^*B - 2E_Y$ is divisible by~$2$ in~$\Pic(\hY)$, hence the double covering is defined).
We have
\begin{equation*}
\pi_{Y*}\widehat{\phi}_*\cO_{\hX} \cong \pi_{Y*}(\cO_{\hY} \oplus \cO_{\hY}(E_Y - \tfrac12\pi_Y^*B)) \cong \cO_Y \oplus \cO_Y(-\tfrac12B),
\end{equation*}
hence the Stein factorization of the morphism~$\pi_Y \circ \widehat{\phi} \colon \hX \to Y$
is given by the double covering of~$Y$ branched at the image~$B = \pi_Y(\hB)$ of~$\hB$, i.e., by~$X$.

Finally, the anticanonical linear system~$|H - E|$ of~$\hX$ 
is the pullback of the linear system~$|H_Y - E_Y|$ from~$\hY$, where~$H_Y$ is the hyperplane class of~$Y$,
hence the anticanonical contraction of~$\hX$ is obtained from the Stein factorization of~$\xi_Y \circ \widehat{\phi}$.
\end{proof} 

\begin{proof}[Proof of Theorem~\textup{\ref{thm:intro-fi1}} in the case~$\g(X) \le 6$]
Recall from Proposition~\ref{prop:base-points} that~$-K_X$ is base point free,
hence it defines a morphism~$X \to \PP^{\g(X) + 1}$.
By Lemmas~\ref{lem:barx-factorial} and~\ref{lem:blowup-covering} 
the anticanonical contraction~$\xi$ of~$\hX$ 
is induced by the linear projection~$\PP^{\g(X) + 1} \dashrightarrow \PP^{\g(X)}$.
Below we describe the map~$\xi$, the anticanonical model~$\bar{X}$ of~$\hX$,
the image~$\bar{Z} \subset \bar{X}$ of the exceptional locus of~$\xi$ explicitly;
in particular, we discuss when the anticanonical contraction~$\xi \colon \hX \to \bar{X}$ is small.
Finally, when~$\bar{Z}$ is finite, hence~$\xi$ is small, we consider the Sarkisov link 
\begin{equation}
\label{eq:sl-factorial}
\vcenter{\xymatrix{
& 
\hX \ar[dl]_{\pi} \ar[dr]^{\xi} \ar@{<-->}[rr] &&
\hX_+ \ar[dl]_{\xi_+} \ar[dr]^{\pi_+}
\\
X &&
\bar{X} &&
X_+,
}}
\end{equation}
where the dashed arrow is the flop, $\xi_+ \colon \hX_+ \to \bar{X}$ is a small resolution of singularities, 
and~$\pi_+ \colon \hX_+ \to X_+$ is a $K$-negative extremal contraction,
and identify its right half.

When~$3 \le \g(X) \le 5$ we choose homogeneous coordinates~$(t_0:t_1:\dots:t_{\g(X)+1})$ on~$\PP^{\g(X) + 1}$ 
such that the image of the singular point~$x_0$ of~$X$ is the point~$(1:0:\dots:0)$ and describe~$X$ by appropriate equations.
We also use the standard isomorphisms
\begin{equation*}
 \Bl_{y_0}(\PP^n) \cong \PP_{\PP^{n-1}}(\cO \oplus \cO(-1))
 \qquad\text{and}\qquad 
 \Bl_{y_0}(Q^n) \cong \Bl_{Q^{n-2}}(\PP^n),
\end{equation*}
where~$y_0$ is a smooth point of a quadric~$Q^n$ and~$Q^{n-2} \subset \PP^{n-1} \subset \PP^n$ is a quadric of the same corank as~$Q^n$.
Our arguments are quite standard, so we omit some details, having no doubt that the interested reader will easily reconstruct them.

\subsection*{Case~$\g(X) = 2$.}
By Lemma~\ref{lem:g2} we have~$X = X_6 \subset \PP(1^4,3)$.
Moreover, it follows from Lemma~\ref{lem:blowup-covering} that the anticanonical contraction~$\xi$
of~$\hX$ is an elliptic fibration
\begin{equation*}
\hX \xrightarrow{\ \widehat\phi\ } \Bl_{y_0}(\PP^3) \cong \PP_{\PP^2}(\cO 
\oplus \cO(-1)) \longrightarrow \PP^2,
\end{equation*}
where~$\widehat{\phi}$ is the double covering branched over a divisor of degree~4 over~$\PP^2$.

\subsection*{Case~$\g(X) = 3$.}
By Lemma~\ref{lem:g34} we have~$X = X_{2,4} \subset \PP(1^5,2)$ and
there are two slightly different situations:
if the quadratic equation of~$X$ depends on the variable of weight~$2$, 
then the anticanonical morphism~$X \to \PP^4$ is a closed embedding and
\begin{aenumerate}
\item[(3f-a)] 
\label{f3:3a}
$X \subset \PP^4$ is a quartic hypersurface.
\end{aenumerate}
Otherwise, $X$ is hyperelliptic, and Lemma~\ref{lem:blowup-covering} proves that
\begin{aenumerate}
\item[(3f-b)] 
\label{f3:3b}
$X$ is a double covering of a smooth quadric~$Q(X) \subset \PP^4$.
\end{aenumerate}
In case~(3f-b) the branch divisor~$B$ is an intersection of~$Q(X)$ with a quartic hypersurface.

\subsection*{Subcase~(3f-a).} 

By Lemma~\ref{lem:barx-factorial} the anticanonical morphism of~$\hX$ is the composition
\begin{equation*}
\hX = \Bl_{x_0}(X) \hookrightarrow \Bl_{x_0}(\PP^4) \longrightarrow \PP^3.
\end{equation*}
Since~$X$ has multiplicity~$2$ at~$x_0$, its equation in standard coordinates can be written as
\begin{equation*}
t_0^2F_2(t_1,t_2,t_3,t_4) + t_0F_3(t_1,t_2,t_3,t_4) + F_4(t_1,t_2,t_3,t_4) = 0,
\end{equation*}
where~$F_d$ are homogeneous polynomials of degree~$d$.
A simple computation shows that the Stein factorization of the morphism~$\hX \to \PP^3$ 
is given by the double covering~$\bar{X} \to \PP^3$
branched at the sextic surface~$\bar{B} \subset \PP^3$
defined by the equation
\begin{equation*}
\bar{B} = \{ F_3^2 - 4F_2F_4 = 0 \}.
\end{equation*}
It is, therefore, a smoothable Fano threefold of type~\typemm{1}{1}.
Moreover, the image~$\bar{Z} \subset \PP^3$ of the exceptional locus of~$\xi \colon \hX \to \bar{X}$ 
is given by the equations
\begin{equation*}
\bar{Z} = \{ F_2 = F_3 = F_4 = 0\}.
\end{equation*}
In some cases~$\bar{Z}$ may be $1$-dimensional (the simplest example is when~$F_3 = 0$), then~$\bar{X}$ has non-isolated canonical singularities.
But in general~$\bar{Z}$ is finite and nonempty, 
$\bar{X}$ has terminal singularities, and therefore~$\xi$ is a flopping contraction.
In this case~$\xi$ extends to a symmetric Sarkisov link,
so that~$\hX_+ \cong \hX$, $X_+ \cong X$, and~$\xi_+ = \bar\tau \circ \xi \colon \hX_+ \to \bar{X}$, 
where~$\bar\tau \colon \bar{X} \to \bar{X}$ is the involution of the double covering~$\bar{X} \to \PP^3$. 

\subsection*{Subcase~(3f-b).} 

In this case we apply Lemma~\ref{lem:blowup-covering} and conclude that the anticanonical contraction of~$\hX$ 
is obtained from the Stein factorization of the composition
\begin{equation*}
\hX = \Bl_{x_0}(X) \xrightarrow{\ \widehat{\phi}\ } \Bl_{y_0}(Q(X)) \cong \Bl_{Q(X,x_0)}(\PP^3) \xrightarrow{\ \xi_Y\ } \PP^3,
\end{equation*}
where~$Q(X) \subset \PP^4$ is a smooth quadric, $y_0\in Q(X)$ is the image of $x_0$, 
and~$Q(X,x_0) \subset \PP^3$ is the smooth conic parameterizing lines on~$Q(X)$ through~$y_0$.
The equations of~$Q(X)$ and the branch divisor~$B \subset Q(X)$ of~$\phi \colon X \to Q(X)$ can be written in standard coordinates as
\begin{equation*}
Q(X) = \{ t_0F_1 + F_2 = 0 \},
\qquad 
B = \{t_0^2G_2 + t_0G_3 + G_4 = 0\}.
\end{equation*}
Moreover, the Stein factorization is given by a double covering~$\bar{X} \to \PP^3$
branched at a sextic surface~$\bar{B} \subset \PP^3$, 
and the equations of~$\bar{B}$ and of the image~$\bar{Z} \subset \PP^3$ of the exceptional locus of the anticanonical morphism are given by
\begin{equation*}
\bar{B} = \{ F_2^2G_2 - F_1F_2G_3 + F_1^2G_4 = 0\},
\qquad 
\bar{Z} = \{F_1 = F_2 = 0\}.
\end{equation*}
In particular, $\bar{X}$ is a smoothable Fano threefold of type~\typemm{1}{1}.
On the other hand, in this case~$\dim(\bar{Z}) = 1$, hence~$\xi$ is not small
and the Sarkisov link does not exist.

\subsection*{Case~$\g(X) = 4$.}

By Lemma~\ref{lem:g34} we have~$X = X_{2,3} \subset \PP^5$.
We denote by~$Q(X) \subset \PP^5$ the unique quadric containing~$X$.
As before, there are two slightly different situations:
\begin{aenumerate}
\item[(4f-a)] 
\label{f3:4a}
the singularity~$x_0$ of~$X$ is a smooth point of~$Q(X)$, or
\item[(4f-b)] 
\label{f3:4b}
the singularity~$x_0$ of~$X$ is a singular point of~$Q(X)$.
\end{aenumerate}
In either case the corank of~$Q(X)$ is less or equal than~$1$.

\subsection*{Subcase~(4f-a).} 

By Lemma~\ref{lem:barx-factorial} the anticanonical morphism of~$\hX$ is the composition
\begin{equation*}
\hX = \Bl_{x_0}(X) \hookrightarrow \Bl_{x_0}(Q(X)) \cong \Bl_{Q(X,x_0)}(\PP^4) \longrightarrow \PP^4,
\end{equation*}
where~$Q(X,x_0) \subset \PP^3$ is the quadric surface parameterizing lines on~$Q(X)$ through~$x_0$ 
(it has the same corank as~$Q(X)$; in particular, it is irreducible).
In this case the equations of~$X$ can be written in standard coordinates as
\begin{equation*}
t_0F_1 + F_2 = t_0G_2 + G_3 = 0,
\end{equation*}
and the quadric surface~$Q(X,x_0) \subset \PP^4$ is given by the equations~$\{F_1 = F_2 = 0\}$.
Moreover, the Stein factorization of the morphism~$\hX \to \PP^4$ is given by a quartic~$\bar{X} \subset \PP^4$ 
and the equations of~$\bar{X}$ and of the image~$\bar{Z} \subset \PP^4$ of the exceptional locus of~$\hX \to \bar{X}$
are given by
\begin{equation*}
\bar{X} = \{ F_1G_3 - F_2G_2 = 0\},
\qquad 
\bar{Z} = \{F_1 = F_2 = G_2 = G_3 = 0\}.
\end{equation*}
In particular, $\bar{X}$ is a smoothable Fano threefold of type~\typemm{1}{2}.

In some cases~$\bar{Z}$ may be $1$-dimensional (the simplest example is when~$F_2 = G_2$),
then~$\bar{X}$ has non-isolated canonical singularities.
But in general~$\bar{Z}$ is finite and nonempty, 
$\bar{X}$ has terminal singularities, and therefore~$\xi$ is a flopping contraction.
To describe the right half of the link~\eqref{eq:sl-factorial}, 
note that~$\bar{X} \cap \{ F_1 = 0 \}$ is the union of two irreducible quadric surfaces
\begin{equation*}
\Sigma = \{ F_1 = G_2 = 0\}
\qquad\text{and}\qquad 
\Sigma_+ = \{ F_1 = F_2 = 0 \},
\end{equation*}
where~$\Sigma_+ = Q(X,x_0)$ and~$\Sigma = \xi(E)$.
Since~$\hX = \Blw{-\Sigma}(\bar{X})$ by Lemma~\ref{lem:barx-factorial}, 
it follows from Corollary~\ref{cor:bl-weil-flop} that
\begin{equation*}
\hX_+ \cong \Blw{\Sigma}(\bar{X}) \cong \Blw{-\Sigma_+}(\bar{X}),
\end{equation*}
and using the argument of Theorem~\ref{thm:intro-nf-ci}\ref{thm:nf-ci-barx-x}
it is easy to see that~$f_+ \colon \hX_+ \to X_+$ contracts the strict transform of the quadric~$\Sigma_+$ 
to a Gorenstein singular point on a Fano threefold of genus~$\g(X_+) = 4$.
Since~$\uprho(\hX_+) = 2$, it follows that~$X_+$ is factorial, and so Lemma~\ref{lem:g34} implies
that~$X_+$ is also a threefold of type~(4f-a);
it is $1$-nodal if~$Q(X)$ is smooth and has a single generalized cusp singularity otherwise.
Conversely, the unique quadric~$Q(X)_+$ containing~$X_+$ is smooth if~$X$ is nodal and has corank~$1$ if~$X$ is cuspidal.

\subsection*{Subcase~(4f-b).} 
In this case~$Q(X) = \rC(\bar{Q}(X))$ is a cone with vertex~$x_0$ over a smooth quadric threefold~$\bar{Q}(X) \subset \PP^4$
and by Lemma~\ref{lem:barx-factorial} the anticanonical morphism of~$\hX$ is the composition
\begin{equation*}
\hX = \Bl_{x_0}(X) \hookrightarrow \Bl_{x_0}(Q(X)) \cong \PP_{\bar{Q}(X)}(\cO \oplus \cO(-1)) \longrightarrow \bar{Q}(X).
\end{equation*}
The equations of~$\bar{Q}(X) \subset \PP^4$ and~$X \subset \PP^5$ can be written in standard coordinates as
\begin{equation*}
\bar{Q}(X) = \{F_2 = 0\}
\qquad\text{and}\qquad
X = \{F_2 = t_0^2G_1 + t_0G_2 + G_3 = 0\}.
\end{equation*}
Moreover, the Stein factorization is given by the double covering~$\bar{X} \to \bar{Q}(X)$ 
branched at a quartic surface~$\bar{B} \subset \bar{Q}(X)$,
and the equation of~$\bar{B}$ and of the image~$\bar{Z} \subset \bar{Q}(X)$ of the exceptional locus of~$\hX \to \bar{X}$
can be written in~$\PP^4$ as
\begin{equation*}
\bar{B} = \{F_2 = G_2^2 - 4G_1G_3 = 0\},
\qquad\text{and}\qquad
\bar{Z} = \{F_2 = G_1 = G_2 = G_3 = 0\}.
\end{equation*}
In particular, $\bar{X}$ is a smoothable Fano threefold of type~\typemm{1}{2}.

In some cases~$\bar{Z}$ may be $1$-dimensional (the simplest example is when~$G_2 = 0$),
then~$\bar{X}$ has non-isolated canonical singularities.
But in general~$\bar{Z}$ is finite and nonempty, 
$\bar{X}$ has terminal singularities, and therefore~$\xi$ is a flopping contraction.
In this case~$\xi$ extends to a symmetric Sarkisov link,
so that~$\hX_+ \cong \hX$, $X_+ \cong X$, and~$\xi_+ = \bar\tau \circ \xi \colon \hX_+ \to \bar{X}$, 
where~$\bar\tau \colon \bar{X} \to \bar{X}$ is the involution of the double covering~$\bar{X} \to \bar{Q}(X)$.

\subsection*{Case~$\g(X) = 5$.}

By Lemma~\ref{lem:g56} we have~$X = X_{2,2,2} \subset \PP^6$.
Since~$x_0 \in X$ is a hypersurface singularity, exactly one of the quadrics through~$X$ must be singular at~$x_0$,
in particular it is the cone~$\rC(\bar{Q}(X))$ over a quadric~$\bar{Q}(X) \subset \PP^5$, which has corank~$0$ or~$1$ 
(otherwise, the scheme~$\Sing(X)$ is either not isolated or its length equals~$4$,
hence~$X$ is neither $1$-nodal nor $1$-cuspidal).
Therefore, by Lemma~\ref{lem:barx-factorial} the anticanonical morphism of~$\hX$ is the composition
\begin{equation*}
\xi \colon \hX = \Bl_{x_0}(X) \hookrightarrow \Bl_{x_0}(\rC(\bar{Q}(X))) \cong \PP_{\bar{Q}(X)}(\cO \oplus \cO(-1)) \longrightarrow \bar{Q}(X).
\end{equation*}
The equations of~$\bar{Q}(X) \subset \PP^5$ and~$X \subset \PP^6$ can be written in standard coordinates as
\begin{equation*}
\bar{Q}(X) = \{F_2 = 0\}
\qquad\text{and}\qquad
X = \{F_2 = t_0G_1 + G_2 = t_0H_1 + H_2 = 0\}.
\end{equation*}
The Stein factorization of the morphism~$\hX \to \bar{Q}(X)$ is a cubic divisor~$\bar{X} \subset \bar{Q}(X)$ 
and the equations of~$\bar{X}$ and of the image~$\bar{Z} \subset \bar{Q}(X)$ of the exceptional locus of~$\hX \to \bar{X}$
are 
\begin{equation*}
\bar{X} = \{F_2 = G_1H_2 - G_2H_1 = 0\}
\qquad\text{and}\qquad
\bar{Z} = \{F_2 = G_1 = H_1 = G_2 = H_2 = 0\}.
\end{equation*}
In particular, $\bar{X}$ is a smoothable Fano threefold of type~\typemm{1}{3}.

By Lemma~\ref{lem:barx-factorial} the scheme~$\bar{Z}$ is always finite and nonempty,
$\bar{X}$ has terminal singularities, $\xi$ is a flopping contraction.
and~$\hX = \Blw{-\Sigma}(\bar{X})$, where~$\Sigma = \{ G_1 = H_1 = 0 \} \cap \bar{Q}(X)$ is an irreducible quadric surface.
By Corollary~\ref{cor:bl-weil-flop} we have~$\hX_+ \cong \Blw{\Sigma}(\bar{X})$,
and the argument of Lemma~\ref{lem:str-blw} shows that~$\hX_+$ is the strict transform of~$\bar{X}$ in
\begin{equation*}
\Bl_{\Sigma}(\bar{Q}(X)) = \{ uG_1 + vH_1 = 0 \} \subset \bar{Q}(X) \times \PP^1_{(u:v)}.
\end{equation*}
The composition~$\Blw{\Sigma}(\bar{X}) \hookrightarrow \Bl_{\Sigma}(\bar{Q}(X)) \to \PP^1$ is the extremal contraction~$\pi_+$, 
and its fiber over a point~$(u:v) \in \PP^1$ is the quartic del Pezzo surface
\begin{equation*}
\{uG_1 + vH_1 = uG_2 + vH_2 = 0\} \subset \bar{Q}(X).
\end{equation*}

\subsection*{Case~$\g(X) = 6$.}

By Lemma~\ref{lem:g56} we have~$X = X_{1,1,1,2} \subset \CGr(2,5)$, i.e., $X$ is a Gushel--Mukai threefold.
We show that~$\bar{X} \subset \PP^6$ is a complete intersection of three quadrics 
containing an irreducible quadric surface~$\Sigma$,
hence it is a smoothable Fano threefold of type~\typemm{1}{4}.

Indeed, when~$X$ is general this was proved in~\cite{DIM}.
For arbitrary~$X$ the argument is similar.
First, a general hyperplane section~$S \subset X$ through~$x_0$ 
is an intersection of a smooth quintic del Pezzo threefold~$Y$ with a quadric singular at~$x_0$.
The projection of~$Y \subset \PP^6$ out of~$x_0$ is an intersection of two quadrics in~$\PP^5$, 
hence the projection~$\bar{S} \subset \PP^5$ of~$S$ is an intersection of three quadrics.
Next, since~$\bar{S}$ is a general hyperplane section of~$\bar{X} \subset \PP^6$,
it follows that~$\bar{X}$ is also an intersection of three quadrics,
and by Lemma~\ref{lem:barx-factorial} it is normal and has terminal singularities, 
hence it is a smoothable Fano threefold of type~\typemm{1}{4}.

Furthermore, Lemma~\ref{lem:barx-factorial} proves that~$\hX = \Blw{-\Sigma}(\bar{X})$,
and it follows from Corollary~\ref{cor:bl-weil-flop} that~\mbox{$\hX_+ \cong \Blw{\Sigma}(\bar{X})$}.
Moreover, $\hX_+$ is the strict transform of~$\bar{X}$ in~$\Bl_{\langle \Sigma\rangle}(\PP^6)$,
where~$\langle \Sigma \rangle = \PP^3$ is the linear span of~$\Sigma$,
and the linear projection out of~$\langle \Sigma \rangle \subset \PP^6$ defines a $\PP^4$-bundle
\begin{equation*}
\pi_+ \colon \Bl_{\langle \Sigma \rangle}(\PP^6) \cong \PP_{\PP^2}(\cO^{\oplus 4} \oplus \cO(-1)) 
\longrightarrow \PP^2.
\end{equation*}
Clearly, two quadratic equations of~$\bar{X}$ contain~$\langle \Sigma \rangle$, 
hence the strict transform in~$\Bl_{\langle \Sigma \rangle}(\PP^6)$ of their intersection 
is a $\PP^2$-subbundle~$\PP_{\PP^2}(\cE) \subset \PP_{\PP^2}(\cO^{\oplus 4} \oplus \cO(-1))$, 
where~$\cE$ is the vector bundle defined by an exact sequence
\begin{equation*}
0 \longrightarrow \cE \longrightarrow \cO^{\oplus 4} \oplus \cO(-1) 
\longrightarrow \cO(1)^{\oplus 2} \longrightarrow 0.
\end{equation*}
Finally, it is clear that the last quadratic equation of~$\bar{X} \subset \PP^6$ 
induces a symmetric morphism~\mbox{$\cE \longrightarrow \cE^\vee$},
hence~$\Blw{\Sigma}(\bar{X})$ is the corresponding conic bundle in~$\PP_{\PP^2}(\cE)$,
and since we have~\mbox{$\det(\cE) \cong \cO(-3)$}, its discriminant divisor has degree~$6$.
\end{proof} 

\subsection{Linear sections of Mukai varieties}
\label{ss:proofs-fi1} 

In this subsection we consider factorial Fano threefolds~$X$ 
with~$\uprho(X) = 1$, $\io(X) = 1$, and~$\g(X) \in \{7,8,9,10\}$.

\begin{proposition}[{cf.~\cite[Theorem~2]{Prokhorov2017}}]
\label{prop:sl-g78910}
Let~$(X,x_0)$ be a factorial Fano threefold with a single node or cusp, $\uprho(X) = 1$, $\io(X) = 1$, and~$\g(X) \in \{7,8,9,10\}$.

\begin{thmenumerate}
\item
\label{it:sl-g78910}
There is a Sarkisov link diagram~\eqref{eq:sl-factorial}
where~\mbox{$\hX \coloneqq \Bl_{x_0}(X)$}, $\pi$ is the blowup morphism,
$\pi_+$ is the blowup of a smooth curve~$\Gamma$ 
on a smooth Fano threefold~$X_+$,
and~$\xi$ and~$\xi_+$ are the anticanonical contractions.
More precisely,
\begin{itemize}
\item 
if~$\g(X) = 7$ then~$X_+ = \,\PP^3$, $\g(\Gamma) = 6$, and~$\deg(\Gamma) = 8$;
\item 
if~$\g(X) = 8$ then~$X_+ = Q^3$, $\g(\Gamma) = 4$, and~$\deg(\Gamma) = 8$;
\item 
if~$\g(X) = 9$ then~$X_+$ is a smooth quartic del Pezzo threefold, $\g(\Gamma) = 0$, and~$\deg(\Gamma) = 4$;
\item 
if~$\g(X) = 10$ then~$X_+$ is a smooth quintic del Pezzo threefold, $\g(\Gamma) = 1$, and~$\deg(\Gamma) = 6$.
\end{itemize}
\item
If~$H_+$ is the ample generator of~$\Pic(X_+)$ and~$E_+ \subset \hX_+$ is the exceptional divisor of~$\pi_+$, 
we have linear equivalences
\begin{equation}
\label{eq:hp-ep-h-e}
H_+ \sim H - 2E
\qquad\text{and}\qquad 
E_+ \sim (\io(X_+)-1)H- (2\io(X_+)-1)E
\end{equation}
in the group~$\Pic(\hX_+) \cong \Cl(\bar{X}) \cong \Pic(\hX)$.
Finally, the curve~$\Gamma$ satisfies the condition
\begin{equation}
\label{eq:nondegeneracy}
\dim|(\io(X_+) - 1)H_+ - \Gamma| = 0.
\end{equation} 
\end{thmenumerate}

Conversely, if~$\Gamma \subset X_+$ is as in part~\ref{it:sl-g78910} and~\eqref{eq:nondegeneracy} holds,
there is a Sarkisov link~\eqref{eq:sl-factorial} where~$X$ is a Fano threefold of the corresponding type with a single node or cusp.
\end{proposition} 

Recall that the anticanonical contraction~$\xi \colon \hX \to \bar{X}$ is small by Lemma~\ref{lem:barx-factorial}.
To prove the first part of the proposition we only need to describe the extremal contraction~$\pi_+$.
In the next two lemmas we show that this contraction is a blowup of a smooth curve.

\begin{lemma}
\label{lem:no-bcd}
If~$X$ is as in Proposition~\textup{\ref{prop:sl-g78910}} then
the contraction~$\pi_+$ cannot be of type~\type{(D)}, \type{(B_2)}, \type{(B_1^0)}, or~\typeb.
Moreover, if~$H_+ \sim H - mE$ with~$m \ge 2$ then~$\pi_+$ cannot be of type~\type{(C)}.
\end{lemma}

\begin{proof}
If~$\pi_+$ is of type~\type{(D)} then by Lemma~\ref{lem:intersection} we have
\begin{align*}
H_+^2 \cdot (-K_{\hX_+}) &= 0, &
H_+ \cdot (-K_{\hX_+})^2 &\le 9.
\intertext{Let~$H_+ \sim aH - bE$.
Since~$-K_X \sim H - E$, using Lemmas~\ref{lem:int-flop} and~\ref{lem:intersection}, we obtain}
a^2(2g - 2) - 2b^2 &= 0, &
a(2g - 2) - 2b &\le 9.
\end{align*}
The equality means that~$g - 1$ is a square, and since~$g \in \{7,8,9,10\}$, 
we conclude that~$g = 10$ and~$(a,b)$ is a multiple of~$(1,\pm 3)$; 
but then the inequality fails.

Similarly, if~$\pi_+$ is of type~\type{(B_2)}, \type{(B_1^0)}, or~\typeb{} then by Lemma~\ref{lem:intersection} we have
\begin{align*}
E_+^2 \cdot (-K_{\hX_+}) &= -2, &
E_+ \cdot (-K_{\hX_+})^2 &= \text{$4$, or~$2$, or~$1$}.
\intertext{Let~$E_+ \sim aH - bE$.
Since~$-K_X \sim H - E$, using Lemmas~\ref{lem:int-flop} and~\ref{lem:intersection}, we obtain}
a^2(2g - 2) - 2b^2 &= -2, &
a(2g - 2) - 2b &= \text{$4$, or~$2$, or~$1$}.
\end{align*}
It follows that type~\typeb{} is impossible (because the left side of the second equality is even while the right is odd).
On the other hand, in type~\type{(B_2)} we obtain
\begin{equation*}
a^2(g - 1) = b^2 - 1
\qquad\text{and}\qquad 
a(g - 1) = b + 2.
\end{equation*}
It follows that~$b + 2$ divides~$b^2 - 1 = (b - 2)(b + 2) + 3$, hence~$b + 2$ divides~$3$, 
hence~$a(g - 1)$ divides~$3$, and since~$g \ge 7$, it follows that~$a = 0$, 
and we arrive at contradiction between the first and second equalities.
Similarly, in type~\type{(B_1^0)} we obtain
\begin{equation*}
a^2(g - 1) = b^2 - 1
\qquad\text{and}\qquad 
a(g - 1) = b + 1.
\end{equation*}
It follows that either~$a = b - 1 \ne 0$, and then~$g = \tfrac{b+1}{b-1} + 1 \le 4$, a contradiction,
or~$a = 0$ and~$b = -1$, hence~$E_+ \sim E$, which is absurd.

Finally, if~$\pi_+$ is of type~\type{(C)} then by Lemma~\ref{lem:intersection} we have
\begin{equation*}
H_+^2 \cdot (-K_{\hX_+}) = 2, 
\end{equation*}
So, assuming~$H_+ \sim H - mE$, we obtain as before~$(2g - 2) - 2m^2 = 2$,
which means~$g = m^2 + 2$ in contradiction with~$g \in \{7,8,9,10\}$.
\end{proof} 

\begin{lemma}
\label{lem:h2e}
If~$X$ is as in Proposition~\textup{\ref{prop:sl-g78910}} then
\begin{equation*}
\dim|H - 2E| = g - 4.
\end{equation*}
Moreover, the linear system~$|H-2E|$ on~$\hX$ has no fixed components 
and its strict transform to~$\hX_+$ induces the extremal contraction~$\pi_+ \colon \hX^+ \to X^+$,
which is of type~\type{(B_1^1)}.
\end{lemma}

\begin{proof}
First, note that the map~$\pi_*$ identifies the linear system~$|H-2E|$ 
with the linear system of hyperplane sections of~$X \subset \PP^{g + 1}$ 
containing the embedded tangent space~$\rT_{x_0}(X) \subset \PP^{g + 1}$ to~$X$ at~$x_0$.
Since~$\dim (\rT_{x_0}(X)) = 4$, we conclude that~$\dim|H - 2E| = g - 4$.

Furthermore, since~$X$ is factorial, we have~$\Cl(X) = \ZZ \cdot H$, 
hence~$|H-2E|$ has no fixed components outside~$E$.
Therefore, we can write 
\begin{equation*}
|H - 2E| = |H - mE| + (m - 2)E
\end{equation*}
for some $m\ge 2$, where $|H - mE|$ is a linear system without fixed components. 

Let~$\Upsilon_+$ be the class of a minimal rational curve contracted by~$\pi_+$.
Since~$-K_{\hX_+} \sim H - E$ (as it is usual, we identify~$\Cl(\hX)$ with~$\Cl(\hX_+)$), we have
\begin{equation}
\label{eq:io-m}
\io(\pi_+) = 
-K_{\hX_+}\cdot \Upsilon_+ =
(H - mE) \cdot \Upsilon_+ + (m - 1)E \cdot \Upsilon_+.
\end{equation}
Since~$\pi_+$ is $K$-negative, the curve~$\Upsilon_+$ is movable
and since~$|H - mE|$ has no fixed components, its base locus has dimension at most~$1$.
Therefore, $(H - mE) \cdot \Upsilon_+ \ge 0$.
On the other hand, we have~$E \cdot \Upsilon_+ \ge 1$,
because~$E$ is positive on the curves contracted by~$\xi$, 
hence it is negative on the curves contracted by~$\xi_+$,
and therefore it must be positive on~$\Upsilon_+$. 

If~$\io(\pi_+) \ge 2$ then Lemma~\ref{lem:no-bcd} shows that~$\pi_+$ must be of type~\type{(C_2)}.
Moreover, if the first summand in the right side of~\eqref{eq:io-m} is zero, 
i.e., $(H - mE) \cdot \Upsilon_+ = 0 $, then~$H_+ \sim H - mE$, and
we obtain a contradiction with the second part of Lemma~\ref{lem:no-bcd}.
Therefore, both summands in the right side must be equal to~$1$, i.e.,
\begin{equation*}
m = 2
\qquad\text{and}\qquad 
(H - 2E) \cdot \Upsilon_+ = E \cdot \Upsilon_+ = 1..
\end{equation*}
But then~$(H - 3E) \cdot \Upsilon_+ = 0$, hence~$H_+ \sim H - 3E$, again in contradiction with Lemma~\ref{lem:no-bcd}.

Therefore, $\io(\pi_+) = 1$, and since~$m \ge 2$, we deduce from~\eqref{eq:io-m} that
\begin{equation*}
m = 2,
\qquad 
(H - 2E) \cdot \Upsilon_+ = 0,
\qquad\text{and}\qquad 
E \cdot \Upsilon_+ = 1.
\end{equation*}
This means that~$H_+ \sim H - 2E$ has no fixed components and defines the contraction~$\pi_+$.
Finally, as Lemma~\ref{lem:no-bcd} shows, $\pi_+$ is of type~\type{(B_1^1)}.
\end{proof}

The following lemma shows that the anticanonical contraction of~$\Bl_\Gamma(X_+)$ is small,
hence the blowup of~$\Gamma$ on~$X_+$ extends to a Sarkisov link.
This lemma is used in Appendix~\ref{sec:links-78910}, which gives a proof of the second part of Proposition~\ref{prop:sl-g78910}.

\begin{lemma}
\label{lem:blg-xplus}
Let~$\Gamma \subset X_+$ be one of the smooth curves on one of the smooth threefolds as in Proposition~\xref{prop:sl-g78910}\ref{it:sl-g78910}
and assume that~\eqref{eq:nondegeneracy} holds.
Let
\begin{equation*}
S_\Gamma \subset X_+ 
\end{equation*}
be the unique divisor in~$|(\io(X_+) - 1)H_+ - \Gamma|$.
Then~$S_\Gamma$ is a del Pezzo surface with at worst Du Val singularities
and~$X_+$ contains only finitely many lines~$\ell$ such that~$\operatorname{length}(\Gamma \cap \ell) \ge \io(X_+)$.
\end{lemma}

\begin{proof}
By definition~$S_\Gamma$ is one of the following surfaces:
\begin{aenumerate}
\item 
\label{it:s3-p3}
a cubic surface in~$X_+ = \PP^3$, or
\item 
\label{it:s4-q3}
a quartic surface in~$X_+ = Q^3 \subset \PP^4$, or
\item 
\label{it:s4-y4}
a quartic surface in a quartic del Pezzo fourfold~$X_+ \subset \PP^4$, or
\item 
\label{it:s5-y5}
a quintic surface in a quintic del Pezzo fourfold~$X_+ \subset \PP^5$.
\end{aenumerate}
Thus, $S_\Gamma$ is a surface of degree~$d$ in~$\PP^d$.
By~\eqref{eq:nondegeneracy} it is not contained in a~$\PP^{d-1}$,
hence~\cite[Theorem~8]{Nagata:RarSurf1} shows that there are three cases:
\begin{itemize}
\item 
\label{it:sg-scroll}
$S_\Gamma$ is a linear projection of a smooth rational scroll of degree~$d$ in~$\PP^{d+1}$, or
\item 
\label{it:sg-dp}
$S_\Gamma$ is a del Pezzo surface with at worst Du Val singularities, or
\item 
\label{it:sg-cone}
$S_\Gamma$ is a cone over a rational or elliptic curve of degree~$d$ in~$\PP^{d-1}$.
\end{itemize}

Assume~$S_\Gamma$ is a projection of a scroll~$\tilde{S}_\Gamma \subset \PP^{d+1}$. 
The strict transform~$\tilde\Gamma \subset \tilde{S}_\Gamma$ of~$\Gamma$ 
is a curve of the same degree and genus as~$\Gamma$.
Using the Castelnuovo genus bound it is easy to check that~$\tilde{\Gamma} \subset \PP^d$,
hence it is contained in the hyperplane section of the scroll, hence
\begin{equation*}
\deg(\Gamma) = \deg(\tilde\Gamma) \le \deg(\tilde{S}_\Gamma) = d.
\end{equation*}
In cases~\ref{it:s3-p3}, \ref{it:s4-q3}, and~\ref{it:s5-y5} this contradicts the assumption of the lemma.
In case~\ref{it:s4-y4} we note that~$X_+$ is a smooth complete intersection, 
hence its hyperplane section~$S_\Gamma$ is normal, 
while any projection of a scroll~$\tilde{S}_\Gamma$ is not,
again a contradiction.

Assume~$S_\Gamma$ is a cone with vertex~$s_0$.
Then~$\PP^d = \rT_{s_0}(S_\Gamma) \subset \rT_{s_0}(X_+) = \PP^3$, hence~$d \le 3$,
hence~$S_\Gamma$ is a cubic cone in~$\PP^3$ and we are in case~\ref{it:s3-p3}.
But the degree of any smooth curve on a cubic cone is divisible by~$3$ (if the curve does not pass through the vertex of the cone)
or congruent to~$1$ modulo three (otherwise), in contradiction with~$\deg(\Gamma) = 8$.

Thus, $S_\Gamma$ is a del Pezzo surface, hence it contains only finitely many lines.
It remains to note that if~$\ell$ is a line on~$X_+$ such that~$\operatorname{length}(\Gamma \cap \ell) \ge \io(X_+)$
then~$S_\Gamma \cdot \ell < 0$, hence~$\ell \subset S_\Gamma$.
\end{proof}

Now we can prove Proposition~\ref{prop:sl-g78910}.
Its first part follows from Lemma~\ref{lem:h2e}, 
and for the second part we use the results of Appendix~\ref{sec:links-78910} (which rely on the first part of the proposition).
After that we prove Theorem~\ref{thm:intro-factorial-ci}\ref{it:fact-i1-gbig} and finish the proof of Theorem~\ref{thm:intro-fi1}.

\begin{proof}[Proof of Proposition~\textup{\ref{prop:sl-g78910}}]
Let~$(X,x_0)$ be a factorial Fano threefold with a single node or cusp, $\uprho(X) = 1$, $\io(X) = 1$, and~$\g(X) \in \{7,8,9,10\}$.
The Sarkisov link diagram exists by Lemma~\ref{lem:barx-factorial}.
The first equality in~\eqref{eq:hp-ep-h-e} is proved in Lemma~\ref{lem:h2e}.
Moreover, the flop~$\hX \dashrightarrow \hX_+$ 
induces an isomorphism~$\Cl(\hX) \cong \Cl(\hX_+)$ that takes~$K_{\hX}$ to~$K_{\hX_+}$, 
and since the contraction~$\pi_+$ is of type~\type{(B_1^1)} by Lemma~\ref{lem:h2e},
we see that~$X_+$ is a smooth Fano threefold and obtain
\begin{equation*}
H - E \sim \io(X_+)H_+ - E_+ \sim \io(X_+)(H - 2E) - E_+.
\end{equation*}
This gives the second equality in~\eqref{eq:hp-ep-h-e} and implies that~$\io(X_+) \ge 2$.

Furthermore, using Lemmas~\ref{lem:hh-g-xi} and~\ref{lem:flop-h}
and Proposition~\ref{prop:hh-g-xsm}, we obtain the equalities
\begin{equation}
\label{eq:hhg}
\hh(X) - 1 = \hh(\hX) = \hh(\hX_+) = \hh(X_+) + \g(\Gamma),
\end{equation}
where~$\Gamma$ is the center of the blowup~$\pi_+$,
hence~$\hh(X_+) \le \hh(X) - 1$.
On the other hand,
\begin{equation*}
\dim |H_+| = \dim |H - 2E| = g - 4 \in \{3,4,5,6\}.
\end{equation*}
Comparing this with the Hodge numbers and the dimensions of the primitive ample linear system 
on del Pezzo threefolds, smooth quadric, and~$\PP^3$,
we identify~$X_+$ in each of the cases.
Then, using~\eqref{eq:hhg} again we determine~$\g(\Gamma)$, 
and then using Lemma~\ref{lem:hh-g-xi} we compute~$\deg(\Gamma)$.
Finally, using~\eqref{eq:hp-ep-h-e} we find
\begin{equation*}
(\io(X_+) - 1)H_+ - E_+ \sim 
(\io(X_+) - 1)(H - 2E) - (\io(X_+) - 1)H + (2\io(X_+) - 1)E \sim E,
\end{equation*}
hence the corresponding linear system is $0$-dimensional; this proves~\eqref{eq:nondegeneracy}.

The converse implication of the proposition 
is proved in Corollaries~\ref{cor:critical-gr25-p3-sl}, \ref{cor:critical-ogr-q3-sl}, \ref{cor:critical-lgr-y4-sl}, and~\ref{cor:critical-lgr-y5-sl},
respectively.
\end{proof}

\begin{proof}[Proof of Theorem~\textup{\ref{thm:intro-factorial-ci}}\ref{it:fact-i1-gbig}]
This is a combination of Corollaries~\ref{cor:critical-gr25-p3-sl}, \ref{cor:critical-ogr-q3-sl},
\ref{cor:critical-lgr-y4-sl}, and~\ref{cor:critical-lgr-y5-sl}.
\end{proof}

\begin{proof}[Proof of Theorem~\textup{\ref{thm:intro-fi1}} for~$\g(X) \in \{7,8,9,10\}$]
The first part follows from a combination of Lemma~\ref{lem:barx-factorial} 
with Corollaries~\ref{cor:critical-gr25-p3-sl}, \ref{cor:critical-ogr-q3-sl},
\ref{cor:critical-lgr-y4-sl}, and~\ref{cor:critical-lgr-y5-sl}.

For the second part we also use the bijection of Corollaries~\ref{cor:critical-gr25-p3}, \ref{cor:critical-ogr-q3},
\ref{cor:critical-gr26-p8}, and~\ref{cor:critical-lgr-y5}.
\end{proof}

\begin{remark}
\label{rem:nodal-production}
Nodal or cuspidal Fano threefolds can be also constructed by appropriate modifications of standard Sarkisov links.
More precisely, consider a Sarkisov link 
that starts with a blowup of a smooth curve~$\Gamma$ or point on a smooth Fano threefold~$X$
and finishes at a blowup of a smooth curve~$\Gamma_+$ on a smooth Fano threefold~$X_+$ ---
some possibilities are listed below

\begin{equation*}
\begin{array}{c@{\qquad}c@{\qquad}c@{\qquad}l@{\qquad}||@{\qquad}c@{\qquad}c@{\qquad}c@{\qquad}l}
\g(X) & \Gamma & X_+ & \Gamma_+ &\g(X) & \Gamma & X_+ & \Gamma_+
\\\hline
12 & (0,0) & \PP^3 & (0,6)&\hphantom{0}9 & (0,3) & \rY_5 & (3,9)
\\
12 & (0,1) & \rY_5 & (0,5)&\hphantom{0}8 & (0,0) & \rY_3 & (0,4)
\\ 
12 & (0,2) & Q^3 & (0,6)&\hphantom{0}7 & (0,0) & \rY_5 & (7,12)
\\
10 & (0,1) & Q^3 & (2,7)&\hphantom{0}7 & (0,2) & Q^3 & (7,10)
\\ 
\hphantom{0}9 & (0,1) & \PP^3 & (3,7)&\hphantom{0}7 & (0,3) & \PP^3 & (7,9)
\end{array}
\end{equation*}
where as usual~$(g,d)$ is the genus and degree of the corresponding curve and~$(0,0)$ means that~$\Gamma$ is a smooth point.
Then we can either take the same smooth threefold~$X_+$ 
and replace~$\Gamma_+$ by a $1$-nodal or 1-cuspidal curve of the same genus and degree 
or replace~$X_+$ by a $1$-nodal or 1-cuspidal threefold of the same type and take the same smooth curve on it,
and run the same link in the opposite direction.
Under appropriate assumptions, we will obtain a $1$-nodal or 1-cuspidal variety~$X$,
which will be factorial if~$\Gamma_+$ is irreducible and~$X_+$ is factorial and nonfactorial otherwise.
For an example of such construction, see~\cite[Theorem~4.1, Remark~4.5]{P:G-Fano:Izv}.
\end{remark}

\appendix


\section{Blowups of Weil divisor classes}
\label{App:blowup}

In this section we remind the definition and basic properties of the blowup of a Weil divisor class;
for more details see~\cite{Kaw88}.
We will use the following simple observation.

\begin{lemma}
\label{lem:pf-reflexive}
Let~$f \colon Y \to Z$ be a small proper birational morphism of normal varieties.
Then for any reflexive sheaf~$\cL$ on~$Y$ the sheaf~$f_*\cL$ is also reflexive.
\end{lemma}

\begin{proof}
Let~$j \colon U \hookrightarrow Z$ be an open subset over which~$f$ is an isomorphism, 
and let~$i \colon U \to Y$ be the lift of~$j$, so that~$j = f \circ i$.
Since~$f$ is small, the complements~$Y \setminus U$ and~$Z \setminus U$ have codimension at least~$2$.
Since~$\codim(Y \setminus U) \ge 2$, we have~$\cL \cong i_*(\cL\vert_U)$, hence
\begin{equation*}
f_*\cL \cong f_*(i_*(\cL\vert_U)) \cong j_*(\cL\vert_U),
\end{equation*}
and since~$\codim(Z \setminus U) \ge 2$ and~$Z$ is normal, this sheaf is reflexive, see~\cite[Tag~0AY6]{stacks-project}.
\end{proof}

\begin{corollary}
\label{cor:small-blowup}
If~$f \colon Y \to Z$ is a small projective morphism of normal varieties then
\begin{equation*}
Y \cong \Proj_Z \left( \moplus_{k = 0}^\infty (\cR^{\otimes k})^{\vee\vee} \right)
\end{equation*}
where~$\cR$ is a reflexive sheaf of rank~$1$ on~$Z$.
\end{corollary}

\begin{proof}
Let~$\cL$ be a relatively ample line bundle on~$Y$.
Then
\begin{equation*}
Y \cong \Proj_Z \left( \moplus_{k = 0}^\infty f_*(\cL^{\otimes k}) \right)
\end{equation*}
Let~$\cR \coloneqq f_*\cL$, it is a reflexive sheaf of rank~$1$ by Lemma~\ref{lem:pf-reflexive}.
It remains to note that
\begin{equation*}
f_*(\cL^{\otimes k}) \cong (\cR^{\otimes k})^{\vee\vee},
\end{equation*}
because both sheaves are reflexive and isomorphic on the open subset~$U \subset Z$ over which~$f$ is an isomorphism,
taking into account that~$\codim(Z \setminus U) \ge 2$.
\end{proof}

Usually, we replace the language of reflexive sheaves by the more geometric language of Weil divisors
(see, e.g.,~\cite[Appendix to~\S1]{Reid:can3fo}).
Consequently, given a Weil divisor class~$D \in \Cl(Z)$ we define
\begin{equation}
\label{eq:bl-d-z}
\Blw{D}(Z) \coloneqq \Proj_Z \left( \moplus_{k = 0}^\infty \cO_Z(-kD) \right),
\end{equation} 
where~$\cO_Z(-kD)\cong (\cO_Z(-D)^{\otimes k})^{\vee\vee}$ is the reflexive sheaf 
corresponding to the Weil divisor class~\mbox{$-kD \in \Cl(Z)$}.
We call the scheme~$\Blw{D}(Z)$ {\sf the blowup of the Weil divisor class~$D$}.
We use square brackets in the notation to distinguish the blowup of a Weil divisor class from the blowup of a subvariety.
Note that~$D$ does not need to be effective, and~$\Blw{D}(Z)$ 
depends only on the class of~$D$ in~$\Cl(Y) / \Pic(Y)$ up to multiplication by a positive integer;
in particular, if~$D$ is a $\QQ$-Cartier divisor then~$\Blw{D}(Z) \cong Z$.

\begin{remark}
\label{rem:blw}
Note that for a normal variety~$Z$ and a Weil divisor class~$D \in \Cl(Z)$
the sheaf of algebras~$\moplus_{k = 0}^\infty \cO_Z(-kD)$ is not necessarily finitely generated
(for instance, this happens if~$Z$ is a cone over a smooth plane cubic curve~$C \subset \PP^2$ and~$D \in \Cl(Z) \cong \Pic(C)$ 
corresponds to a non-torsion divisor class on~$C$ of degree~$0$).
However, if~$D$ comes from a small morphism~$f \colon Y \to Z$ as in Corollary~\ref{cor:small-blowup}, 
the algebra is finitely generated.
Moreover, in this case if~$D_Y \in \Cl(Y)$ is the strict transform of the class~$D$
then~$D_Y$ is $\QQ$-Cartier and~$-D_Y$ is $f$-ample,
because on the open subset~$U \subset Z$ over which~$f$ is an isomorphism we have~$\cO_U(-D_Y) \cong \cO_{Y/Z}(1)$.
\end{remark}

\begin{corollary}
\label{cor:bl-weil-flop}
If~$f \colon Y \to Z$ is a flopping contraction of threefolds and~\mbox{$Y \cong \Blw{D}(Z)$}
then~$Y_+ \coloneqq \Blw{-D}(Z)$ is the flop of~$Y$.
\end{corollary}

\begin{proof}
Let~$U \subset Z$ be the open subset over which~$Y \to Z$ and~$Y_+ \to Z$ are isomorphisms.
By definition of a flop there are line bundles~$\cL$ on~$Y$ an~$\cL_+$ on~$Y_+$, both ample over~$Z$, 
such that~$\cL\vert_U \cong \cL_+^{-1}\vert_U$.
On the other hand, by Corollary~\ref{cor:small-blowup} we have~$Y \cong \Blw{D}(Z)$ and~$Y_+ \cong \Blw{D_+}(Z)$,
where~$D$ and~$D_+$ correspond to the extensions of~$\cL^{-1}\vert_U$ and~$\cL_+^{-1}\vert_U$ from~$U$ to~$Z$;
these line bundles are mutually inverse, hence~$D_+ = -D$.
\end{proof} 

Sometimes the blowup of a Weil divisor coincides with the strict transform in the blowup.

\begin{lemma}
\label{lem:str-blw}
Let~$Z \subset M$ be a closed embedding of normal varieties and let~$D \subset Z$ be an effective Weil divisor.
Let~$Z' \subset \Bl_D(M)$ be the strict transform of~$Z$ and let~$Z''$ be the normalization of~$Z'$.
If the morphism~$Z' \to Z$ is small then~\mbox{$Z'' \cong \Blw{D}(Z)$}.
\end{lemma}

\begin{proof}
Let~$\pi \colon \Bl_D(M) \to M$ be the blowup and let~$E \subset \Bl_D(M)$ be the exceptional divisor.
Then~$-E$ is ample over~$M$, hence~$-E\vert_{Z'}$ and its pullback to~$Z''$ is ample over~$Z$,
and if~$U \subset Z$ is the open subset over which~$\pi'' \colon Z'' \to Z$ is an isomorphism,
then~$E \cap U$ (in~$Z''$) coincides with~$D \cap U$ (in~$Z$),
which means that over~$U$ the sheaf~$\pi''_*\cO_{Z''}(-E\vert_{Z''})$ coincides with~$\cO_Z(-D)$.
Therefore, Corollary~\ref{cor:small-blowup} implies that~$Z'' \cong \Blw{D}(Z)$.
\end{proof} 

\begin{lemma}
\label{lem:ddprime}
Let~$Y$ be a threefold with terminal Gorenstein singularities and let~$S \subset Y$ 
be a surface smooth at every point of~$S \cap \Sing(Y)$.
Then the sheaf of algebras~$\moplus_{k=0}^\infty \cO_Y(-kS)$ is finitely generated,
and if~$S' \subset Y' \coloneqq \Blw{-S}(Y) \xrightarrow{\ \eta\ } Y$ is the strict transform of~$S$
then~$S'$ is~$\eta$-ample and the morphism~$\eta\vert_{S'} \colon S' \xrightiso{} S$ is an isomorphism.
\end{lemma}

\begin{proof}
By assumption~$Y$ has terminal Gorenstein, hence isolated hypersurface singularities.
All our claims are local around singular points of~$Y$, so we may assume
that~$\Sing(Y) = \{y_0\}$, $S$ is smooth, and~$Y \subset M$, where~$M$ is a smooth fourfold.
Since~$S$ is smooth, it is a local complete intersection in~$M$, 
so shrinking~$M$ if necessary, we may assume that
\begin{equation*}
S = \{ \varphi_1 = \varphi_2 = 0 \} \subset M
\qquad\text{and}\qquad 
Y = \{ \varphi_1 \psi_2 - \varphi_2 \psi_1 = 0 \} \subset M,
\end{equation*}
where~$\varphi_i$ and~$\psi_i$ are functions on~$M$.
The condition~$\Sing(Y) = \{y_0\}$ implies an equality of schemes~$\{\varphi_1 = \varphi_2 = \psi_1 = \psi_2 = 0\} = \{y_0\}$,
hence the subvariety~$S \cap \{\psi_1 = 0\}$ is a Cartier divisor in~$S$.
Now consider the surface
\begin{equation*}
R \coloneqq \{\varphi_1 = \psi_1 = 0 \} \subset M.
\end{equation*}
It is clear that~$Y \cap \{\varphi_1 = 0\} = S \cup R$, hence~$S + R \sim 0$ in~$\Cl(Y)$.
Note also that~$R$ is smooth at~$y_0$, so shrinking~$M$ if necessary we may assume that~$R$ is smooth everywhere.

Consider the blowup~$\pi \colon \Bl_R(M) \to M$. 
Then~$\Bl_R(M)$ is smooth, hence the strict transform~$Y'$ of~$Y$ in~$\Bl_R(M)$ is Cohen--Macaulay.
On the other hand, $R$ is a Cartier divisor on~$Y \setminus \{y_0\}$ (because~$Y \setminus \{y_0\}$ is smooth), 
hence the morphism~$\pi\vert_{Y'} \colon Y' \to Y$ is an isomorphism over~$Y \setminus \{y_0\}$ and~$Y'$ is smooth over~$Y \setminus \{y_0\}$.
Finally, $\dim(\pi^{-1}(y_0)) = 1$, hence~$Y'$ is normal and the morphism~$\pi\vert_{Y'}$ is small,
so applying Lemma~\ref{lem:str-blw} we conclude that
\begin{equation*}
Y' \cong \Blw{R}(Y) \cong \Blw{-S}(Y),
\end{equation*}
in particular, the sheaf of algebras~$\moplus_{k=0}^\infty \cO_Y(-kS)$ is finitely generated
and the divisor~$S'$ is $\eta$-ample by Remark~\ref{rem:blw}.
Now it remains to note that the strict transform~$S'$ of~$S$ in~$\Bl_R(M)$ is isomorphic to~$S$
because~$S \cap R = S \cap \{\psi_1 = 0\}$ is a Cartier divisor in~$S$, 
and it is obvious that~$S'$ coincides with the strict transform of~$S$ in~$Y' \subset \Bl_R(M)$.
\end{proof}


\section{Factoriality of some complete intersections}

In this section we prove that $1$-nodal Fano complete intersection threefolds in Mukai varieties are factorial.
We also compute the Picard number of small resolutions of some complete intersections in products of Fano varieties.
Note that for a smooth Fano variety~$X$ we have~$\uprho(X) = \h^{1,1}(X)$, where the right-hand side is a Hodge number.

Our first result is similar to~\cite[Theorem~4.1]{Rams}.

\begin{proposition}
\label{prop:h11-tx}
Let~$Y$ be a projective fourfold with isolated singularities.
Let~$\whom_Y$ be the reflexive hull of~$\Omega^1_Y$.
Let~$D \in \Pic(Y)$ be a Cartier divisor such that
\begin{arenumerate}
\item 
\label{it:h123-ox}
$H^1(Y, \cO_Y(-D)) = H^2(Y, \cO_Y(-D)) = H^2(Y, \cO_Y(-2D)) = H^3(Y, \cO_Y(-2D)) = 0$,

\item 
\label{it:h12-om}
$H^1(Y,\whom_Y(-D)) = H^2(Y,\whom_Y(-D)) = 0$.

\end{arenumerate}
Let~$X' \subset Y$ be a smooth divisor in~$|D|$
and let~$X \subset Y$ be a $1$-nodal divisor in~$|D|$, both contained in the smooth locus of~$Y$.
If the node~$x_0$ of~$X$ is not a base point of the linear system~$|D + K_X|$ on~$X$
then~$\h^{1,1}(\Bl_{x_0}(X)) = \h^{1,1}(X') + 1$.
\end{proposition}

\begin{proof}
Let~$\tY \coloneqq \Bl_{x_0}(Y)$ and~$\tX \coloneqq \Bl_{x_0}(X)$ be the blowups.
Let~$E_Y \subset \tY$ and~$E_X \subset \tX$ be the exceptional divisors, so that~$E_Y \cong \PP^3$ and~$E_X \cong \PP^1 \times \PP^1$.
Below we compare the exact sequences computing~$\h^{1,1}(X')$ and~$\h^{1,1}(\tX)$ in terms of
the embeddings~$X' \subset Y$ and~$\tX \subset \tY$.

We start with some preparations.
First of all, since~$x_0$ is a smooth point of~$Y$, we have an exact sequence
\begin{equation}
\label{eq:omty}
0 \longrightarrow \pi^*\whom_Y \longrightarrow \whom_{\tY} \longrightarrow i_*\Omega^1_{E_Y} \longrightarrow 0,
\end{equation}
where~$\pi \colon \tY \to Y$ is the blowup, $i \colon E_Y \to \tY$ is the embedding,
and~$\whom_{\tY}$ is the reflexive hull of~$\Omega^1_{\tY}$.
Since~$E_Y \cong \PP^3$, we have~$H^1(E_Y,\Omega^1_{E_Y}) = \Bbbk$.
We claim that the morphism
\begin{equation*}
H^1(E_Y,\Omega^1_{E_Y}) \to H^2(\tY, \pi^*\whom_Y)
\end{equation*}
from the cohomology sequence of~\eqref{eq:omty} vanishes.
Indeed, $H^1(E_Y,\Omega^1_{E_Y}) = H^0(E_Y, \cO_{E_Y})$ via the Euler sequence on~$E_Y$,
so the vanishing follows from the equality~$\Ext^2(\cO_{E_Y}, \pi^*\whom_Y) = 0$,
which itself follows from Grothendieck duality for the morphism~$i$.
Therefore, we have
\begin{equation*}
H^1(\tY, \whom_{\tY}) = H^1(Y, \whom_Y) \oplus \Bbbk.
\end{equation*}
Further, twisting sequence~\eqref{eq:omty} by~$\cO_{\tY}(2E_Y - \pi^*D)$, 
pushing it forward to~$Y$, and taking into account that the third term vanishes,
because~$H^\bullet(E_Y, \Omega^1_{E_Y}(2E_Y)) = H^\bullet(\PP^3,\Omega^1_{\PP^3}(-2)) = 0$ by the Euler sequence on~$E_Y$,
we see that~$H^i(\tY, \whom_{\tY}(2E_Y - \pi^*D))$ is equal to the cohomology of~$\pi^*(\whom_Y(-D)) \otimes \cO_{\tY}(2E_Y)$.
Since~\mbox{$\mathbf{R}\pi_*\cO_{\tY}(2E_Y) \cong \cO_Y$}, the projection formula gives
\begin{equation*}
H^i(\tY, \whom_{\tY}(2E_Y - \pi^*D)) = H^i(Y, \whom_Y(-D)) = 0
\quad\text{for~$i = 1,2$},
\end{equation*}
where the second equality is the hypothesis~\ref{it:h12-om}.

Now we start the computations for~$\tX$ and~$X'$.
Since~$\tX$ belongs to the linear system~$|\pi^*D - 2E_Y|$ and is contained in the smooth locus of~$\tY$, 
while~$X'$ belongs to~$|D|$ and is contained in the smooth locus of~$Y$,
we have exact sequences
\begin{align*}
0 \longrightarrow \whom_{\tY}(2E_Y - \pi^*D) \longrightarrow \whom_{\tY} \longrightarrow {}&\Omega^1_{\tY}\vert_{\tX\hphantom{{}'}} \longrightarrow 0,
\\
0 \longrightarrow \whom_{Y}(-D) \longrightarrow \whom_{Y} \longrightarrow {}&\Omega^1_{Y}\vert_{X'} \longrightarrow 0,
\end{align*}
and it follows that
\begin{equation*}
H^1(\tX, \Omega^1_{\tY}\vert_{\tX}) = 
H^1(\tY, \whom_{\tY}) = 
H^1(Y, \whom_Y) \oplus \Bbbk =
H^1(X', \Omega^1_Y\vert_{X'}) \oplus \Bbbk.
\end{equation*}
Further, the conormal bundle of~$\tX \subset \tY$ is~$\cO_{\tX}(2E_X - \pi^*D)$, and we have an exact sequence
\begin{equation*}
0 \longrightarrow \cO_{\tX}(E_X - \pi^*D) \longrightarrow \cO_{\tX}(2E_X - \pi^*D) \longrightarrow \cO_{E_X}(2E_X) \longrightarrow 0.
\end{equation*}
We have~$H^2(E_X,\cO_{E_X}(2E_X)) = H^2(\PP^1 \times \PP^1, \cO_{\PP^1 \times \PP^1}(-2,-2)) = \Bbbk$ 
and the other cohomology spaces vanish, hence the exact sequence in cohomology takes the form
\begin{multline*}
0 \longrightarrow H^2(\tX, \cO_{\tX}(E_X - \pi^*D)) \longrightarrow H^2(\tX, \cO_{\tX}(2E_X - \pi^*D)) 
\\ \longrightarrow \Bbbk \longrightarrow 
H^3(\tX, \cO_{\tX}(E_X - \pi^*D)) \longrightarrow H^3(\tX, \cO_{\tX}(2E_X - \pi^*D)) \longrightarrow 0.
\end{multline*}
Combining the Grothendieck duality for the morphisms~$i$ and~$\pi$ with Serre duality on~$X$,
we identify the morphism~$\Bbbk \to H^3(\tX, \cO_{\tX}(E_X - \pi^*D))$ with the dual of
the evaluation morphism~$H^0(X, \cO_{X}(D + K_X)) \longrightarrow \Bbbk$ at point~$x_0$,
so, if~$x_0 \not\in \Bs(|D + K_X|)$, the evaluation morphism is surjective, 
hence the morphism in the above cohomology exact sequence is injective, 
and we conclude that for~$i = 1,2$ we have
\begin{equation*}
H^i(\tX, \cO_{\tX}(2E_X - \pi^*D)) \cong 
H^i(\tX, \cO_{\tX}(E_X - \pi^*D)) \cong
H^i(X, \cO_{X}(-D)).
\end{equation*}
where the second isomorphism follows from the projection formula and~$\mathbf{R}\pi_*\cO_{\tX}(E_X) \cong \cO_X$.
Moreover, using the Koszul exact sequences for~$\cO_X(-D)$ and~$\cO_{X'}(-D)$
and hypothesis~\ref{it:h123-ox}, we obtain
\begin{equation*}
H^i(X, \cO_{X}(-D)) = H^i(X', \cO_{X'}(-D)) = 0
\end{equation*}
for~$i = 1,2$.
Finally, using the above vanishings and the cohomology exact sequences of the conormal exact sequences
\begin{align*}
0 \longrightarrow \cO_{\tX}(2E_X - \pi^*D) \longrightarrow \Omega^1_{\tY}\vert_{\tX\hphantom{{}'}} \longrightarrow 
\Omega^1_{\tX\hphantom{{}'}} \longrightarrow 0,
&\\
0 \longrightarrow \cO_{X'}(-D) \longrightarrow \Omega^1_{Y}\vert_{X'} \longrightarrow \Omega^1_{X'} \longrightarrow 0,&
\end{align*}
we obtain
\begin{equation*}
H^1(\tX, \Omega^1_{\tX}) = 
H^1(\tX, \Omega^1_{\tY}\vert_{\tX}) = 
H^1(X', \Omega^1_Y\vert_{X'}) \oplus \Bbbk = 
H^1(X', \Omega^1_{X'}) \oplus \Bbbk,
\end{equation*}
which is equivalent to the desired equality~$\h^{1,1}(\tX) = \h^{1,1}(X') + 1$.
\end{proof} 

To apply Proposition~\ref{prop:h11-tx} we verify its hypotheses in the relevant cases.
Recall from the Introduction that we define~$\rM_2 = \PP(1^4,3)$, $\rM_3 = \PP(1^5,2)$, 
$\rM_4 = \PP^5$, $\rM_5 =\PP^6$, \mbox{$\rM_6 = \CGr(2,5)$},
$\rM_7 = \OGr_+(5,10)$, 
$\rM_{8\hphantom{0}} = \Gr(2,6)$, 
$\rM_{9\hphantom{0}} = \LGr(3,6)$, 
and~$\rM_{10} = \GtGr(2,7)$.
Furthermore, if~$M$ is a Cohen--Macaulay variety with~$\Cl(M) \cong \ZZ$, 
we denote by~$Y_{d_1,\dots,d_k}$ a complete intersection of Cartier divisors corresponding to the elements~$d_1,\dots,d_k$ of~$\Cl(M)$.

\begin{lemma}
\label{lem:ci-fourfolds}
Let~$Y$ be one of the following fourfolds:
\begin{enumerate}
\item 
\label{it:pre-ci-i2}
$\PP(1^3,2,3)$, 
$\PP(1^4,2)$, 
$\PP^4$, 
or a smooth quadric~$Q^4$;
\item 
\label{it:pre-ci-i1}
$\rM_2$
or a smooth complete intersection~$Y_4 \subset \rM_3$,
$Y_3 \subset \rM_4$,
$Y_{2,2} \subset \rM_5$;
\item 
\label{it:pre-ci-mukai}
a smooth complete intersection~$Y_{1,1,2} \subset \rM_6$
or~$Y_{1^{\dim(\rM_g)-4}} \subset \rM_g$, $7 \le g \le 10$.
\end{enumerate}
Let~$\whom_Y$ be the reflexive hull of~$\Omega^1_Y$.
Then for all ample Cartier divisor classes~$D \in \Pic(Y)$ 
the hypotheses of Proposition~\textup{\ref{prop:h11-tx}} hold.
\end{lemma}

\begin{proof}
Let~$M$ be a (well-formed) weighted projective space or one of the Mukai varieties~$\rM_g$, where~$6 \le g \le 10$.
Note that in all cases~$\Cl(M) \cong \ZZ$ and by the Kawamata--Viehweg vanishing theorem \cite[Theorem~2.70]{Kollar-Mori:book} we have
\begin{equation}
\label{eq:h-om-d}
H^i(M, \cO_M(-D)) = 0 
\qquad\text{for~$i < \dim(M)$}
\end{equation}
for any ample Weil divisor~$D$.

Further, if~$M \cong \PP(a_0,\dots,a_N)$ is a weighted projective space, 
the reflexive hull~$\whom_M$ of the sheaf of differentials on~$M$ fits into
the Euler sequence (see~\cite[Theorem~8.1.6]{CLS})
\begin{equation*}
0 \longrightarrow \whom_M(-D) \longrightarrow \moplus_{j=0}^{N} \cO_M(-D-a_jH) \longrightarrow \cO_M(-D) \longrightarrow 0,
\end{equation*}
(where~$H$ is the ample generator of~$\Cl(M)$),
hence for any ample Cartier divisor~$D$ we have
\begin{equation}
\label{eq:h-omm-d}
H^i(M, \whom_{M}(-D)) = 0
\qquad\text{for~$i < \dim(M)$}.
\end{equation} 
If~$M = \rM_g$ with~$g \ge 7$ we take~$\whom_M = \Omega^1_M$ 
and if~$M = \rM_6$ we take~$\whom_M$ to be the pushforward of the sheaf of differentials 
of the blowup of~$M$ at the vertex with logarithmic singularities on the exceptional divisor.
In both cases the vanishing~\eqref{eq:h-omm-d} is proved in Lemma~\ref{lem:omd-mukai}.

Now let~$Y$ be a $4$-dimensional complete intersection of ample Cartier divisors~$D_1,\dots,D_n$ in~$M$.
Then for any ample Cartier divisor class~$D$ we have Koszul resolutions
\begin{equation*}
0 \to \cO_M(-D-D_1-\ldots-D_n) \longrightarrow \ldots \longrightarrow 
\moplus_{i=1}^{n} \cO_M(-D-D_i) \longrightarrow 
\cO_M(-D) \longrightarrow 
\cO_Y(-D) \to 0.
\end{equation*}
Since~$D$ and~$D_i$ are ample, 
\eqref{eq:h-om-d} shows that its terms have cohomology only in top degree.
Since the length of the resolution is equal to the codimension of~$Y$ in~$M$, 
the hypercohomology spectral sequence implies that
\begin{equation}
\label{eq:h123_om}
H^i(Y, \cO_Y(-D)) = 0
\qquad\text{for~$i < \dim(Y) = 4$.}
\end{equation}
In particular, hypothesis~\ref{it:h123-ox} of Proposition~\ref{prop:h11-tx} holds for~$Y$.
Similarly, tensoring the above Koszul resolution by~$\whom_M$ and using~\eqref{eq:h-omm-d}, we obtain
\begin{equation}
\label{eq:h123_omm}
H^i(Y, \whom_{M}\vert_Y(-D)) = 0
\qquad\text{for~$i < \dim(Y) = 4$}.
\end{equation} 

Furthermore, since~$Y$ is a complete intersection, there is an exact sequence
\begin{equation}
\label{eq:whom-restriction}
0 \longrightarrow \moplus_{i=1}^{n} \cO_Y(-D_i) \longrightarrow 
\whom_{M}\vert_Y \longrightarrow \whom_Y \longrightarrow 0.
\end{equation}
Indeed, if~$Y$ is a weighted projective space, we have~$M = Y$ and there is nothing to prove,
and in all other cases~$Y$ is smooth and contained in the smooth locus of~$M$,
hence~\eqref{eq:whom-restriction} coincides with the usual conormal sequence. 
Finally, twisting~\eqref{eq:whom-restriction} by~$\cO_Y(-D)$, and using~\eqref{eq:h123_om} and~\eqref{eq:h123_omm},
we conclude that
\begin{equation*}
H^i(Y, \whom_Y(-D)) = 0
\qquad\text{for~$i < 3$}.
\end{equation*}
In particular, hypothesis~\ref{it:h12-om} of Proposition~\ref{prop:h11-tx} holds.
\end{proof} 

Recall the sheaf~$\whom_M$ defined for~$M = \rM_g$ with~$g \ge 6$ in the proof of Lemma~\ref{lem:ci-fourfolds}.

\begin{lemma}
\label{lem:omd-mukai}
For any Mukai variety~$M = \rM_g$ with~$6 \le g \le 10$ we have
\begin{equation*}
\label{eq:hi-whom}
H^i(M, \whom_{M}(-D)) = 0
\end{equation*}
for any ample divisor~$D$ on~$M$ and~$i < \dim(M)$.
\end{lemma}

\begin{proof}
First, let~$M = \rM_g$ with~$g \ge 7$ or~$M = \rM'_6 \coloneqq \Gr(2,5)$.
Then~$M$ is smooth and by Akizuki--Nakano vanishing we have~$H^i(M, \Omega^1_M(-D)) = 0$ for all~$i \le \dim(M) - 2$. 
Moreover, 
\begin{equation*}
H^{\dim(M) - 1}(M, \Omega^1_M(-D)) = H^1(M, \cT_M(K_M + D))^\vee
\end{equation*}
by Serre duality, so it remains to check that~$H^1(M, \cT_M(D')) = 0$ for~$D' > K_M$.
If~$D' < 0$ this is proved in~\cite[Theorem~2.2]{BFM}, so assume~$D' \ge 0$.
If~$g \le 9$ then~$\cT_M$ is a globally generated irreducible equivariant vector bundle, 
hence its higher cohomology vanishes by the Borel--Bott--Weil Theorem.
Finally, if~$g = 10$, then~$\cT_M$ is an equivariant extension
\begin{equation*}
0 \longrightarrow \cT'_M \longrightarrow \cT_M \longrightarrow \cO_M(H) \longrightarrow 0,
\end{equation*} 
where~$\cT'_M$ is irreducible, 
and for~$D' > 0$ the bundles~$\cT'_M(D')$ and~$\cO_M(H + D')$ are globally generated,
hence have no higher cohomology, hence the same is true for~$\cT_M(D')$,
while for~$D' = 0$ the bundle~$\cT'_M$ is acyclic (this can be easily checked by the Borel--Bott--Weil Theorem)
and~$\cO_M(H)$ is globally generated and the same argument applies.

Now let~$M = \rM_6 = \CGr(2,5)$ be the cone over~$\rM'_6$.
Let~$\tM \coloneqq \PP_{\rM'_6}(\cO \oplus \cO(-1))$ be the blowup of~$\rM_6$ at the vertex, 
let~$E \subset \tM$ be its exceptional divisor, so that
\begin{equation*}
\whom_M \coloneqq \pi_*(\Omega^1_{\tM}(\log E)),
\end{equation*}
where~$\pi \colon \tM \to M$ is the blowup morphism.
Let~$p \colon \tM \to \rM'_6$ be the natural projection.
Then it is easy to see that there is an exact sequence
\begin{equation*}
0 \longrightarrow p^*\Omega^1_{\rM'_6} \longrightarrow \Omega^1_{\tM}(\log E) \longrightarrow \cO_{\tM}(-H) \longrightarrow 0,
\end{equation*}
where~$H$ is the pullback of the ample generator of~$\Cl(M) = \Pic(M)$.
Now, using the projection formula for~$\pi$, we see that it is enough to check the vanishings
\begin{equation*}
H^i(\tM, p^*\Omega^1_{\rM'_6}(-tH)) = 0
\qquad\text{and}\qquad 
H^i(\tM, \cO_{\tM}(-(t+1)H)) = 0
\end{equation*}
for~$i < \dim \tM$ and~$t > 0$.
For this we note that
\begin{equation*}
\mathbf{R} p_*\cO_{\tM}(-tH) \cong \moplus_{s=1}^{t-1} \cO_{\rM'_6}(-s)[-1].
\end{equation*}
and see that the required vanishings follow from~\eqref{eq:hi-whom} and~\eqref{eq:h-om-d} for~$\rM'_6$.
\end{proof}

\begin{corollary}
\label{cor:ci-factorial}
Let~$(X,x_0)$ be a $1$-nodal complete intersection of one of the following types:
\begin{enumerate}
\item 
\label{it:ci-i2}
$X_6 \subset \PP(1^3,2,3)$, 
$X_4 \subset \PP(1^4,2)$, 
$X_3 \subset \PP^4$, 
$X_{2,2} \subset \PP^5$;
\item 
\label{it:ci-i1}
$X_6 \subset \rM_2 = \PP(1^4,3)$,
$X_{2,4} \subset \rM_3 = \PP(1^5,2)$,
$X_{2,3} \subset \rM_4 = \PP^5$,
$X_{2,2,2} \subset \rM_5 = \PP^6$;
\item 
\label{it:ci-mukai}
$X_{1^3,2} \subset \rM_6$
and
$X_{1^{\dim(\rM_g)-3}} \subset \rM_g$, $7 \le g \le 10$.
\end{enumerate}
Then~$X$ is factorial.
\end{corollary}

\begin{proof}
To start with, we observe that, in the cases where the ambient variety~$M$ is singular,
the tangent space at any of its singular points has dimension larger than~$\dim(M) + 1$, 
hence~$X$ must be contained in the smooth locus of~$M$.

Now we check that in each case we can choose a fourfold~$Y$ as in Lemma~\ref{lem:ci-fourfolds} 
such that~$X \subset Y$.

First, if~$X$ is a hypersurface in a weighted projective space, there is nothing to check.

Now, let~$X = X_{2,2} \subset \PP^5$.
Then a general quadric in the pencil of quadrics through~$X$ is smooth at~$x_0$ because~$\dim T_{x_0}(X) = 4$
and smooth away from~$X$ by Bertini's Theorem.
Moreover, any of these quadrics is smooth along~$X \setminus \{x_0\}$ because~$X$ is a Cartier divisor on it.
Thus, we can take~$Y$ to be a general quadric through~$X$.

Obviously, the same argument works in all other cases of the corollary, 
because, as we noticed above, the subvariety~$X$ is contained in the smooth locus of the ambient variety~$M$.

Furthermore, in all cases~$X = Y \cap D$, where~$D$ is a Cartier divisor on~$M$, 
and it is easy to check that~$D + K_X$ is always equal to the restriction of an effective Weil divisor class on~$M$.
Moreover, when~$M$ is smooth, the corresponding linear system on~$M$ is base point free, hence the same is true on~$X$,
and when~$M$ is singular, the base locus (when nonempty) coincides with one of the singular points of~$M$, 
hence~$x_0$ lies away of this base locus.

Now, Lemma~\ref{lem:ci-fourfolds} proves that all hypotheses of Proposition~\ref{prop:h11-tx} are satisfied, 
hence
\begin{equation*}
\h^{1,1}(\Bl_{x_0}(X)) = \h^{1,1}(X') + 1,
\end{equation*}
where~$X'$ is a smooth complete intersection of the same type as~$X$, 
i.e., a smooth Fano variety with Picard number~$1$.
It follows that~$\h^{1,1}(\Bl_{x_0}(X)) = 2$, hence~$\uprho(\Bl_{x_0}(X)) = 2$, hence~$X$ is factorial.
\end{proof} 

A similar argument gives the following result. 
Recall that the blowup of a Weil divisor class is discussed in Appendix~\ref{App:blowup},
and that the blowup~$\Bl_{[\Sigma]}(\bar{X})$ that appears in the next proposition
is well-defined by Lemma~\ref{lem:ddprime}.

\begin{proposition}
\label{prop:rho-3}
Let~$Z_1,\, Z_2$ be two smooth Fano varieties with~$\uprho(Z_1) = \uprho(Z_2) = 1$.
Let~$H_1$ and~$H_2$ be the ample generators of~$\Pic(Z_1)$ and~$\Pic(Z_2)$. 
Let~$L_1 \subset Z_1$ and~$L_2 \subset Z_2$ be lines such that 
the restriction morphisms~$H^0(Z_i, \cO_{Z_i}(H_i)) \to H^0(L_i, \cO_{L_i}(1))$ are surjective.
Let
\begin{equation*}
\Sigma \coloneqq L_1 \times L_2 \subset Z_1 \times Z_2 \eqqcolon Z.
\end{equation*}
Let~$\bar{X} \subset Z$ be a $3$-dimensional complete intersection of ample divisors~$D_i \in |a_iH_1 + b_iH_2|$ 
containing~$\Sigma$ and such that
\begin{equation*}
\textstyle\sum a_i = \io(Z_1) - 1,
\qquad
\textstyle\sum b_i = \io(Z_2) - 1.
\end{equation*}
If~$\bar{X}$ has terminal singularities and~$\Blw{\Sigma}(\bar{X})$ is smooth 
then~$\uprho(\Blw{\Sigma}(\bar{X})) = \uprho(\Blw{-\Sigma}(\bar{X})) = 3$.
\end{proposition}

\begin{proof}
Consider the blowup~$\tZ \coloneqq \Bl_{\Sigma}(Z)$ and let~$\tE \subset \tZ$ be its exceptional divisor.
The blowup morphism~$\pi \colon\tZ \to Z$ induces an isomorphism of cohomology
\begin{equation*}
H^i(\tZ, \cO(aH_1 + bH_2 - \tE)) \longrightarrow H^i(Z, \cI_\Sigma(aH_1 + bH_2))
\end{equation*}
for all~$a,b \in \ZZ$ by the projection formula and isomorphism~$\mathbf{R}\pi_*\cO(-\tE) \cong \cI_\Sigma$.
Note also that 
\begin{align}
\label{eq:rps-van}
\mathbf{R}p_* \cO_{\tE}(t\tE) &\cong
\begin{cases}
0, & \text{for~$1 \le t \le \dim(Z) - 3$},\\
\det(\cN_{\Sigma/Z})[3 - \dim(Z)], & \text{for~$t = \dim(Z) - 2$},
\end{cases}
\intertext{
where~$\cN_{\Sigma/Z}$ is the normal bundle and~$p \colon \tE \cong \PP_\Sigma(\cN_{\Sigma/Z}) \to \Sigma$ is the projection.
Similarly, }
\label{eq:rzs-van}
\mathbf{R}^i\pi_* \cO_{\tZ}(t\tE) &\cong
\begin{cases}
\cO_Z, & \text{$1 \le t \le \dim(Z) - 2$ and~$i = 0$},\\
0, & \text{$1 \le t \le \dim(Z) - 3$ and~$i > 0$},\\
0, & \text{$\hphantom{1 \le {}} t = \dim(Z) - 2$ and~$i \not\in \{0, \dim(Z) - 3\}$},\\
\det(\cN_{\Sigma/Z}), & \text{$\hphantom{1 \le {}} t = \dim(Z) - 2$ and~$i = \dim(Z) - 3$}.
\end{cases}
\end{align}
Indeed, \eqref{eq:rzs-van} follows easily from~\eqref{eq:rps-van} by induction on~$t$.
Finally, the exact sequence
\begin{equation}
\label{eq:om-zz}
0 \longrightarrow 
\pi^*(\Omega^1_{Z_1} \oplus \Omega^1_{Z_2}) \longrightarrow 
\Omega^1_{\tZ} \longrightarrow 
\epsilon_*(p^*\cN_{\Sigma/Z}^\vee \otimes \cO_{\tE}(\tE)) \longrightarrow 
\epsilon_*\cO_{\tE} \longrightarrow 0,
\end{equation}
where~$\epsilon \colon \tE \to \tZ$ is the embedding, implies (see the argument of Proposition~\ref{prop:h11-tx}) that
\begin{equation}
\label{eq:om-tz}
H^i(\tZ, \Omega^1_{\tZ}) \cong
\begin{cases}
H^i(Z, \Omega^1_{Z_1}) \oplus H^i(Z, \Omega^1_{Z_2}) \oplus \Bbbk, &\text{if~$i = 1$},\\
H^i(Z, \Omega^1_{Z_1}) \oplus H^i(Z, \Omega^1_{Z_2}), &\text{otherwise}.
\end{cases}
\end{equation}

Now let~$\tX \subset \tZ$ be the intersection of the strict transforms~$\tD_i \in |a_iH_1 + b_iH_2 - \tE|$ 
of divisors~$D_i \in |a_iH_1 + b_iH_2|$ that cut out~$\bar{X} \subset Z$.
By the argument of Lemma~\ref{lem:ddprime} we have
\begin{equation*}
\tX \cong \Blw{\Sigma}(\bar{X}).
\end{equation*}
Therefore, by our assumption, it is a smooth threefold.
Now we compute the Hodge number~$\h^{1,1}(\tX)$ using the complete intersection representation of~$\tX \subset \tZ$.

First, we show that
\begin{equation}
\label{eq:h-om-twisted}
H^\bullet(\tZ, \Omega^1_{\tZ}(-aH_1 - bH_2 + c\tE)) = 0
\end{equation} 
if either
\begin{aenumerate}
\item 
\label{it:c-small}
$1 \le c \le \dim(Z) - 4$, $1 \le a \le \io(Z_1) - 1$, and~$1 \le b \le \io(Z_2) - 1$, or
\item 
\label{it:c-big}
$c = \dim(Z) - 3$, $a = \io(Z_1) - 1$, and~$b = \io(Z_2) - 1$.
\end{aenumerate}
For this we twist~\eqref{eq:om-zz} appropriately and consider the hypercohomology spectral sequence.
For~$1 \le c \le \dim(Z) - 3$, using~\eqref{eq:rzs-van} we identify cohomology of the first term with
\begin{equation*}
H^\bullet(Z, \Omega^1_{Z_1}(-aH_1 - bH_2)) \oplus H^\bullet(Z, \Omega^1_{Z_2}(-aH_1 - bH_2)).
\end{equation*}
Now applying the K\"unneth formula and Kodaira vanishing we see that this cohomology is zero
(for the first summand we use the inequality~$1 \le b \le \io(Z_2) - 1$ 
and for the second summand we use~$1 \le a \le \io(Z_1) - 1$).
The fourth term
\begin{equation*}
H^\bullet(\tZ, \cO_{\tZ}(-aH_1 - bH_2 + c\tE) \otimes \epsilon_*\cO_{\tE}),
\end{equation*}
also vanishes by~\eqref{eq:rps-van}.
Finally, cohomology of the third term can be identified with
\begin{equation*}
H^\bullet(\tE, p^*\cN_{\Sigma/Z}^\vee \otimes \cO_{\tE}(-aH_1 - bH_2 + (c + 1)\tE)).
\end{equation*}
In case~\ref{it:c-small} this is zero by~\eqref{eq:rps-van}, 
and in case~\ref{it:c-big} using~\eqref{eq:rps-van} we can rewrite this as
\begin{equation*}
H^{\bullet + 3 - \dim(Z)}(\Sigma, \cN_{\Sigma/Z}^\vee \otimes \det(\cN_{\Sigma/Z}) \otimes \cO_\Sigma(1 - \io(Z_1), 1 - \io(Z_2))).
\end{equation*}
But the adjunction formula gives~$\det(\cN_{\Sigma/Z}) \otimes \cO_E(1 - \io(Z_1), 1 - \io(Z_2)) \cong \cO_E(-1,-1)$,
while~$\cN_{\Sigma/Z}^\vee \cong \cN^\vee_{L_1/Z_1} \oplus \cN^\vee_{L_2/Z_2}$ because~$\Sigma = L_1 \times L_2$.
Therefore, this cohomology also vanishes,
and summarizing the above computation, we finally obtain~\eqref{eq:h-om-twisted}.

Next, tensoring the Koszul resolution of~$\tX \subset \tZ$ by~$\Omega^1_{\tZ}$ and using~\eqref{eq:h-om-twisted}, we see that
\begin{equation}
\label{eq:om-tz-vert-tx}
H^\bullet(\tX, \Omega^1_{\tZ}\vert_{\tX}) \cong 
H^\bullet(\tZ, \Omega^1_{\tZ}).
\end{equation}

On the other hand, we check that
\begin{equation}
\label{eq:conormal-coh}
h^q(\cO_{\tX}(\tE - a_iH_1 - b_iH_2)) = 0
\qquad\text{for~$q < \dim(\tX) = 3$.}
\end{equation}
Indeed, since~$\cO_{\tX}(K_{\tX}) = \cO_{\tX}(-H_1-H_2)$ by adjunction formula,
it is enough to show that
\begin{equation*}
h^q(\cO_{\tX}((a_i - 1)H_1 + (b_i - 1)H_2 - \tE)) = 0
\qquad\text{for~$q > 0$,}
\end{equation*}
and to compute this we can use the exact sequence
\begin{equation*}
0 \to \cO_{\tX}((a_i - 1)H_1 + (b_i - 1)H_2 - \tE) \to \cO_{\tX}((a_i - 1)H_1 + (b_i - 1)H_2) \to \cO_\Sigma(a_i-1,b_i-1) \to 0
\end{equation*}
and note that the the second and third terms have cohomology only in degree~$0$ (because~$a_i,b_i \ge 1$)
and the restriction morphism on global sections is surjective by assumption.

Finally, combining~\eqref{eq:conormal-coh}, \eqref{eq:om-tz-vert-tx}, and~\eqref{eq:om-tz}
with the conormal exact sequence
\begin{equation*}
0 \longrightarrow 
\moplus_{i=1}^{\dim Z - 3} \cO_{\tX}(\tE - a_iH_1 - b_iH_2) \longrightarrow 
\Omega^1_{\tZ}\vert_{\tX} \longrightarrow 
\Omega^1_{\tX} \longrightarrow 
0
\end{equation*}
we obtain
\begin{equation*}
\h^{1,1}(\tX) = \dim H^1(\tX, \Omega^1_{\tX}) =\dim H^1(Z_1, \Omega^1_{Z_1}) + \dim H^1(Z_2, \Omega^1_{Z_2}) + 1 = 3,
\end{equation*}
and since~$H^2(\tX, \cO_{\tX}) = H^2(\bar{X}, \cO_{\bar{X}}) = 0$ 
(because~$\bar{X}$ is a Fano variety with terminal singularities), 
we conclude that~$\uprho(\tX) = 3$.
\end{proof} 


\section{Special linear sections of Mukai varieties}
\label{sec:links-78910}

Recall the Mukai varieties~$\rM_g$, $g \in \{7,8,9,10\}$, defined in the Introduction:
\begin{equation*}
\rM_{7\hphantom{0}} = \OGr_+(5,10),
\quad
\rM_{8\hphantom{0}} = \Gr(2,6),
\quad
\rM_{9\hphantom{0}} = \LGr(3,6),
\quad
\rM_{10} = \GtGr(2,7),
\end{equation*}
where the last is the adjoint Grassmannian of the simple algebraic group of Dynkin type~$\mathrm{G}_2$.
In this section we show that any factorial Fano threefold~$X$ with a single node or cusp,
\mbox{$\uprho(X) = 1$}, \mbox{$\io(X) = 1$}, and~$g \coloneqq \g(X) \in \{7,8,9,10\}$
is a dimensionally transverse linear section of the Mukai variety~$\rM_{g}$,
thus providing an extension of the Mukai's description of smooth Fano threefolds.
We also show that the Sarkisov link~\eqref{eq:sl-factorial} is induced by a Sarkisov link of~$\rM_{g}$ constructed in~\cite[\S2]{KP21},
and its midpoint~$\bar{X}$ is a linear section of~$\rM_{g - 1}$.

For the reader's convenience we provide here an outline of our argument:
\begin{itemize}[wide]
\item 
We choose a smooth quadric~$Q$ of maximal dimension in the Mukai variety~$\rM_{g-1}$ 
and identify the blowup~$\Bl_Q(\rM_{g-1})$ as a projective bundle over (or as a blowup of) a simpler variety,
see Propositions~\ref{prop:bl-q-rm6}, \ref{prop:bl-q-rm7}, \ref{prop:bl-q-rm8}, and~\ref{prop:bl-q-rm9}.
\item 
By passing to linear sections we prove a bijection between 
the set of 3-dimensional linear sections~$\bar{X} \subset \rM_{g-1}$ 
such that~$\Sigma = \bar{X} \cap Q$ is an irreducible quadrics surface and~$\Blw{\Sigma}(\bar{X})$ is smooth and
the set of pairs~$(X_+,\Gamma)$, where~$\Gamma$ is a smooth curve on a smooth Fano threefold~$X_+$ as in Proposition~\ref{prop:sl-g78910},
see Corollaries~\ref{cor:critical-gr25-p3}, \ref{cor:critical-ogr-q3}, \ref{cor:critical-gr26-p8}, and~\ref{cor:critical-lgr-y5}.
\item 
We consider the Sarkisov link
\begin{equation}
\label{eq:sl-rm}
\vcenter{\xymatrix@!C{
& 
\Bl_{x_0}(\rM_g) \ar[dl]_\pi \ar[dr]^\xi \ar@{<-->}[rr] &&
\widehat{\rM}_g^+ \ar[dl]_{\xi_+} \ar[dr]^{\pi_+}
\\
\rM_g &&
\overline{\rM}_g &&
\rM_g^+
}}
\end{equation}
constructed in~\cite[Theorems~2.2 and~4.4]{KP21}, 
and by passing to linear sections we construct a link 
that goes from the blowup~$\Bl_{x_0}(X)$ of a linear section~$(X,x_0)$ of~$\rM_g$ with a single node or cusp to the blowup~$\Bl_\Gamma(X_+)$.
Then we apply the first part of Proposition~\ref{prop:sl-g78910} and the uniqueness of Sarkisov links 
to identify any factorial Fano threefold~$X$ with a single node or cusp with a linear section of~$\rM_g$,
see Corollaries~\ref{cor:critical-gr25-p3-sl}, \ref{cor:critical-ogr-q3-sl}, \ref{cor:critical-lgr-y4-sl}, and~\ref{cor:critical-lgr-y5-sl}.
\end{itemize}

The general line of argument is the same in all cases, but technical details differ very much,
so we have to keep a separate section for each case.

\subsection{Genus 7}

All maximal quadrics in~$\Gr(2,5)$ have the form~$Q^4 = \Gr(2,4) \subset \Gr(2,5)$.
Moreover, the group~$\PGL(5)$ acts transitively on the Hilbert scheme of 4-dimensional quadrics in~$\Gr(2,5)$,
so from now on we fix one such quadric~$Q^4 \subset \Gr(2,5)$.

\begin{proposition}
\label{prop:bl-q-rm6}
There is a canonical isomorphism
\begin{equation}
\label{eq:gr25-p3}
\Bl_{Q^4}(\Gr(2,5)) \cong \PP_{\PP^3}(\cO(-1) \oplus \cT(-2)),
\end{equation}
where~$\cT$ is the tangent bundle.
If~$\bar{H}$ and~$\bar{E}$ denote the hyperplane class of~$\Gr(2,5)$ and the class of the exceptional divisor in~\eqref{eq:gr25-p3}
then the pullback of the hyperplane class of~$\PP^3$ is
\begin{equation}
\label{eq:hp-g6}
H_+ = \bar{H} - \bar{E}
\end{equation}
and the relative hyperplane class for~$\PP_{\PP^3}(\cO(-1) \oplus \cT(-2))$ coincides with~$\bar{H}$.
\end{proposition}

\begin{proof}
Follows from two projective bundle structures of the partial flag variety~$\Fl(1,2;5)$.
\end{proof}

Recall from~\cite{DK1} that an ordinary Gushel--Mukai threefold
is a complete intersection in~$\Gr(2,5)$ of two hyperplanes and a quadric.

\begin{corollary}
\label{cor:critical-gr25-p3}
There is a bijection between the sets of all
\begin{itemize}[wide]
\item 
ordinary Gushel--Mukai threefolds~$\bar{X} \subset \Gr(2,5)$ with isolated hypersurface singularities
such that~$\Sigma = \bar{X} \cap Q^4$
is an irreducible quadric surface and~$\Blw{\Sigma}(\bar{X})$ is smooth, and
\item 
smooth connected curves~$\Gamma \subset \PP^3$ such that~$\deg(\Gamma) = 8$, $\g(\Gamma) = 6$, 
and~\eqref{eq:nondegeneracy} holds.
\end{itemize}
If~$\bar{X}$ corresponds to~$\Gamma$, there is an isomorphism
\begin{equation}
\label{eq:x6-g}
\Blw{\Sigma}(\bar{X}) \cong \Bl_\Gamma(\PP^3)
\end{equation} 
and the morphism~$\Bl_\Gamma(\PP^3) \to \bar{X}$ is the anticanonical contraction.
\end{corollary}

\begin{proof}
Let~$\bar{X} \subset \Gr(2,5)$ be a Gushel--Mukai threefold containing a quadric surface~$\Sigma \subset Q^4$.
Since~$\bar{X} \subset \Gr(2,5)$ is a complete intersection of two hyperplanes and a quadric 
while~$\Sigma \subset Q^4$ is a complete intersection of two hyperplanes, 
we may assume that the quadratic equation of~$\bar{X}$ vanishes on~$Q^4$, 
while the hyperplanes are dimensionally transverse to it.
Consider the intersection~$\bar{X}' \subset \Bl_{Q^4}(\Gr(2,5))$
of the corresponding divisor of class~$2\bar{H} - \bar{E}$ and two divisors of class~$\bar{H}$.
It is clear that the projection~$\bar{X}' \to \bar{X}$ 
has at most 1-dimensional fibers over the singularities of~$\bar{X}$ 
and is an isomorphism elsewhere.
Since~$\bar{X}$ has isolated singularities,
$\bar{X}'$ is a normal irreducible threefold and it coincides with the strict transform of~$\bar{X}$, 
therefore Lemma~\ref{lem:str-blw} implies that
\begin{equation*}
\bar{X}' \cong \Blw{\Sigma}(\bar{X}) \subset \Bl_{Q^4}(\Gr(2,5)) \cong \PP_{\PP^3}(\cO(-1) \oplus \cT(-2)).
\end{equation*}
In other words, $\Blw{\Sigma}(\bar{X})$ is an irreducible complete intersection,
and since~$2\bar{H} - E \sim \bar{H} + H_+$ by~\eqref{eq:hp-g6},
this complete intersection corresponds to a morphism
\begin{equation*}
\psi \colon \cO(-1) \oplus \cT(-2) \longrightarrow \cO^{\oplus 2} \oplus \cO(1).
\end{equation*}
Now, if~$\Gamma \subset \PP^3$ is the degeneracy locus of~$\psi$ then~$\dim(\Gamma) \le 1$ 
(otherwise the preimage of~$\Gamma$ 
under the map~$\Blw{\Sigma}(\bar{X}) \hookrightarrow \PP_{\PP^3}(\cO(-1) \oplus \cT(-2)) \to \PP^3$ would be an irreducible component),
and applying~\cite[Lemma~2.1]{K16} we obtain an isomorphism~$\Blw{\Sigma}(\bar{X}) \cong \Bl_\Gamma(\PP^3)$,
and since the left side is smooth, we conclude that~$\Gamma$ is smooth as well.
Finally, computing the determinants of the source and target of~$\psi$ we see that~$\Ker(\psi) \cong \cO(-4)$, 
hence after dualization and twist we obtain an exact sequence
\begin{equation}
\label{eq:res-g86}
0 \longrightarrow 
\cO(-5) \oplus \cO(-4)^{\oplus 2} \xrightarrow{\ \psi^\vee\ } 
\cO(-3) \oplus \Omega^1(-2) \xrightarrow{\quad} 
\cO \xrightarrow{\quad} 
\cO_\Gamma \longrightarrow 0.
\end{equation} 
Using it to compute the cohomology of~$\cO_\Gamma$ we see that~$\Gamma$ is connected and~$\g(\Gamma) = 6$,
and computing similarly the cohomology of~$\cO_\Gamma(3)$ we see that~$\deg(\Gamma) = 8$
and~\eqref{eq:nondegeneracy} holds.

Conversely, let~$\Gamma \subset \PP^3$ be a smooth connected curve 
such that~$\deg(\Gamma) = 8$, $\g(\Gamma) = 6$, and~\eqref{eq:nondegeneracy} holds.
We check that its structure sheaf has a resolution of the form~\eqref{eq:res-g86},
by decomposing~$\cI_\Gamma$ with respect to the exceptional collection~$(\cO(-5),\cO(-4),\Omega^1(-2),\cO(-3))$.
It is not hard to see that to show that the decomposition has form~\eqref{eq:res-g86},
it is enough to verify the following equalities
\begin{equation*}
H^\bullet(\cI_\Gamma(1)) = \Bbbk[-2],
\qquad 
H^\bullet(\cI_\Gamma(2)) = \Bbbk[-1],
\qquad 
H^\bullet(\cI_\Gamma(3)) = \Bbbk.
\end{equation*}
By condition~\eqref{eq:nondegeneracy} and the Riemann--Roch Theorem it is enough to check that
\begin{equation*}
\dim H^1(\Gamma, \cO_\Gamma(1)) \le 1,
\qquad 
H^1(\Gamma, \cO_\Gamma(2)) = H^1(\Gamma, \cO_\Gamma(3)) = 0.
\end{equation*}
The last two equalities are obvious, because~$\deg(\cO_\Gamma(2)) > 2\g(\Gamma) - 2$.
If the first condition fails then it follows that~$\Gamma$ is hyperelliptic 
and~$\cO_\Gamma(1)$ is isomorphic to the fourth power of the hyperelliptic line bundle.
But then the morphism~$\Gamma \to \PP^3$ given by this line bundle factors through the hyperelliptic covering,
a contradiction.

Thus, we obtain a resolution~\eqref{eq:res-g86}, and applying~\cite[Lemma~2.1]{K16}
we conclude that~$\Bl_\Gamma(\PP^3)$ is a complete intersection in~$\Bl_{Q^4}(\Gr(2,5))$
of a divisor of class~$2\bar{H} - \bar{E}$ and two divisors of class~$\bar{H}$
and there is a linear equivalence~$\bar{H} \sim 4H_+ - E_+$, where~$E_+$ is the exceptional divisor of~$\Bl_\Gamma(\PP^3)$.
Therefore, the morphism~$\Bl_\Gamma(\PP^3) \to \Gr(2,5)$ is anticanonical, hence small by Lemma~\ref{lem:blg-xplus},
its image~$\bar{X}\subset\Gr(2,5)$ is an ordinary Gushel--Mukai threefold containing a quadric surface~$\Sigma$
and~$\Bl_\Gamma(\PP^3)$ is the strict transform of~$\bar{X}$.
In particular, $\bar{X}$ has terminal, hence isolated hypersurface singularities, 
and~$\Bl_\Gamma(\PP^3) \cong \Blw{\Sigma}(\bar{X})$ by Lemma~\ref{lem:str-blw}.

These two constructions are mutually inverse and define the required bijection.
Finally, the linear equivalences~$\bar{E} \sim \bar{H} - H_+ \sim 3H_+ - E_+$
show that~$\Sigma$ is dominated by the strict transform 
of the unique (by~\eqref{eq:nondegeneracy}) divisor~$S_\Gamma \subset \PP^3$ in~$|3H_+ - \Gamma|$,
hence~$\Sigma$ is irreducible.
\end{proof}

\begin{remark}
The above proof shows that there is a contraction~$\Bl_\Gamma(S_\Gamma) \to \Sigma$, given by the class~$4H_+ - E_+$.
If~$S_\Gamma$ is smooth, we have~$\Bl_\Gamma(S_\Gamma) \cong S_\Gamma$, 
hence~$\Sigma$ is smooth as well,
and the morphism contracts the five $4$-secant lines of~$\Gamma$.
\end{remark} 

\begin{corollary}
\label{cor:critical-gr25-p3-sl}
If~$\Gamma \subset X_+ = \PP^3$ is a smooth curve with~$\deg(\Gamma) = 8$ and~$\g(\Gamma) = 6$ satisfying~\eqref{eq:nondegeneracy},
there is a Sarkisov link~\eqref{eq:sl-factorial}, where
\begin{itemize}[wide]
\item 
$X$ is a linear section of the Mukai variety~$\rM_7$ with a single node or cusp~$x_0 \in X$,
\item 
$\bar{X} \subset \Gr(2,5)$ is an ordinary Gushel--Mukai threefold with terminal singularities
containing an irreducible quadric surface~\mbox{$\Sigma \subset \bar{X}$},
\item 
$\pi$ and~$\pi_+$ are the blowups of~$x_0$ and~$\Gamma$, and
\item 
$\xi$ and~$\xi_+$ are the blowups of the Weil divisor classes~$-\Sigma$ and~$\Sigma$, respectively.
\end{itemize}
Moreover, any factorial Fano threefold~$X$ with a single node or cusp, $\uprho(X) = 1$, $\io(X) = 1$, and~$\g(X) = 7$
is a dimensionally transverse linear section of the Mukai variety~$\rM_7$.
\end{corollary}

\begin{proof}
Assume~$\Gamma \subset X_+ = \PP^3$ is a smooth curve, $\deg(\Gamma) = 8$, $\g(\Gamma) = 6$,
and~\eqref{eq:nondegeneracy} holds. 
The proof of Corollary~\ref{cor:critical-gr25-p3} shows that~$\cO_\Gamma$ has a resolution of the form~\eqref{eq:res-g86}.
Combining it with a twist of the Euler 
sequence~$0 \longrightarrow \cO(-5) \longrightarrow \cO(-4)^{\oplus 4} \longrightarrow \cT(-5) \longrightarrow 0$,
we see that~$\cO_\Gamma$ also has a resolution of the form
\begin{equation} 
\label{eq:resolution-g86}
0 \longrightarrow 
\cO(-4)^{\oplus 6} \xrightarrow{\ \varphi^\vee\ } 
\cO(-3) \oplus \Omega^1(-2) \oplus \cT(-5) \xrightarrow{\quad} 
\cO \xrightarrow{\quad} 
\cO_\Gamma \longrightarrow 0.
\end{equation}

On the other hand, recall from~\cite[Theorem~2.2 and~\S2.4]{KP21} that there is a Sarkisov link~\eqref{eq:sl-rm},
where~$\rM_7^+ = \PP^4$ and~$\widehat{\rM}_7^+ \cong \PP_{\PP^4}(\cE)$ with~$\cE = \cO(-1) \oplus \Omega^2_{\PP^4}(2)$,
and relations
\begin{equation*}
H_+ \sim H - 2E
\qquad\text{and}\qquad 
\bar{H} \sim H - E
\qquad\text{in~$\Pic(\Bl_{x_0}(\rM_7)) \cong \Pic(\PP_{\PP^4}(\cE))$,}
\end{equation*}
where~$H_+$ and~$\bar{H}$ denote the hyperplane class of~$\PP^4$ and the relative hyperplane class of~$\PP_{\PP^4}(\cE)$, respectively. 
Note that the restriction of~$\cE^\vee(-4)$ to a hyperplane~$\PP^3 \subset \PP^4$ 
is isomorphic to the second term of~\eqref{eq:resolution-g86}, 
and the first map in~\eqref{eq:resolution-g86} is the twisted dual of a morphism~$\varphi \colon \cE\vert_{\PP^3} \to \cO^{\oplus 6}$.
Now, applying~\cite[Lemma~2.1]{K16}, we see that~$\Bl_\Gamma(\PP^3)$ 
can be identified with a linear section of~$\PP_{\PP^3}(\cE\vert_{\PP^3})$ of codimension~$6$, 
hence it is an intersection in~$\widehat{\rM}_7^+ = \PP_{\PP^4}(\cE)$ of one divisor 
in the linear system~$|H_+| = |H-2E|$ and six divisors in~$|\bar{H}| = |H-E|$.
The diagram~\eqref{eq:sl-rm} then shows that there is a flop from~$\Bl_\Gamma(\PP^3)$ 
onto the blowup~$\Bl_{x_0}(X)$ of a linear section~$X \subset \rM_7$ of codimension~$7$ which is singular at~$x_0$
and we obtain a Sarkisov link that has the form~\eqref{eq:sl-factorial}.

Since the midpoint~$\bar{X}$ of the link is obtained from the anticanonical contraction of~$\Bl_\Gamma(\PP^3)$,
Corollary~\ref{cor:critical-gr25-p3} shows that~$\bar{X}$ is an ordinary Gushel--Mukai threefold with terminal singularities
containing an irreducible quadric surface~$\Sigma$ and~$\Bl_\Gamma(\PP^3) \cong \Blw{\Sigma}(\bar{X})$.
Then Corollary~\ref{cor:bl-weil-flop} shows that~$\Bl_{x_0}(X) \cong \Blw{-\Sigma}(\bar{X})$.
Finally, the argument of Theorem~\ref{thm:intro-nf-ci} identifies the strict transform of~$\Sigma$
with the exceptional divisor of~$\Bl_{x_0}(X)$, 
and we conclude that~$X$ is 1-nodal or 1-cuspidal.
In particular, $X$ is a Fano threefold with~$\uprho(X) = 1$, $\io(X) = 1$, and~$\g(X) = 7$,
and it is factorial by Corollary~\ref{cor:ci-factorial} and Remark~\ref{rem:blowup-nc}.
This proves the first part of the corollary.

To prove the second part we start with a factorial Fano threefold~$X$,
consider the Sarkisov link~\eqref{eq:sl-factorial} that exists by the first part of Proposition~\ref{prop:sl-g78910},
and obtain from it a smooth curve~$\Gamma \subset \PP^3$.
Since~\mbox{$\uprho(\Bl_\Gamma(\PP^3)) = 2$}, this Sarkisov link 
must be inverse to the link constructed in the first part of the corollary,
so we conclude that~$X$ is a linear section of~$\rM_7$.
\end{proof} 

\subsection{Genus 8}

All maximal quadrics in~$\OGr_+(5,10)$ have the form
\begin{equation*}
Q^6 = \OGr_+(4,8) \subset \OGr_+(5,10).
\end{equation*}
Moreover, the group~$\SO(10)$ acts transitively on the Hilbert scheme of 6-dimensional quadrics in~$\OGr_+(5,10)$,
so from now on we fix one such quadric~$Q^6 \subset \OGr_+(5,10)$. 

\begin{proposition}[{\cite[Proposition~5.8]{K18:spinor}}]
\label{prop:bl-q-rm7}
There is a canonical isomorphism
\begin{equation}
\label{eq:ogr-q6}
\Bl_{Q^6}(\OGr_+(5,10)) \cong \PP_{Q^6}(\cO(-1) \oplus \cS_4),
\end{equation}
where~$\cS_4$ is a spinor bundle of rank~$4$ on~$Q^6$ with~$\det(\cS_4) \cong \cO(-2)$.
If~$\bar{H}$ and~$\bar{E}$ denote the hyperplane class of~$\OGr_+(5,10)$ and the class of the exceptional divisor in~\eqref{eq:ogr-q6}
then the pullback of the hyperplane class of~$Q^6$ is
\begin{equation}
\label{eq:hp-g7-a}
H_+ = \bar{H} - \bar{E}
\end{equation}
and the relative hyperplane class for~$\PP_{Q^6}(\cO(-1) \oplus \cS_4)$ coincides with~$\bar{H}$.
\end{proposition} 

\begin{corollary}
\label{cor:critical-ogr-q3}
There is a bijection between the sets of all
\begin{itemize}[wide]
\item 
$3$-dimensional linear sections~$\bar{X} \subset \OGr_+(5,10)$ with isolated hypersurface singularities
such that~$\Sigma \coloneqq \bar{X} \cap Q^6$ is an irreducible quadric surface and~$\Blw{\Sigma}(\bar{X})$ is smooth, and
\item 
smooth connected curves~$\Gamma \subset Q^3$ such that~$\deg(\Gamma) = 8$, $\g(\Gamma) = 4$, 
and~\eqref{eq:nondegeneracy} holds.
\end{itemize}
If~$\bar{X}$ corresponds to~$\Gamma$, there is an isomorphism
\begin{equation}
\label{eq:x7-g}
\Blw{\Sigma}(\bar{X}) \cong \Bl_\Gamma(Q^3)
\end{equation} 
and the morphism~$\Bl_\Gamma(Q^3) \to \bar{X}$ is the anticanonical contraction.
\end{corollary}

\begin{proof}
The argument is similar to that of Corollary~\ref{cor:critical-gr25-p3}.

This time~$\bar{X} \subset \OGr_+(5,10)$ is a complete intersection of seven hyperplanes
while~$\Sigma \subset Q^6$ is a complete intersection of four hyperplanes, 
hence we may assume that three hyperplanes contain~$Q^6$, 
while the other four are dimensionally transverse to it.
Consider the intersection~$\bar{X}' \subset \Bl_{Q^6}(\OGr_+(5,10))$
of the corresponding three divisors of class~$\bar{H} - \bar{E}$ and four divisors of class~$\bar{H}$.
As in Corollary~\ref{cor:critical-gr25-p3} we conclude that~$\bar{X}'$ coincides with the strict transform of~$\bar{X}$ and
\begin{equation*}
\bar{X}' \cong \Blw{\Sigma}(\bar{X}) \subset \Bl_{Q^6}(\OGr_+(5,10)) \cong \PP_{Q^6}(\cO(-1) \oplus \cS_4).
\end{equation*}
In other words, $\Blw{\Sigma}(\bar{X})$ is an irreducible complete intersection, 
and since~$\bar{H} - \bar{E} \sim H_+$ by~\eqref{eq:hp-g7-a},
the first three divisors define a subquadric~$Q^3 \subset Q^6$, 
and the other four correspond to a morphism
\begin{equation*}
\psi \colon \cO(-1) \oplus \cS_2^{\oplus 2} \cong (\cO(-1) \oplus 
\cS_4)\vert_{Q^3} \longrightarrow \cO^{\oplus 4},
\end{equation*}
Now, if~$\Gamma \subset Q^3$ is the degeneracy locus of~$\psi$
then~$\dim(\Gamma) \le 1$ (because~$\Blw{\Sigma}(\bar{X})$ is irreducible)
and applying~\cite[Lemma~2.1]{K16} we obtain an isomorphism~$\Blw{\Sigma}(\bar{X}) \cong \Bl_\Gamma(Q^3)$,
and since the left side is smooth, we conclude that~$Q^3$ and~$\Gamma$ are smooth as well.
Finally, computing the determinants we see that~$\Ker(\psi) \cong \cO(-3)$, 
and deduce an exact sequence
\begin{equation}
\label{eq:res-g84}
0 \longrightarrow 
\cO(-3)^{\oplus 4} \xrightarrow{\ \psi^\vee\ } 
\cO(-2) \oplus \cS_2(-2)^{\oplus 2} \xrightarrow{\quad} 
\cO \xrightarrow{\quad} 
\cO_\Gamma \longrightarrow 0,
\end{equation} 
and as before it follows that~$\Gamma$ is connected, $\deg(\Gamma) = 8$, $\g(\Gamma) = 4$, and~\eqref{eq:nondegeneracy} holds.

Conversely, let~$\Gamma \subset Q^3$ be a smooth connected curve 
such that~$\deg(\Gamma) = 8$, $\g(\Gamma) = 4$, and~\eqref{eq:nondegeneracy} holds.
We check that its structure sheaf has a resolution of the form~\eqref{eq:res-g84},
by decomposing~$\cI_\Gamma$ with respect to the exceptional collection~$(\cO(-3),\cS_2(-2),\cO(-2),\cO(-1))$.
To show that the decomposition has that form it is enough to check the following
\begin{equation*}
H^\bullet(\cI_\Gamma) = \Bbbk^4[-2],
\qquad 
H^\bullet(\cI_\Gamma(1)) = 0,
\qquad 
H^\bullet(\cI_\Gamma(2)) = \Bbbk.
\end{equation*}
The first is obvious and~\eqref{eq:nondegeneracy} implies that for the last two it is enough to check that
\begin{equation*}
\dim H^1(\Gamma, \cO_\Gamma(1)) = H^1(\Gamma, \cO_\Gamma(2)) = 0,
\end{equation*}
which is also obvious, because~$\deg(\cO_\Gamma(1)) > 2\g(\Gamma) - 2$.

Thus, we obtain a resolution~\eqref{eq:res-g84}, and applying~\cite[Lemma~2.1]{K16}
we conclude that~$\Bl_\Gamma(Q^3)$ is a complete intersection in~$\Bl_{Q^6}(\OGr_+(5,10))$
of three divisors of class~$\bar{H} - \bar{E}$ and four divisors of class~$\bar{H}$
and there is a linear equivalence~$\bar{H} \sim 3H_+ - E_+$, where~$E_+$ is the exceptional divisor of~$\Bl_\Gamma(Q^3)$.
Therefore, the morphism~$\Bl_\Gamma(Q^3) \to \OGr_+(5,10)$ is anticanonical, hence small by Lemma~\ref{lem:blg-xplus},
its image~$\bar{X}$ is a linear section of codimension~$7$ 
containing a quadric surface~$\Sigma$ 
dominated by the strict transform of the unique (by~\eqref{eq:nondegeneracy}) 
divisor~$S_\Gamma \subset Q^3$ in~$|2H_+ - \Gamma|$,
and hence irreducible.
These two constructions are mutually inverse and define the required bijection.
\end{proof}

\begin{remark}
The above proof shows that there is a contraction~$\Bl_\Gamma(S_\Gamma) \to \Sigma$, given by the class~$3H_+ - E_+$,
and if~$S_\Gamma$ is smooth then~$\Sigma$ is smooth as well,    
and the morphism contracts the four $3$-secant lines of~$\Gamma$.
\end{remark} 

\begin{corollary}
\label{cor:critical-ogr-q3-sl}
If~$\Gamma \subset X_+ = Q^3$ is a smooth curve with~$\deg(\Gamma) = 8$ and~$\g(\Gamma) = 4$ satisfying~\eqref{eq:nondegeneracy},
there is a Sarkisov link~\eqref{eq:sl-factorial}, where
\begin{itemize}[wide]
\item 
$X$ is a linear section of the Mukai variety~$\rM_8$ with a single node or cusp~$x_0 \in X$,
\item 
$\bar{X} \subset \OGr_+(5,10)$ is a linear section of codimension~$7$ with terminal singularities
containing an irreducible quadric surface~\mbox{$\Sigma \subset \bar{X}$},
\item 
$\pi$ and~$\pi_+$ are the blowups of~$x_0$ and~$\Gamma$, and
\item 
$\xi$ and~$\xi_+$ are the blowups of the Weil divisor classes~$-\Sigma$ and~$\Sigma$, respectively.
\end{itemize}
Moreover, any factorial Fano threefold~$X$ with a single node or cusp, $\uprho(X) = 1$, $\io(X) = 1$, and~$\g(X) = 8$
is a dimensionally transverse linear section of the Mukai variety~$\rM_8$.
\end{corollary}

\begin{proof}
Assume~$\Gamma \subset X_+ = Q^3$ is a smooth curve, $\deg(\Gamma) = 8$, $\g(\Gamma) = 4$, and~\eqref{eq:nondegeneracy} holds.
The proof of Corollary~\ref{cor:critical-ogr-q3} shows that~$\cO_\Gamma$ has a resolution of the form~\eqref{eq:res-g84}.

On the other hand, recall from~\cite[Theorem~2.2 and~\S2.3]{KP21} that there is a Sarkisov link~\eqref{eq:sl-rm},
where~$\rM_8^+ = Q^4$, $\widehat{\rM}_8^+ = \PP_{Q^4}(\cE)$ with~$\cE = \cO(-1) \oplus \cS_2^{\oplus 2}$ and relations
\begin{equation*}
H_+ \sim H - 2E
\qquad\text{and}\qquad 
\bar{H} \sim H - E
\qquad\text{in~$\Pic(\Bl_{x_0}(\rM_8)) \cong \Pic(\PP_{Q^4}(\cE))$,}
\end{equation*}
where~$H_+$ and~$\bar{H}$ denote the hyperplane class of~$Q^4$ 
and the relative hyperplane class of~$\PP_{Q^4}(\cE)$, respectively. 
Note that the restriction of~$\cE^\vee(-3)$ to a hyperplane section~$Q^3 \subset Q^4$ 
is isomorphic to the second term of~\eqref{eq:res-g84}, 
and the first map in~\eqref{eq:res-g84} is the twisted dual of a morphism~$\varphi \colon \cE\vert_{Q^3} \to \cO^{\oplus 4}$.
Now, applying~\cite[Lemma~2.1]{K16}, we see that~$\Bl_\Gamma(Q^3)$ 
can be identified with a linear section of~$\PP_{Q^3}(\cE\vert_{Q^3})$ of codimension~$4$, 
hence it is an intersection in~$\widehat{\rM}_8^+ = \PP_{Q^4}(\cE)$ of one divisor 
in the linear system~$|H_+| = |H-2E|$ and four divisors in~$|\bar{H}| = |H-E|$.
The diagram~\eqref{eq:sl-rm} then shows that there is a flop from~$\Bl_\Gamma(Q^3)$ 
onto the blowup~$\Bl_{x_0}(X)$ of a linear section~$X \subset \rM_8$ of codimension~$5$ which is singular at~$x_0$
and we obtain a Sarkisov link that has the form~\eqref{eq:sl-factorial}.

Since the midpoint~$\bar{X}$ of the link is obtained from the anticanonical contraction of~$\Bl_\Gamma(Q^3)$,
Corollary~\ref{cor:critical-ogr-q3} shows that~$\bar{X}$ is a linear section of~$\OGr_+(5,10)$ of codimension~$7$ with terminal singularities
containing an irreducible quadric surface~$\Sigma$ and~$\Bl_\Gamma(Q^3) \cong \Blw{\Sigma}(\bar{X})$.
Then Corollary~\ref{cor:bl-weil-flop} shows that~$\Bl_{x_0}(X) \cong \Blw{-\Sigma}(\bar{X})$.
Finally, the argument of Theorem~\ref{thm:intro-nf-ci} identifies the strict transform of~$\Sigma$
with the exceptional divisor of~$\Bl_{x_0}(X)$, 
and we conclude that~$X$ is 1-nodal or 1-cuspidal.
In particular, $X$ is a Fano threefold with~$\uprho(X) = 1$, $\io(X) = 1$, and~$\g(X) = 8$,
and it is factorial by Corollary~\ref{cor:ci-factorial} and Remark~\ref{rem:blowup-nc}.
This proves the first part of the corollary.

The proof of the second part is analogous to that of Corollary~\ref{cor:critical-gr25-p3-sl}.
\end{proof} 

\subsection{Genus 9}

All maximal quadrics in~$\Gr(2,6)$ have the form~$Q^4 = \Gr(2,4) \subset \Gr(2,6)$.
Moreover, the group~$\PGL(6)$ acts transitively on the Hilbert scheme of 4-dimensional quadrics in~$\Gr(2,6)$,
so from now on we fix one such quadric~$Q^4 \subset \Gr(2,6)$.

\begin{proposition}
\label{prop:gr26-p8}
\label{prop:bl-q-rm8}
There is a canonical isomorphism
\begin{equation}
\label{eq:gr26-p8}
\Bl_{Q^4}(\Gr(2,6)) \cong \Bl_{\PP^1 \times \PP^3}(\PP^8).
\end{equation}
If~$\bar{H}$ and~$\bar{E}$ denote the hyperplane class of~$\Gr(2,6)$ and the class of the exceptional divisor of the left side,
while~$H_+$ and~$E_+$ are the hyperplane class of~$\PP^8$ and the class of the exceptional divisor on the right side of~\eqref{eq:gr26-p8}
then
\begin{equation}
\label{eq:hp-g7-b}
H_+ = \bar{H} - \bar{E},
\qquad
E_+ = \bar{H} - 2\bar{E}.
\end{equation}
\end{proposition}

\begin{proof}
This is completely analogous to the isomorphism 
\begin{equation*} 
\Bl_{\Gr(2,3)}(\Gr(2,5)) \cong \Bl_{\PP^1 \times \PP^2}(\PP^6) 
\end{equation*}
constructed in~\cite[Lemma~5.2]{KP18}, we leave details to the interested reader.
\end{proof} 

\begin{corollary}
\label{cor:critical-gr26-p8}
There is a bijection between the sets of all
\begin{itemize}[wide]
\item 
$3$-dimensional linear sections~$\bar{X} \subset \Gr(2,6)$ with isolated hypersurface singularities
such that~$\Sigma \coloneqq \bar{X} \cap Q^4$ is an irreducible quadric surface and~$\Blw{\Sigma}(\bar{X})$ is smooth, and
\item 
pairs~$(Y_4,\Gamma)$, where~$Y_4$ is a smooth quartic del Pezzo threefold, 
\mbox{$\Gamma \subset Y_4$} is a smooth connected curve such that~$\deg(\Gamma) = 4$, and~$\g(\Gamma) = 0$, 
and~\eqref{eq:nondegeneracy} holds.
\end{itemize}
If~$\bar{X}$ corresponds to~$(Y_4,\Gamma)$, there is an isomorphism
\begin{equation}
\label{eq:x8-g}
\Blw{\Sigma}(\bar{X}) \cong \Bl_\Gamma(Y_4)
\end{equation} 
and the morphism~$\Bl_\Gamma(Y_4) \to \bar{X}$ is the anticanonical contraction.
\end{corollary}

\begin{proof}
The argument is similar to that of Corollaries~\ref{cor:critical-gr25-p3} and~\ref{cor:critical-ogr-q3}.

This time~$\bar{X} \subset \Gr(2,6)$ is a complete intersection of five hyperplanes
while~$\Sigma \subset Q^4$ is a complete intersection of two hyperplanes, 
hence we may assume that three hyperplanes contain~$Q^4$, 
while the other two are dimensionally transverse to it.
Consider the intersection~$\bar{X}' \subset \Bl_{Q^4}(\Gr(2,6))$
of the corresponding three divisors of class~$\bar{H} - \bar{E}$ and two divisors of class~$\bar{H}$.
As in Corollary~\ref{cor:critical-gr25-p3} we conclude that~$\bar{X}'$ coincides with the strict transform of~$\bar{X}$ and
\begin{equation*}
\bar{X}' \cong \Blw{\Sigma}(\bar{X}) \subset \Bl_{Q^4}(\Gr(2,6)) \cong \Bl_{\PP^1 \times \PP^3}(\PP^8).
\end{equation*}
In other words, $\Blw{\Sigma}(\bar{X})$ is an irreducible complete intersection, 
and since~$\bar{H} - \bar{E} \sim H_+$ and~$\bar{H} \sim 2H_+ - E_+$ by~\eqref{eq:hp-g7-b},
the first three divisors define a~$\PP^5 \subset \PP^8$ such that
\begin{equation}
\label{eq:g40}
\Gamma \coloneqq \PP^5 \cap (\PP^1 \times \PP^3)
\end{equation}
is a curve of degree~$4$ and genus~$0$ contained in a unique hyperplane
and the other two divisors correspond to a pair of quadrics containing~$\PP^1 \times \PP^3$.
Thus, the intersection of all divisors is a quartic del Pezzo threefold~$Y_4 \subset \PP^5$
containing the curve~$\Gamma$ which satisfies~\eqref{eq:nondegeneracy}
and we have an isomorphism~$\Blw{\Sigma}(\bar{X}) \cong \Bl_\Gamma(Y_4)$.
Since~$\Bl_{\Sigma}(\bar{X})$ is smooth, we conclude that~$Y_4$ and~$\Gamma$ are both smooth.

Conversely, let~$\Gamma \subset Y_4 \subset \PP^5$ 
be a smooth quartic del Pezzo threefold 
and a smooth connected curve such that~$\deg(\Gamma) = 4$, $\g(\Gamma) = 0$, and~\eqref{eq:nondegeneracy} holds.
Consider an embedding~$\Gamma \hookrightarrow \PP^1 \times \PP^3$ as a curve of bidegree~$(1,3)$.
This defines an embedding of the linear span~$\langle \Gamma \rangle = \PP^4$ into the Segre space~$\PP^7$
such that~$\Gamma = \PP^4 \cap (\PP^1 \times \PP^3)$, 
and this embedding can be extended to an embedding~$\PP^5 \to \PP^8$ such that~\eqref{eq:g40} holds.
Furthermore, two quadratic equations of~$Y_4 \subset \PP^5$ 
can be extended to two quadrics on~$\PP^8$ containing~$\PP^1 \times \PP^3$,
and together with the three linear equations of~$\PP^5 \subset \PP^8$, 
they define a complete intersection in~\mbox{$\Bl_{\PP^1 \times \PP^3}(\PP^8)\cong \Bl_{Q^4}(\Gr(2,6))$} 
of three divisors of class~$H_+ = \bar{H} - \bar{E}$ and two divisors of class~$2H_+ - E_+ = \bar{H}$.
Therefore the morphism~$\Bl_\Gamma(Y_4) \to \Gr(2,6)$ is anticanonical, hence small by Lemma~\ref{lem:blg-xplus},
its image~$\bar{X}$ is a linear section of codimension~$5$ 
containing a quadric surface~$\Sigma$ 
dominated by the strict transform of the unique (by~\eqref{eq:nondegeneracy}) 
divisor~$S_\Gamma \subset Y_4$ in~$|H_+ - \Gamma|$,
and hence irreducible.
These two constructions are mutually inverse and define the required bijection.
\end{proof}

\begin{remark}
The above proof shows that there is a contraction~$\Bl_\Gamma(S_\Gamma) \to \Sigma$, given by the class~$2H_+ - E_+$
and if~$S_\Gamma$ is smooth then~$\Sigma$ is smooth as well,
and the morphism contracts four $2$-secant lines of~$\Gamma$.
\end{remark}

\begin{corollary}
\label{cor:critical-lgr-y4-sl}
If~$\Gamma \subset X_+$ is a smooth curve on a smooth quartic del Pezzo threefold
with~$\deg(\Gamma) = 4$ and~$\g(\Gamma) = 0$ satisfying~\eqref{eq:nondegeneracy},
there is a Sarkisov link~\eqref{eq:sl-factorial}, where
\begin{itemize}[wide]
\item 
$X$ is a linear section of the Mukai variety~$\rM_9$ with a single node or cusp~$x_0 \in X$,
\item 
$\bar{X} \subset \Gr(2,6)$ is a linear section of codimension~$5$ with terminal singularities
containing an irreducible quadric surface~\mbox{$\Sigma \subset \bar{X}$},
\item 
$\pi$ and~$\pi_+$ are the blowups of~$x_0$ and~$\Gamma$, and
\item 
$\xi$ and~$\xi_+$ are the blowups of the Weil divisor classes~$-\Sigma$ and~$\Sigma$, respectively.
\end{itemize}
Moreover, any factorial Fano threefold~$X$ with a single node or cusp, $\uprho(X) = 1$, $\io(X) = 1$, and~$\g(X) = 9$
is a dimensionally transverse linear section of the Mukai variety~$\rM_9$.
\end{corollary}

\begin{proof}
Assume~$\Gamma \subset X_+ = Y_4 \subset \PP^5$ is a smooth curve on a smooth quartic del Pezzo threefold 
such that~$\deg(\Gamma) = 4$, $\g(\Gamma) = 0$, and~\eqref{eq:nondegeneracy} holds.
Clearly, $\Gamma$ can be realized as a hyperplane section of the Veronese surface~$\mathrm{v}_2(\PP^2) \subset \PP^5 \subset \PP^6$,
and the two equations of~$Y_4$ can be extended to two quadrics in~$\PP^6$ containing~$\mathrm{v}_2(\PP^2)$.

On the other hand, recall from~\cite[Theorem~4.4]{KP21} that there is a Sarkisov link~\eqref{eq:sl-rm},
where~$\rM_9^+ = \PP^6$, $\widehat{\rM}_9^+ = \Bl_{\mathrm{v}_2(\PP^2)}(\PP^6)$, and relations
\begin{equation*}
H_+ \sim H - 2E
\qquad\text{and}\qquad 
E_+ \sim H - 3E
\qquad\text{in~$\Pic(\Bl_{x_0}(\rM_9)) \cong \Pic(\Bl_{\mathrm{v}_2(\PP^2)}(\PP^6))$,}
\end{equation*}
where~$H_+$ and~$E_+$ denote the hyperplane class of~$\PP^6$ 
and the exceptional divisor of~$\Bl_{\mathrm{v}_2(\PP^2)}(\PP^6)$, respectively. 
We see that~$\Bl_\Gamma(Y_4)$ is an intersection in~$\Bl_{\mathrm{v}_2(\PP^2)}(\PP^6)$ 
of one divisor in the linear system~$|H_+| = |H-2E|$ and two divisors in~$|2H_+ - E_+| = |H-E|$.
The diagram~\eqref{eq:sl-rm} then shows that there is a flop from~$\Bl_\Gamma(Y_4)$ 
onto the blowup~$\Bl_{x_0}(X)$ of a linear section~$X \subset \rM_9$ of codimension~$3$ which is singular at~$x_0$
and we obtain a Sarkisov link that has the form~\eqref{eq:sl-factorial}.

Since the midpoint~$\bar{X}$ of the link is obtained from the anticanonical contraction of~$\Bl_\Gamma(Y_4)$,
Corollary~\ref{cor:critical-gr26-p8} shows that~$\bar{X}$ is a linear section of~$\Gr(2,6)$ of codimension~$5$ with terminal singularities
containing an irreducible quadric surface~$\Sigma$ and~$\Bl_\Gamma(Y_4) \cong \Blw{\Sigma}(\bar{X})$.
Then Corollary~\ref{cor:bl-weil-flop} shows that~$\Bl_{x_0}(X) \cong \Blw{-\Sigma}(\bar{X})$.
Finally, the argument of Theorem~\ref{thm:intro-nf-ci} identifies the strict transform of~$\Sigma$
with the exceptional divisor of~$\Bl_{x_0}(X)$, 
and we conclude that~$X$ is 1-nodal or 1-cuspidal.
In particular, $X$ is a Fano threefold with~$\uprho(X) = 1$, $\io(X) = 1$, and~$\g(X) = 9$,
and it is factorial by Corollary~\ref{cor:ci-factorial} and Remark~\ref{rem:blowup-nc}.
This proves the first part of the corollary.

The proof of the second part is analogous to that of Corollary~\ref{cor:critical-gr25-p3-sl}.
\end{proof} 

\subsection{Genus~10}

Maximal quadrics in~$\LGr(3,6)$ have the form~$Q^3 = \LGr(2,4) \subset \LGr(3,6)$.
Moreover, the group~$\Sp(6)$ acts transitively on the Hilbert scheme of 3-dimensional quadrics in~$\LGr(3,6)$,
so from now on we fix one such quadric~$Q^3 \subset \LGr(3,6)$.

Recall that a smooth quintic del Pezzo fivefold~$Y^5_5$ is a smooth hyperplane section of the Grassmannian~$\Gr(2,5)$;
all such fivefolds are isomorphic to each other.
We denote by~$\cU_2$ and~$\cU_2^\perp$ the restrictions to~$Y^5_5$ 
of the tautological subbundles of rank~$2$ and~$3$ on~$\Gr(2,5)$.

\begin{proposition}
\label{prop:bl-q-rm9}
There is a canonical isomorphism
\begin{equation}
\label{eq:lgr-y5}
\Bl_{Q^3}(\LGr(3,6)) \cong \PP_{Y^5_5}(\cE),
\end{equation}
where~$Y^5_5$ is a smooth quintic del Pezzo fivefold
and~$\cE$ is the vector bundle of rank~$2$ on~$Y^5_5$ defined by the exact sequence
\begin{equation}
\label{eq:ce-y55}
0 \longrightarrow \cU_2(-1) \longrightarrow \cO(-1) \oplus \cU_2^\perp(-1) 
\longrightarrow \cE \longrightarrow 0.
\end{equation}
If~$\bar{H}$ and~$\bar{E}$ denote the hyperplane class of~$\LGr(3,6)$ 
and the class of the exceptional divisor in~\eqref{eq:lgr-y5}
then the pullback of the hyperplane class of~$Y_5^5$ is
\begin{equation}
\label{eq:hp-g9}
H_+ = \bar{H} - \bar{E}
\end{equation}
and the relative hyperplane class for~$\PP_{Y^5_5}(\cE)$ coincides with~$\bar{H}$.
\end{proposition}

\begin{remark}
\label{rem:ce-unique}
It is easy to see that~$\Hom(\cU_2,\cU_2^\perp) = \Bbbk$, 
the unique nontrivial morphism is induced by the equation of~$Y^5_5 \subset \Gr(2,5)$,
and its cokernel is isomorphic to the ideal of the unique~$3$-space~$\Pi \subset Y^5_5$.
Therefore, the definition~\eqref{eq:ce-y55} of~$\cE$ can be rewritten as a non-split (because~$\cE$ is locally free) exact sequence
\begin{equation*}
0 \longrightarrow \cO(-1) \longrightarrow \cE \longrightarrow \cI_\Pi(-1) 
\longrightarrow 0,
\end{equation*}
hence~$\cE$ is a twist of a bundle produced from~$\Pi$ by Serre's construction.
It is also easy to check that~$\Ext^1(\cI_\Pi(-1),\cO_{Y^5_5}(-1)) = \Bbbk$, hence
such a bundle is unique.
\end{remark} 

\begin{proof}
To construct the isomorphism~\eqref{eq:lgr-y5}, consider a hyperplane~$V_5 \subset V_6$ in a 6-dimensional vector space.
Using the two projective bundle structures of the flag variety~$\Fl(2,3;V_6)$ it is easy to obtain an isomorphism
\begin{equation*}
\Bl_{\Gr(3,V_5)}(\Gr(3,V_6)) \cong \PP_{\Gr(2,V_5)}(\cO \oplus V_5/\cU_2),
\end{equation*}
where~$\cU_2$ is the tautological subbundle on~$\Gr(2,V_5)$, and a commutative diagram
\begin{equation*}
\xymatrix{
0 \ar[r] & 
\cU_2 \ar[r] \ar@{=}[d] & 
\cU_3 \ar[r] \ar[d] & 
\cL \ar[r] \ar[d] & 
0
\\
0 \ar[r] & 
\cU_2 \ar[r] & 
\cO \oplus (V_5 \otimes \cO) \ar[r] & 
\cO \oplus V_5/\cU_2 \ar[r] & 
0,
}
\end{equation*}
where~$\cU_3$ is the restriction of the tautological subbundle of~$\Gr(3,V_6)$ 
and~$\cL \hookrightarrow \cO \oplus V_5/\cU_2$ is the tautological line subbundle of~$\PP_{\Gr(2,V_5)}(\cO \oplus V_5/\cU_2)$.

Now choose a symplectic form~$\lambda$ on~$V_6$ and consider the zero locus 
of the corresponding section of the vector bundle~$\wedge^2\cU_3^\vee$.
The dual of the wedge square of the first row in the above diagram has the form
\begin{equation*}
0 \longrightarrow \cU_2^\vee \otimes \cL^\vee \longrightarrow \wedge^2\cU_3^\vee 
\longrightarrow \wedge^2\cU_2^\vee \longrightarrow 0.
\end{equation*}
Therefore, $\lambda$ gives a global section of~$\wedge^2\cU_2^\vee$ on~$\Gr(2,V_5)$
that corresponds to the restriction of the form~$\lambda$ to the subspace~$V_5 \subset V_6$.
This restriction has rank~$4$, hence its zero locus is a smooth hyperplane section~$Y^5_5 \subset \Gr(2,V_5)$.
Furthermore, $\lambda$ also gives a global section 
of the restriction of the vector bundle~$\cU_2^\vee \otimes \cL^\vee$ to~$\PP_{Y^5_5}(\cO \oplus V_5/\cU_2)$.
It corresponds to a morphism~$\cU_2 \to \cO \oplus \cU_2^\perp$ on~$Y^5_5$.
Since~$\lambda$ is nondegenerate, this is a fiberwise monomorphism, hence we have an exact sequence
\begin{equation*}
0 \longrightarrow \cU_2 \longrightarrow \cO \oplus \cU_2^\perp \longrightarrow 
\cE' \longrightarrow 0,
\end{equation*}
where~$\cE'$ is a self-dual vector bundle on~$Y^5_5$ of rank~$2$, 
and summarizing the above we obtain an identification 
of the zero locus of~$\lambda$ on~$\PP_{\Gr(2,V_5)}(\cO \oplus V_5/\cU_2)$ with~$\PP_{Y^5_5}(\cE')$.

On the other hand, the zero locus of~$\lambda$ on~$\Gr(3,V_6)$ is~$\LGr(3,V_6)$.
Moreover,
\begin{equation*}
Q_{V_5} \coloneqq \Gr(3,V_5) \cap \LGr(3,V_6) 
\end{equation*}
is a smooth 3-dimensional quadric; in particular, the intersection is transverse.
Therefore, we obtain an isomorphism
\begin{equation*}
\Bl_{Q_{V_5}}(\LGr(3,V_6)) \cong \PP_{Y^5_5}(\cE').
\end{equation*}
It remains to note that the defining sequence of~$\cE'$ coincides with a twist of~\eqref{eq:ce-y55},
hence~$\cE' \cong \cE(1)$, and so~\eqref{eq:lgr-y5} holds.
\end{proof}

Below we denote by~$Y_5$ the smooth quintic del Pezzo threefold (unique up to isomorphism).

\begin{corollary}
\label{cor:critical-lgr-y5}
There is a bijection between the sets of all
\begin{itemize}[wide]
\item 
$3$-dimensional linear sections~$\bar{X} \subset \LGr(3,6)$ with isolated hypersurface singularities
such that~$\Sigma \coloneqq \bar{X} \cap Q^3$ is an irreducible quadric surface and~$\Blw{\Sigma}(\bar{X})$ is smooth, and
\item 
smooth connected curves~$\Gamma \subset Y_5$ such that~$\deg(\Gamma) = 6$, $\g(\Gamma) = 1$, 
and~\eqref{eq:nondegeneracy} holds.
\end{itemize}
If~$\bar{X}$ corresponds to~$\Gamma$, there is an isomorphism
\begin{equation}
\label{eq:x9-g}
\Blw{\Sigma}(\bar{X}) \cong \Bl_\Gamma(Y_5)
\end{equation} 
and the morphism~$\Bl_\Gamma(Y_5) \to \bar{X}$ is the anticanonical contraction.
\end{corollary} 

\begin{proof}
The argument is similar to that of Corollaries~\ref{cor:critical-gr25-p3}, \ref{cor:critical-ogr-q3}, and~\ref{cor:critical-gr26-p8}.

This time~$\bar{X} \subset \LGr(3,6)$ is a complete intersection of three hyperplanes
while~$\Sigma \subset Q^3$ is a hyperplane section,
hence we may assume that two hyperplanes contain~$Q^3$, 
while the last one is dimensionally transverse to it.
Consider the intersection~$\bar{X}' \subset \Bl_{Q^3}(\LGr(3,6))$
of the corresponding two divisors of class~$\bar{H} - \bar{E}$ and a divisor of class~$\bar{H}$.
As in Corollary~\ref{cor:critical-gr25-p3} we conclude that~$\bar{X}'$ coincides with the strict transform of~$\bar{X}$ and
\begin{equation*}
\bar{X}' \cong \Blw{\Sigma}(\bar{X}) \subset \Bl_{Q^3}(\LGr(3,6)) \cong \PP_{Y^5_5}(\cE).
\end{equation*}
In other words, $\Blw{\Sigma}(\bar{X})$ is an irreducible complete intersection, 
and since~$\bar{H} - \bar{E} \sim H_+$ by~\eqref{eq:hp-g9},
the first two divisors define a linear section~$Y_5 \subset Y_5^5$, 
and the last divisor corresponds to a morphism
\begin{equation*}
\psi \colon \cE\vert_{Y_5} \longrightarrow \cO,
\end{equation*}
Now, if~$\Gamma \subset Y_5$ is the zero locus of~$\psi$
then~$\dim(\Gamma) \le 1$ (because~$\Blw{\Sigma}(\bar{X})$ is irreducible)
and applying~\cite[Lemma~2.1]{K16} we obtain an isomorphism~$\Blw{\Sigma}(\bar{X}) \cong \Bl_\Gamma(Y_5)$,
and since the left side is smooth, we conclude that the del Pezzo threefold~$Y_5$ and the curve~$\Gamma$ are smooth.
Finally, computing the determinants we see that~$\Ker(\psi) \cong \cO(-2)$ and deduce an exact sequence
\begin{equation}
\label{eq:res-g61}
0 \longrightarrow 
\cO(-2) \xrightarrow{\ \psi^\vee\ } 
\cE\vert_{Y_5} \xrightarrow{\quad} 
\cO \xrightarrow{\quad} 
\cO_\Gamma \longrightarrow 0,
\end{equation} 
and as before it follows that~$\Gamma$ is connected, $\deg(\Gamma) = 6$, $\g(\Gamma) = 1$, and~\eqref{eq:nondegeneracy} holds.

Conversely, let~$\Gamma \subset Y_5$ be a smooth connected curve 
such that~$\deg(\Gamma) = 6$, $\g(\Gamma) = 1$, and~\eqref{eq:nondegeneracy} holds.
We check that its structure sheaf has a resolution 
\begin{equation}
\label{eq:res-g61-2}
0 \longrightarrow \cO(-2) \oplus \cU_2(-1) \longrightarrow \cO(-1) \oplus 
\cU_2^\perp(-1) \longrightarrow \cO \longrightarrow \cO_\Gamma \longrightarrow
0,
\end{equation}
and that this resolution is equivalent to~\eqref{eq:res-g61}.
For this we decompose~$\cI_\Gamma$ with respect to the exceptional collection~$(\cO(-2),\cU_2(-1),\cU_2^\perp(-1),\cO(-1))$.
To show that the decomposition has the form~\eqref{eq:res-g61-2}, it is enough to check the following
\begin{equation*}
H^\bullet(\cI_\Gamma) = \Bbbk[-2],
\qquad 
H^\bullet(\cU_2^\vee \otimes \cI_\Gamma) = \Bbbk[-1],
\qquad 
H^\bullet(\cI_\Gamma(1)) = \Bbbk.
\end{equation*}
The first is obvious and~\eqref{eq:nondegeneracy} implies that for the last it is enough to check that
\begin{equation*}
\dim H^1(\Gamma, \cO_\Gamma(1)) = 0,
\end{equation*}
which is also obvious.
Similarly, for the second equality it is enough to check that
\begin{equation*}
H^0(Y_5, \cU_2^\vee \otimes \cI_\Gamma) = 0
\qquad\text{and}\qquad 
H^1(Y_5, \cU_2^\vee \otimes \cO_\Gamma) = 0.
\end{equation*}
To prove these vanishings we use the exact sequences
\begin{eqnarray*}
& 0 \longrightarrow \cU_2^\vee \otimes \cI_\Gamma \longrightarrow \cU_2^\vee 
\longrightarrow \cU_2^\vee\vert_\Gamma \longrightarrow 0,
\\
& 0 \longrightarrow \cU_2 \otimes \cO_\Gamma \longrightarrow V_5 \otimes 
\cO_\Gamma \longrightarrow (V_5/\cU_2)\vert_\Gamma \longrightarrow 0.
\end{eqnarray*}
We see that if~$H^0(Y_5, \cU_2^\vee \otimes \cI_\Gamma)$ or~$H^0(Y_5, \cU_2 \otimes \cO_\Gamma)$ is nonzero
then~$\Gamma$ is contained in the zero locus of a section of~$\cU_2^\vee$ or of~$V_5/\cU_2$ on~$Y_5$.
But the former zero locus is a conic and the latter is a point or a line, hence neither can contain~$\Gamma$.
Thus, both these spaces are zero, and Serre duality applied to the second 
implies that~$H^1(Y_5, \cU_2^\vee \otimes \cO_\Gamma)$ also vanishes.

Thus, we obtain the required vanishings, and hence resolution~\eqref{eq:res-g61-2}. 
The argument of Remark~\ref{rem:ce-unique} shows that the cokernel of~$\cU_2(-1) \to \cO(-1) \oplus \cU_2^\perp(-1)$ 
is an extension of~$\cI_L(-1)$ by~$\cO(-1)$, where~$L$ is a line on~$Y_5$.
This extension surjects on~$\cI_\Gamma$, hence it is non-split, 
hence it is isomorphic to the restriction of the bundle~$\cE$ from an appropriate quintic del Pezzo fivefold,
and we obtain~\eqref{eq:res-g61}.

Applying~\cite[Lemma~2.1]{K16} to~\eqref{eq:res-g61}
we conclude that~$\Bl_\Gamma(Y_5)$ is a complete intersection in~$\Bl_{Q^3}(\LGr(3,6))$
of two divisors of class~$\bar{H} - \bar{E}$ and one divisor of class~$\bar{H}$
and there is a linear equivalence~$\bar{H} \sim 2H_+ - E_+$, where~$E_+$ is the exceptional divisor of~$\Bl_\Gamma(Y_5)$.
Therefore, the morphism~$\Bl_\Gamma(Y_5) \to \LGr(3,6)$ is anticanonical, hence small by Lemma~\ref{lem:blg-xplus},
its image~$\bar{X}$ is a linear section of codimension~$3$ 
containing a quadric surface~$\Sigma$ 
dominated by the strict transform of the unique (by~\eqref{eq:nondegeneracy}) 
divisor~$S_\Gamma \subset Y_5$ in~$|H_+ - \Gamma|$,
and hence irreducible.
These two constructions are mutually inverse and define the required bijection.
\end{proof}

\begin{remark}
The above proof shows that 
there is a contraction~$\Bl_\Gamma(S_\Gamma) \to \Sigma$, given by the class~$2H_+ - E_+$,
and if~$S_\Gamma$ is smooth
then~$\Sigma$ is smooth as well,
and the morphism contracts the three $2$-secant lines of~$\Gamma$.
\end{remark}

\begin{corollary}
\label{cor:critical-lgr-y5-sl}
If~$\Gamma \subset X_+$ is a smooth curve on a smooth quintic del Pezzo threefold
with~$\deg(\Gamma) = 6$ and~$\g(\Gamma) = 1$ satisfying~\eqref{eq:nondegeneracy},
there is a Sarkisov link~\eqref{eq:sl-factorial}, where
\begin{itemize}[wide]
\item 
$X$ is a linear section of the Mukai variety~$\rM_{10}$ with a single node or cusp~$x_0 \in X$,
\item 
$\bar{X} \subset \LGr(3,6)$ is a linear section of codimension~$3$ with terminal singularities
containing an irreducible quadric surface~\mbox{$\Sigma \subset \bar{X}$},
\item 
$\pi$ and~$\pi_+$ are the blowups of~$x_0$ and~$\Gamma$, and
\item 
$\xi$ and~$\xi_+$ are the blowups of the Weil divisor classes~$-\Sigma$ and~$\Sigma$, respectively.
\end{itemize}
Moreover, any factorial Fano threefold~$X$ with a single node or cusp, $\uprho(X) = 1$, $\io(X) = 1$, and~$\g(X) = 10$
is a dimensionally transverse linear section of the Mukai variety~$\rM_{10}$.
\end{corollary}

\begin{proof}
Assume~$\Gamma \subset X_+ = Y_5$ is a smooth curve on a smooth quintic del Pezzo threefold 
such that~$\deg(\Gamma) = 6$, $\g(\Gamma) = 1$, and~\eqref{eq:nondegeneracy} holds.
The proof of Corollary~\ref{cor:critical-lgr-y5} shows that~$\cO_\Gamma$ has a resolution of the form~\eqref{eq:res-g61},
where recall from Remark~\ref{rem:ce-unique} that~$\cE$ is the vector bundle 
on a quintic del Pezzo fivefold~$Y_5^5$ containing~$Y_5$
obtained by Serre's construction from the unique 3-space~$\Pi \subset Y^5_5$.
Note that~$\Pi \cap Y_5$ is a line, hence~$\cE\vert_{Y_5}$ is the vector bundle obtained by Serre's construction from a line.
Obviously, each line on~$Y_5$ 
can be represented as a hyperplane section of a plane~$\Pi_0 = \Gr(2,3) \subset \Gr(2,5)$.
Moreover, two (out of three) linear equations of~$Y_5 \subset \Gr(2,5)$ vanish on~$\Pi_0$, 
hence define a smooth quintic del Pezzo fourfold~$Y_5^4 \subset \Gr(2,5)$ containing~$\Pi_0$,
and it follows that~$\cE\vert_{Y_5} \cong \cE_0(-1)\vert_{Y_5}$, 
where~$\cE_0$ is the vector bundle on~$Y_5^4$ obtained by Serre's construction from~$\Pi_0$.
Therefore, $\cO_\Gamma$ also has a resolution 
\begin{equation}
\label{eq:resolution-g84-b}
0 \longrightarrow 
\cO(-2) \xrightarrow{\ \varphi^\vee\ } 
\cE_0(-1)\vert_{Y_5} \xrightarrow{\quad} 
\cO \xrightarrow{\quad} 
\cO_\Gamma \longrightarrow 0.
\end{equation}

On the other hand, recall from~\cite[Theorem~2.2 and~\S2.5]{KP21} that there is a Sarkisov link~\eqref{eq:sl-rm},
where~$\rM_{10}^+ = Y_5^4$, $\widehat{\rM}_{10}^+ = \PP_{Q^4}(\cE_0(-1))$,
and relations
\begin{equation*}
H_+ \sim H - 2E
\qquad\text{and}\qquad 
\bar{H} \sim H - E
\qquad\text{in~$\Pic(\Bl_{x_0}(\rM_{10})) \cong \Pic(\PP_{Y_5^4}(\cE_0(-1)))$,}
\end{equation*}
where~$H_+$ and~$\bar{H}$ denote the hyperplane class of~$Y_5^4$ 
and the relative hyperplane class of~$\PP_{Y_5^4}(\cE_0(-1))$, respectively. 
Note that~$\cE_0^\vee \cong \cE_0$
and the first map in~\eqref{eq:resolution-g84-b} 
is the twisted dual of a morphism~$\varphi \colon \cE_0(-1)\vert_{Y_5} \to \cO$.
Now, applying~\cite[Lemma~2.1]{K16}, we see that~$\Bl_\Gamma(Y_5)$ 
can be identified with a relative hyperplane section of~$\PP_{Y_5}(\cE_0(-1)\vert_{Y_5})$,
hence it is an intersection in~$\widehat{\rM}_{10}^+ = \PP_{Y_5^4}(\cE_0(-1))$ of one divisor 
in the linear system~$|H_+| = |H-2E|$ and one divisor in~$|\bar{H}| = |H-E|$.
The diagram~\eqref{eq:sl-rm} then shows that there is a flop from~$\Bl_\Gamma(Y_5)$ 
onto the blowup~$\Bl_{x_0}(X)$ of a linear section~$X \subset \rM_{10}$ of codimension~$2$ which is singular at~$x_0$
and we obtain a Sarkisov link that has the form~\eqref{eq:sl-factorial}.

Since the midpoint~$\bar{X}$ of the link is obtained from the anticanonical contraction of~$\Bl_\Gamma(Y_5)$,
Corollary~\ref{cor:critical-lgr-y5} shows that~$\bar{X}$ is a linear section of~$\LGr(3,6)$ of codimension~$3$ with terminal singularities
containing an irreducible quadric surface~$\Sigma$ and~$\Bl_\Gamma(Y_5) \cong \Blw{\Sigma}(\bar{X})$.
Then Corollary~\ref{cor:bl-weil-flop} shows that~$\Bl_{x_0}(X) \cong \Blw{-\Sigma}(\bar{X})$.
Finally, the argument of Theorem~\ref{thm:intro-nf-ci} identifies the strict transform of~$\Sigma$
with the exceptional divisor of~$\Bl_{x_0}(X)$, 
and we conclude that~$X$ is 1-nodal or 1-cuspidal.
In particular, $X$ is a Fano threefold with~$\uprho(X) = 1$, $\io(X) = 1$, and~$\g(X) = 10$,
and it is factorial by Corollary~\ref{cor:ci-factorial} and Remark~\ref{rem:blowup-nc}.
This proves the first part of the corollary.

The proof of the second part is analogous to that of Corollary~\ref{cor:critical-gr25-p3-sl}.
\end{proof}

\section{Nonfactorial 1-nodal Fano threefolds in the literature}
\label{app:table}

As we mentioned in the Introduction, 
the nonfactorial 1-nodal Fano threefolds and the corresponding Sarkisov links~\eqref{eq:intro-sl}
already appeared in the literature.
For the readers' convenience we gather some relevant references in the next table:

\begin{table}[H]
\begin{tabular}{l|l|l|l}
\multirow{2}{*}{Type  }&\multicolumn{3}{c}{References}
\\\hhline{~---}
 & \cite{Jahnke-Peternell-Radloff-II}& \cite{Takeuchi:DP}& \multicolumn{1}{c}{Other references}
\\
\noalign{\hrule height 0.1em}
{\ntype{1}{12}{na}} &&&
\cite[Prop.~2.9]{Cutrone-Marshburn},
\cite[Ex.~4.8(iii)]{BL},
\cite[Type~I]{P:V22}
\\
{\ntype{1}{12}{nb}} &
7.7.1&&
\cite[Type~II]{P:V22}
\\
{\ntype{1}{12}{nc}} &
7.4.1&&
\cite[Type~III]{P:V22}
\\
{\ntype{1}{12}{nd}} &&2.13.1&
\cite[Example~2.13]{Yasu},
\cite[Type~IV]{P:V22}
\\
{\ntype{1}{10}{na}} &
7.7.2&&\cite[Ex.~4.8(ii)]{BL}
\\
{\ntype{1}{10}{nb}} &
7.4.2&
2.13.3
\\
\hphantom{$\mathbf{0}$}%
{\ntype{1}{9}{na}} &
7.4.5&
2.13.4
\\
\hphantom{$\mathbf{0}$}%
{\ntype{1}{9}{nb}} &
7.1.4&
2.3.8
\\
\hphantom{$\mathbf{0}$}%
{\ntype{1}{8}{na}} &
7.4.7&
2.11.4
\\
\hphantom{$\mathbf{0}$}%
{\ntype{1}{8}{nb}} &&&
\cite[\S3.4]{P:ratFano2:22}
\\
\hphantom{$\mathbf{0}$}%
{\ntype{1}{7}{n}} &
7.4.8&
2.11.5&\cite[Ex.~4.8(i)]{BL}
\\
\hphantom{$\mathbf{0}$}%
{\ntype{1}{6}{n}} &
7.4.10&
2.9.3&\cite[Example~4.7]{P:ratF-1}
\\
\hphantom{$\mathbf{0}$}%
{\ntype{1}{5}{n}} &
7.2.5&
2.9.4&\cite[Example~4.6]{P:ratF-1}
\\
\hphantom{$\mathbf{0}$}%
{\ntype{1}{4}{n}} &
Prop.~2.7(3)& 2.7.3 &
\cite[\S2]{Jahnke2011},
\cite[Example~4.3]{P:ratF-1}
\\
\hphantom{$\mathbf{0}$}%
{\ntype{1}{2}{n}} &
7.1.17&
2.5.2
\end{tabular}
\end{table}

\end{document}